%% file: main_revision.tex
\documentclass[11pt]{article}

\usepackage{fullpage}

\usepackage{natbib}
\usepackage{geometry}

\usepackage{psfrag,amsmath,amssymb,enumerate,color}
\usepackage[ruled,vlined]{algorithm2e}
\usepackage{enumerate}
\usepackage{mathtools}
\usepackage{subfigure,float,hyperref,cleveref}
\usepackage{multirow}
\usepackage{makecell}
\usepackage{ulem}
\usepackage{algorithmic}
\usepackage{graphicx}
\usepackage{bbm}

\usepackage{caption}
\usepackage{booktabs}
\usepackage{tabularx} 
\newcolumntype{Y}{>{\centering\arraybackslash}X}

\providecommand{\U}[1]{\protect\rule{.1in}{.1in}}
\providecommand{\U}[1]{\protect\rule{.1in}{.1in}}
\newtheorem{thm}{Theorem}
\newtheorem{proposition}{Proposition}
\newtheorem{lemma}{Lemma}
\newtheorem{corollary}{Corollary}
\newtheorem{assumption}{Assumption}
\newtheorem{remark}{Remark}

\newtheorem{claim}{Claim}

\newcommand{\EE}{\mathbb{E}}

\newtheorem{proof}{Proof}

\input{math_command}

\def\EB{{\mathbb E}}
\def\RB{{\mathbb R}}
\def\ttheta{{\widetilde{\vartheta}}}
\def\De{{\widetilde{\Delta}}}
\def\sr{{r}}
\def\su{{u}}
\def\sw{{\omega}}
\def\snu{{\nu}}
\def\ph{{\phi}}
\def\bG{{ \overline{G} }}
\def\vtheta{{\vartheta^{\star}}}
\def\FM{{\mathcal F}}
\def\tconst{{\widetilde{C}}}
\def\Bpsi{{\widetilde{\psi}}}
\newcommand{\vertiii}[1]{{\left\vert\kern-0.25ex\left\vert\kern-0.25ex\left\vert #1 \right\vert\kern-0.25ex\right\vert\kern-0.25ex\right\vert}}
\newcommand{\myinner}[2]{\langle #1, #2 \rangle}
\allowdisplaybreaks[4]

\begin{document}

\title{Convergence and Inference of Stream SGD, with
Applications to Queueing Systems and Inventory
Control}
\author{
	Xiang Li\thanks{Equal contribution, School of Mathematical Sciences, Peking University; email: \texttt{lx10077@pku.edu.cn}. } \\
	\and
	Jiadong Liang\thanks{Equal contribution, School of Mathematical Sciences, Peking University; email: \texttt{jdliang@pku.edu.cn}. } \\
	\and
 	Xinyun Chen\thanks{School of Data Science, The Chinese University of Hong Kong, Shenzhen; email: \texttt{chenxinyun@cuhk.edu.cn}. }\\
	\and
	Zhihua Zhang\thanks{School of Mathematical Sciences, Peking University; email: \texttt{zhzhang@math.pku.edu.cn}. } \\
}

\maketitle

\abstract{%
Stream stochastic gradient descent (SGD) is a simple and efficient method for solving online optimization problems in operations research (OR), where data is generated by parameter-dependent Markov chains. Unlike traditional approaches which require increasing batch sizes during iterations, stream SGD uses a single sample per iteration, significantly improving sample efficiency. 
This paper establishes a systematic framework for analyzing stream SGD, leveraging the Poisson equation solution to address gradient bias and statistical dependence. We prove optimal \( \gO(T^{-1}) \) convergence rates and the state-of-the-art \( \gO(\log T) \) regret, while also introducing an online inference method for uncertainty quantification and supporting it by a novel functional central limit theorem.  We propose a novel Wasserstein-type divergence to describe the framework's conditions, which makes the assumptions in question directly verified via coupling techniques tailored to underlying OR models. We consider applications in queueing systems and inventory management, demonstrating the practicality and broad relevance, as well as providing new insights into the effectiveness of stream SGD in OR fields. 
}%







\section{Introduction}

Many operations research (OR) problems aim to minimize long-term costs by optimizing system parameters, where the data is evaluated under the stationary distribution of a parameter-dependent Markov chain. This can be formulated as the following minimization problem:
\begin{equation}\label{eq: objective_previous}
\min_{\vartheta\in\Theta}~ f(\vartheta) :=
\int_{\mathcal{X}} c(\vartheta,x)\pi_\vartheta(\rd x),
\end{equation}
where \( \vartheta \in \Theta \) is the control parameter, \( \Theta \) is the bounded feasible domain, and \( \pi_\vartheta \) is the steady-state distribution of the Markov chain \( P_\vartheta \). The function \( c(\vartheta, x) \) represents the transient cost incurred when the system state is \( x \) and the control parameter is \( \vartheta \). Examples include pricing and capacity sizing in queueing systems \citep{chen2023online} and base-stock replenishment in inventory systems \citep{Huh2009}, which are the focus of this paper.
In these OR problems, direct access to the distribution \( \pi_\vartheta \) is not available. Instead, at each step \( t \), the current control parameter \( \vartheta_t \) is applied to the system, transitioning the state \( x_{t-1} \) to \( x_t \), where:
\[
x_{t}\sim P_{\vartheta_t}(x_{t-1},\cdot).
\]
With the observed \( x_t \), the next control parameter \( \vartheta_{t+1} \) is obtained. This process alternates until the optimal parameter \( \vartheta^\star = \arg\min_{\vartheta \in \Theta} f(\vartheta) \) is found. 

Stochastic gradient descent (SGD), introduced by \citet{robbins1951stochastic}, is a powerful stochastic approximation method for this iterative optimization. It alternates between observing random samples and updating parameters. 
Due to its simplicity and computational efficiency, SGD has been widely used in online optimization \citep{shalev2012online, bubeck2015convex} 
and in OR applications \citep{heyman2004stochastic, chen2022elements}.
For the problem \eqref{eq: objective_previous}, we assume that the gradient takes the following form
\begin{equation*}
	\nabla f (\vartheta) :=	\int_\gX H(\vartheta, x) \pi_{\vartheta}(\rd x),
\end{equation*}
where \( H(\vartheta, x) \) presents a (known) gradient estimator for \( \vartheta \).
As an immediate   application of stochastic approximation, SGD updates the parameter by:
\begin{equation*}
\vartheta_{t+1} \leftarrow \Pi_\Theta(\vartheta_t - \eta_t H(\vartheta_t, x_t)),
\end{equation*}
where \( \Pi_\Theta \) projects \( \vartheta_t \) onto the feasible domain \( \Theta \). 

However, this approach 
presents two challenges.
First, the gradient estimator  is significantly biased. Since \( x_t \) is drawn from the conditional distribution \( P_{\vartheta_t}(x_{t-1}, \cdot) \), which is typically different from the steady-state distribution \( \pi_\vartheta \), the gradient estimator (denoted \( H(\vartheta_t, x_t) \)) satisfies \( \EB_{x_t \sim P_{\vartheta_t}(x_{t-1},\cdot)} [H(\vartheta_t, x_t)] \neq \nabla f(\vartheta_t) \). Such a biased estimate could impede the convergence to the true optimum \citep{ajalloeian2020convergence}. 
Second, parameter-dependent Markov chains pose challenges for theoretical analysis. The dependence of \( P_\vartheta \) on \( \vartheta \) causes the sample distribution to evolve with the parameter, making the update potentially unstable and complicating rigorous analysis.

To address these challenges, previous work has considered periodic updates that approximate the stationary distribution by collecting multiple samples from the same Markov chain at each iteration. This approach uses an increasing batch size \( B_t \), and at each iteration \( t \), \( B_t \) sequential samples are drawn from the same Markov chain \( P_{\vartheta_t} \) and used to compute the gradient. By making \( B_t \)  sufficiently large, the sequential samples closely approximate the steady-state distribution \( \pi_{\vartheta_t} \). For instance, \citet{chen2023online} suggested \( B_t \sim \log t \), while \citet{Huh2009} proposed \( B_t \sim t^\beta \) for \( \beta > 0 \). Although this approach facilitates theoretical analysis, it is highly sample inefficient and potentially slows down the convergence of $\vartheta_t$. This inefficiency is particularly problematic in online scenarios, where data samples are streamed via real-time interaction with real systems and result in immediate costs.

\begin{figure}[t!]
    \centering
    \includegraphics[width=0.8\textwidth]{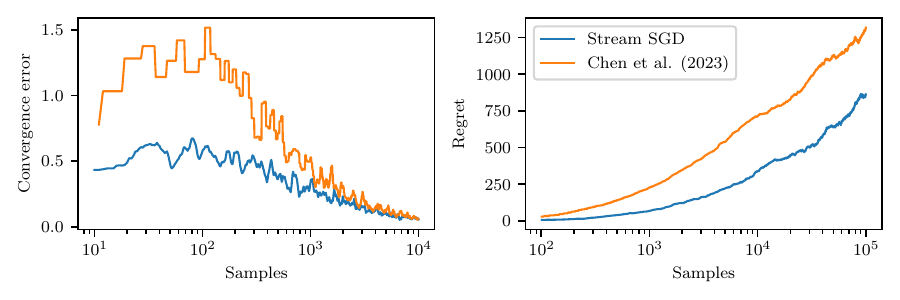} \\
    \includegraphics[width=0.8\textwidth]{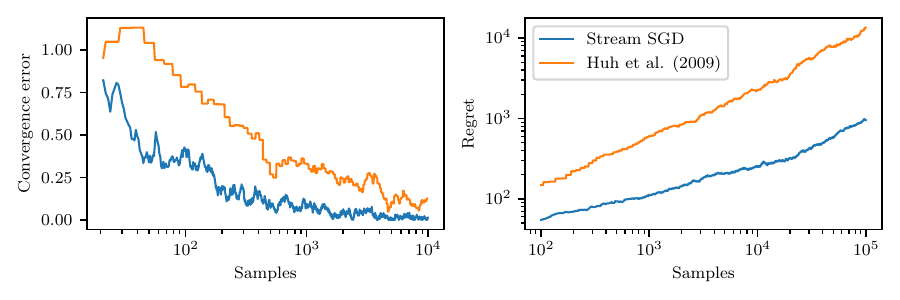}
    \caption{Stream SGD exhibits greater sample efficiency compared to increasing-batch variants, with better convergence rates and cumulative regret. Top: queueing systems \citep{chen2023online}. Bottom: inventory systems \citep{Huh2009}. The simulation setup follows the experimental section.}
   \label{fig:intro}
\end{figure}

Interestingly, simulation results show that such large batch sizes are unnecessary. 
A single sample (\( B_k = 1 \)) from \( P_{\vartheta_t}(x_{t-1}, \cdot) \) often enjoys fast and stable convergence.
The numerical simulations in Figure \ref{fig:intro} demonstrate that single-sample SGD achieves faster convergence rates and smaller cumulative regret compared to its batch counterparts.
We refer to this single-sample SGD method as \textit{stream SGD} to distinguish it from other variants.

A natural question arises: \textit{why does stream SGD work so effectively in practice?} While this question has been explored in the optimization or statistics literature, it remains insufficiently addressed in the context of OR. For instance, \citet{li2022state, karimi2019non, khodadadian2022finite} established convergence results based on regularity assumptions that are too restrictive for the OR applications considered here. Classical studies such as \citet{liang2010trajectory} examine the asymptotic properties of stream SGD under broad conditions, but these conditions are often impractical to verify. 
In short, existing studies are either too general or too narrowly scoped to fully explain the practical success of stream SGD in OR applications, leaving a notable gap in theoretical understanding. 
In this paper, we develop an analytic framework, suitable for a variety of OR problems of interest, to establish convergence rate and regret bound for stream SGD, which provides a theoretical assurance of the benefits stream SGD has compared to previous methods that rely on batch data. 
More recently, \citet{ethan2024stochastic}\footnote{We noted this work during the preparation of our current revision.} argued that stream SGD achieves convergence rates comparable to increasing batch methods (up to logarithmic factors) in common OR settings. In contrast, our analysis employs different methods and yields distinct (in fact stronger) results.

Another critical yet understudied aspect of online optimization in OR problems is uncertainty quantification. Although the previous works \citep{chen2023online,Huh2009} have established convergence rate results on $\theta_t$, those results cannot provide quantitative measures on the uncertainty of SGD outputs, as an estimate of the true $\theta^*$, in finite running time. A common approach to uncertainty quantification is to construct consistent confidence intervals, which has been studied in a variety of SGD settings.  
However, most existing methods rely on assumptions of i.i.d. data \citep{chen2020statistical, zhu2021online, su2023higrad, lee2021fast} or parameter-independent Markov chains \citep{ramprasad2021online, li2023online}, thus are unsuitable for the OR problems considered here. In this paper, based on our theoretic framework, we are able to develop an  online inference algorithm to construct asymptotically valid confidence intervals for stream SGD in a computationally efficient way. 

\subsection{Our Contributions}
We summarize our contributions as follows:

\paragraph{A general framework with easy-to-verify conditions}
To address a broad range of OR applications, we propose a general framework to evaluate the efficiency of stream SGD. At its essence, the framework leverages the \textit{Poisson equation solution} to effectively capture the statistical dependence between sequential samples in parameter-dependent Markov chains. 
Different from the prior studies \citep{karimi2019non, li2022state}, we provide a systematic approach from first principles: identifying conditions under which the Poisson equation solution exists and is sufficiently smooth. These conditions fall into two categories.
The first is the regularity conditions on the stationary loss \( f(\vartheta) \), such as strong convexity and smoothness, which are commonly used in the SGD literature.
The other is the continuity conditions on the transition kernel \( P_{\vartheta} \). We formulate those conditions using Wasserstein-type divergences so that they can be directly verified by applying coupling techniques to the underlying OR models.
As illustrative examples, we apply  our framework to two OR problems of different natures: (i) pricing and capacity sizing in queueing systems, and (ii) base-stock replenishment in inventory systems.

\paragraph{Convergence and regret analysis}
Based on our framework, we establish a finite-sample convergence rate and an upper bound for cumulative regret of the stream SGD algorithm. Specifically, we demonstrate that stream SGD achieves a convergence rate of \( \EB \| \vartheta_T - \vtheta \|^2 = \gO(T^{-1}) \) when the step size is \( \eta_t \sim t^{-1} \). 
This result improves upon the convergence rate of $\gO((\log T)^2/T)$ in the recent work \citep{ethan2024stochastic}.
Furthermore, we show that the cumulative regret over \( T \) observations is \( \gO(\log T) \), representing the state-of-the-art regret bound to the best of our knowledge. Compared to the regret bounds of $\gO((\log T)^2) $ and $\gO(T^{2/3})$ obtained in \cite{chen2023online} and \cite{Huh2009}, respectively, for batch data methods,  our regret bound for stream SGD possesses a smaller growing rate in $T$. These theoretic results indicate the effectiveness and efficiency of stream SGD without increasing batch size when applied to a variety of stochastic systems.

\paragraph{Uncertainty quantification}
To measure the uncertainty in the output of the stream SDG algorithm, we develop an online inference method for uncertainty quantification, supported by a novel functional central limit theorem (FCLT). Different from the classic CLT \citep{liang2010trajectory, lei2024variance} focusing solely on the sample average, our FCLT captures the asymptotic behavior of the entire trajectory of stream SGD iterates. This more comprehensive characterization enables the construction of confidence regions via the cancellation method \citep{glynn1990simulation}, eliminating the need for explicit estimation of the asymptotic variance matrix and resulting in a computationally efficient online inference algorithm. The derivation of FCLT is also tied to our theoretic framework and relies heavily on the Poisson equation solution.
In addition to validating our online inference method, we would believe the FCLT result itself is also of independent theoretic interest in general SGD scenarios.

\subsection{Organization of the Paper}
The paper is organized as follows.
Section \ref{sec: literture review} reviews related work, emphasizing the distinctions between our work and prior studies.  
Section \ref{sec: main results} introduces our framework, detailing the assumptions, presenting the convergence and regret analysis, establishing the FCLT, and describing the proposed online uncertainty quantification method.  
Section \ref{sec:examples} demonstrates the applicability of our framework through two OR applications, providing detailed theoretical results for each.  
Finally, in Section \ref{sec: numermical}, we present a comprehensive comparison of the convergence and regret performance of stream SGD with prior methods and showcase the effectiveness of the proposed online inference methods.  
All proofs are provided in the appendix.

\section{Literature Review}
\label{sec: literture review}

\subsection{Stream SGD in OR Studies}

Stream SGD algorithm and its variants have been extensively applied to optimizing control policies of stochastic systems in OR literature. In early works,  it was applied to M/M/1 queues in \cite{suri1988perturbation} and was later extended to GI/GI/1 queues by \cite{Fu1990} using their regenerative structure. \cite{LEcuyer1994} analyzed convergence properties with various gradient estimators, including IPA, finite-difference, and likelihood-ratio methods, while \cite{LEcuyer94b} provided detailed numerical studies. Recently, \cite{chen2023online} established finite-step convergence for online learning in GI/GI/1 queues, and \cite{ravner2023Nash} developed an SGD-based algorithm for Nash equilibrium in queueing games. 
In these works, SGD algorithms relied on batch data to manage gradient bias, using fixed or regenerative cycle-based batch sizes \citep{LEcuyer1994, chen2023online, Fu1990, ravner2023Nash}.

Beyond queueing systems, SGD methods have also been applied to steady-state optimization in inventory systems. In early work, \citet{huh2009nonparametric} employed stream SGD for lost-sales inventory systems with censored demand.
They identified an unbiased gradient estimator given the assumption of zero replenishment lead time. Similarly, \citet{yuan2021marrying} combined SGD with cycle-based batch sizes and bandit controls to address stochastic inventory systems with fixed costs, also assuming zero lead time.
However, for inventory systems with general lead times, unbiased gradient estimators are typically unavailable. To address this, researchers have developed batch-based SGD variants to manage gradient bias. \citet{Huh2009} proposed an increasing batch approach, while \citet{zhang2020closing} introduced a cycle-based batch size method.

\subsection{Theoretic Study of Stream SGD}

There is a vast body of literature on the theoretical analysis of stream SGD and general SGD-type algorithms. 
\citet{liang2010trajectory} focused on the asymptotic behavior of stream SGD rather instead of finite-time convergence.
Other studies, such as \citet{duchi2012ergodic}, \citet{doan2022finite}, and \citet{even2023stochastic}, established finite-time convergence rates under the assumption that data originate from a fixed, suitably ergodic Markov chain. Similarly, \citet{mendler2020stochastic} and \citet{drusvyatskiy2023stochastic} assumed that data are sampled directly from the steady-state distribution \( \pi_\vartheta \). These assumptions, however, are incompatible with the parameter-dependent Markov chains commonly encountered in OR problems. Furthermore, some works impose stronger conditions, such as uniformly bounded gradient estimators \citep{sun2018markov} or almost surely (strong) convexity and smoothness \citep{li2022state}, which differ fundamentally from the settings of interest in our study.

Our analysis employs the Poisson equation framework (e.g., \citet{benveniste2012adaptive}) to address the non-stationarity inherent in Markov data. While the Poisson equation has been used in prior analyses of SGD-type algorithms \citep{karimi2019non, tadic2017asymptotic, li2022state}, these works typically assume the existence and smoothness of the Poisson equation solution without verifying. In contrast, we establish a set of sufficient conditions to guarantee these properties, making our approach more practical for applying stream SGD to specific stochastic systems. 
Under these tailored assumptions, we prove finite-time convergence rates of \( \gO(T^{-1}) \), matching the minimax lower bound due to gradient noise \citep{lan2020first, even2023stochastic, beznosikov2024first} and outperforming the results of the concurrent work \citet{ethan2024stochastic} by logarithmic factors. 

\subsection{Statistical Inference for Stream SGD} 
Statistical inference quantifies the uncertainty of estimated parameters, helping us understand inherent randomness and make better decisions under uncertainty.
Statistical inference for stream SGD often begins with establishing a central limit theorem (CLT) for the averaged iterates \citep{polyak1992acceleration, ruppert1988efficient, lei2024variance}, followed by estimating the asymptotic variance matrix \citep{chen2020statistical, zhu2021online}. The combination of asymptotic normality from the CLT and variance estimates enables the construction of confidence intervals, which has been widely applied to SGD and its variants under i.i.d. data \citep{chen2020statistical, toulis2017asymptotic, liang2019statistical}. For Markovian data, however, methods like online bootstrapping \citep{ramprasad2021online} often rely on multiple simulation oracles, while other approaches assume data generated from parameter-dependent Markov chains \citep{li2023online}, limiting their practicality.

Our approach diverges from existing methods by analyzing data from parameter-dependent Markov chains and leveraging the entire trajectory of SGD iterates to quantify uncertainty. Specifically, we establish a functional central limit theorem (FCLT) in which the partial-sum process converges weakly to a Brownian motion. This enables the construction of confidence intervals that incorporate information from the full trajectory rather than relying solely on simple averages. While our approach draws inspiration from \citet{lee2021fast,lee2022fast}, it addresses challenges posed by parameter-dependent Markov data and connects to broader applications in various fields~\citep{li2021statistical, li2021polyak, li2023online, chen2021online}.
A major challenge lies in establishing the FCLT for parameter-dependent Markov data, as it requires addressing the temporal dependencies among SGD iterates. To tackle this, we utilize the Poisson solution equation to decompose the partial-sum process into several simpler components, enabling us to analyze their asymptotic behaviors individually.

\section{Main Results}
\label{sec: main results}
In Section \ref{subsec:algorithm and assumptions}, we provide a formal introduction to the stream SGD algorithm along with technical assumptions.  Section \ref{subsec: convergence and regret} presents the main theoretic results on the finite-step convergence rate and the regret upper bound, while Section \ref{subsec: inference} covers the CLT analysis and resulted online inference method. The complete proofs are deferred to~\ref{sec:prf_of_conv&regret} and~\ref{sec:prf_of_inference}.

\subsection{The Algorithm and Assumptions}
\label{subsec:algorithm and assumptions}
\input{Journal_version/Assumption}

\subsection{Finite-time Convergence and Regret Bound}
\label{subsec: convergence and regret}
\input{Journal_version/convergence_result}
\input{Journal_version/Tailored_Regret_Analysis_for_SAMCMC}

\subsection{Functional CLT and Online Inference}
\label{subsec: inference}
\input{Journal_version/online_inference}

\section{Online Optimization Examples}\label{sec:examples}
We now apply the framework developed in Section \ref{sec: main results} to two application examples: (1) revenue management  for queueing systems in Section \ref{subsec:queue}; (2) replenishment for inventory systems in Section \ref{subsec:inventory}. For both settings, we develop efficient stream SGD algorithms and corresponding online inference methods, leveraging our framework, with theoretical performance guarantees. We also test their numerical performance and the results will be reported in Section \ref{sec: numermical}.

\subsection{Pricing and Capacity sizing in Queueing Systems}
\label{subsec:queue}
\input{Journal_version/Application_GG1_Queue_Service_System}

\subsection{Base-stock Replenishment in Inventory Systems}
\label{subsec:inventory}
\input{Journal_version/Application_Inventory_Control_with_Base_Stock_Policy}

\section{Numerical Experiments}
\label{sec: numermical}
\input{Journal_version/numerical}

\section{Discussion}

In this paper, we study stream SGD, a method that uses only one sample from the current parameter-dependent Markov chain to compute the stochastic gradient for online OR problems. We develop a systematic analytical framework to support this simple yet effective approach.
From an optimization perspective, we establish finite-time convergence rates and perform regret analysis. From a statistical perspective, we derive a functional central limit theorem and propose an online inference method. We illustrate the framework in two OR applications: (i) pricing and capacity sizing in queueing systems and (ii) base-stock replenishment in inventory systems, showing how it advances online OR problem-solving.


Our theory and experiments consistently show that batching samples for gradient computation is unnecessary. When the Poisson equation solution is sufficiently smooth, a single sample suffices. This smoothness captures the statistical dependence between consecutive Markov chain samples, enabling convergence rate analysis and FCLTs. Importantly, practitioners need not compute the Poisson equation solution---it is purely a theoretical tool. In our examples, standard assumptions ensure the required smoothness.


Several open questions remain. First, our work focuses on box constraints, but practical problems often involve data-driven constraints \citep{boob2023stochastic, yang2024data}, whose integration into our framework is nontrivial.
Second, stream SGD is a single-step-size scale stochastic approximation method, while many RL problems involve multi-scale settings \citep{hong2023two}. Recent work on convergence \citep{han2024finite} and uncertainty quantification \citep{han2024decoupled} in such settings often assumes i.i.d. data, which is unrealistic for many OR problems. Extending our framework to multi-scale stochastic approximation with Markovian data is a promising direction.
Finally, our framework targets stable and relatively simple dynamic systems, but real-world systems—such as multi-server queues—can be unstable or highly complex. Coupling techniques \citep{blanchet2016rates,blanchet2018perfect} may allow our approach to extend to these systems if suitable Wasserstein-type divergence conditions can be verified. However, such extensions require refined analysis beyond the scope of this work.
{In concrete OR applications such as queueing systems, interesting directions include relaxing strong convexity and handling instability from suboptimal decision parameters, which may require advanced algorithmic designs (e.g., regularization) and transient analysis of unstable dynamics.} Addressing these challenges could substantially broaden the applicability of stream SGD, and we leave them for future work.

\normalem
\bibliographystyle{plainnat}
\bibliography{reference}

\clearpage
\appendix
\section{{Online Computation of Confidence Intervals}}\label{subsec: online computation}
\input{Journal_version/online_comput_conf_intv}
\section{Proofs for Section~\ref{subsec: convergence and regret}}\label{sec:prf_of_conv&regret}
\input{Journal_version/continuous_Poisson}

\input{Proof_for_Convergence_Result}
\input{Journal_version/proof_of_regret_lemma}

\section{Proofs for Section~\ref{subsec: inference}}\label{sec:prf_of_inference}
\input{Journal_version/inference}

\subsection{Proof of Proposition~\ref{prop: covariance convergence}}
\label{sec:prf of prop_covar_conver}
\input{Journal_version/Convergence_of_the_Asymptotic_Covariance}

\section{Proofs for Section~\ref{sec:examples}}
\input{Journal_version/proof_of_queue}

\input{Journal_version/proof_of_inventory}

\input{Journal_version/proofs_of_auxiliary_lemmas}

\section{Experimental Results for Other Queue Systems}
\input{Journal_version/numerical_more}

\section{{Additional Numerical Results on Convergence Rate}}\label{appx: additional results}
\input{Journal_version/additional_results_convergence_rate}



  



\end{document}

%% file: math_command.tex
\usepackage{amsmath,amsfonts,bm}
\usepackage{bbm}

\def\rd{{\textnormal{d}}}

\def\vtheta{{\bm{\theta}}}

\DeclareMathAlphabet{\mathsfit}{\encodingdefault}{\sfdefault}{m}{sl}
\SetMathAlphabet{\mathsfit}{bold}{\encodingdefault}{\sfdefault}{bx}{n}

\def\gA{{\mathcal{A}}}
\def\gB{{\mathcal{B}}}

\def\gD{{\mathcal{D}}}

\def\gF{{\mathcal{F}}}

\def\gL{{\mathcal{L}}}

\def\gO{{\mathcal{O}}}
\def\gP{{P}}  

\def\gT{{\mathcal{T}}}
\def\gU{{\mathcal{U}}}
\def\gV{{\mathcal{V}}}
\def\gW{{\mathcal{W}}}
\def\gX{{\mathcal{X}}}

\def\0{{\bf 0}}
\def\1{{\bf 1}}

\def\CM{{\mathcal C}}

\def\FM{{\mathcal F}}

\def\NM{{\mathcal N}}

\def\XM{{\mathcal X}}

\def\EB{{\mathbb E}}

\def\NB{{\mathbb N}}

\def\PB{{\mathbb P}}

\def\RB{{\mathbb R}}

\def\ph{\mbox{\boldmath$\phi$\unboldmath}}

\def\De{\mbox{\boldmath$\Delta$\unboldmath}}

\def\eps{{\varepsilon}}

\def\floor#1{\lfloor #1 \rfloor}

\def\argmin{\mathop{\rm argmin}}

\newcommand{\Var}{\mathrm{Var}}

\newcommand{\ssum}[3]{\sum\limits_{{#1}={#2}}^{#3}}

\def\inner#1#2{\left\langle #1, #2 \right\rangle}
\newcommand{\norm}[1]{\left\|{#1}\right\|}
\newcommand{\abs}[1]{\left|{#1}\right|}

%% file: Journal_version/Assumption.tex
The objective is to minimize the long-term cost of a discrete-time Markov system \( x_t \) with state space \( \mathcal{X} \subseteq \mathbb{R}^n \). The system's transition kernel \( P_\vartheta \), controlled by \( \vartheta \in \Theta \subseteq \mathbb{R}^m \), has the stationary distribution \( \pi_\vartheta \). Each step incurs a transient cost \( c(\vartheta, x) \), and the objective is to find \( \vartheta^\star \) that minimizes 
\begin{equation}
\label{eq: objective}
\min_{\vartheta \in \Theta}~ f(\vartheta) := \int_{\mathcal{X}} c(\vartheta, x) \pi_\vartheta(\rd x) .
\end{equation}
In practical applications, the function \( f(\vartheta) \) is typically assumed to be convex and differentiable, with gradient \( \nabla f(\vartheta) \). To facilitate analysis, we impose the following convexity assumption, as commonly used in previous studies \citep{chen2023online,ethan2024stochastic}.
\begin{assumption}[Convexity and smoothness]\label{assmpt: convexity}
There exist constants \( C_\Theta \), \( K_1 > K_0 > 0 \) such that:
\begin{enumerate}
\item[$(a)$] \( \vartheta^\star \) is an interior point of the bounded domain \( \Theta \), i.e., \( \sup_{\vartheta \in \Theta} \|\vartheta\| \le C_\Theta \);
\item[$(b)$] \( \langle \vartheta - \vartheta^\star, \nabla f(\vartheta) \rangle \geq K_0 \|\vartheta - \vartheta^\star\|^2 \); and (c) \( \|\nabla f(\vartheta)\| \leq K_1 \|\vartheta - \vartheta^\star\| \).
\end{enumerate}
\end{assumption}

\begin{remark}[Discussion on Strong Convexity Condition]
    {By Condition~(b) of Assumption~\ref{assmpt: convexity}, the objective function $f(\vartheta)$ is required to be strongly convex over the feasible set $\Theta$. This condition is essential for developing our inference method for stream SGD, as it guarantees the optimal $\mathcal{O}(1/T)$ convergence rate and allows us to control the remainder terms beyond the fluctuation term in the FCLT proof (see Lemma~\ref{lem:error-analysis} for details).}
    \end{remark}

There exists a bivariate function $H$ so that for any \( \vartheta \in \Theta \), the gradient \( \nabla f(\vartheta) \) can be expressed as
\begin{equation}\label{eq: H gradient esimator}
  \nabla f(\vartheta) :=  \int_{\mathcal{X}} H(\vartheta, x) \pi_\vartheta(\rd x) . 
\end{equation}
We apply the stream SGD algorithm (see Algorithm~\ref{alg:online_sgd}) to solve the above problem.

\begin{algorithm}[t!]
\small
\caption{\textsc{Stream SGD}}
\label{alg:online_sgd}
\begin{algorithmic}
\STATE \textbf{Input:} Step sizes $\{\eta_t\}$, initial state $x_0$, initial control parameter $\vartheta_1$, and total step $T$.
\FOR {Iteration $t = 1, \ldots, T$}
    \STATE Observe $x_t$ from the transition  $P_{\vartheta_t}(x_{t-1}, \cdot)$ 
    \STATE Compute the gradient estimator by $H(\vartheta_t, x_t)$.
    \STATE Update $\vartheta_{t+1} \leftarrow \Pi_\Theta(\vartheta_t - \eta_t H(\vartheta_t, x_t))$, where $\Pi_\Theta$ is the projection operator.
\ENDFOR
\STATE Output the final parameter $\vartheta_{T+1}$.
\end{algorithmic}
\end{algorithm}

\begin{assumption}[Conditions on step sizes]
\label{assump:step-size} Step size $\{\eta_t\}$ decreases in $t$. There exist constants $C_U >0$ and $0 <\varkappa \le \frac{K_0}{2}$ so that \textnormal{(}a\textnormal{)}  $0<\eta_t<\gamma_0 := \min\left\{\frac{2K_0}{C_U}, \frac{2}{K_0}, \frac{K_0}{\varkappa + 3K_0}\right\}$ and \textnormal{(}b\textnormal{)} $\left\vert\frac{1}{\eta_{t}}-\frac{1}{\eta_{t-1}}\right\vert \le \varkappa$.
\end{assumption}

We assume that the step size \( \eta_t \) satisfies the conditions in Assumption \ref{assump:step-size}, common in the SGD literature \citep[e.g.,][]{polyak1992acceleration, moulines2011non}. 
In practice, common choices for the step size include \( \eta_t \propto t^{-\alpha} \) with \( \alpha \in (0,1) \), or \( \eta_t = \gamma t^{-1} \) where \( \gamma \leq \gamma_0 \) for some constant $\gamma_0$.


The conventional SGD usually assumes i.i.d. data samples \citep{shalev2012online}, whereas our work considers data \( x_t \) generated by a parameter-dependent Markov chain \( P_{\vartheta} \). Consequently, \( H(\vartheta_t, x_t) \) is typically a \textit{biased} estimator of \( \nabla f(\vartheta_t) \), because the distribution \( P_{\vartheta_t}(x_{t-1}, \cdot) \) generally differs from the steady-state distribution \( \pi_{\vartheta_t} \). 
To analyze this gradient bias and its evolution through state transitions and parameter updates, we introduce a set of continuity conditions for the transition probabilities. These conditions, which are formulated by using potential functions and Wasserstein-type divergences, are detailed below.




The Wasserstein distance quantifies the discrepancy between two probability distributions by considering the ``cost'' of transporting one distribution to the other \citep{villani2009optimal}. By relaxing some of its metric properties, it can be extended to a Wasserstein-type divergence, which provides a more flexible measure of transport-based differences. 
In our work, this divergence facilitates the use of coupling techniques to characterize the continuity of transient distributions with respect to the parameter \( \vartheta \) or the state \( x \). In detail, for any given cost function \( d(\cdot, \cdot) \) on the state space \( \mathcal{X} \),\footnote{The cost function does not need to be a strict distance metric; for specific requirements, see Theorem 5.10 in \citep{villani2009optimal}.} we define the Wasserstein-type divergence between two distributions \( \mu \) and \( \nu \) on \( \mathcal{X} \) as
\begin{equation}\label{eq:def_w_dist}
\mathcal{W}_d(\mu, \nu) := \inf_{\pi \in \Gamma(\mu, \nu)} \mathbb{E}_{(X, Y) \sim \pi} \left[d(X, Y)\right],
\end{equation}
where \( \Gamma(\mu, \nu) \) represents the set of all couplings of \( \mu \) and \( \nu \):
\[
\Gamma(\mu, \nu) := \left\{\pi :~ \mathrm{Supp}(\pi) \subseteq \mathcal{X} \times \mathcal{X},~ \int \pi(\cdot, \mathrm{d} y) = \mu(\cdot),~ \int \pi(\mathrm{d} x, \cdot) = \nu(\cdot)\right\}.
\]

Selecting different cost functions allows us to define specific Wasserstein-type divergences for particular application settings. In this paper, we construct the cost function \( d(x, y) \) as $$ d(x, y) := \rho(x, y)(1 + V(x) + V(y)) ,$$ where \( \rho \) is typically either the Euclidean distance \( \norm{x - y} \) or an indicator function \( \mathbbm{1}\{x \neq y\} \) and  $V (\cdot)$ is a positive and measurable potential function  satisfying 
$$\sup_{\vartheta\in \Theta}\int_{\gX} V(y)^4 P_{\vartheta}(\mathrm{d}y|x) < V(x)^4 + b,~ \forall x\in\gX$$ with a given constant $b$. By appropriately selecting \( V \) and \( \rho \), our analytical framework adapts seamlessly to a wide range of OR problems.
Here and later, we denote by \( \gW_{\norm{\cdot}, V} \) and \( \gW_{\mathbbm{1}, V} \)  the  Wasserstein-type divergences corresponding to  cost functions  \( d(x, y) = \norm{x - y} \big(1 + V(x) + V(y)\big) \) and \( d(x, y) = \mathbbm{1}\{x \neq y\} \big(1 + V(x) + V(y)\big) \), respectively.



For any given control parameter $\vartheta$ and state point $x \in \gX$, we define the function $P_{\vartheta}f(x) \coloneqq \int f(y)P_{\vartheta}(\rd y\mid x)$ and the measure $\delta_{x}P_{\vartheta}(\cdot)\coloneqq P_{\vartheta}(\cdot, x)$.
Then, we are ready to state the continuity conditions for the transition probabilities described by the Wasserstein-type divergence. 
Specifically, Assumption \ref{ass:continue_H} ensures the joint $(\vartheta, x)$-continuity of the one-step transition \( \gP_{\vartheta} H(\vartheta, x) \) on the gradient estimator $H(\vartheta, x)$, while Assumption \ref{ass:continue_P} focuses on the $\vartheta$-continuity of the transition dynamics from a fixed initial state \( x \). Beyond these continuity requirements, Assumption \ref{ass:W-contraction} further demands that the influence of the initial state on the $n$-step transition dynamics diminishes over time, approximately at an exponential rate.


\begin{assumption}[Joint continuity of $P_{\vartheta}H$]\label{ass:continue_H}
There exist a universal constant \( L > 0 \) and a potential function \( V(x) > 0 \) such that for any \( \vartheta, \vartheta' \in \Theta \) and \( x, x' \in \mathcal{X} \), 
\[
\norm{\gP_{\vartheta} H(\vartheta, x) - \gP_{\vartheta'} H(\vartheta', x')} \leq L \big(V(x) + 1\big) \big(\norm{x - x'} + \norm{\vartheta - \vartheta'}\big).
\]
\end{assumption}

\begin{assumption}[Continuity of transition kernel with respect to \( \vartheta \)]\label{ass:continue_P}
Given the potential function \( V(x) \) from Assumption~\ref{ass:continue_H}, the transition kernel \( \gP_\vartheta \) satisfies a nearly \( L \)-Lipschitz continuity property. Specifically, there exists a universal constant \( L > 0 \) such that
\[
\gW_{\norm{\cdot}, V^2}(\delta_x \gP_{\vartheta}, \delta_x \gP_{\vartheta'}) \leq L \norm{\vartheta - \vartheta'} \big(1 + V(x)^2\big),
\]
where \( \gW_{\norm{\cdot}, V^2} \) is the Wasserstein-type divergence with the cost function \( \norm{x - y}(1 + V(x)^2 + V(y)^2) \).
\end{assumption}

\begin{assumption}[Wasserstein contraction]\label{ass:W-contraction}
The transition kernel \( \gP_\vartheta \) exhibits a Wasserstein contraction property for both the \( \gW_{\norm{\cdot}, \tilde{V}} \) and \( \gW_{\mathbbm{1}, \tilde{V}} \) distances, where \( \tilde{V} \in \{V, V^2\} \). Specifically, there exists a constant \( c > 0 \) such that for any \( x, y \in \mathcal{X} \) and \( \vartheta \in \Theta \), we have:
\[
\gW_{\rho, \tilde{V}}(\delta_x \gP_{\vartheta}^n, \delta_y \gP_{\vartheta}^n) \leq C \rho(x, y) e^{-cn} \big(2 + \tilde{V}(x)^2 + \tilde{V}(y)^2\big),
\]
where \( \rho \in \{\norm{\cdot}, \mathbbm{1}\} \) is the cost function, and \( C > 0 \) is a universal constant.
\end{assumption}

\begin{remark}[Discussion on Assumption \ref{ass:W-contraction}]
Classic works such as \citep{liang2010trajectory,care2019poisson} {replace Assumption~\ref{ass:W-contraction} with a combination of a drift condition and a Lipschitz continuity assumption of the transition kernel with respect to $\vartheta$ in total variation distance}. However, this approach has two main drawbacks: (1) the drift condition is often difficult to verify in complex OR applications, as it requires characterizing the density of $x_t$ for all time steps $t$ and parameters $\vartheta$; {(2) it imposes a Lipschitz requirement more stringent than Assumption~\ref{ass:continue_P}, further complicating verification. This highlights that the drift condition cannot fully replace Assumption~\ref{ass:W-contraction}}.
In contrast, the Wasserstein contraction assumption can be directly verified using coupling techniques, which are inherent in the definition of the Wasserstein-type divergence in \eqref{eq:def_w_dist}. Since coupling methods are widely used in OR model analysis \citep{lindvall2002lectures,kalashnikov2013mathematical}, our framework can be extended to general OR settings by leveraging existing results in the literature.
\end{remark}

\begin{remark}[Comparison with \citep{benveniste2012adaptive}]
{The assumptions we adopt are akin to those suggested in Chapter 2 of \citet{benveniste2012adaptive}, and are similarly used to establish certain smoothness conditions of the solution to the Poisson equation (see Lemma \ref{lem:reg_poisson_eq}). Nonetheless, due to some key differences in analytical approaches, they require extra conditions including contraction properties between two Markov systems with different model parameters, which are more technical and less straightforward to verify. In contrast, our approach achieves the same result under simpler and more practical assumptions. Employing a telescoping technique, we show that our assumptions are sufficient and extra assumptions in \cite{benveniste2012adaptive} can be derived from our assumptions.}
\end{remark}

\begin{remark}[The role of potential functions]
The potential function \( V(x) \), defined as a function of the state \( x \), measures the system's local stability and smoothness. It acts as a ``local Lipschitz constant,'' reflecting how sensitive the system is around different states. 
If \( V(x) \) is a trivial constant function, the above assumptions reduce to the stronger uniform ergodicity condition as  used in \citep{chandak2022concentration,kaledin2020finite}. 
For many OR models, however, the uniform ergodicity condition does not hold because the Wasserstein-type divergence is dependent on the states $x$ and $y$. As illustrated by concrete examples in Section~\ref{sec:examples}, the potential function can be chosen properly in accordance with the system's dynamics and the coupling rule.
\end{remark}

%% file: Journal_version/convergence_result.tex

In the heart of our theoretic analysis is the solution  \( U_h(\vartheta, x) \) to the following Poisson equation responding to the Markov system:
\begin{equation}
\label{eq:poisson}
U_h(\vartheta,x) - \gP_\vartheta U_h(\vartheta, x) = h(x) - \EB_{\pi_{\vartheta}}[h],
\end{equation}
where \( \EB_{\pi_\vartheta}[h] \equiv \int_{\XM} h(x) \pi_{\vartheta}(x) \rd x \). The following Lemma \ref{lem:reg_poisson_eq} shows that, under the continuity assumptions introduced in the previous section, the Poisson equation \eqref{eq:poisson} does have a solution, and in addition, its solution $U_h(\vartheta,x)$ is sufficiently smooth with respect to the parameter \( \vartheta \). We refer the readers to \ref{sec:prf_sketch} for the proof of Lemma \ref{lem:reg_poisson_eq}.

\begin{lemma}[Existence and smoothness of the Poisson equation solution]
\label{lem:reg_poisson_eq}
Suppose Assumptions \ref{ass:continue_H}--\ref{ass:W-contraction} hold for a given potential function \( V \in \gV \). Then, for any control parameter \( \vartheta \in \Theta \), there exists a function \( U_{P_{\vartheta}H}(\vartheta, \cdot) \) that solves Eqn.~\eqref{eq:poisson}, where \( h(\cdot) \) is chosen as \( P_\vartheta H(\vartheta, \cdot) \). 
Moreover, there exists a universal constant \( C > 0 \), independent of \( \vartheta \) and \( x \), such that for any \( \vartheta_1, \vartheta_2 \in \Theta \) and \( x \in \mathcal{X} \),
\[
\abs{U_{P_{\vartheta_1} H}(\vartheta_1, x) - U_{P_{\vartheta_2} H}(\vartheta_2, x)} \le
C\norm{\vartheta_1 - \vartheta_2}\left(1 + {V}(x)^4 +  \EB_{\pi_{\vartheta_1} + \pi_{\vartheta_2}}[V^4]\right),
\]
where \( \EB_{\pi_{\vartheta_1} + \pi_{\vartheta_2}}[V^4] = \int_{\XM} V(x)^4 \big(\pi_{\vartheta_1}(x) + \pi_{\vartheta_2}(x)\big) \rd x \) . 
\end{lemma}

\begin{corollary}
\label{cor:U_H-exist}
From Lemma \ref{lem:reg_poisson_eq}, there exists a function \( U_{H}(\vartheta, \cdot) \) that solves Eqn.~\eqref{eq:poisson} when  \( h(\cdot) = H(\vartheta, \cdot) \). It is given by $U_{H}(\vartheta, \cdot) = H(\vartheta, \cdot) - \nabla f(\vartheta) +  U_{P_{\vartheta}H}(\vartheta, \cdot)$ and satisfies $ U_{P_{\vartheta}H} = P_{\vartheta} U_{H}$.
\end{corollary}



Lemma \ref{lem:reg_poisson_eq} shows $U_{P_{\vartheta}H}$ exists and is smooth, as a result of which, $U_{H}$ also exists, as shown in Corollary \ref{cor:U_H-exist}.
In our analysis, the solution  $U_{P_{\vartheta}H}$ not only captures the deviation of the expected gradient estimator under the transient distribution from that under stationary distribution, i.e. the true gradient, but also quantifies the memory effect of the Markov chain and its impact on the gradient estimator variance.\footnote{The memory effect in the context of parameter-dependent Markov chains refers to the influence of past states of the Markov chain on its current state and behavior.}
Based on this foundation, we then are able to separate the transient effects of stream SGD from its long-term behavior and conduct finite-time analyses of convergence rates (using $U_{P_{\vartheta}H}$) and uncertainty quantification (using $U_{H}$). Before presenting our main theoretic results, we would like to point out that, for implementing stream SGD in practice, one does not need to compute the solution of the Poisson equation; it is used purely as a theoretical tool for analysis. 

\subsubsection{Finite-step Convergence Rate}\label{subsec: convergence result}

Our first main result provides a finite-time analysis of the mean square error $\EB \| \vartheta_{t} - \vtheta \|^2$, which we denote by \( \Delta_t  \) to simplify the notation.

\begin{thm}
\label{thm:L2-convergence}  
Consider applying the stream SGD algorithm to an online learning problem with the objective of the form~\eqref{eq: objective}, and let the observed state sequence during the algorithm execution be \(\{x_t\}_{t=1}^\infty\). Assume that \(\sup\limits_{t > 0} \EB V(x_t)^4 < \infty\).
Suppose Assumptions~\ref{assmpt: convexity}--\ref{ass:W-contraction} hold. There exist two problem-dependent constants $\overline{\Delta}$ and $C_{\star} >0$, such that for all $t\ge 1$,
\[
\Delta_{t+1} \le
\prod_{s=2}^t \left( 1 - K_0 \eta_s \right) \Delta_1 
+ \eta_t \left[ \left(1 + 8 \kappa\right) \overline{\Delta} + 18 \kappa C_{\star} \right].
\]
\end{thm}
\begin{corollary}
\label{cor:convergence}
Under the same conditions of Theorem \ref{thm:L2-convergence}, if \( \eta_t = \gamma / t \) with  \( K_0 \gamma \ge 1 \). Then, for all \( t \ge 1 \), $\Delta_{t+1} \le  \frac{1}{t} \cdot \left[ \Delta_1 + \gamma \left[(1 + 8 \kappa) \overline{\Delta} + 18 \kappa C_\star\right]\right].$
In other words, the mean square error \( \EB\|\vartheta_t - \vtheta\|^2 \) achieves a convergence rate of \( \gO(t^{-1}) \). 
\end{corollary}

\begin{remark}
    Prior works, such as \citep{li2022state,ethan2024stochastic}, have studied state-dependent stochastic algorithms similar to Theorem~\ref{thm:L2-convergence}. In the strongly convex case, \citet{ethan2024stochastic} achieved a convergence rate of $\mathcal{O}((\log t)^2 / t)$, while \citet{li2022state} obtained a rate of $\mathcal{O}(1/t)$ under the assumption that $H(\theta, x)$ is strongly monotone for all $x \in \mathcal{X}$---a condition significantly more stringent than Assumption~\ref{assmpt: convexity}. In the absence of strong convexity, \citet{ethan2024stochastic} shows that the rate further degrades to $\mathcal{O}((\log t)^2 / \sqrt{t})$.
\end{remark}

The \( \eta_t = \gO(t^{-1}) \) convergence rate in Corollary~\ref{cor:convergence} matches the rate of standard SGD with unbiased sampling data, see for example, Theorem 2.1 of Chapter 10 in \citet{kushner2003stochastic}. 

%% file: Journal_version/Tailored_Regret_Analysis_for_SAMCMC.tex
\subsubsection{Regret Analysis for Stream SGD}\label{sec:regret_analysis}
To evaluate the performance of stream SGD when applied online, we also establish an upper bound for the cumulative regret.
In general, cumulative regret quantifies the difference between the total cost incurred by an online algorithm and the best possible cost achievable in hindsight. In our setting, the cumulative regret is specified as
\begin{equation}
\label{eq:regret}
R(T) := \ssum{t}{1}{T} \big(\EB [c(\vartheta_t, X_t)] - \EB [c(\vartheta^*, X_\infty(\vartheta^*))] \big),
\end{equation}
where \( c(\cdot, \cdot) \) is the cost function defined in \eqref{eq: objective}, \( \vartheta_t \) is the parameter obtained by stream SGD at iteration \( t \), \( X_t \) represents the state in the Markov system at iteration \( t \), and \( X_\infty(\vartheta^*) \) is the (auxiliary) state from the stationary distribution of the dynamics under the optimal parameter \( \vartheta^* \).

To bound the regret, we impose the following Lipschitz assumption on the cost function \( c(\vartheta, x) \). 

\begin{assumption}[Continuity of cost function]\label{ass:cont_cost}
    There is a universal constant $L_{\text{Lip}}$ such that $c(\vartheta, x) - c(\vartheta, y) \le L_{\text{Lip}}\norm{x - y},~\forall \vartheta \in \Theta,~ \forall x,~y \in \gX$.
\end{assumption}

Following \cite{chen2021online},  we first decompose the cumulative regret into two terms as
\begin{equation*}
\begin{aligned}
    R(T)= \underbrace{\ssum{t}{1}{T}\EB[c(\vartheta_t, X_t) - c(\vartheta_t, X_\infty(\vartheta_t))]}_{R_1(T)} + 
    \underbrace{\ssum{t}{1}{T}\EB[c(\vartheta_t, X_\infty(\vartheta_t)) - c(\vartheta^*, X_\infty(\vartheta^*) )]}_{R_2(T)}.
\end{aligned}
\end{equation*}
The first term $R_1(T)$ captures the discrepancy caused by the distributional difference between the current state \( X_t \) and the steady state \( X_\infty(\vartheta_t) \) (regret of nonstationarity), while the second term $R_2(T)$ reflects the optimization gap due to incomplete convergence (regret of suboptimality).

Given the finite-time convergence rate established in Corollary \ref{cor:convergence}, the regret of suboptimality
\begin{equation}\label{eq:bd_estimation_err}
\begin{aligned}
   &R_2(T)= \ssum{t}{1}{T}\EB\big\{
    c(\vartheta_t, X_\infty(\vartheta_t)) - c(\vartheta^*, X_\infty(\vartheta^*) )
    \big\}
    = \ssum{t}{1}{T} [f(\vartheta_t) - f(\vartheta^*)]  \\
    \overset{(a)}{\le}& \ssum{t}{1}{T}\EB\left\{
    \inner{\nabla f(\vartheta^*) }{\vartheta_t - \vartheta^*} + \frac{L_{s}}{2}\norm{\vartheta_t - \vartheta^*}^2
    \right\} \overset{(b)}{=} 
     \frac{L_{s}}{2} \ssum{t}{1}{T}\EB
    \norm{\vartheta_t - \vartheta^*}^2 
    \le C\ssum{t}{1}{T}\frac{1}{t} = \gO(\log T),
\end{aligned}
\end{equation}
where inequality $(a)$ follows from Assumption~\ref{assmpt: convexity} and equality $(b)$ holds as  $\nabla f(\vartheta^*)=0$.

To bound $R_1(T)$, we establish the following Lemma \ref{lem:bd_trans_err_of_regret} on the transient error of expected cost incurred by stream SGD. 

\begin{lemma}[Bounding the transient error]\label{lem:bd_trans_err_of_regret}
    Under the assumptions of Theorem~\ref{thm:L2-convergence} and Assumption~\ref{ass:cont_cost}, if the step size $\eta_t \sim t^{-1}$, the transient error satisfies $ \EB [c(\vartheta_t, X_t) - c(\vartheta_t, X_\infty(\vartheta_t))] = \gO(t^{-1}).$
\end{lemma}

Combining the results from Eq.~\eqref{eq:bd_estimation_err} and Lemma~\ref{lem:bd_trans_err_of_regret}, we obtain an upper bound of order $\gO(\log T)$ for the cumulative regret of stream SGD.
\begin{thm}[Regret analysis]\label{thm:regret}
Consider applying the stream SGD algorithm to an online learning problem with the objective of the form~\eqref{eq: objective}. 
Suppose that the assumptions in Theorem~\ref{thm:L2-convergence} and Assumption~\ref{ass:cont_cost} are satisfied.
    If the step size is set as \( \eta_t \sim t^{-1} \), the cumulative regret of stream SGD (defined in \eqref{eq:regret}) is bounded by
    \[
    R(T) = \gO(\log T).
    \]
\end{thm}
Our \( \gO(\log T) \) regret also matches that for SGD  with i.i.d. data \citep{shalev2012online} and parameter-independent Markov chains \citep{agarwal2012generalization}.  
To the best of our knowledge, Theorem \ref{thm:regret} provides the tightest regret bound for parameter-dependent Markov chains in the literature. Furthermore, relative to the regret bounds established for batch-data techniques \citep{chen2023online, Huh2009}, Theorem \ref{thm:regret} shows that stream SGD not only speeds up convergence but also effectively reduces cumulated regrets, which also explains our finding in the simulation experiments as reported in Figure \ref{fig:intro} in Introduction.

%% file: Journal_version/online_inference.tex
This subsection presents how to perform uncertainty quantification using stream SGD outputs. We establish an FCLT for the averaging trajectory in Section \ref{subsec: FCLT}, outline confidence interval construction in Section \ref{subsec: CI}, and introduce an efficient online algorithm in  \ref{subsec: online computation}.

\subsubsection{Functional Central Limit Theorem}\label{subsec: FCLT}

\begin{assumption}[Regularity Conditions for FCLT]
\label{assmpt: more condtion FCLT}  
Assume that
\begin{enumerate}
\item[(a)] There exist a positive definite matrix $G$ and constants $L_G, \delta_G$ such that for any $\|\vartheta - \vtheta\| \leq \delta_G$, $\| \nabla f (\vartheta) - G(\vartheta - \vtheta) \| \leq L_G \cdot \|\vartheta - \vtheta\|^2.$
\item[(b)] The step size $\eta_t$ satisfies Assumption~\ref{assump:step-size} and the following conditions:
\begin{equation}
\label{eq:new-step-size}
\eta_t \downarrow 0, \quad t \eta_t \uparrow \infty, \quad \frac{\eta_{t-1}-\eta_t}{\eta_{t-1}} = o(\eta_{t-1}), \quad \sum_{t=1}^\infty \eta_t = \infty, \quad \sum_{t=1}^\infty \eta_t^{1.5} < \infty.
\end{equation}
\end{enumerate}
\end{assumption}

To establish the FCLT, Assumption~\ref{assmpt: more condtion FCLT} introduces two key prerequisites. Condition (a) ensures that $\nabla f (\vartheta)$ exhibits linear behavior near $\vtheta$, with $G$ representing the curvature information of $f$ at $\vtheta$. A sufficient condition for this is the continuity of $\nabla^2 f(\vartheta)$ in a neighborhood of the optimum $\vartheta^\star$.
Condition (b) places restrictions on the decay rate of the step size $\eta_t$. A practical alternative is $\eta_t = \gamma \cdot t^{-\alpha}$, where $\alpha \in (2/3, 1)$ and $\gamma > 0$, which satisfies both Assumption~\ref{assump:step-size} and Eqn.~\eqref{eq:new-step-size}.
These conditions are quite common in the literature \citep{polyak1992acceleration,li2023online}.

Before presenting the FCLT, we define the limiting covariance matrix $S$, which appears in the theorem. Recall the definition of $U_H(\vartheta, x)$ in Corollary \ref{cor:U_H-exist} with the gradient estimator $H(\vartheta,x)$. Based on this function, the covariance matrix is defined as:  
\[
\Var_{\vartheta}\big(U_H(\vtheta, x)\big) \equiv P_{\vartheta}[U_H(\vtheta,x)U_H(\vtheta,x)^\top] - \left(P_{\vartheta}U_H(\vtheta,x)\right)\left(P_{\vartheta}U_H(\vtheta,x)\right)^{\top},
\]
which represents the conditional covariance of $U(\vartheta, x)$ starting from $x$, with transitions governed by $P_{\vartheta}$. Consequently, $\Var_{\vartheta_t}(U_H(\vtheta, x_{t-1}))$ depends on both $x_{t-1}$ and $\vartheta_t$, capturing discrepancies from the stream SGD method in gradient estimation $H(x_t, \vartheta_t)$.  

\begin{proposition}[Limiting Covariance Matrix]  
\label{prop: covariance convergence}  
Under Assumptions \ref{assmpt: convexity} -- \ref{assmpt: more condtion FCLT}, we have 
\[
\frac{1}{T} \sum_{t=1}^T \Var_{\vartheta_t}(U_H(\vtheta, x_{t-1})) \overset{p.}{\to} S := \EB_{x \sim \pi_{\vartheta^*}} \Var_{\vartheta^*}(U_H(\vtheta, x)).
\]
\end{proposition}

\begin{thm}[Functional Central Limit Theorem (FCLT)]  
\label{thm:fclt}  
Under Assumptions \ref{assmpt: convexity} -- \ref{assmpt: more condtion FCLT}, we have 
\[
{\ph}_T(r) := \frac{1}{\sqrt{T}} \sum_{t=1}^{\floor{Tr}} (\vartheta_t - \vartheta^\star) \overset{d.}{\to} \psi(r) := G^{-1} S^{1/2} W(r)
\]
in the topology $\gD([0,1]; \vertiii{\cdot})$, representing the space of c\`adl\`ag functions on $[0,1]$ with the maximum norm $\vertiii{\phi} := \max_{r \in [0, 1]} \|\phi(r)\|_\infty$. Here, $G$ is defined in Assumption~\ref{assmpt: more condtion FCLT}, $S$ is the limiting covariance matrix from Proposition~\ref{prop: covariance convergence}, and $W(r)$ is a standard $d_\vartheta$-dimensional Brownian motion.
\end{thm}

Theorem~\ref{thm:fclt} establishes the FCLT for the partial-sum process ${\ph}_T(r)$, which represents a scaled average of the error vector $\vartheta_t - \vartheta^\star$ over the first $r$ fraction of iterations. This result demonstrates that the process weakly converges to a scaled Brownian motion. As Corollary~\ref{cor:functional} illustrates, this FCLT provides a stronger result than the conventional CLT \citep{liang2010trajectory}. The proof leverages the continuous mapping theorem\footnote{It states that any continuous functional of the process will weakly converge to the corresponding functional on the scled Brownian motion.} \citep{van2000asymptotic} by selecting an the functional $h(\phi) = \phi(1)$.

\begin{corollary}  
\label{cor:functional}  
Under Theorem~\ref{thm:fclt}, then for any continuous functional $h\colon \gD([0,1]; \vertiii{\cdot}) \to \RB^d$, 
\[
h({\ph}_T) \overset{d.}{\to} h(\psi) \quad \text{as } T \to \infty.
\]
In particular, if we use the functional $h(\phi) = \phi(1)$,  then
\[
{\ph}_T(1) \overset{d.}{\to} \psi(1), \quad \text{i.e.,} \quad \frac{1}{\sqrt{T}} \sum_{t=1}^T (\vartheta_t - \vtheta) \overset{d.}{\to} \NM(0, G^{-1} S G^{-\top}).
\]
\end{corollary}


\subsubsection{Construction of Confidence Intervals}
\label{subsec: CI}
{With Corollary~\ref{cor:functional}, we develop our inference procedure. 
Direct use of its asymptotic normality requires knowing $G^{-1} S^{1/2}$, which is difficult to estimate in Markovian systems.
We avoid this by employing a scale-invariant functional $h$, satisfying $h(A\ph) = h(\ph)$ for any non-singular $A \in \mathbb{R}^{d \times d}$ and c\`adl\`ag process $\ph$. Corollary~\ref{cor:functional} then yields
\[
h(\ph_T) \overset{d}{\to} h(W),
\]
where the law of $h(W)$ is independent of $G^{-1} S^{1/2}$ and can be simulated since $W$ and $h$ are known. This enables valid confidence intervals for any scale-invariant $h(\cdot)$, as in Proposition~\ref{prop:CI}.
}
\begin{remark}
This idea of leveraging scale-invariant functional to eliminate dependence on nuisance parameters aligns with the cancellation methods described in \citep{glynn1990simulation}. However, while in \citep{glynn1990simulation}  the FCLT is assumed to hold directly, our work establishes the FCLT from scratch and applies it to two concrete OR applications. Furthermore, we give a concrete functional in Eqn.~\eqref{eq:h-pivotal} that facilitates online computation.
\end{remark}

\begin{proposition}
\label{prop:CI}
Under the same conditions in Theorem~\ref{thm:fclt}, given any unit vector $v \in \RB^d$ and a continuous scale-invariant functional $h$ on $\gD([0,1]; \vertiii{\cdot})$, it follows that $\lim_{T \to \infty}\PB \left(v^\top {\vtheta} \in \CM(\alpha, h,  \phi_T) \right) = 1 - \alpha,$
where $\CM(\alpha, h,  \phi_T) := \left\{  v^\top \vtheta \in \RB  :  |h(v^\top \ph_T)| \le q_{\alpha, h}\right\}$ is the $\alpha$-level confidence set
and $ q_{\alpha, h} := \sup\{q:\PB( |h(W)| \ge q ) \le \alpha\}$ is the critical value.
\end{proposition}

According to Proposition~\ref{prop:CI}, the remaining task is to select an appropriate scale-invariant functional $h(\cdot)$. The widely used \textit{batch-mean method} for inference can be incorporated into this framework by choosing a specific $h(\cdot)$ as:  
\[
h(\phi) = \frac{\phi(1)}{\sqrt{\frac{1}{m(m-1)}\sum_{i=1}^m \left(\frac{\phi\left(\tau_{i+1}\right)-\phi\left(\tau_{i}\right)}{\tau_{i+1}-\tau_{i}} - \phi(1) \right)^2 }},
\]  
where the integer $m$ and endpoints $\{\tau_i : 1 \leq i \leq m\}$ are algorithm hyper-parameters corresponding to different variants of the batch-mean method. For instance, the classic batch-mean method \citep{glynn1990simulation} uses evenly split batch sizes, i.e., $\tau_{i+1} - \tau_i = 1/m$. In contrast, for SGD inference, \cite{chen2020statistical} and \cite{zhu2021constructing} proposed a scheme of increasing batch sizes.

We propose an alternative functional $h(\cdot)$ to avoid the complexity of analyzing algorithm hyper-parameters and to facilitate efficient online computation. Our choice is inspired by the \textit{random scaling} idea \citep{lee2022fast} and takes the form:  
\begin{equation}
\label{eq:h-pivotal}
h(\phi) = \frac{\phi(1)}{\sqrt{\int_0^1 (\phi(r) - r \phi(1))^2 \, \mathrm{d}r}}.
\end{equation}  
This functional is not only scale-invariant (i.e., $h(a \phi) = h(\phi)$ for any non-negative $a \geq 0$) but also symmetric (i.e., $h(-\phi) = -h(\phi)$) for any c\`adl\`ag process $\phi$.  
Substituting this $h(\cdot)$ into Proposition \ref{prop:CI}, we derive explicit confidence intervals in the following corollary.

\begin{corollary}  
\label{cor:CI}  
Under the conditions of Theorem~\ref{thm:fclt}, we have that
\begin{equation*}
\lim_{T \to \infty}
\PB\left( v^\top\vtheta \in \left[
v^\top \bar{\vartheta}_T - q_{\alpha, h} \frac{\bar{\sigma}_T}{\sqrt{T}}, \, v^\top \bar{\vartheta}_T + q_{\alpha, h} \frac{\bar{\sigma}_T}{\sqrt{T}}
\right] \right) \to 1 - \alpha,
\end{equation*}  
where $\bar{\vartheta}_T = \frac{1}{T} \sum_{i=1}^T \vartheta_i$ is the average of the first $T$ samples and $\bar{\sigma}_T = \frac{1}{T} \sqrt{\sum_{i=1}^T i^2 \left|v^\top (\bar{\vartheta}_i - \bar{\vartheta}_T) \right|^2}$ approximates the denominator $\sqrt{\int_0^1 (v^\top \ph_T(r) - r v^\top \ph_T(1))^2 \, \mathrm{d}r}$.  
\end{corollary}  

 The critical values $q_{\alpha, h}$ can be obtained via numerical simulation and listed in Table~\ref{table:critical}. In \ref{subsec: online computation}, we present an online method to compute $\bar{\sigma}_T$ efficiently in both computation and memory, which perfectly fits online learning settings where decisions may be required at every iteration. Furthermore, this online inference algorithm offers decision-makers greater flexibility. For instance, they can adapt the stream SGD procedure in real time based on decision quality, such as the half-width of the confidence interval.


%% file: Journal_version/Application_GG1_Queue_Service_System.tex
\subsubsection{Problem Formulation}

In this section, we apply the previously established framework to solve the pricing and capacity sizing problem in a GI/GI/1 queue. 
This problem involves determining the optimal service fee $p$ and capacity $\mu$ to minimize costs by balancing customer demand, congestion, and staffing expenses.
The decisions $(\mu, p)$ directly impact both the dynamics and overall cost of the GI/GI/1 queue system. With $\lambda(p)$ denoting the demand arrival rate, the interarrival and service times are modeled using two scaled random sequences, $\{U_n/\lambda(p)\}$ and $\{V_n/\mu\}$, where $U_1, U_2, \ldots$ and $V_1, V_2, \ldots$ are independent i.i.d.\ sequences with unit means, i.e., $\EE[U_n] = \EE[V_n] = 1$. The service times $V_i$ are assumed to have a continuous density function with a maximum value of $p_V$, while no regularity condition is imposed on the distribution of $U_i$, allowing it to be singular.
The service fee $p$ influences the demand arrival rate $\lambda(p)$, determining the interarrival times as $U_i/\lambda(p)$. Similarly, the service times depend on capacity $\mu$ and are given by $V_i/\mu$. Customers are charged a fee $p > 0$ upon joining the queue, while maintaining a service rate $\mu$ incurs a continuous staffing cost of $c(\mu)$ per time unit. Congestion, measured by the steady-state mean queue length $\EE[Q_\infty(\mu, p)]$, contributes to the cost at a rate proportional to $h_0$. Here, $h_0$ represents the cost per unit of congestion, encompassing penalties such as delays or customer dissatisfaction.

Putting these components together, we attempt to find the optimal service fee $p^\star$ and service capacity $\mu^\star$ within the feasible set $[\underline{\mu}, \bar{\mu}] \times [\underline{p}, \bar{p}]$ by minimizing the steady-state expected cost rate:
\begin{equation}
\label{eqObjmin}
\min_{(\mu, p) \in \mathcal{B}} f(\mu, p) \equiv h_0 \EE[Q_{\infty}(\mu, p)] + \zeta(\mu) - p\lambda(p), \qquad \mathcal{B} \equiv [\underline{\mu}, \bar{\mu}] \times [\underline{p}, \bar{p}].
\end{equation} 
This formulation provides a comprehensive framework to evaluate and optimize the trade-offs inherent in the pricing and capacity sizing problem.

To apply the framework, we impose assumptions on the demand and cost functions, as well as on the interarrival and service time distributions, which are used previously \citep{chen2023online}.

\begin{assumption}[Demand rate, staffing cost, and uniform stability]
\label{assmpt: uniform}
We assume that
\begin{enumerate}
\item[$(a)$] The arrival rate $\lambda(p)$ is continuously differentiable and non-increasing in $p$.
\item[$(b)$] The staffing cost $\zeta(\mu)$ is continuously differentiable and non-decreasing in $\mu$.
\item[$(c)$] The lower bounds $\underline{p}$ and $\underline{\mu}$ satisfy that  $\lambda(\underline{p})<\underline{\mu}$ so that the system is uniformly stable for all feasible choices of the pair $(\mu,p)$.
\end{enumerate}
\end{assumption}

\begin{assumption}[Light-tailed service and inter-arrival times]
\label{assmpt: light tail}
There exists a sufficiently small positive constant $\eta>0$ such that the moment-generating functions
$$\EE[\exp(\eta V_n)]<\infty\quad \text{and}\quad \EE[\exp(\eta U_n)]<\infty. $$ 
In addition, there exist constants ${0<\theta <\eta/2\bar{\mu}}$, $0<a<(\underline{\mu}-\lambda(\underline{p}))/(\underline{\mu}+\lambda(\underline{p}))$ and $\gamma >0$ such that
\begin{equation*}
\phi_U(-\theta/(1-a)) < -\theta -\gamma/2\quad \text{and}\quad \phi_V(\theta/(1+a)) < \theta -\gamma/2,
\end{equation*}
where $\phi_V(\cdot)$ and $\phi_U(\cdot)$ are the cumulant generating functions of $V$ and $U$, respectively.
\end{assumption}

To obtain an appropriate gradient estimator, \citet{chen2023online} considered the joint dynamics of the pair $(W_t, Y_t)$ which consist of the waiting time $W_t$ and the busy period $Y_t$.
In detail, for each $t \geq 1$, the observed busy period $Y_t$ is the age of the server’s busy time observed by customer $t$ upon arrival. If an arriving customer finds the server idle, both the waiting time and observed busy time corresponding to this customer are equal to 0. Under a fixed pair of control parameters $\vartheta\equiv (\mu, p)^\top \in\mathcal{B}$, $Y_t$ evolves jointly with the waiting time $W_t$ according to the following recursion:
\begin{equation}\label{eq:gnr rule of w&y}
W_{t+1}=\left(W_t - \frac{U_t}{\lambda(p)}+\frac{V_t}{\mu}\right)_+, \quad Y_{t+1}=\left(Y_t + \frac{U_t}{\lambda(p)}\right)\cdot \mathbbm{1}(W_{t+1}>0),
\end{equation}
where $(z)_{+} =\max(z, 0)$ and $\mathbbm{1}(\cdot)$ is the indicator function.  

It is known that under the stability condition $\mu>\lambda(p)$, the process $(W_t, Y_t)$ has a unique stationary distribution. We denote by $W_\infty(\mu,p)$ and $Y_\infty(\mu,p)$ the steady-state waiting time and observed busy period corresponding to control parameters $(\mu,p)$, respectively. Then, the derivative of $\EE[W_\infty(\mu,p)]$ is connected to $Y_\infty(\mu,p)$ as follows.

\begin{lemma}[Gradient Estimator of Queueing System]\label{lmm: gradient estimator} Under Assumptions \ref{assmpt: uniform} and \ref{assmpt: light tail}, then
\begin{equation*}
\begin{aligned}
    &\frac{\partial}{\partial p}\EE[W_\infty(\mu,p)]=\frac{\lambda'(p)}{\lambda(p)}\EE[Y_\infty(\mu,p)]~\text{and}~ \frac{\partial}{\partial \mu}\EE[W_\infty(\mu,p)]=-\left(\EE[Y_\infty(\mu,p)]+\frac{\EE[W_\infty(\mu,p)]}{\mu}\right).
\end{aligned}
\end{equation*}
\end{lemma}

Back to the form of the objective function~\eqref{eqObjmin}, by Little's law, the steady-state expectation is $\EE[Q_\infty(\mu,p)]=\lambda(p)(\EE[W_\infty(\mu,p)]+1/\mu)$.
Hence, denoting $\vartheta = (\mu,p)$ and $x = (w,y)$, we can rewrite the gradient estimator $H(\vartheta, x)$ as 
\begin{equation}\label{eq: Q gradient}
  H(\vartheta, x) = \left({-}\lambda(p) {-} p\lambda'(p)+h_0\lambda'(p)\left(w {+} y {+} \frac{1}{\mu}\right),  \zeta'(\mu) {-}h_0\frac{\lambda(p)}{\mu}\left(w {+}y {+} \frac{1}{\mu}\right)\right). 
\end{equation}

\subsubsection{Application of the Framework}

We now show that the GI/GI/1 queue example can be framed within the theoretical framework developed in Section \ref{sec: main results} and present the resulting implications. To do so, we verify the general conditions outlined in Section \ref{subsec:algorithm and assumptions}, specifically Assumptions~\ref{ass:continue_H}---\ref{ass:W-contraction}. The following proposition identifies the potential function used in this context and confirms the validity of these assumptions.

\begin{proposition}\label{prop:verify_queue}
Let Assumptions \ref{assmpt: uniform} and \ref{assmpt: light tail} hold for the above pricing and capacity sizing problem. With the gradient estimator $H(\vartheta, x)$ as defined in \eqref{eq: Q gradient}, Assumptions~\ref{ass:continue_H}---\ref{ass:W-contraction} are satisfied, where the potential function $V$ is given by $V(w, y) := y + e^{\kappa w}$, with $\kappa > 0$ as a problem-dependent constant.
\end{proposition}

We briefly explain why these assumptions hold and defer the detailed technical verification to the appendix, offering insights for potential future applications.
As a reminder, Assumption \ref{ass:continue_H} concerns the joint $(\vartheta, x)$-continuity of the one-step transition on the gradient estimator $\gP_{\vartheta} H(\vartheta, x)$. Assumption \ref{ass:continue_P} focuses on the $\vartheta$-continuity of the one-step transition dynamic from the same state $x$, while Assumption \ref{ass:W-contraction} addresses the initial state dependence in the $n$-step transition kernels.  
Each assumption fixes one argument---either the parameter $\vartheta$ or the state $x$---and varies the other to varies the other to examine its effect on the system dynamics. This modular structure motivates verification through synchronous coupling, which is precisely the approach we employed.

{
Take Assumption \ref{ass:W-contraction} as an example. We verify it by coupling two processes with the same transition kernel \eqref{eq:gnr rule of w&y} but different initial states. Defining a stopping time, we show that before it, their distance is bounded by the initial gap, and afterward, they coincide. The problem thus reduces to analyzing the pre-stopping-time behavior.
}
This approach reduces the problem to analyzing the behavior of the processes before the stopping time. To achieve this, we establish an intermediate result: for a fixed parameter $\vartheta = (\mu, p)$ and initial states $x=(w, y)$, 
there exist universal constants $\iota_w, \iota_y, \iota_y^\prime > 0$ such that for any $a >0$ and $n \ge 1$, $\PB(W_n \ge a) \lesssim e^{-\iota_w(a - w)_+}$ and $\PB(Y_n \ge a) \lesssim e^{-\iota_y(a - y - \iota_y^\prime w)_+}$.
These tail bounds control any integrable function of the joint process and thus the process difference before the stopping time. They also show that the system’s scale grows exponentially with the initial states, motivating the exponential potential function $V$, which captures this growth pattern.
\begin{corollary}
Assuming Assumptions~\ref{assmpt: convexity}, \ref{assmpt: uniform}, and \ref{assmpt: light tail} hold for the pricing and capacity sizing problem.
The cost function here is \( c\big((p,\mu), (w,y)\big) := h_0\lambda(p)(w + 1/\mu) + \zeta(\mu) - p\lambda(p) \) which satisfies  Assumption~\ref{ass:cont_cost}.
When setting the step size \( \eta_t \sim t^{-1} \), the stream SGD iterates achieve a linear convergence rate \( \EB \big[\norm{\vartheta_T - \vartheta^*}^2\big] = \gO(1/T) \) and a logarithmic regret \( R(T) = \gO(\log T) \).\footnote{The regret is defined in continuous time rather than discrete time. However, the results remain analogous. For a more detailed discussion on this part, one can refer to \ref{sec:queue_regret}. 
} Furthermore, if $\nabla^2 f(\vartheta)$ is continuous around the optima $\vartheta^\star$ and the step size is chosen as $\eta_t \sim t^{-\alpha}$ with $\alpha \in (\frac{2}{3}, 1)$, the FCLT in Theorem~\ref{thm:fclt} holds.
\end{corollary}
\begin{remark}\label{rmk: queue non-convexity}
{The objective function~\eqref{eqObjmin} can be non-convex, particularly when $p \to \infty$ and $\lambda(p) \to 0$; see Section~6.2 of \cite{chen2023online}. Sufficient conditions ensuring strong convexity are given in \cite{chen2023online}, which require the initial point to lie within the strongly convex region and not too far from the optimum.  
Although this region is generally unknown before implementation, managers often possess market knowledge or face regulatory price caps, making their initial guess unlikely to be far from optimal.~\ref{appx: additional results} presents numerical results indicating that stream SGD is relatively robust to the choice of initial points across various model settings.}
\end{remark}

%% file: Journal_version/Application_Inventory_Control_with_Base_Stock_Policy.tex
This subsection examines a periodic review inventory system with lost sales, positive replenishment lead times, and stochastic demand, aiming to minimize expected holding and lost sales costs under a base-stock policy. We begin by introducing the basic settings, demonstrate how the problem can be integrated into our general framework, and conclude with the corresponding theoretical results.

\subsubsection{Problem Formulation}

We consider a periodic review inventory system with a single product, where \( t \in \{1, 2, \ldots\} \) denotes the time period, and \( D_t \) represents the demand during period \( t \). The demands \( D_t \) are assumed to be i.i.d.\ random variables, each following the same distribution as a generic demand \( D \). The demand \( D \) is nonnegative, continuous, and satisfies \( \mu := \mathbb{E}[D] > 0 \), with cumulative distribution function \( F_D \).
The system operates with a replenishment lead time \( \tau \geq 1 \), and we adopt the classic base-stock policy, a widely used approach in inventory control~\citep{zipkin2000foundations}. Under this policy, a base-stock level \( S \geq 0 \) is set, and the quantity ordered in period \( t \), denoted as \( Q_t \), arrives at the beginning of period \( t + \tau \). The ordering rule is given by:
\begin{equation}\label{eq:base_stock_update}
Q_t = \left(S - X_t \cdot \mathbf{1}^\tau\right)_{+},
\end{equation}
where \( X_t \) is the inventory vector summarizing the system's state at period \( t \), defined as $X_t := \left(Q_{t-1}, Q_{t-2}, \ldots, Q_{t-\tau+1}, I_t\right).$
Here, \( I_t \) denotes the on-hand inventory after delivery in period \( t \), which evolves according to $I_{t+1} = (I_t - D_t)_+ + Q_{t-\tau + 1}.$
The inventory cost during period \( t \), given the on-hand inventory \( I_t \) and current demand $D_t$, is
\begin{equation}\label{eq:inventory_rndm_cost}
C(I_t, D_t) = h \cdot [(I_t - D_t)_{+}] + b \cdot [(D_t - I_t)_{+}],
\end{equation}
where \( h > 0 \) is the per-unit holding cost, covering storage and related expenses, and \( b > 0 \) is the per-unit penalty for lost sales, reflecting the cost of unmet demand.

The objective is to determine the base-stock level \( S \) that minimizes the long-run average inventory cost under the steady state \( X_\infty(S) \). This can be formulated as:
\begin{equation}
\label{eq:pr_obj_inv}
\inf_{S}~\EB [C(I_{\infty}(S), D)] = \inf_{S}~\EB \left[h\cdot (I_{\infty}(S) - D)_{+} + b\cdot(D - I_{\infty}(S))_{+}\right].
\end{equation}

We now introduce the gradient estimator for the base-stock policy. Let \( Q_t^\prime:= \frac{d}{dS} Q_t \) and \( I_t^\prime := \frac{d}{dS} I_t \) denote the sample derivatives of the order quantities and on-hand inventory levels, respectively. By directly computing the derivative, we obtain the gradient estimator of the cost function with respect to \( S \). Specifically, let \( Y_t := (Q_{t-1}^\prime, Q_{t-2}^\prime, \ldots, Q_{t-\tau+1}^\prime, I_t^\prime) \) represent the derivative vector of \( X_t \). Then, the gradient estimator is given by
\begin{equation}\label{eq:Inventory_grad_inf}
\begin{aligned}
H(S, (X_t, Y_t, D_t)):=\frac{d}{dS} C\left(I_t, D_t\right)= h \cdot \left[\mathbbm{1}\left\{D_t < I_t\right\} \cdot I_t^\prime\right] 
- b \cdot \left[\mathbbm{1}\left\{D_t \ge I_t\right\} \cdot I_t^\prime\right],
\end{aligned}
\end{equation}
where $S$ is the base-stock level and $(X_t, Y_t)$ is the state of the Markov chain.\footnote{
To fit within our framework, the gradient estimator used for assumption verification is \( \EB_{D_t}[H] \), which introduces a level of smoothness without violating the posed assumptions. This formulation implies that the Markov chain involves only \( (X_t, Y_t) \). For practical purposes, however, we explicitly include the dependence on \( D_t \) as shown in \eqref{eq:Inventory_grad_inf}.
}

To analyze the gradient estimator, \citet{Huh2009} introduced some technical assumptions.
They showed that the sequence of inventory vectors \( \{X_t\}_{t=1}^\infty \) is ergodic under the following assumption.
\begin{assumption}\label{assmpt:demmand}
There exists a known lower bound \( \underline{M} \) for the optimal base-stock level \( S^\star \) and an known upper bound \( \overline{M} \) for inventory capacity. Additionally, The demand \( D \) is non-negative and has a continuous and bounded density. The probability \( \gamma(\underline{M}) := \PB(D_t \leq \frac{\underline{M}}{3(\tau+1)}) > 0\). 
\end{assumption}
Moreover, they established a recursive relationship between \( \{Y_t\} \) and \( \{X_t\} \) under the $S$-level base-stock policy (see Lemma \ref{lem:drvt_update}). This relationship implies that \( Q_t^\prime \) and \( I_t^\prime \) can be computed iteratively, and thus the gradient estimator can also be calculated recursively.
\begin{lemma}[Theorem 1 of \citep{Huh2009}]\label{lem:drvt_update}
Under the $S$-level base-stock policy, the derivatives with respect to $S$ of the order quantity $Q_t$ and on-hand inventory level $I_t$ satisfy the following recursive relationships.
\begin{equation}\label{eq:inv_drvt_update}
\begin{aligned}
Q_t^\prime = (1 - Y_t \cdot \mathbf{1}^\tau) \cdot \mathbbm{1}\left\{S \geq X_t \cdot \mathbf{1}^\tau\right\},\quad
I_t^\prime = I_{t-1}^\prime \cdot \mathbbm{1}\left\{D_{t-1} < I_{t-1}\right\} + Q_{t-\tau}^\prime,
\end{aligned}
\end{equation}
where the initial conditions are given by \( I_0^\prime = 0 \) and \( Q_t^\prime = 0 \) for all \( t \leq 0 \).
\end{lemma}

\subsubsection{Application of the Framework}
Similar to the approach for the GI/GI/1 queue example, we demonstrate how this inventory control problem can be incorporated into our framework. As before, it suffices to verify that the problem satisfies Assumptions~\ref{ass:continue_H} through~\ref{ass:W-contraction}.

\begin{proposition}\label{prop:verify_inventory}
Let Assumption~\ref{assmpt:demmand} hold for the inventory control problem under the base-stock policy. With the gradient estimator defined in \eqref{eq:Inventory_grad_inf}, Assumptions~\ref{ass:continue_H}--\ref{ass:W-contraction} are satisfied, where the potential function \( V \equiv 1 \) is a constant.
\end{proposition}
We outline the approach for verifying the assumptions using synchronous coupling, focusing on the most challenging case, Assumption~\ref{ass:W-contraction}. We couple two inventory dynamics $(X_t, Y_t)$ that follow the same update rules and demand sequence $\{D_t\}_{t=1}^\infty$ but start from different initial states $z_i = (x_i, y_i)$, $i=1,2$. Let $(X_t(z_i), Y_t(z_i))$ denote the coupled trajectories. A stopping time is defined after which the trajectories coincide, and we analyze their pre-stopping-time behavior to assess the effect of initial states on the dynamics.
To quantify this effect, we apply the mean value theorem and analyze the derivative with respect to the initial states. The mean value theorem implies:
\begin{equation*}
\begin{aligned}
    \EB \left[\norm{X_t(z_1) - X_t(z_2)}\right] &= \EB\left[\norm{\int_0^1(z_2-z_1)  dX_t\big(\lambda z_2 + (1 - \lambda) z_1\big) \rd \lambda}\right] \\
    &\le \norm{z_1 - z_2} \int_0^1 \EB\left[\norm{dX_t\big(\lambda z_2 + (1 - \lambda) z_1\big)}\right] \rd \lambda.
\end{aligned}
\end{equation*}
Here $dX_t$ represents the derivative of $X_t(z)$ with respect to the initial state $z$.
Thus, continuity with respect to initial states is guaranteed if we can uniformly bound \( \EB[\norm{d X_t(z)}] \) for any \( z \).
To establish this bound, we note that it suffices to focus on \( dI_t(z) \) and \( dQ_t(z) \), given the definition that
$dX_t(z) = \big(dQ_{t-1}(z), \ldots, dQ_{t-\tau+1}(z), dI_t(z)\big)$,
where \( dI_t(z) \) and \( dQ_t(z) \) are defined analogously to \( dX_t(z) \). 
We prove that these sequences decay exponentially in expectation over time. Specifically, there exist universal constants \( C, \iota > 0 \), independent of \( t, z, \) and \( S \), such that:
\begin{equation}
\label{eq:I-Q-exp-decay}
\EB \left[\norm{dI_t(z)}\right] \le C e^{-\iota t}, \quad
\EB \left[\norm{dQ_t(z)}\right] \le C e^{-\iota t}.
\end{equation}
This exponential decay establishes the uniform boundedness of \( \EB[\norm{d X_t(z)}] \), thereby verifying the continuity condition with respect to initial states. 

As a byproduct, the uniform boundedness of the inventory dynamics is also ensured, given the capacity limit \( \overline{M} \) of the inventory system and the uniform bounds implied by \eqref{eq:I-Q-exp-decay}. Consequently, our potential function \( V \) can be chosen as a constant.

Given Proposition~\ref{prop:verify_inventory}, we translate the general theoretical results into the following corollary.

\begin{corollary}
Assuming Assumptions~\ref{assmpt: convexity} and \ref{assmpt:demmand} hold for the inventory control problem.
The cost function here is $c(S, I) = \EB_D[C(I,D)]$ which satisfies Assumption~\ref{ass:cont_cost}.
When setting \( \eta_t \sim t^{-1} \), the stream SGD iterates achieve a linear convergence rate \( \EB \big[\abs{S_T - S^*}^2\big] = \gO(1/T) \) and logarithmic regret \( R(T) = \gO(\log T) \).
Furthermore, if $\nabla^2 f(\vartheta)$ is continuous around the optima $\vartheta^\star$ and the step size is chosen as $\eta_t \sim t^{-\alpha}$ with $\alpha \in (\frac{2}{3}, 1)$, the FCLT in Theorem~\ref{thm:fclt} holds.
\end{corollary}

%% file: Journal_version/numerical.tex
In this section, we present numerical experiments to evaluate the performance of the stream SGD algorithm and our online inference method in two OR applications. 

\subsection{Pricing and Capacity Sizing in Single-Server Queue}
\label{sec:MM1}

We implement stream SGD on three GI/GI/1 queueing models with different inter-arrival and service time distributions, as in \cite{chen2023online}. For all settings, we use a logistic demand function $\lambda(p) = M \cdot \frac{\exp(a-p)}{1 + \exp(a-p)}$, and set $M=10$, $a=4.1$, and the holding cost coefficient $h_0=1$. For brevity, we present the setup and results for the M/M/1 model, which is the most classic example with Poisson arrivals and exponential service times.
Results on other models are provided in the appendix. Following \cite{chen2023online}, we use a quadratic cost function $\zeta(\mu) = \frac{1}{10} \cdot \mu^2$. In this case, the objective function \eqref{eqObjmin} has a closed-form expression from which we can numerically obtain the optimal solution $(\mu^\star, p^\star) = (7.103, 4.0234)$ with $\mathcal{B} = [6.56, 15] \times [3.5, 10]$.
As a benchmark, we use a simplified version of GOLiQ from \cite{chen2023online},\footnote{For a fair comparison, we ignore the leftovers from previous cycles in our implementation of GOLiQ.} which estimates the gradient $\nabla f(\vartheta_t)$ with batch sizes of $1 + b_c \log t$, where we choose $b_c \in \{1, 5, 10, 20\}$. Stream SGD corresponds to $b_c=0$. All experiments start with $(\mu_0, p_0) = (8, 3.5)$, using step sizes $\eta_t^\mu = \frac{12.5}{t}$ and $\eta_t^p = \frac{1.25}{t}$.
{To address the potential violation of strong convexity discussed in Remark~\ref{rmk: queue non-convexity}, we provide additional numerical results on the convergence rate of stream SGD across various model settings---particularly with different initial points---in~\ref{appx: additional results}.}

\subsubsection{Convergence and Regret Results}
\label{sec:convergence-MM1}

\begin{figure}[t!]
    \centering
    \includegraphics[width=\textwidth]{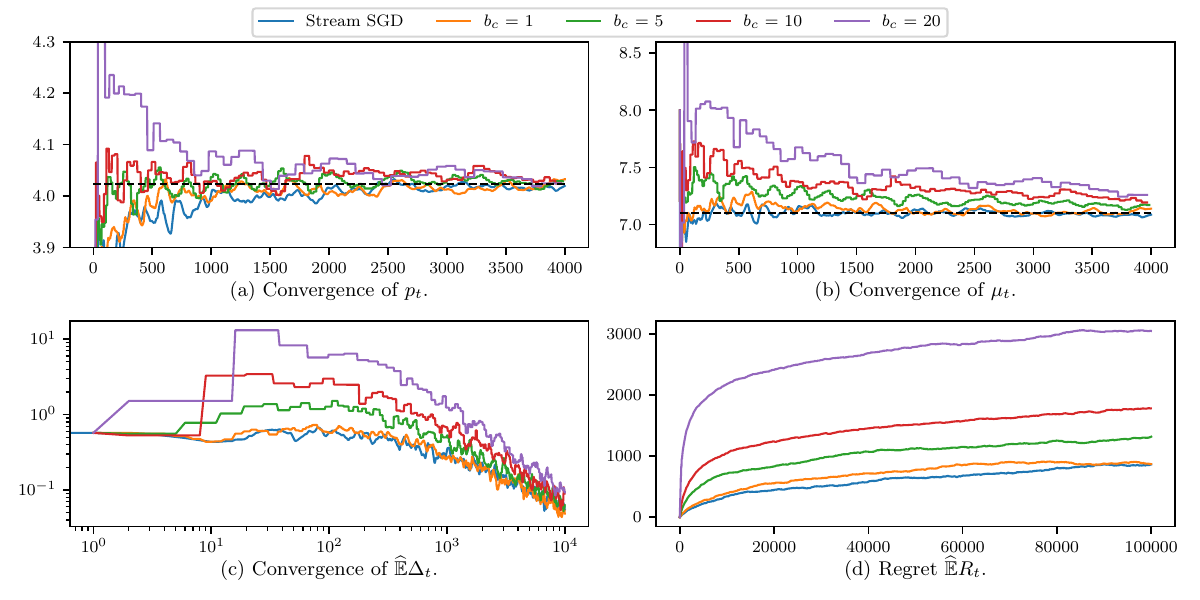}
    \caption{Convergence and regret in the M/M/1 setting, averaged over 200 repetitions.}  
    \label{fig:conv_MM1}
\end{figure}

The numerical results of stream SGD and GOLiQ with varying $b_c$ are shown in Figures \ref{fig:conv_MM1}. The top panels display the averaged trajectories of $p_t$ and $\mu_t$, while the bottom panels report the average $L_2$ difference between $\vartheta_t$ and $\vartheta^*$ (left) and the average cumulated regret (right), with $x$-axes representing the number of data samples (customers served).
The results show that increasing batch size does not improve convergence rate or reduce regret, which is consistent with our theoretical result that single-point gradient estimation suffices for convergence. The performance gap between stream SGD and GOLiQ becomes significant when $b_c \geq 5$, likely due to excessive emphasis on gradient refinement over optimization progress. When $b_c=1$, the batch size $1+b_c\log(t)$ is close to $1$ for small $t$. As a consequence, we see that the two curves corresponding to stream SGD and GOLiQ with $b_c=1$ are similar in the early iterations but gradually deviate from each other in the later iterations.

\subsubsection{Numerical Performance of Online Inference Method} 
\label{subsec: inference performance}

\begin{figure}[t!]
    \centering
    \includegraphics[width=\textwidth]{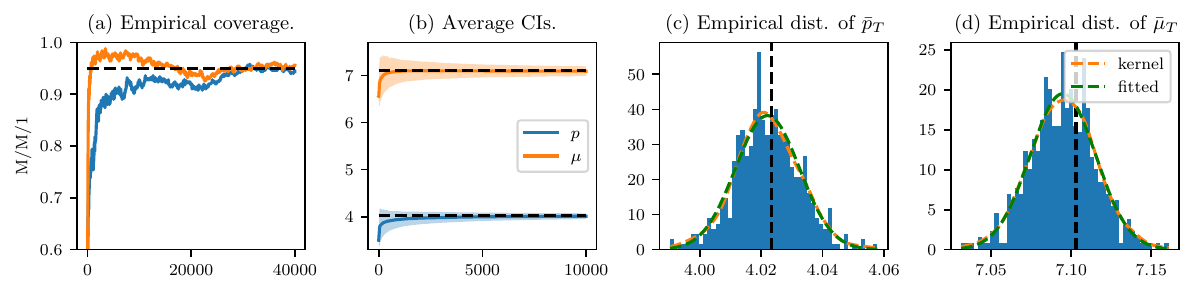}
    \caption{Online inference results in the M/M/1 setting, averaged over 500 repetitions.}  
   \label{fig:inf}
\end{figure}

To align with our theoretical framework, we use polynomial step sizes for all experiments: $\eta_{t,p} = \frac{c_p}{(1+t)^\alpha}$ and $\eta_{t,\mu} = \frac{c_\mu}{(1+t)^\alpha}$, with $c_p = 1$, $c_\mu = 10$, and $\alpha = 0.99$ for the M/M/1 model. Figure \ref{fig:inf} presents the inference results: (a) empirical coverage rates across 500 repetitions, (b) averaged confidence intervals, and (c), (d) histograms of averaged estimators $\bar{p}_T$ and $\bar{\mu}_T$.
As predicted by Corollary \ref{cor:CI}, Panel (a) shows empirical coverage rates approaching the target $95\%$ with increasing iterations and sample sizes. Panel (b) illustrates consistent coverage of true parameters by confidence intervals, with lengths decreasing as more data are processed. The histograms in Panels (c) and (d) reveal that $\bar{p}_T$ and $\bar{\mu}_T$ closely follow fitted Gaussian distributions, consistent with Corollary~\ref{cor:functional}.

\subsubsection{Comparison with Different Inference Methods}
\label{sec:other_reference_methods}

Finally, we compare our inference method with two benchmark batch-mean methods: (1) even-splitting (ES) with $\tau_i = \frac{i}{m}$~\citep{glynn1990simulation}, and (2) increasing batch size (IBS) with $\tau_i = (\frac{i}{m})^2$~\citep{chen2020statistical,zhu2021constructing}.\footnote{The original IBS proposed in \citep{chen2020statistical} uses $\tau_i = (\frac{i}{m})^{1/(1-\alpha)}$ for $\eta_t \propto t^{-\alpha}$ with $\alpha \in (0.5, 1)$. However, in our experiments, fixing $\alpha = 0.5$ performs better even when $\eta_t \propto t^{-0.99}$. So we present the results of $\alpha=0.5$ for IBS.}
Batch-mean methods are not inherently suited for online inference, so we perform the comparisons offline by storing all historical iterates in memory.
For a fair comparison, we apply all three methods to the same trajectories $\{\vartheta_t\}_{t \in [T]}$ generated by the SGD stream over 500 simulations with $T = 5 \times 10^4$ iterations for the M/M/1 case. Batch-mean methods are evaluated with batch numbers $m \in \{10, 20, 30\}$, and the best-performing results are reported.


Table~\ref{table:empirical_coverage} presents the empirical coverage rates (in percentages) estimated from 500 simulations, along with their standard errors, for parameters $p$ and $\mu$ across the four queueing settings. Results are provided for subsets of trajectory data with $n \in \{T/2^4, T/2^3, T/2^2, T/2, T\}$. 
The results show that our method consistently achieves higher empirical coverage rates when the number of iterations $n$ is small, while delivering performance comparable to IBS when $n$ is large. Specifically, when $n$ is small, the ranking is typically: Ours $>$ IBS $>$ ES. As $n$ increases, the empirical coverage rates of all methods converge to the nominal $95\%$ coverage, confirming their validity in large-sample settings.

\begin{table}[t!]
\caption{Comparison of empirical coverage rates between our method and the batch mean method in the M/M/1 setting. 
The standard errors of coverage rates $\widehat{p}$ are reported in the bracket (computed via $\sqrt{\widehat{p}(1-\widehat{p})/500} \times 100 \%$). 
} 
\centering 
\small
\begin{tabular}{c|c|ccccc} 
\toprule
Parameters &\multicolumn{1}{c|}{Methods} & $n=T/2^4$ & $n=T/2^3$ & $n=T/2^2$ & $n=T/2$ & $n=T$\\
\midrule \multirow{3}{*}{$p$}&
Ours & 89.2(1.39)&90.4(1.32)&93.2(1.13)&92.4(1.19)&94.8(0.99)\\ 
& ES& 80.6(1.77)&86.2(1.54)&90.2(1.33)&91.6(1.24)&92.4(1.19)\\ 
& IBS& 84.8(1.61)&92.4(1.19)&94.0(1.06)&95.6(0.92)&96.4(0.83)\\ 
\midrule
\multirow{3}{*}{$\mu$}&
 Ours & 97.0(0.76)&97.0(0.76)&95.8(0.9)&94.0(1.06)&94.4(1.03)\\ 
&ES& 97.8(0.66)&97.4(0.71)&96.6(0.81)&93.8(1.08)&93.6(1.09)\\ 
&IBS& 95.4(0.94)&94.6(1.01)&90.8(1.29)&89.2(1.39)&90.4(1.32)\\ 
\bottomrule
\end{tabular}
\label{table:empirical_coverage} 
\end{table}

\subsection{Base-stock Replenishment in Inventory Systems}

We apply stream SGD to solve the inventory control problem by finding the optimal base-stock level $S^\star$. The loss objective function is defined as $\min_{S} \EB[C(I_{\infty}(S))]$, as given in \eqref{eq:pr_obj_inv}. For all the experiments, we set $h=1$ and $b=10$. The demand $D$ follows an exponential distribution with an expectation of $\EB[D]=1$, and the feasible domain is $[1, 10]$. Two replenishment lead times are considered: $\tau \in \{1, 2\}$.
Since the optimal base-stock level $S^\star$ does not have a closed-form solution, we first use stream SGD to solve the problem with $10^5$ samples and compute the averaged iterates over 500 repetitions as the true parameter. For $\tau=1$, the computed optimal base-stock level is $S^\star = 4.4054$, and for $\tau=2$, it is $S^\star = 5.1406$.
We use the adaptive algorithm proposed in \cite{Huh2009} as a benchmark, as it also explores the application of batched stream SGD to inventory control. This algorithm, denoted by Adaptive$(\alpha, \beta)$, requires a step size $\eta_t \sim t^{-\alpha}$ and a batch size of the form $t^\beta$. 
To evaluate the effect of batch sizes, we use different values of $\beta \in \{0.5, 1, 1.5, 2\}$ for each gradient estimate. The stream SGD algorithm corresponds to the special case where $\beta=0$. For a fair comparison, we set the step size for all the algorithms to $\eta_t = \frac{2}{t^{0.67}}$, corresponding to $\alpha=0.67$, and initialize the base-stock level at $S_1=2$. 

\subsubsection{Convergence and Regret Results}

\begin{figure}[t!]
    \centering
    \includegraphics[width=\textwidth]{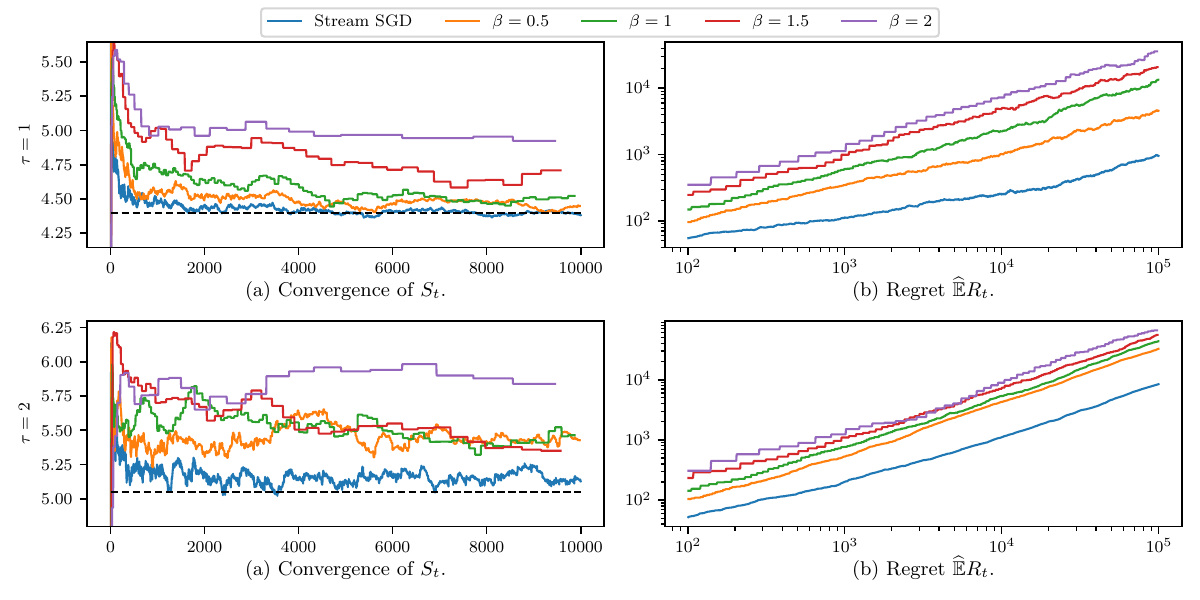}
    \caption{Convergence and regret in the inventory control problem, averaged over 200 repetitions. }  
    \label{fig:conv_inv}
\end{figure}

The numerical results of stream SGD and Adaptive$(\alpha, \beta)$ are presented in Figure \ref{fig:conv_inv}, with the top row corresponding to $\tau=1$ and the bottom row to $\tau=2$. As in Figure \ref{fig:conv_MM1}, the $x$-axes represent the number of samples. 
Adaptive$(\alpha, \beta)$ uses batched samples to compute gradients, leading to piecewise constant convergence and regret during the sample collection phases. From the left column, we observe that stream SGD remains consistently closer to the true parameter (denoted by the black dotted line), and from the right column, it demonstrates a smaller cumulative regret compared to Adaptive$(\alpha, \beta)$. 
As $\beta$ increases, the number of samples used for gradient computation grows, but Adaptive$(\alpha, \beta)$ performs worse than stream SGD in terms of both convergence and regret. Similar to Figure \ref{fig:conv_MM1}, increasing batch sizes does not improve sample efficiency. These results further support the conclusion that using a single sample for gradient estimation is sufficient to solve the inventory control problem effectively.

\subsubsection{Numerical Performance of Online Inference Method} 
\label{sec:exp-online-inference}

Under the same setup and parameter choices, we evaluate the performance of our online inference method for inventory control problems. Figure \ref{fig:inf_inv} summarizes the results: Panel (a) shows the empirical coverage rate over 500 repetitions, Panel (b) illustrates the averaged confidence intervals, and Panel (c) presents the histogram of the last iterate $\bar{S}_T$ for $T=10^5$ in the 500 experiments.
The observations are consistent with those in the M/M/1 setting. Panel (a) demonstrates that the empirical coverage rate quickly stabilizes around the target 95\% as iterations progress. Panel (b) reveals that the length of the averaged confidence intervals decreases with more iterations, reflecting improved precision and confidence. Notably, the confidence intervals for $\tau=2$ are wider than those for $\tau=1$ due to the greater stochasticity in the transition system for $\tau=2$. 
Finally, Panel (c) shows that the distribution of $\bar{S}_T$ approximates a Gaussian distribution, in agreement with Corollary~\ref{cor:functional}.

\begin{figure}[t!]
    \centering
    \includegraphics[width=0.7\textwidth]{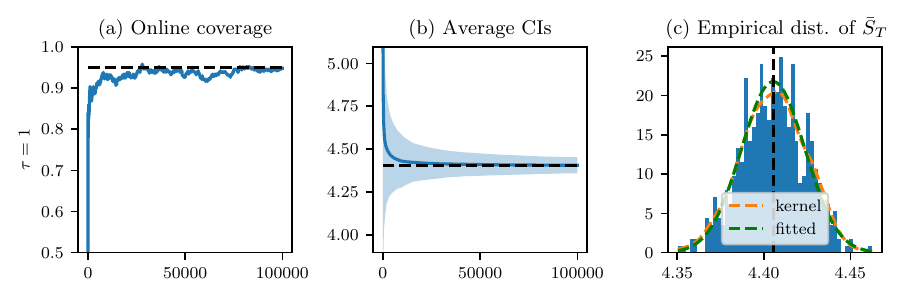} \\
    \includegraphics[width=0.7\textwidth]{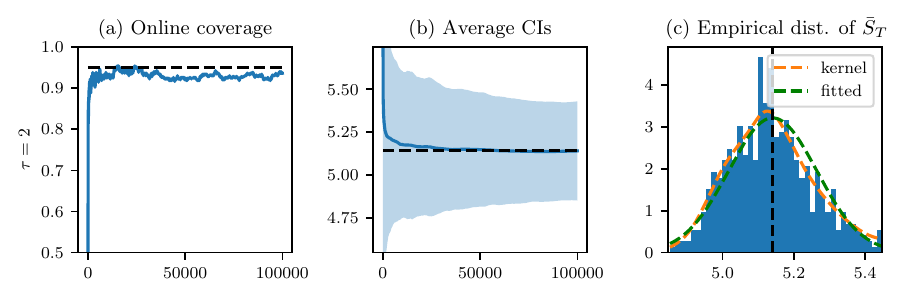}
    \caption{Online inference results in the inventory control problem, averaged over 500 repetitions.}  
    \label{fig:inf_inv}
    \vspace{-0.1in}
\end{figure}

\subsubsection{Comparison with Other Inference Methods} 
\label{sec:exp-inference-comparison}

Finally, we compare our method with the two baselines in Section~\ref{sec:other_reference_methods}. Table~\ref{table:empirical_coverage_inv} shows empirical coverage rates (in $\%$) and standard errors from 500 simulations for parameter \$S\$ in two inventory control settings, using trajectory subsets with $n \in {T/2^4, T/2^3, T/2^2, T/2, T}$ and $T=8000$. The coverage rate generally improves with sample size for both our method and ES. When the sample size is limited (e.g., $n \approx T/2^4$), IBS attains higher coverage but declines as $n$ grows, eventually becoming the lowest—indicating overestimation of randomness with limited samples. For instance, at $n = T/2^4$, the average interval length is $1.2169$ for IBS, $1.1601$ for ours, and $0.7843$ for the ES method. For $n > T/2^2$, our method consistently yields the highest coverage, showing robustness and accuracy with larger samples.

\begin{table}[t!]
\caption{Comparison of empirical coverage rates between our method and the batch mean method in the inventory control problem. 
The standard errors of coverage rates $\widehat{p}$ are reported in the bracket (computed via $\sqrt{\widehat{p}(1-\widehat{p})/500} \times 100 \%$). 
} 
\centering 
\small
\begin{tabular}{c|c|ccccc} 
\toprule
Cases &\multicolumn{1}{c|}{Methods} & $n=T/2^4$ & $n=T/2^3$ & $n=T/2^2$ & $n=T/2$ & $n=T$\\
\midrule \multirow{3}{*}{$\tau=1$}&
Ours & 84.4(1.62)&89.2(1.39)&86.6(1.52)&89.2(1.39)&93.6(1.09)\\ 
& ES& 74.4(1.95)&81.6(1.73)&83.2(1.67)&83.8(1.65)&88.0(1.45)\\ 
& IBS& 96.0(0.88)&92.6(1.17)&86.4(1.53)&90.0(1.34)&91.6(1.24)\\ 
\midrule
\multirow{3}{*}{$\tau=2$}&
 Ours & 87.2(1.49)&90.2(1.33)&91.8(1.23)&91.8(1.23)&92.6(1.17)\\ 
&ES& 78.2(1.85)&83.6(1.66)&88.0(1.45)&88.8(1.41)&88.0(1.45)\\ 
&IBS& 93.0(1.14)&89.0(1.4)&86.2(1.54)&81.8(1.73)&84.4(1.62)\\ 
\bottomrule
\end{tabular}
\label{table:empirical_coverage_inv} 
\end{table}

%% file: Journal_version/online_comput_conf_intv.tex
{
In this section, we discuss how to compute the confidence intervals of $\vtheta$ efficiently in an online manner. This requires incremental updates of both $\bar{\vartheta}_t$ and $\bar{\sigma}_t$. The moving average $\bar{\vartheta}_t$ can be computed directly using
\[
\bar{\vartheta}_{t+1} = \frac{t}{t+1}\bar{\vartheta}_t + \frac{1}{t+1}\vartheta_{t+1}.
\] 
For $\bar{\sigma}_t$, the term $\sum_{i=1}^t i^2 |v^\top(\bar{\vartheta}_i-\bar{\vartheta}_t)|^2$ is decomposed into individual components that allow incremental updates:  
\begin{equation*}
\sum_{i=1}^t i^2 |v^\top(\bar{\vartheta}_i-\bar{\vartheta}_t)|^2 
= \underbrace{\sum_{i=1}^t i^2 (v^\top\bar{\vartheta}_i)^2}_{O_{t,1}}
- 2 \underbrace{\sum_{i=1}^t i^2 v^\top\bar{\vartheta}_i}_{O_{t,2}}
\cdot  (v^\top\bar{\vartheta}_t) 
+ \underbrace{\sum_{i=1}^t i^2}_{O_{t, 3}} \cdot (v^\top\bar{\vartheta}_t)^2.
\end{equation*}
Each $O_{t,j}$ can be updated to $O_{t+1,j}$ in $O(1)$ time using:  
\begin{equation}
\label{eq:O-update}
    O_{t+1,1} = O_{t,1} + (t+1)^2(v^\top\bar{\vartheta}_{t+1})^2,~~
    O_{t+1,2} = O_{t,2} + (t+1)^2(v^\top\bar{\vartheta}_{t+1}),~~
    O_{t+1,3} = O_{t,3} + (t+1)^2.
\end{equation}
With these updates, transitioning from $\bar{\sigma}_t$ to $\bar{\sigma}_{t+1}$ requires minimal computation and memory, relying on three one-dimensional sequences $\{O_{t,j}\}_{j \in [3]}$ and the average $\bar{\sigma}_t$. This approach is summarized in Algorithm~\ref{alg:inf}. }
\begin{algorithm}[ht]
\small
\caption{\textsc{Stream SGD with Online Confidence Intervals}}
\label{alg:inf}
\begin{algorithmic}
\STATE \textbf{Input:} Step sizes $\{\eta_t\}$, linear coefficient $v$, initial state $x_0$, and parameter $\vartheta_0$.
\FOR {Iteration $t \geq 1$}
    \STATE Sample $x_t$ from the transition kernel $P_{\vartheta_t}(x_{t-1}, \cdot)$.
    \STATE Update $\vartheta_{t+1} \leftarrow \Pi_\Theta(\vartheta_t - \eta_t H(\vartheta_t, x_t))$, where $\Pi_\Theta$ is the projection operator.
    \STATE Compute $\bar{\vartheta}_{t+1} = \frac{t}{t+1}\bar{\vartheta}_t + \frac{1}{t+1}\vartheta_{t+1}$.
    \STATE Compute the value of the inner product $v^\top \bar{\vartheta}_{t+1}$.
    \STATE Perform the updates in~\eqref{eq:O-update} to obtain $\{O_{t+1, j}\}_{j \in [3]}$.
    \STATE Compute $\bar{\sigma}_{t+1} = \frac{1}{t+1} \sqrt{O_{t+1,1} - 2 O_{t+1,2} \cdot (v^\top \bar{\vartheta}_{t+1}) + O_{t+1,3} \cdot (v^\top \bar{\vartheta}_{t+1})^2}$.
    \STATE Output the current confidence interval: $\left[
    v^\top \bar{\vartheta}_{t+1} - q_{\alpha, h} \bar{\sigma}_{t+1}, 
    v^\top \bar{\vartheta}_{t+1} + q_{\alpha, h} \bar{\sigma}_{t+1}
    \right].$
\ENDFOR
\end{algorithmic}
\end{algorithm}

\begin{table}[t]
\centering 
\caption{Asymptotic critical values $q_{\alpha, h}$ with $h$ set to~\eqref{eq:h-pivotal}, are computed via simulations. The Brownian motion $W$ is approximated by normalized sums of i.i.d.\ $\NM(0, 1)$ random variables using 1,000 steps and 50,000 replications.
} 
\label{table:critical} 
\begin{tabular}{c|ccccccccc} 
\toprule
$1-\alpha$& $1\%$ & $2.5\%$ & $5\%$ & $10\%$ & $50\%$ & $90\%$ & $95\%$ & $97.5\%$ &$99\%$ \\
\midrule		
$q_{\alpha, h}$&  -8.628&-6.758&-5.316&-3.873& 0.000 & 3.873& 5.316& 6.758 & 8.628\\
\bottomrule
\end{tabular}
\end{table}

%% file: Journal_version/continuous_Poisson.tex
\subsection{Proofs for Continuity of the Solution to Poisson's Equation}\label{sec:prf_sketch}
In this section, our main goal is to prove the Lipschitz continuity of Poisson's equation solution $U_{P_{\vartheta}H}$ for the equation \eqref{eq:poisson} with $h(\cdot) = P_{\vartheta}H(\vartheta, \cdot)$. 
We use a slightly difference notation to denote the expectation for simplicity:
\[
\inner{\pi_\vartheta}{f} = \EB_{x \sim \pi_{\vartheta} }f(x) = \int_{\mathcal{X}} f(x) \rd\pi_{\vartheta}(x).
\]
The proof of its existence highly depends on the Neumann expansion~\citep{werner2011funktionalanalysis}, i.e.,
\begin{equation}
\label{eq:equi-def-U}
        U_f(\vartheta, x)
        = (I - \gP_{\vartheta})^{-1}(f(x) - \inner{\pi_\vartheta}{f})
        = \ssum{n}{0}{\infty} \{\gP^n_{\vartheta}f(x) - \inner{\pi_\vartheta}{f}\}
\end{equation}
is a solution to the Poisson's equation in \eqref{eq:poisson}.
From the last equation, one can also find that
\[
U_{P_{\vartheta} f} = P_{\vartheta}U_{ f}.
\]
For any $V\in\gV$ and metric $\rho(x,y) \in \big\{\norm{x-y}, \mathbbm{1}\{x\neq y\}\big\}$ in the space $\gX$, we define a function class $\gF_{\rho, V}$  as
    \[
    \gF_{\rho, V} := \left\{f: \abs{f(x) - f(y)} \leq \rho(x,y)(2+V(x)+V(y)),~\forall x,y \in \gX\right\}.
    \]
    We will frequently use the function class $\gF_{\rho, V}$ in the following proof.
By the weak duality of Wasserstein distance, the Wasserstein-style divergence $W_{\rho, V}$ (defined in \eqref{eq:def_w_dist}) can be lower bounded by 
\begin{equation*}
\gW_{\rho,V}(\mu,\nu) \ge \sup_{f \in \gF_{\rho, V}} \left\{ \EB_{X\sim \mu} [f(X)] - \EB_{Y\sim \nu} [f(Y)]\right\},~~\text{for}~~\rho \in \{\norm{\cdot},\mathbbm{1}\}.
\end{equation*}
When verifying the conditions for two OR applications, we specify the function $\rho(x, y)$, which is chosen as either the Euclidean norm $\norm{x-y}$ or the indicator function $\1\{x \neq y\}$. For simplicity, we denote the function class $\gF_{\rho, V}$ with $\rho(x, y) = \norm{x-y}$ as $\gF_{\norm{\cdot}, V}$ and the function class with $\rho(x, y) = \mathbbm{1}\{x \neq y\}$ as $\gF_{\mathbbm{1}, V}$.

In the rest of this section, we will prove Lemma~\ref{lem:reg_poisson_eq}. The proof structure involves using Lemmas~\ref{lem:prop_p^mf} -- \ref{lem:cont_Pnf-pi_f} to establish bounds for the key term:  
\[
\abs{
    \{\gP_{\vartheta_1}^n f(x) - \inner{\pi_{\vartheta_1}}{f}\} - \{\gP_{\vartheta_2}^n f(x) - \inner{\pi_{\vartheta_2}}{f}\}
}
\]
for any test function \( f \in \gF_{\norm{\cdot},V} \). Finally, by leveraging the Neumann expansion, we prove the continuity of the Poisson equation solution in Lemma~\ref{lem:cont_sol_poisson_eq}.

\begin{lemma}\label{lem:prop_p^mf}
For any choice of $\rho$ (either $\rho(x, y) = \norm{x - y}$ or $\rho(x, y) = \mathbbm{1}{x \neq y}$), the following result holds.
Given any function $f \in \gF_{\rho, V}$, there exist positive constants $c$ and $C$, depending on $\rho$ and $V$, such that for all $m \in \NB$ and $\vartheta \in \Theta$, we have:
  \[
  \frac{e^{cm}}{C}\gP_{\vartheta}^m f(\cdot) \in \gF_{\rho, {V^2}}.
  \]
\end{lemma}
\begin{proof}{Proof of Lemma~\ref{lem:prop_p^mf}}
    The proof of Lemma~\ref{lem:prop_p^mf} can be found in~\ref{proof:lem:prop_p^mf}.
\end{proof}

\begin{lemma}\label{lem:prop_L_mf}
    For any $m \in \NB$, any $f \in \gF_{\norm{\cdot}, V}$ and any $\vartheta_1, \vartheta_2 \in \Theta$ with $\vartheta_1 \neq \vartheta_2$, there exist positive constants $c$ and $C$ such that:
    \[
    \abs{(\gP_{\vartheta_1} - \gP_{\vartheta_2})\gP_{\vartheta_2}^m f(x)} \leq C\norm{\vartheta_1 - \vartheta_2}e^{-cm}\big(1 + {V}(x)^2\big).
    \]
    As a direct corollary, we also have:
    \[
    \frac{e^{cm}}{C\norm{\vartheta_1 - \vartheta_2}} (\gP_{\vartheta_1} - \gP_{\vartheta_2})\gP_{\vartheta_2}^m f(\cdot) \in \gF_{\mathbbm{1}, {V^2}}.
    \]
\end{lemma}
\begin{proof}{Proof of Lemma~\ref{lem:prop_L_mf}}
    The proof of Lemma~\ref{lem:prop_L_mf} can be found in~\ref{proof:lem:prop_L_mf}.
\end{proof}

\begin{lemma}\label{lem:prop_pl_Lm}
    For any $l,m \in \NB$, any $f \in \gF_{\norm{\cdot},V}$, and any $\vartheta_1, \vartheta_2 \in \Theta$ with $\vartheta_1 \neq \vartheta_2$, there exist positive constants $c$ and $C$ such that:
    \begin{equation*}
    \begin{aligned}
    &\abs{\gP_{\vartheta_1}^l(\gP_{\vartheta_1} - \gP_{\vartheta_2})\gP_{\vartheta_2}^m f(x) - \inner{\pi_{\vartheta_1}}{\gP_{\vartheta_1}^l(\gP_{\vartheta_1} - \gP_{\vartheta_2})\gP_{\vartheta_2}^m f(\cdot)}}\\
    &\qquad \leq CL\norm{\vartheta_1 - \vartheta_2}e^{-c(l + m)}\big(2 + {V}(x)^4 + \EB_{X\sim\pi_{\vartheta_1}}{V}(X)^4\big).
    \end{aligned}
    \end{equation*}
\end{lemma}
\begin{proof}{Proof of Lemma~\ref{lem:prop_pl_Lm}}
    The proof of Lemma~\ref{lem:prop_pl_Lm} can be found in~\ref{proof:lem:prop_pl_Lm}.
\end{proof}

\begin{lemma}\label{lem:lip_pi}
Recall that $\pi_{\vartheta}$ is the invariant (or stationary) distribution of the transition kernel $\gP_{\vartheta}$. 
There exits a positive number $C_1$ such that
    \begin{equation*}
    \gW_{\norm{\cdot},V^2}(\pi_{\vartheta_1}, \pi_{\vartheta_2}) \leq C_1 L \norm{\vartheta_1 - \vartheta_2}\left(2 + \inner{\pi_{\vartheta_1}}{V^4} + \inner{\pi_{\vartheta_2}}{V^4}\right).
    \end{equation*}
\end{lemma}

\begin{proof}{Proof of Lemma~\ref{lem:lip_pi}}
    The proof of Lemma~\ref{lem:lip_pi} can be found in~\ref{proof:lem:lip_pi}.
\end{proof}

Recall the expansion of the solution to Poisson's equation~\eqref{eq:equi-def-U}. The following Lemma characterizes the continuity of every difference term with respect to $\vartheta$.

\begin{lemma}\label{lem:cont_Pnf-pi_f}
    For any function $f\in\gF_{\norm{\cdot},V}$, we can find a positive number $C_2$ satisfying
    \begin{equation*}
    \begin{aligned}
    &\abs{
    \{\gP_{\vartheta_1}^n f(x) {-} \inner{\pi_{\vartheta_1}}{f}\} {-} \{\gP_{\vartheta_2}^n f(x) {-} \inner{\pi_{\vartheta_2}}{f}\}
    }\\ 
    &\quad\quad\quad\quad\quad{\leq}  C_2L \norm{\vartheta_1 - \vartheta_2}ne^{-cn}\left(1 {+} {V}(x)^4 {+} \inner{\pi_{\vartheta_1}{+} \pi_{\vartheta_2}}{V^4}\right).
    \end{aligned}
    \end{equation*}
\end{lemma}
\begin{proof}{Proof of Lemma~\ref{lem:cont_Pnf-pi_f}}
    The proof of Lemma~\ref{lem:cont_Pnf-pi_f} can be found in~\ref{proof:lem:cont_Pnf-pi_f}.
\end{proof}

\begin{lemma}[Continuity of the solution to Poisson's equation]\label{lem:cont_sol_poisson_eq}
    For any function $f\in \gF_{\norm{\cdot},V}$, 
    \[
    \abs{U_f(\vartheta_1, x) - U_f(\vartheta_2,x)}
    \leq C_3 L \norm{\vartheta_1 - \vartheta_2}\left(1+ {V}(x)^4 + \inner{\pi_{\vartheta_1} + \pi_{\vartheta_2}}{V^4}\right)
    \]
\end{lemma}
\begin{proof}{Proof of Lemma~\ref{lem:cont_sol_poisson_eq}}
    The proof of Lemma~\ref{lem:cont_sol_poisson_eq} can be found in~\ref{proof:lem:cont_sol_poisson_eq}.
\end{proof}

\subsubsection{Proof of Lemma~\ref{lem:prop_p^mf}}
\label{proof:lem:prop_p^mf}
\begin{proof}{Proof of Lemma~\ref{lem:prop_p^mf}}
    For any $f\in \gF_{\rho, V}$ and any $x,y \in \gX$, by Assumption~\ref{ass:W-contraction}, we have
    \begin{equation*}
    \begin{aligned}
    \gP_{\vartheta}^mf(x) - \gP_{\vartheta}^mf(y)
    &= \EB_{X \sim \delta_x \gP_{\vartheta}^m}f(X) - \EB_{X \sim \delta_y \gP_{\vartheta}^m}f(X)\\
    &\leq \gW_{\rho, V}(\delta_x \gP_{\vartheta}^m, \delta_y \gP_{\vartheta}^m)\\
    &\leq Ce^{-cm}\cdot \rho(x,y)\big(2+{V}(x)^2+ {V}(y)^2\big),
    \end{aligned}
    \end{equation*}
    which yields
    \[
    \frac{e^{cm}}{C}\gP_{\vartheta}^m f(x) - 
    \frac{e^{cm}}{C}\gP_{\vartheta}^m f(y) \leq \rho(x,y)\big(2+{V}(x)^2 + {V}(y)^2\big).
    \]
    Then the lemma holds because $\frac{e^{cm}}{C}\gP_{\vartheta}^m f \in \gF_{\rho, {V^2}}$.
\end{proof}
\subsubsection{Proof of Lemma~\ref{lem:prop_L_mf}}
\label{proof:lem:prop_L_mf}
\begin{proof}{Proof of Lemma~\ref{lem:prop_L_mf}}
    By Lemma~\ref{lem:prop_p^mf}, we know that $\frac{e^{cm}}{C}\gP_{\vartheta_2}^m f \in \gF_{\norm{\cdot}, V^2}$. 
    Furthermore, by making use of Assumption~\ref{ass:continue_P}, we obtain that
    \begin{equation*}
    \begin{aligned}
    \abs{\frac{e^{cm}}{C}(\gP_{\vartheta_1} - \gP_{\vartheta_2})\gP_{\vartheta_2}^m f(x)}
    &= \abs{\EB_{X\sim \delta_x \gP_{\vartheta_1}}\frac{e^{cm}}{C}\gP_{\vartheta_2}^m f(X) - \EB_{X\sim \delta_x \gP_{\vartheta_1}}\frac{e^{cm}}{C}\gP_{\vartheta_2}^m f(X)}\\
    &\leq \gW_{\norm{\cdot}, V^2}(\delta_x \gP_{\vartheta_1}, \delta_x \gP_{\vartheta_2}) \\
    &\leq L\norm{\vartheta_1 - \vartheta_2}\big(1+{V^2}(x)\big).
    \end{aligned}
    \end{equation*}
We complete the proof by multiplying both sides of the inequality above by ${Ce^{-cm}}$ and redefining $C$ as $CL$.
\end{proof}
\subsubsection{Proof of Lemma~\ref{lem:prop_pl_Lm}}
\label{proof:lem:prop_pl_Lm}
\begin{proof}{Proof of Lemma~\ref{lem:prop_pl_Lm}}
    For simplicity, we denote the function $(\gP_{\vartheta_1} - \gP_{\vartheta_2})\gP_{\vartheta_2}^m f(x)$ as $g(x)$. According to Lemma~\ref{lem:prop_L_mf}, it can be obtained that $\frac{e^{cm}}{CL\norm{\vartheta_1 - \vartheta_2}}g(\cdot) \in \gF_{\mathbbm{1},V^2}$. 
    Consequently, the following inequality is obtained by combining the properties of $g$ with Assumption~\ref{ass:W-contraction}:
    \begin{equation}\label{eq:prop_pl_Lm_eq1}
    \begin{aligned}
        \abs{\gP_{\vartheta_1}^l g(x) - \gP_{\vartheta_1}^l g(y)} &\leq CL\norm{\vartheta_1 - \vartheta_2}e^{-cm}\gW_{2,V^2}(\delta_x \gP_{\vartheta_1}^l, \delta_y \gP_{\vartheta_1}^l)\\
        &\leq CL\norm{\vartheta_1 - \vartheta_2}e^{-c(l+m)}\left(2+{V}(x)^4 + {V}(y)^4\right).
    \end{aligned}
    \end{equation}
Note that the above inequality holds for any choice of $y$. Now, let $y$ be sampled from the distribution $\pi_{\vartheta_1}$; then:
    \begin{equation*}
    \begin{aligned}
        \abs{\gP_{\vartheta_1}^l g(x) - \inner{\pi_{\vartheta_1}}{\gP_{\vartheta_1}^l g(\cdot)}}
        &=
        \abs{\EB_{y\sim \pi_{\vartheta_1}}[\gP_{\vartheta_1}^l g(x) - \gP_{\vartheta_1}^l g(y)]
        }
        \leq \EB_{y\sim \pi_{\vartheta_1}}\left[\abs{
        \gP_{\vartheta_1}^l g(x) - \gP_{\vartheta_1}^l g(y)
        }\right]\\
        &\overset{\eqref{eq:prop_pl_Lm_eq1}}{\leq}
        \EB_{y\sim \pi_{\vartheta_1}}\left[CL\norm{\vartheta_1 - \vartheta_2}e^{-c(l+m)}\big(2+ {V}(x)^4 + {V}(y)^4\big)\right]\\
        &\leq
        CL\norm{\vartheta_1 - \vartheta_2}e^{-c(l+m)}\big(2+{V}(x)^4 +\EB_{X\sim \pi_{\vartheta_1}}\big[{V}(X)^4\big]\big).
    \end{aligned}
    \end{equation*}
    We then complete the proof by plugging the definition of $g(x)$ into the above inequality.
\end{proof}
\subsubsection{Proof of Lemma~\ref{lem:lip_pi}}
\label{proof:lem:lip_pi}
\begin{proof}{Proof of Lemma~\ref{lem:lip_pi}}
It suffices to bound the value of $\abs{\inner{\pi_{\vartheta_1}}{f} - \inner{\pi_{\vartheta_2}}{f}}$ for any $f \in \gF_{\norm{\cdot},V^2}$. In fact, for any positive number $K$, it follows that
    \begin{equation*}
    \begin{aligned}
        \abs{\inner{\pi_{\vartheta_1}}{f} - \inner{\pi_{\vartheta_2}}{f}} &= 
        \abs{\inner{\pi_{\vartheta_1}}{\gP_{\vartheta_1}^Kf} - \inner{\pi_{\vartheta_1}}{\gP_{\vartheta_1}^Kf}}\\
        &\leq \abs{\inner{\pi_{\vartheta_1}}{(\gP_{\vartheta-1}^K - \gP_{\vartheta_2}^K) f}} + \abs{\inner{\pi_{\vartheta_1} - \pi_{\vartheta_2}}{\gP_{\vartheta_2}^K f}}\\
        &\overset{(a)}{\leq}
        \abs{\inner{\pi_{\vartheta_1}}{(\gP_{\vartheta-1}^K - \gP_{\vartheta_2}^K) f}}
        +
        Ce^{-cK}\gW_{\mathbbm{1},{V}^4}(\pi_{\vartheta_1}, \pi_{\vartheta_2})\\
        &\overset{(b)}{\leq} \ssum{m}{1}{K-1}\abs{
        \inner{\pi_{\vartheta_1}}{\gP_{\vartheta_1}^{K-m}(\gP_{\vartheta_1} - \gP_{\vartheta_2})\gP_{\vartheta_2}^{m-1} f}}
        +
        Ce^{-cK}\gW_{\mathbbm{1},{V^4}}(\pi_{\vartheta_1}, \pi_{\vartheta_2})\\
        &\overset{(c)}{=}\ssum{m}{1}{K-1}\abs{
        \inner{\pi_{\vartheta_1}}{(\gP_{\vartheta_1} - \gP_{\vartheta_2})\gP_{\vartheta_2}^{m-1} f}
        }+ Ce^{-cK}\gW_{\mathbbm{1},{V^4}}(\pi_{\vartheta_1}, \pi_{\vartheta_2}).
    \end{aligned}
    \end{equation*}
    Here $(a)$ holds by applying Lemma~\ref{lem:prop_p^mf} and the definition of $\gW_{\mathbbm{1},V^4}(\pi_{\vartheta_1}, \pi_{\vartheta_2})$, $(b)$ holds due to the telescoping equivalence
    \begin{equation*}
    \begin{aligned}
        \gP_{\vartheta_1}^K - \gP_{\vartheta_2}^K = \ssum{m}{1}{K-1}\left(\gP_{\vartheta_1}^{K-m+1}\gP_{\vartheta_2}^{m-1} - \gP_{\vartheta_1}^{K-m}\gP_{\vartheta_2}^{m}\right)
        = \ssum{m}{1}{K-1}\gP_{\vartheta_1}^{K-m}(\gP_{\vartheta_1} - \gP_{\vartheta_2})\gP_{\vartheta_2}^{m-1},
    \end{aligned}
    \end{equation*}
    and $(c)$ holds because $\pi_{\vartheta_1}$ is the invariant distribution of $\gP_{\vartheta_1}$.

    As for $\abs{\inner{\pi_{\vartheta_1}}{(\gP_{\vartheta_1} - \gP_{\vartheta_2})\gP_{\vartheta_2}^{m-1}f}}$, by Lemma~\ref{lem:prop_L_mf}, we have
    \[
    \abs{(\gP_{\vartheta_1} - \gP_{\vartheta_2})\gP_{\vartheta_2}^m f(x)} \leq CL\norm{\vartheta_1 - \vartheta_2}e^{-cm}\big(1 + V^{4}(x)\big).
    \]
    As a result,
    \begin{equation*}
    \begin{aligned}
    \abs{\inner{\pi_{\vartheta_1}}{(\gP_{\vartheta_1} - \gP_{\vartheta_2})\gP_{\vartheta_2}^{m-1}f}} 
    &\le \inner{\pi_{\vartheta_1}}{\abs{(\gP_{\vartheta_1} - \gP_{\vartheta_2})\gP_{\vartheta_2}^{m-1}f}}\\
    &\leq
    CL\norm{\vartheta_1 - \vartheta_2}\inner{\pi_{\vartheta_1}}{e^{-c(m-1)}\big(1+V^{4}(\cdot)\big)}\\
    &= CL\norm{\vartheta_1 - \vartheta_2}e^{-c(m-1)}\big(1 + \EB_{\pi_{\vartheta_1}}\big[V^{4}\big]\big).
    \end{aligned}
    \end{equation*}
Integrating all the above inequalities, we conclude that
\begin{equation*}
    \begin{aligned}
        \abs{\inner{\pi_{\vartheta_1}}{f} - \inner{\pi_{\vartheta_2}}{f}} 
        &\leq CL\norm{\vartheta_1 - \vartheta_2}\big(1+\EB_{\pi_{\vartheta_1}}\big[V^4\big]\big)\ssum{m}{1}{K-1}e^{-c(m-1)} + Ce^{-cK}\gW_{\mathbbm{1},V^4}(\pi_{\vartheta_1}, \pi_{\vartheta_2})\\
        &\leq
        \frac{C}{c}L\norm{\vartheta_1 - \vartheta_2}(2+\EB_{\vartheta_1}[V^{4}] + \EB_{\vartheta_2}[V^{4}]) + Ce^{-cK}\gW_{\mathbbm{1},V^4}(\pi_{\vartheta_1}, \pi_{\vartheta_2}).
    \end{aligned}
    \end{equation*}
    Note that $\gW_{\mathbbm{1},V^4}(\pi_{\vartheta_1},\pi_{\vartheta_2})$ is always upper bounded by $\inner{\pi_{\vartheta_1} + \pi_{\vartheta_2}}{1 + V^4(\cdot)}$. 
    We then complete the proof by setting the new constant $C_1 := \frac{C}{c}$, and $K \to \infty$.
\end{proof}
\subsubsection{Proof of Lemma~\ref{lem:cont_Pnf-pi_f}}\label{proof:lem:cont_Pnf-pi_f}
\begin{proof}{Proof of Lemma~\ref{lem:cont_Pnf-pi_f}}
    By using the telescoping, we have
\begin{equation*}
\begin{aligned}
        \gP_{\vartheta_1}^n f - \gP_{\vartheta_2}^n f = \ssum{m}{1}{n-1}\left(\gP_{\vartheta_1}^{n-m+1}\gP_{\vartheta_2}^{m-1} - \gP_{\vartheta_1}^{n-m}\gP_{\vartheta_2}^{m}\right)f
        = \ssum{m}{1}{n-1}\gP_{\vartheta_1}^{n-m}(\gP_{\vartheta_1} - \gP_{\vartheta_2})\gP_{\vartheta_2}^{m-1}f.
    \end{aligned}
    \end{equation*}
Taking the expectation with respect to the distribution $\pi_{\vartheta_1}$ on both sides of the above equation yields
\begin{align*}
    \inner{\pi_{\vartheta_1}}{f} - \inner{\pi_{\vartheta_1}}{\gP_{\vartheta_2}^nf} &= \inner{\pi_{\vartheta_1}}{\gP_{\vartheta_1}^n f} - \inner{\pi_{\vartheta_1}}{\gP_{\vartheta_2}^n f}\\
    &=\ssum{m}{1}{n-1}\inner{\pi_{\vartheta_1}}{\gP_{\vartheta_1}^{n-m}(\gP_{\vartheta_1} - \gP_{\vartheta_2})\gP_{\vartheta_2}^{m-1}f}.
\end{align*}
Combining these equations, we have that
\begin{equation}\label{eq:cont_Pnf-pi_f_eq1}
\begin{aligned}
    &\abs{
    \{\gP_{\vartheta_1}^n f(x) - \inner{\pi_{\vartheta_1}}{f}\} - \{\gP_{\vartheta_2}^n f(x) - \inner{\pi_{\vartheta_2}}{f}\}
    }\\
    \leq&
    \abs{
    \big(\gP_{\vartheta_1}^n f(x) - \gP_{\vartheta_2}^n f(x)\big) - \left(\inner{\pi_{\vartheta_1}}{f} - \inner{\pi_{\vartheta_1}}{\gP_{\vartheta_2}^n f(x)}\right) + \left(\inner{\pi_{\vartheta_1}}{\gP_{\vartheta_2}^n f} - \inner{\pi_{\vartheta_2}}{f}\right)
    }\\
    \leq&
    \abs{
    \big(\gP_{\vartheta_1}^n f(x) - \gP_{\vartheta_2}^n f(x)\big) - \left(\inner{\pi_{\vartheta_1}}{f} - \inner{\pi_{\vartheta_1}}{\gP_{\vartheta_2}^n f(x)}\right)
    } + \abs{\inner{\pi_{\vartheta_1} - \pi_{\vartheta_2}}{\gP_{\vartheta_2}^n f}}\\
    =& 
    \abs{
    \ssum{m}{1}{n-1} \left(
    \gP_{\vartheta_1}^l(\gP_{\vartheta_1} - \gP_{\vartheta_2})\gP_{\vartheta_2}^m f(x) - \inner{\pi_{\vartheta_1}}{\gP_{\vartheta_1}^l(\gP_{\vartheta_1} - \gP_{\vartheta_2})\gP_{\vartheta_2}^m f}
    \right)
    } + \abs{\inner{\pi_{\vartheta_1} - \pi_{\vartheta_2}}{\gP_{\vartheta_2}^n f}}\\
    \leq&
    \ssum{m}{1}{n-1}\abs{
    \gP_{\vartheta_1}^{n-m}(\gP_{\vartheta_1} - \gP_{\vartheta_2})\gP_{\vartheta_2}^m f(x) - \inner{\pi_{\vartheta_1}}{\gP_{\vartheta_1}^{n-m}(\gP_{\vartheta_1} - \gP_{\vartheta_2})\gP_{\vartheta_2}^m f}
    } + \abs{\inner{\pi_{\vartheta_1} - \pi_{\vartheta_2}}{\gP_{\vartheta_2}^n f}}.
\end{aligned}
\end{equation}
By Lemma~\ref{lem:prop_pl_Lm}, it follows that
\begin{equation*}
\begin{aligned}
\abs{
    \gP_{\vartheta_1}^{n-m}(\gP_{\vartheta_1} - \gP_{\vartheta_2})\gP_{\vartheta_2}^m f(x) - \inner{\pi_{\vartheta_1}}{\gP_{\vartheta_1}^{n-m}(\gP_{\vartheta_1} - \gP_{\vartheta_2})\gP_{\vartheta_2}^m f}
    }\\ \le CL\norm{\vartheta_1 - \vartheta_2}e^{-cn}\big(2+V^4(x)+\EB_{\pi_{\vartheta_1}}[V^4]\big).
\end{aligned}
\end{equation*}
By Lemma~\ref{lem:prop_p^mf}, we have that $\frac{e^{cn}}{C}\gP_{\vartheta_2}^n f \in \gF_{\norm{\cdot},V^{2}}$, which allows us to apply  Lemma~\ref{lem:lip_pi}. 
By Lemma~\ref{lem:lip_pi}, the following holds
\begin{equation*}
\begin{aligned}
\abs{
\inner{\pi_{\vartheta_1} - \pi_{\vartheta_2}}{\gP_{\vartheta_2}^n f} 
}
&\leq Ce^{-cn}\gW_{\norm{\cdot},V^{2}}(\pi_{\vartheta_1}, \pi_{\vartheta_2})\\
&\leq
CC_1L\norm{\vartheta_1 - \vartheta_2}e^{-cn}\left(2+\inner{\pi_{\vartheta_1}}{V^{4}} + \inner{\pi_{\vartheta_2}}{V^{4}}\right).
\end{aligned}
\end{equation*}
Plugging these into equation~\eqref{eq:cont_Pnf-pi_f_eq1} yields
\begin{equation*}
\begin{aligned}
    &\abs{
    \{\gP_{\vartheta_1}^n f(x) - \inner{\pi_{\vartheta_1}}{f}\} - \{\gP_{\vartheta_2}^n f(x) - \inner{\pi_{\vartheta_2}}{f}\}
    }\\
    \leq&
    \ssum{m}{1}{n-1}CL\norm{\vartheta_1 - \vartheta_2}e^{-cn}\big(2+V^4(x) + \inner{\pi_{\vartheta_1}}{V^4}\big) + CC_1L\norm{\vartheta_1 - \vartheta_2}e^{-cn}\big(2+\inner{\pi_{\vartheta_1} + \pi_{\vartheta_2}}{V^{4}}\big)\\
    \leq&
    C_2L \norm{\vartheta_1 - \vartheta_2}ne^{-cn}\big(1 + {V}(x)^4 + \inner{\pi_{\vartheta_1}+ \pi_{\vartheta_2}}{V^4}\big),
\end{aligned}
\end{equation*}
where the last inequality holds by setting $C_2 := 2(C + CC_1)$.
\end{proof}
\subsubsection{Proof of Lemma~\ref{lem:cont_sol_poisson_eq}}
\label{proof:lem:cont_sol_poisson_eq}
\begin{proof}{Proof of Lemma~\ref{lem:cont_sol_poisson_eq}}
    Combining Eqn.~\eqref{eq:equi-def-U} with Lemma~\ref{lem:cont_Pnf-pi_f} implies
    \begin{equation*}
    \begin{aligned}
        \abs{U_f(\vartheta_1 ,x) - U_f(\vartheta_2, x)} &=
        \abs{\ssum{n}{0}{\infty}\left\{\gP_{\vartheta_1}^n f(x) - \inner{\pi_{\vartheta_1}}{f} - \gP_{\vartheta_2}^n f(x) + \inner{\pi_{\vartheta_2}}{f}\right\}
        }\\
        &\leq \ssum{n}{0}{\infty}\abs{
        \gP_{\vartheta_1}^n f(x) - \inner{\pi_{\vartheta_1}}{f} - \gP_{\vartheta_2}^n f(x) + \inner{\pi_{\vartheta_2}}{f}
        }\\
        &\leq
        C_2L\norm{\vartheta_1 - \vartheta_2}\left(1 + {V}(x)^4 + \inner{\pi_{\vartheta_1} + \pi_{\vartheta_2}}{V^4}\right)\ssum{n}{0}{\infty}n e^{-cn}\\
        &\leq
        \frac{C_2}{c}L\norm{\vartheta_1 - \vartheta_2}\big(1 + {V}(x)^4 +\inner{\pi_{\vartheta_1} + \pi_{\vartheta_2}}{V^4}\big).
    \end{aligned}
    \end{equation*}
    Here at the last inequality, we use the following fact,
    \[
    \ssum{n}{0}{\infty}ne^{-cn} = \frac{1}{c}\ssum{n}{0}{\infty}cn\cdot e^{-cn} \leq \frac{1}{c}\int_0^\infty xe^{-x} \mathrm{d}x = \frac{1}{c}.
    \]
    We then complete the proof by setting $C_3:= \frac{C_2}{c}$. 
\end{proof}
\subsubsection{Proof of Lemma~\ref{lem:reg_poisson_eq}}
By Assumption~\ref{ass:continue_H}, we have that $\frac{1}{L}H(\vartheta, x) \in \gF_{\norm{\cdot},V}$ for any $\vartheta$. Hence, by Lemma~\ref{lem:cont_sol_poisson_eq} we can obtain
\[
\abs{U_{\frac{1}{L}P_{\vartheta_1}H}(\vartheta_1, x) - U_{\frac{1}{L}P_{\vartheta_2}H}(\vartheta_2,x)}
    \leq C_3 L \norm{\vartheta_1 - \vartheta_2}\left(1+ {V}(x)^4 + \inner{\pi_{\vartheta_1} + \pi_{\vartheta_2}}{V^4}\right).
\]
On the other hand, it is straightforward to verify the homogeneity of the solution to Poisson's equation with respect to the test function \( f \). Specifically, for any \( f \in \gF_{\norm{\cdot},V} \) and any constant \( c \), we have:  
\[
U_{cf}(\vartheta, x) = c U_f(\vartheta, x).
\]  
Therefore, Lemma~\ref{lem:reg_poisson_eq} follows by substituting \( f \) and \( c \) in the above equation with \( P_{\vartheta_1}H \) and \( \frac{1}{L} \), respectively.

%% file: Proof_for_Convergence_Result.tex
\subsection{Proofs for Convergence Results}\label{sec:prf_for_conv_results}
\input{convergence}

%% file: convergence.tex
\def\bG{{ \overline{G} }}
\def\tconst{{\widetilde{C}}}



Lemma~\ref{lem:reg_poisson_eq} provides the foundation for the following lemma, which incorporates the most critical technical conditions, particularly those concerning $U_H(\vartheta, x)$, in the convergence rate proof and the FCLT proof.

\begin{lemma}[Continuity and bounded moments]
\label{assump:U}
Suppose we utilize the stream SGD method to determine the optimal parameter choice $\vartheta^*$ for the online learning problem, and $\{x_t\}_{t=0}^\infty$ represents the data we gather during this process. Then we can establish that:
\begin{enumerate}
\item[$(a)$] For any $x \in \mathcal{X}$ and $\vartheta \in \Theta$, $\|P_{\vartheta^*}H(\vartheta^*, x) - P_{\vartheta}H(\vartheta, x)\| \le L_H(x) \cdot \|\vartheta-\vartheta^*\|$.
\item[$(b)$] For any $x \in \mathcal{X}$ and $\vartheta, \vartheta' \in \Theta$, $\|\gP_{\vartheta}U_H(\vartheta, x) - \gP_{\vartheta'}U_H(\vartheta', x)\| \le L_U(x) \cdot \|\vartheta-\vartheta'\|$.
\item[$(c)$] There exist $C_U, C_H>0$ so that for any $\alpha = 1,2$,
\[
\sup_{t \ge 1} {\EB L_H^\alpha(x_t)} \le C_H^\alpha \;
~\text{and}~ \;
\sup_{t \ge 1} {\EB L_U^\alpha(x_t)} \le C_U^\alpha.
\]
\item[$(d)$] There exist $\sigma, \sigma_U >0$ such that for any $\alpha = 1, 2$,
\[
\sup_{t \ge 0} \mathbb{E}\|H(\vartheta^*, x_t)\|^\alpha \le \sigma_H^\alpha\;
~\text{and}~ \;
\sup_{t \ge 0} \mathbb{E}\|\gP_{\vartheta^*}U_H(\vartheta^*, x_t)\|^\alpha \le \sigma_U^\alpha.
\]
\end{enumerate}
\end{lemma}
The proof of Lemma~\ref{assump:U} is deferred to Appendix~\ref{proof:assump:U}.
Lemma~\ref{assump:U} shows that $P_{\vartheta}U_H(\vartheta, x)$ is $L_U(x)$-Lipschitz w.r.t.\ $\vartheta$ for any given data point $x$.
Similarly, the $H(\vartheta, x)$ is average-$L_H(x)$-Lipschitz at the parameter $\vartheta^*$ for any given data point $x$.
Note that the Lipschitz modules $L_{H}(\cdot)$ and $L_U(\cdot)$ are allowed to depend on the data $x$ rather than a universal constant in previous works~\citep{li2022state,li2023online}.
In order to ensure stability, we require both $\EB L_H^2(x_t)$ and $\EB L_U^2(x_t)$ are uniformly bounded.
As a result, we have the following result 
\begin{lemma}
\label{lem:bounded-growth}
It follows from Assumption~\ref{assmpt: convexity} and Lemma~\ref{assump:U}  that for any $t \ge 0$,
\begin{equation*}
\begin{aligned}
\EB \| H(\vartheta_t, x_{t})\|^2 &\le 2(\sigma_H^2 + 4C_{\Theta}^2 C_H^2), \\
\EB \norm{H(\vartheta_t, x_t)}^4 &\le {8(\sigma_H^4 + 16L^4C_\Theta^4C_H^4)},\\
\EB \left\|P_{\vartheta_{t}}U_H(\vartheta_{t}, x_{t})
\right\|^2 &\le 2(\sigma_U^2 +4C_{\Theta}^2 C_U^2).
\end{aligned}
\end{equation*}
\end{lemma}
\begin{proof}{Proof of Lemma~\ref{lem:bounded-growth}}
The proof can be found in~\ref{proof:lem:bounded-growth}.
\end{proof}
Lemma~\ref{lem:bounded-growth} implies that both the incremental update $H(\vartheta_t, x_{t})$ and the conditional functions $P_{\vartheta_{t}}U_H(\vartheta_{t}, x_{t})$ have uniformly bounded (at most forth-order) moments.
The analysis of Lemma~\ref{lem:bounded-growth} heavily depends on the fact that all iterates $\{\vartheta_t\}_{t \ge 0}$ locate in $\Theta$ that is contained in a ball with radius $C_{\Theta}$.
If we allow $C_{\Theta}$ to be infinity, then we might force $L_H(x)$ and $L_H(x)$ to be constant again so as to ensure stability and convergence.




\subsubsection{Proof of Theorem~\ref{thm:L2-convergence}} \label{sec:prf_of_thm_convergence}
In the following, we will establish Theorem~\ref{thm:L2-convergence} step by step.
We first capture the one-step progress of our Markov SGD in Lemma~\ref{lem:itera}.

\begin{lemma} 
	\label{lem:itera}
	Define $\Delta_t = \EB \| \vartheta_{t} - \vartheta^* \|^2 $ for simplicity.
	Under Assumptions~\ref{assmpt: convexity}--\ref{ass:W-contraction}, it holds
	 \begin{equation}
	 		 \label{eq:square1}
	 		\Delta_{t+1} \le \left( 1-2 \eta_t K_0 \right) \Delta_t + 2(\sigma_H^2 + 4C_{\Theta}^2 C_H^2)\eta_t^2 - 2\eta_t\EB\myinner { \vartheta_t -\vartheta^* }{ H(\vartheta_t, x_t) - \nabla f (\vartheta_t) }.
	 \end{equation}
\end{lemma}
\begin{proof}{Proof of Lemma~\ref{lem:itera}}
	The proof can be found in~\ref{proof:lem:itera}.
\end{proof}

Once $\eta_t$ is sufficiently small (e.g., $\eta_t \le \frac{1}{K_0}$), we have $1-2 \eta_t K_0 \le 1- \eta_t K_0 < 1$.
It implies that the first term in the r.h.s.~of \eqref{eq:square1} is a contraction. 
To simplify notations, we introduce the scalar product $G_{m:n} = \prod_{t=m}^n (1 - 2\eta_t K_0)$, for $n > m \geq 1$, and $G_{m:n} = 1$ if $n \leq m$. 
Solving the recursion in \eqref{eq:square1} yields
\begin{equation}
\begin{aligned}
    \label{eq:iterate}
    \Delta_{t+1} & \textstyle \leq G_{1:t} \Delta_1 + 2 (\sigma_H^2 + 4C_{\Theta}^2 C_H^2) \sum_{s=1}^{t} G_{s+1:t} \eta_{s}^2 \\
    &\qquad + 2 \sum_{s=1}^{t} G_{s+1:t} \eta_{s} \EB \myinner{ \vartheta^* - \vartheta_s }{ H(\vartheta_t, x_t) - \nabla f (\vartheta_t)   } .
\end{aligned}
\end{equation}

It can be shown that the first two terms in~\eqref{eq:iterate} are bounded by ${\cal O}(\eta_k)$.
Indeed, the first term would decay exponentially fast by using the numerical inequality $1-x \le e^{-x}$.
For the second term, by the following Lemma~\ref{lem:o_gamma_k}, we immediately know that $\ \sum_{s=1}^{t} G_{s+1:t} \eta_{s}^2\le \frac{\kappa}{K_0} \eta_t$, given the step size condition in Assumption~\ref{assump:step-size}.

\begin{lemma}\label{lem:o_gamma_k}
	Let $a>0$ and $\{\eta_{t}\}_{t \geq 1}$ be a non-increasing sequence such that $\eta_{1}< 2 / a$. If $\eta_{t-1} / \eta_t \leq 1 + (a/2)\eta_t$ for any $k \geq 1$, then for any $t \geq 2$, 
	\[
	\sum_{s=1}^{t} \eta_{s}^{2} \prod_{\tau=s+1}^{t}\left(1-\eta_{\tau} a\right) \leq  \frac{2}{a} \eta_{t}.
	\]
\end{lemma}
\begin{proof}{Proof of Lemma~\ref{lem:o_gamma_k}}
		The proof can be found in~\ref{proof:lem:o_gamma_k}.
\end{proof}

In the following, we focus on analyzing the last term in~\eqref{eq:iterate} which we refer to as the cross term, and present the result in Lemma~\ref{lem:cross}.
When the samples $\{ x_{t} \}_{t \geq 1}$ are drawn according to $x_{t} \sim P_{\vartheta_{t}}(x_{t-1}, \cdot)$, to bound the expectation of this cross term in Lemma~\ref{lem:cross}, we make use of the Poison equation and decompose the gradient error $H(\vartheta_s, x_s) - g(\vartheta_s)$ into martingale and finite difference terms.
This analysis approach is inspired by~\citet{benveniste2012adaptive} and has recently been used by~\citet{atchade2017perturbed,karimi2019non,li2022state} for the finite-time convergence analysis of their interested algorithms.

\begin{lemma} \label{lem:cross}
	Under Assumptions~\ref{assump:step-size}--\ref{ass:W-contraction}, it holds
    \begin{equation*}
	\begin{aligned}
		 \sum_{s=1}^{t} &\eta_{s} G_{s+1:t} \EB \myinner{ \vartheta^* - \vartheta_s }{ H(\vartheta_s, x_s) - g(\vartheta_s)   } \\
   &\le G_{2:t} C_0 + \sum_{s=2}^{t}  \eta_{s-1}\eta_sG_{s+1:t}  \left(
	C_1 \Delta_s + C_2 \Delta_{s-1}
	\right)  + \eta_t \left(\frac{1}{2}\Delta_t + C_3\right)
	\end{aligned}
    \end{equation*}
	where we have defined the following constants for simplicity
    \begin{equation}
    \begin{aligned}
    \label{eq:cross-term-constant}
    C_0 &= \eta_1\left(
    \frac{\Delta_{1}}{2}+ \sigma_U^2 +4C_{\Theta}^2 C_U^2
    \right), \quad
    C_1 = \frac{L_U}{2}, \quad
    C_2 = \frac{\varkappa + K_0}{2}, \quad
    C_3 = C_{\star},
    \end{aligned}
    \end{equation}
    where 
    \begin{equation}\label{eq:C_star}
    C_{\star} := {\frac{\kappa}{K_0}\left(4\sigma_U^4 + 32\sigma_H^4 + 512L^4C_{\Theta}^4 C_H^4\right)} + \left(1+ \frac{\kappa(1+\varkappa + K_0)}{K_0} \right)
    \left(\sigma_U^2 +4C_{\Theta}^2 C_U^2\right).
    \end{equation}
    \end{lemma}
\begin{proof}{Proof of Lemma~\ref{lem:cross}}
	The proof can be found in~\ref{proof:lem:cross}.
\end{proof}

To proceed the proof, we substitute Lemma~\ref{lem:cross} and Lemma~\ref{lem:o_gamma_k} into~\eqref{eq:iterate} and obtain
\begin{align}
	\label{eq:Delta-final-iterate}
	\Delta_{t+1} & \leq G_{1:t} \Delta_1 + \frac{2\kappa(\sigma_H^2 + 4C_{\Theta}^2C_H^2)}{K_0} \eta_t + 2G_{2:t} C_0  + 2\eta_t \left(\frac{1}{2}\Delta_t + C_3\right) \notag \\
 &\qquad + 2\sum_{s=2}^{t}  \eta_{s-1}\eta_sG_{s+1:t}  \left(
	C_1 \Delta_s + C_2 \Delta_{s-1}
	\right) 
 \notag \\
	&\le G_{1:t} (\Delta_1 + 2C_0 ) +  \sum_{s=2}^{t}  \eta_{s-1}\eta_s G_{s+1:t}  \left(
	L_U\Delta_s + 2C_2 \Delta_{s-1}
	\right)  + \eta_t \left( \Delta_t + 4C_3\right) \notag \\
	&= G_{1:t}  \tconst_0 + \sum_{s=2}^{t}  \eta_{s-1}\eta_s G_{s+1:t} (\tconst_1 \Delta_s + \tconst_2 \Delta_{s-1}) + \eta_t (\Delta_t + 4C_3 ),
\end{align}
where we introduce constants $\{ \tconst_i \}_{0\le i \le 3}$ for notation simplicity
\begin{equation}
			\label{eq:convergence-constant}
	\tconst_0 = \Delta_1 + 2C_0,  \quad
	\tconst_1 = L_U,\quad
	\tconst_2 = 2C_2 = \varkappa + K_0,\quad
	  \tconst_3= 3C_3.
\end{equation}

\begin{lemma}
	\label{lem:convergence}
	Suppose that $\{ \Delta_t \}_{t \geq 1}$ satisfy \eqref{eq:Delta-final-iterate} and the step sizes $\{ \eta_t\}_{t \geq 1}$ satisfy Assumption~\ref{assump:step-size}. 
	It holds {(a)}
	\begin{equation}
		 \label{eq:olDelta}
			\sup_{t \geq 1} \Delta_t \leq \overline{\Delta} :=   \min \left\{ 6 \left(\Delta_1 + \kappa K_0   \sum_{s=1}^{\infty} \eta_s^2\right), 4C_{\Theta}^2 \right\}
	\end{equation}
	and {(b)} the following inequality holds for any $t \geq 1$,
	\begin{equation}
	\label{eq:bdDelta}
 \Delta_{t+1} \le
	\prod_{s=2}^t \left( 1 -K_0 \eta_s \right) \Delta_1 
		+ \eta_t \left[  \left(1 + 8 \kappa    \right) \overline{\Delta} + 18\kappa C_3\right],
	\end{equation}
where $C_3$ is given in Lemma~\ref{lem:cross}.

\end{lemma}
\begin{proof}{Proof of Lemma~\ref{lem:convergence}}
	The proof can be found in~\ref{proof:lem:convergence}.
\end{proof}

Proving the above lemma requires one to establish the stability of the system \eqref{eq:Delta-final-iterate}, which demands a sufficiently small $\eta_t$ to control the remainder term $\eta_t   \Delta_t$. Our analysis relies on the special structure of this inequality system which is also used by~\citet{li2022state}.
The convergence bound \eqref{eq:bdDelta} follows from the boundedness of $\Delta_t$ in~\eqref{eq:bdDelta}. 
Applying Lemma~\ref{lem:convergence} finishes the proof of Theorem~\ref{thm:L2-convergence}.

\subsubsection{Proof of Lemma~\ref{assump:U}}\label{proof:assump:U}

\begin{proof}{Proof of Lemma~\ref{assump:U}}
   For point \((a)\), by applying Assumption~\ref{ass:continue_H}, we have  
\[
\norm{P_{\vartheta^*}H(\vartheta^*,x) - P_{\vartheta}H(\vartheta,x)} \leq L\big(V(x) + 1\big)\norm{\vartheta - \vartheta^*}.
\]  

For point \((b)\), using Neumann's expansion, we obtain:
\begin{align*}
P_{\vartheta}U_H(\vartheta,x) &= P_{\vartheta} \sum_{n=0}^\infty P_{\vartheta}^n \big(H(\vartheta,x) - \inner{\pi_{\vartheta}}{H(\vartheta,\cdot)}\big) \\
&= \sum_{n=0}^\infty P_{\vartheta}^n \big(P_{\vartheta}H(\vartheta,x) - \inner{\pi_{\vartheta}}{P_{\vartheta}H(\vartheta,\cdot)}\big) 
= U_{P_{\vartheta}H}(\vartheta,x).
\end{align*}  

Applying Lemma~\ref{lem:reg_poisson_eq}, we then have:
\begin{align*}
\norm{P_{\vartheta}U_H(\vartheta,x) - P_{\vartheta^\prime}U_H(\vartheta^\prime,x)} &= \norm{U_{P_{\vartheta}H}(\vartheta,x) - U_{P_{\vartheta^\prime}H}(\vartheta^\prime,x)} \\
&\leq C\left(1 + V(x)^4 + \EB_{\pi_{\vartheta}+\pi_{\vartheta^\prime}}[V^4]\right)\norm{\vartheta - \vartheta^\prime}.
\end{align*}  

By taking $L_H(x)$ as $V(x)+1$, we know that
\[
\sup\limits_{t\ge 1} L_{H}^\alpha \precsim L\sup\limits_{t\ge 1}\left\{\EB V(x_t)^\alpha + 1\right\} \eqqcolon C_{H}^\alpha < \infty.
\]
Similarly, let $L_U(x) = C\left(1 + V(x)^4 + \sup\limits_{\vartheta\in \Theta}\EB_{\pi_{\vartheta}}[V^4]\right)$.
The same analysis can be used on $L_U(x_t)$.
\end{proof}

\subsubsection{Proof of Lemma~\ref{lem:bounded-growth}}
\label{proof:lem:bounded-growth}

\begin{proof}{Proof of Lemma~\ref{lem:bounded-growth}}
\begin{equation*}
\begin{aligned}
\EB \|P_{\vartheta_t} H(\vartheta_t, x_{t}) - P_{\vartheta^*} H(\vartheta^*, x_{t})  \|^2
&\le \EB \left[ L_H(x_{t-1})^2 \|\vartheta_t -\vartheta^*\|^2 \right]\\
&\overset{(a)}{\le} 4C_{\Theta}^2 \cdot \EB L_H(x_{t-1})^2 \overset{(b)}{\le} 4C_{\Theta}^2 C_H^2
\end{aligned}
\end{equation*}
where (a) uses $\|\vartheta_t -\vartheta^*\| \le 2 C_{\Theta}$ and (b) uses Lemma~\ref{assump:U}.

Similarly, we have
\begin{equation*}
\begin{aligned}
    \EB \norm{P_{\vartheta_t}H(\vartheta_t, x_t) - P_{\vartheta^*}H(\vartheta^*, x_t)}^4 &\le L^4 \EB (\norm{x_t} + 1)^4\norm{\vartheta_t - \vartheta^\star}^4\\
    &\le 16L^4C_\Theta^4\EB (\norm{x_t} + 1)^4\\
    &\le 16L^4C_\Theta^4 \EB L_H(x_t)^4 = 16L^4C_\Theta^4C_H^4,
\end{aligned}
\end{equation*}
and
\begin{equation*}
\begin{aligned}
\EB \left\|P_{\vartheta_{t}}U_H(\vartheta_{t}, x_{t})-P_{\vartheta^*}U_H(\vartheta, x_{t})
\right\|^2 
&\le \EB \left[ L_U^2(x_{t}) \|\vartheta_{t} - \vartheta^*\|^2 \right]\\
&\le 4C_{\Theta}^2 \cdot \EB L_U(x_{t})^2 \le 4C_{\Theta}^2 C_U^2.
\end{aligned}
\end{equation*}
By the same argument, we have $\EB \left\|P_{\vartheta_{t}}U_H(\vartheta_{t}, x_{t-1})-P_{\vartheta^*}U_H(\vartheta, x_{t-1})
\right\|^2  \le 4C_{\Theta}^2 C_U^2$ for any $t \ge 1$.
Combining the above results with the following inequalities yields the lemma,
\begin{align*}
    \EB\norm{H(\vartheta_t,x_t) - P_{\vartheta_t}H(\vartheta_t,x_{t-1})}^2 &\le \sigma_H^2;\\
    \EB\norm{H(\vartheta_t,x_t) - P_{\vartheta_t}H(\vartheta_t,x_{t-1})}^4 &\le \sigma_H^4;\\
    \EB\norm{P_{\vartheta^*}U_H(\vartheta^*, x_t)} \le \sigma_U^2.
\end{align*}
\end{proof}

\subsubsection{Proof of Lemma~\ref{lem:itera}}
\label{proof:lem:itera}
\begin{proof}{Proof of Lemma~\ref{lem:itera}}
	Let $\FM_t = \sigma(\{ x_{\tau}\}_{\tau \le t})$ be the $\sigma$-field generated by all randomness before iteration $t$.
	Then $\vartheta_{t+1} \in \FM_{t}$.
 By the non-expansiveness of projections, it follows that
        \begin{equation*}
            \begin{aligned}
		\EB\left\|\vartheta_{t+1}-\vartheta^*\right\|^{2} &
    = \EB\left\|\Pi \left( \vartheta_{t}-\eta_t H( \vartheta_{t}, x_{t} ) \right)  - \vartheta^*
 \right\|^{2} 
  \le \EB\left\|\vartheta_{t}-\vartheta^*-\eta_t H( \vartheta_{t}, x_{t} )\right\|^{2} \\
		& = \EB\left\|\vartheta_{t}-\vartheta^*\right\|^{2}- \, 2 \eta_{t} \EB\underbrace{ \langle\vartheta_t -\vartheta^*, H( \vartheta_t, x_{t} ) \rangle } _{=:B_1} + \eta_{t}^2 \underbrace{ \EB \left\|  H(\vartheta_t, x_t)  \right\|^{2} }_{ =: B_2 }.
	\end{aligned}
        \end{equation*}
	The inner product can be lower bounded as
	\begin{equation*}
		\begin{aligned}
			B_{1}&= \EB \myinner{ \vartheta_t -\vartheta^*}{  H(\vartheta_t, z_t) }\\
			&= \myinner{  \vartheta_t -\vartheta^* } { \nabla f (\vartheta_t) } + \EB\myinner { \vartheta_t -\vartheta^* }{  H(\vartheta_t, x_t) - \nabla f (\vartheta_t) } \\
			&\ge K_0 \cdot \EB \|\vartheta_t -\vartheta^*\|^2
			+ \EB\myinner { \vartheta_t -\vartheta^* }{  H(\vartheta_t, x_t) - \nabla f (\vartheta_t) }.
		\end{aligned}
	\end{equation*}
	where the last inequality is due to the $\mu$-strong convexity of $g(\cdot)$.

	Furthermore, by Lemma~\ref{lem:bounded-growth},
        $
        B_2 = \EB \| H(\vartheta_t, x_{t}) \|^2 \le 2(\sigma_H^2 + 4C_{\Theta}^2 C_H^2).
        $
	Combing the bounds for $B_1$ and $B_{2}$, we can get the desired inequality.
	\begin{equation*}
		\begin{aligned}
			\EB\left\|\vartheta_{t+1}-\vartheta^*\right\|^2 
			&\leq\EB \left\|\vartheta_{t}-\vartheta^*\right\|^2+ 2(\sigma_H^2 + 4C_{\Theta}^2 C_H^2)\eta_t^2  \\
			&\quad -2\eta_t \left( K_0 \cdot \EB \|\vartheta_t -\vartheta^*\|^2
			+ \EB\myinner { \vartheta_t -\vartheta^* }{ H(\vartheta_t, x_t) - \nabla f (\vartheta_t)  } \right)\\
			&\le \left( 1-2 \eta_t K_0\right)\EB \left\|\vartheta_{t}-\vartheta^*\right\|^2 + 2(\sigma_H^2 + 4C_{\Theta}^2 C_H^2)\eta_t^2  \\
			&\qquad - 2\eta_t\EB\myinner { \vartheta_t -\vartheta^* }{ H(\vartheta_t, x_t) - \nabla f (\vartheta_t)  } .
		\end{aligned}
	\end{equation*}
\end{proof}

\subsubsection{Proof of Lemma~\ref{lem:o_gamma_k}}
\label{proof:lem:o_gamma_k}
\begin{proof}{Proof of Lemma~\ref{lem:o_gamma_k}} It follows that
\begin{equation*}
\begin{aligned}
	\sum_{j=1}^{t} \eta_{j}^{2} &\prod_{s=j+1}^{t}\left(1-\eta_{s} a\right)   = \eta_t\sum_{j=1}^{t} \eta_{j} \prod_{s=j+1}^{t} \frac{ \eta_{s-1} }{ \eta_s } \left(1-\eta_{s} a\right) \\
		& \leq \eta_t\sum_{j=1}^{t} \eta_{j} \prod_{s=j+1}^{t} (1 + (a/2) \eta_s ) \left(1-\eta_{s} a\right) 
		 \leq \eta_t\sum_{j=1}^{t} \eta_{j} \prod_{s=j+1}^{t} \left(1-\eta_{s} (a/2) \right) \\
		& = \frac{2 \eta_t}{a} \sum_{j=1}^{t} \left( \prod_{s=j+1}^{t} (1 - \eta_s a/2 ) - \prod_{s'=j}^{t} (1 - \eta_{s'} a/2 ) \right) 
		 = \frac{2 \eta_t}{a} \left( 1 - \prod_{s'=1}^{t} (1 - \eta_{s'} a/2 ) \right) \leq \frac{2 \eta_t}{a}.
\end{aligned}
\end{equation*}
\end{proof}

\subsubsection{Proof of Lemma~\ref{lem:cross}}
\label{proof:lem:cross}
\begin{proof}{Proof of Lemma~\ref{lem:cross}}
It follows from $x_t \sim P_{\vartheta_t}(x_{t-1}, \cdot)$ that
    \[
    \EB \myinner{\vartheta_t-\vartheta^*}{U_H(\vartheta_t, x_t)} = \EB \myinner{\vartheta_t-\vartheta^*}{\EB[U_H(\vartheta_t, x_t)|\FM_{t-1}]} = \EB \myinner{\vartheta_t-\vartheta^*}{P_{\vartheta_{t}}U_H(\vartheta_t, x_{t-1})}
    \]
   
	By applying the Poisson equation, we can simplify the sum of inner products into
    \begin{equation*}
        \begin{aligned}
        \lefteqn{\sum_{s=1}^{t} G_{s+1:t} \; \eta_{s} \EB [\myinner{ \vartheta - \vartheta_s }{  H(\vartheta_s, x_s) - g(\vartheta_s) } ] } \\
        &=
        \sum_{s=1}^{t} G_{s+1:t} \; \eta_{s} \EB [\myinner{ \vartheta^* - \vartheta_s }{ P_{\vartheta_{s}}U_H(\vartheta_s, x_{s-1})- P_{\vartheta_s} U_H(\vartheta_s, x_s)   } ].
        \end{aligned}
        \end{equation*}
	In the following, the key idea is to decompose 
	\[
	\sum_{s=1}^{t} G_{s+1:t} \eta_{s} \EB \myinner{ \vartheta^* - \vartheta_s }{ P_{\vartheta_{s}}U_H(\vartheta_s, x_{s-1})- P_{\vartheta_s} U_H(\vartheta_s, x_s)   } ] = A_1 + A_2 + A_3 + A_4
	\]
	into four sub-terms $\{A_i\}_{i \in [4]}$ which can be bounded using the smoothness assumption.
	In particular, we have 
        \begin{equation*}
	\begin{aligned}
		&A_{1}:=-\sum_{s=2}^{t}\eta_{s}G_{s+1:t} \EB \myinner{ \vartheta_{s}-\vartheta^* }{
			P_{\vartheta_{s}}U_H(\vartheta_s, x_{s-1})-P_{\vartheta_{s-1}}U_H(\vartheta_{s-1}, x_{s-1})
		}, \\
		&A_{2}:=-\sum_{s=2}^{t}\eta_{s}G_{s+1:t} \EB\myinner{ \vartheta_{s}-\vartheta_{s-1} } { 
			P_{\vartheta_{s-1}}U_H(\vartheta_{s-1}, x_{s-1})
		}, \\
		&A_{3}:=-\sum_{s=2}^{t}(\eta_{s}G_{s+1:t}-\eta_{s-1}G_{s:t}) \EB \myinner{ \vartheta_{s-1}-\vartheta^* }{ P_{\vartheta_{s-1}}U_H(\vartheta_{s-1}, x_{s-1})
		}, \\
		&A_{4}:=-\eta_{1}G_{2:t}\EB\myinner{ \vartheta_{1}-\vartheta^* } { 
			P_{\vartheta_{1}}U_H(\vartheta_{1}, x_{0})
		} +	
		\eta_{t}  \EB\myinner{ \vartheta_{t}-\vartheta^* }{ P_{\vartheta_{t}}U_H(\vartheta_{t}, x_{t}) }.
	\end{aligned}
        \end{equation*}
	
	For $A_{1}$, as $\vartheta_{t+1}=\vartheta_{t}-\eta_{t} H(\vartheta_t, x_t)$, we get $\vartheta_{s}-\vartheta_{s-1}=-\eta_{s-1} H(\vartheta_{s-1}, x_{s-1})$. Applying the smoothness condition of $P_{\vartheta} U_H(\vartheta, x)$ in $\vartheta$, we have that
    \begin{align*}
		A_{1}&=-\sum_{s=2}^{t}\eta_{s}G_{s+1:t}\EB \myinner{ \vartheta_{s}-\vartheta^* } { 
			P_{\vartheta_{s}}U_H(\vartheta_{s}, x_{s-1})-
			P_{\vartheta_{s-1}}U_H(\vartheta_{s-1}, x_{s-1})
		} \notag \\
		&
		\leq  \sum_{s=2}^t \eta_{s} G_{s+1:t} \EB \left[\left\|\vartheta_{s}-\vartheta^*\right\| L_U(x_{s-1}) \left\|\vartheta_{s}-\vartheta_{s-1}\right\| \right] \notag
		\\
		&= \sum_{s=2}^t  \eta_{s-1} \eta_{s} G_{s+1:t} \EB \left[\left\|\vartheta_{s}-\vartheta^*\right\| L_U(x_{s-1}) \left\| H(\vartheta_{s-1}, x_{s-1}) \right\| \right] \notag \\
		&\overset{(a)}{\le} \frac{1}{2}\sum_{s=2}^t  \eta_{s-1} \eta_{s} G_{s+1:t}  \left(  
		\EB\|\vartheta_{s}-\vartheta^*\|^2 + \EB \left[L_U(x_{s-1})^2\left\| H(\vartheta_{s-1}, x_{s-1})\right\| ^2 \right]
		\right) \notag \\
		&  {\le} \frac{1}{2} \sum_{s=2}^t  \eta_{s-1} \eta_{s} G_{s+1:t} 
		\left( \Delta_s + 4(\sigma_U^4 + \EB\norm{H(\vartheta_{s-1}, x_{s-1})}^4 \right)\nonumber \\
            & \overset{(b)}{\le}  \frac{1}{2} \sum_{s=2}^t  \eta_{s-1} \eta_{s} G_{s+1:t} 
		\left( \Delta_s + 4(\sigma_U^4 + 8(\sigma_H^4 + 16L^4C_\Theta^4C_H^4) \right).
        \end{align*}
	where (a) uses $ab \leq \frac{1}{2}(a^2+ b^2) $ for any $a,b\in \RB$ and (b) uses Lemma~\ref{lem:bounded-growth}.

	For $A_{2}$, we observe that
	\begin{align*}
		\label{eq:A2}
		A_{2}&=-\sum_{s=2}^{t}\eta_{s}G_{s+1:t} \EB \myinner{ \vartheta_{s}-\vartheta_{s-1} } { 
			P_{\vartheta_{s-1}}U_H(\vartheta_{s-1}, x_{s-1})
		} \notag \\ 
		&\leq \sum_{s=2}^{t}\eta_{s} G_{s+1:t}\EB\left\|\vartheta_{s}-\vartheta_{s-1}\right\|\cdot \left\|
		P_{\vartheta_{s-1}}U_H(\vartheta_{s-1}, x_{s-1})
		\right\|  \notag \\
		&= \sum_{s=2}^{t}\eta_{s} \eta_{s-1} G_{s+1:t}\EB\left\| H(\vartheta_{s-1}, x_{s-1})\right\|\cdot \left\|
		P_{\vartheta_{s-1}}U_H(\vartheta_{s-1}, x_{s-1})
		\right\|  \notag \\
		&\le \frac{1}{2} \sum_{s=2}^{t}\eta_{s} \eta_{s-1} G_{s+1:t}
		\left(
		\EB	\left\| H(\vartheta_{s-1}, x_{s-1})\right\|^2 + \EB \left\|
		P_{\vartheta_{s-1}}U_H(\vartheta_{s-1}, x_{s-1})
		\right\|^2
		\right)   \notag \\
		&\overset{(a)}{\le}  \sum_{s=2}^{t}\eta_{s} \eta_{s-1} G_{s+1:t}
		\left(
		\sigma_H^2 + 4C_{\Theta}^2 C_H^2 + \sigma_U^2 +4C_{\Theta}^2 C_U^2
		\right)  
	\end{align*}
	where (a) uses Lemma~\ref{lem:bounded-growth}.
	
	For $A_{3}$, notice that $|\eta_{s}-\eta_{s-1}| \le \eta_s \eta_{s-1} \varkappa$ and thus
        \begin{equation*}
	\begin{aligned}
		| \eta_{s}G_{s+1:t}-\eta_{s-1}G_{s:t} | 
		&= | \eta_s - \eta_{s-1} (1-K_0\eta_s)| \cdot G_{s+1:t}\\
		&\le \left( |\eta_s - \eta_{s-1}| + \eta_{s-1}\eta_s K_0
		\right)\cdot G_{s+1:t} 
		\le  \left(  \varkappa + K_0 \right) \eta_{s-1}\eta_s G_{s+1:t}.
	\end{aligned}
        \end{equation*}
	Using this inequality, we have that
        \begin{equation}
	\begin{aligned}
		\label{eq:A3}
		A_{3}&=-\sum_{s=2}^{t}(\eta_{s}G_{s+1:t}-\eta_{s-1}G_{s:t}) \EB \myinner{ \vartheta_{s-1}-\vartheta^* }{ P_{\vartheta_{s-1}}U_H(\vartheta_{s-1}, x_{s-1})
		} \notag \\
		&\leq \sum_{s=2}^{t} | \eta_{s}G_{s+1:t}-\eta_{s-1}G_{s:t} | \cdot  \left\|\vartheta_{s-1}-\vartheta^* \right\|\cdot \left\| P_{\vartheta_{s-1}}U_H(\vartheta_{s-1}, x_{s-1}) \right\|   \notag \\
		&\le  \left(  \varkappa + K_0 \right)  \sum_{s=2}^{t} \eta_{s-1}\eta_s G_{s+1:t} \EB \left\|\vartheta_{s-1}-\vartheta^* \right\|\cdot \left\| P_{\vartheta_{s-1}}U_H(\vartheta_{s-1}, x_{s-1}) \right\|   \notag \\
		&\le \left(  \varkappa + K_0 \right)  \sum_{s=2}^{t} \eta_{s-1}\eta_s G_{s+1:t}
  \left( \frac{1}{2}\EB\left\|\vartheta_{s-1}-\vartheta^* \right\|^2 + \frac{1}{2} \EB \left\| P_{\vartheta_{s-1}}U_H(\vartheta_{s-1}, x_{s-1}) \right\|^2    \right)
  \notag \\
		&\le  \left(  \varkappa + K_0 \right)  \sum_{s=2}^{t} \eta_{s-1}\eta_s G_{s+1:t}  \left(
		\frac{1}{2} \Delta_{s-1} + \sigma_U^2 +4C_{\Theta}^2 C_U^2
		\right),
	\end{aligned}
        \end{equation}
 where the last inequality again uses Lemma~\ref{lem:bounded-growth}.
 
	Finally, for $A_{4}$, we have 
        \begin{equation}
	\begin{aligned}
        \label{eq:A4}
        A_{4}&=-\eta_{1}G_{2:t}\EB\myinner{ \vartheta_{1}-\vartheta^* } { 
        P_{\vartheta_{1}}U_H(\vartheta_{1}, x_{0})
        } +	
        \eta_{t}  \EB\myinner{ \vartheta_{t}-\vartheta^* }{ P_{\vartheta_{t}}U_H(\vartheta_{t}, x_{t}) } \notag\\
        &\le \eta_1 G_{2:t} \EB\| \vartheta_{1}-\vartheta^*\| \|	P_{\vartheta_{1}}U_H(\vartheta_{1}, x_{0})\| + \eta_t \EB\|\vartheta_{t}-\vartheta^*\| \|P_{\vartheta_{t}}U_H(\vartheta_{t}, x_{t}) \|\notag\\
        &\le  G_{2:t} \frac{\eta_1}{2}\left(\EB\|\vartheta_1-\vartheta^*\|^2 + \EB\|	P_{\vartheta_{1}}U_H(\vartheta_{1}, x_{0})\|^2 \right) \notag \\
        &\qquad   + \frac{\eta_t}{2} \left( 
        \EB\|\vartheta_{t}-\vartheta^*\|^2 + \EB \|P_{\vartheta_{t}}U_H(\vartheta_{t}, x_{t})\|^2
        \right)\notag \\
        &\le \eta_1 G_{2:t} \left(
        \frac{\Delta_{1}}{2}+ \sigma_U^2 +4C_{\Theta}^2 C_U^2
        \right)  + \eta_t \left(
       \frac{\Delta_{t} }{2}+ \sigma_U^2 +4C_{\Theta}^2 C_U^2
        \right).
	\end{aligned}
        \end{equation}
	
	Summing up $A_{1}$ to $A_{4}$ yields
        \begin{equation}
	\begin{aligned} 
		\label{eq:sum-A}
		\sum_{i=1}^4 A_i
		&\le  \frac{L_U}{2} \sum_{s=2}^t  \eta_{s-1} \eta_{s} G_{s+1:t} 
		\left( \Delta_s + 2(\sigma_H^2 + 4C_{\Theta}^2 C_H^2) \right) \notag  \\
            & \qquad +  \sum_{s=2}^{t}\eta_{s} \eta_{s-1} G_{s+1:t}
		\left(
		\sigma_H^2 + 4C_{\Theta}^2 C_H^2 + \sigma_U^2 +4C_{\Theta}^2 C_U^2
		\right)   \notag \\
		&\qquad + \left(  \varkappa + K_0 \right)  \sum_{s=2}^{t} \eta_{s-1}\eta_s G_{s+1:t}  \left(
		\frac{1}{2} \Delta_{s-1} + \sigma_U^2 +4C_{\Theta}^2 C_U^2
		\right), \notag \\
            &\qquad + \eta_1 G_{2:t} \left(
        \frac{\Delta_{1}}{2}+ \sigma_U^2 +4C_{\Theta}^2 C_U^2
        \right)  + \eta_t \left(
       \frac{\Delta_{t} }{2}+ \sigma_U^2 +4C_{\Theta}^2 C_U^2
        \right) \notag\\
		&= \sum_{s=2}^{t}  \eta_{s-1}\eta_s G_{s+1:t}  \left( \frac{L_U }{2}
		\Delta_s +   \frac{\varkappa+K_0}{2}	 \Delta_{s-1}
		\right) \notag\\
		&\qquad +  \sum_{s=2}^{t} \eta_{s-1} \eta_s G_{s+1:t}  \left( 
(L_U+1)(\sigma_H^2 + 4C_{\Theta}^2 C_H^2) + (1+\varkappa+K_0)(\sigma_U^2 + 4C_{\Theta}^2 C_U^2)
		\right) \notag \\ 
            &\qquad + \eta_1 G_{2:t} \left(
        \frac{\Delta_{1}}{2}+ \sigma_U^2 +4C_{\Theta}^2 C_U^2
        \right)  + \eta_t \left(
       \frac{\Delta_{t} }{2}+ \sigma_U^2 +4C_{\Theta}^2 C_U^2
        \right).
	\end{aligned} 
        \end{equation}
	To further simplify the expression, we make use of Lemma~\ref{lem:o_gamma_k}, which, together with the fact that $\eta_{s-1} \le \kappa \eta_{s}$ and $\varkappa \le K_0$, imply that 
	\[
	\sum_{s=2}^{t}  \eta_{s-1} \eta_s G_{s+1:t}  \le  \kappa
	\sum_{s=2}^{t}  \eta_s^2 G_{s+1:t}  \le  \kappa \sum_{s=1}^{t}  \eta_s^2 G_{s+1:t}  \le \frac{\kappa }{K_0} \eta_t.
	\]
	Plugging the last inequality into~\eqref{eq:sum-A} yields that
	\[
	\sum_{i=1}^4 A_i\le G_{2:t} C_0 + \sum_{s=2}^{t}  \eta_s^2 G_{s+1:t}  \left(
	C_1 \Delta_s + C_2 \Delta_{s-1}
	\right)  + \eta_t \left( \Delta_t + C_3\right).
	\]
	where $\{C_i\}_{0 \le i \le 3}$ are define in the lemma.
\end{proof}

\subsubsection{Proof of Lemma~\ref{lem:convergence}}
\label{proof:lem:convergence}
\begin{proof}{Proof of Lemma~\ref{lem:convergence}}
	Recall that the sequence $\{\Delta_t\}_{t \ge 0}$ satisfies the following inequality
	\begin{equation*}
		\Delta_{t+1}  \le G_{1:t}  \tconst_0 + \sum_{s=2}^{t}   \eta_{s-1}\eta_s G_{s+1:t} (\tconst_1 \Delta_s + \tconst_2 \Delta_{s-1}) + \eta_t ( \Delta_t + \tconst_3).
	\end{equation*}
	The inequality in~\eqref{eq:Delta-final-iterate} is not easy to handle.
	We then introduce a non-negative auxiliary sequence $\{ {\rm U}_t \}_{t \geq 0}$ that always serves as a valid upper bound of $\Delta_t$.
	This sequence is defined by the recursion: ${\rm U}_{1} = \Delta_1 $ and 
	\begin{equation}
		\label{eq:U-iterate}
		{\rm U}_{t+1}  = G_{2:t}  \tconst_1 + \sum_{s=2}^{t} \eta_{s-1} \eta_s G_{s+1:t} (\tconst_1 {\rm U}_{s}  + \tconst_2 {\rm U}_{s-1}  ) + \eta_t (  {\rm U}_{t}   + \tconst_3).
	\end{equation}
	By induction, it is easy to see that $\Delta_t \leq {\rm U}_t$ for any $t \geq 1$. 
	
	Note that $G_{s, t} = G_{s, t-1}(1-2\eta_t K_0)$ for any $s \ge 1$. Arrangement on~\eqref{eq:U-iterate} implies that
        \begin{equation*}
	\begin{aligned}
		{\rm U}_{t+1} &= (1-2\eta_tK_0) {\rm U}_{t} +\eta_{t-1} \eta_t (\tconst_1{\rm U}_{t} + \tconst_2{\rm U}_{t-1}) \\
		&\qquad + 
		 (\eta_t {\rm U}_{t} -   (1-2\eta_tK_0)\eta_{t-1}  {\rm U}_{t-1} ) +
		\tconst_3(\eta_t - (1-2\eta_tK_0)\eta_{t-1})\\
		&= \left(  1-2\eta_tK_0 + \eta_{t-1}\eta_t\tconst_1 \right){\rm U}_{t} 
		+  \eta_{t-1}\eta_t \left( \tconst_2  + 2K_0  	\right){\rm U}_{t-1} \\
		&\qquad +  (\eta_t{\rm U}_{t} -\eta_{t-1}{\rm U}_{t-1} )
		+ 	\tconst_3(\eta_t - (1-2\eta_tK_0)\eta_{t-1})\\
		&\le \left(1 - K_0 \eta_t \right){\rm U}_{t}  + \eta_{t-1} \eta_t (\tconst_2  + 2 K_0   ) {\rm U}_{t-1} 
		+  (\eta_t{\rm U}_{t} -\eta_{t-1}{\rm U}_{t-1} ) + 3\eta_{t-1}\eta_t K_0 \tconst_3,
	\end{aligned}
        \end{equation*}
	where the last inequality uses the inequalities on step sizes in Assumption~\ref{assump:step-size}, namely the following results
	\begin{itemize}
		\item  $\eta_t \le \frac{K_0}{\tconst_1} = \frac{2K_0}{L_U}$ for all $t \ge 1$;
		\item $\eta_t - (1-2\eta_tK_0)\eta_{t-1}) \le \eta_t \eta_{t-1}(\varkappa + 2K_0) \le 3 \eta_t \eta_{t-1} K_0$.
	\end{itemize}
	
	Given the established inequality
	\begin{equation}
		\label{eq:U-iterate-new}
		{\rm U}_{s+1}  \le \left(1 - K_0 \eta_s \right){\rm U}_{s}  + \eta_{s-1} \eta_s (\tconst_2  +  2K_0   ) {\rm U}_{s-1} 
		+  (\eta_s{\rm U}_{s} -\eta_{s-1}{\rm U}_{s-1} ) + 3\eta_{s-1}\eta_s  K_0\tconst_3,
	\end{equation}
	we are ready to establish the lemma.
	
	We first prove the part ${(a)}$.
	Summing up the inequality~\eqref{eq:U-iterate-new} over $s=1$ to $s=t$ yields 
 \begin{equation*}
 \begin{aligned}
    \sum_{s=1}^t {\rm U}_{s+1}  &\le	
    \sum_{s=1}^t   \left[  \left(1 - K_0 \eta_s \right){\rm U}_{s} +
    3\eta_{s-1}\eta_s K_0\tconst_3
    \right] \\
    & \qquad +
    \sum_{s=1}^t     \eta_{s-1} \eta_s (\tconst_2  + 2 K_0   ) {\rm U}_{s-1} 
    +  \sum_{s=1}^t  (\eta_s{\rm U}_{s} -\eta_{s-1}{\rm U}_{s-1} ) .
 \end{aligned}
 \end{equation*}
	Rearranging the last inequality give
        \begin{equation*}
	\begin{aligned}
		{\rm U}_{t+1} 
		&\overset{(a)}{\le}  {\rm U}_{1} + \sum_{s=1}^t  \left[
		\eta_{s-1} \eta_s (\tconst_2  + 2 K_0   )  {\rm U}_{s-1}  - K_0 \eta_s {\rm U}_{s} 
		\right]    +  \eta_t	{\rm U}_{t}  + 3\kappa K_0 \tconst_3 \sum_{s=1}^t \eta_s^2 \\
		&\le  {\rm U}_{1} + \sum_{s=1}^t 
		\left(\eta_{s-1} \eta_s (\tconst_2  +  2K_0   )- K_0\eta_{s-1}  \right) {\rm U}_{s-1}   
		+  \eta_t	{\rm U}_{t}  + 3 \kappa K_0 \tconst_3 \sum_{s=1}^t \eta_s^2 \\
		&\overset{(b)}{\le} {\rm U}_{1} +  \eta_t	{\rm U}_{t}  + 3 \kappa K_0 \tconst_3 \sum_{s=1}^{\infty} \eta_s^2 
		\overset{(c)}{\le} 0.5 \cdot {\rm U}_{t} + \left( {\rm U}_{1}  + 3 \kappa K_0 \tconst_3 \sum_{s=1}^{\infty} \eta_s^2  \right) \\
		&\overset{(d)}{\le} 0.5^t \cdot {\rm U}_1  + \sum_{\tau=0}^{t-1} 0.5^{\tau} \cdot  \left( {\rm U}_{1}  + 3\kappa K_0 \tconst_3 \sum_{s=1}^{\infty} \eta_s^2  \right) 
		\le 6\left({\rm U}_1 + \kappa K_0 \tconst_3 \sum_{s=1}^{\infty} \eta_s^2 \right),
	\end{aligned}
        \end{equation*}
	where (a) uses $\eta_{s-1} \le \kappa \eta_{s}$, (b) uses $\eta_t \le \frac{K_0}{\tconst_2+2K_0 )} = \frac{K_0}{\varkappa + 3K_0}$, (c) uses $ \eta_t \le \frac{K_0 }{\tconst_2+2K_0} \le 0.5$, and (d) is due to repeatedly using (c) $t-1$ times

	We now proceed to prove part ${(b)}$ of the lemma. 
	For simplicity, we define $\bG_{m:n} = \prod_{s=m}^n (1 - \eta_s K_0)$ for $n \ge m$ and $\bG_{m:n} =1$ if $n < m$.
	We observe from~\eqref{eq:U-iterate-new} that
        \begin{equation*}
	\begin{aligned}
		{\rm U}_{t+1}  &\le  \bG_{2:t} {\rm U}_{1}  + \eta_{s-1} \eta_s \sum_{s=1}^t \bG_{s+1:t} \left[ (\tconst_2  +  2K_0   ) {\rm U}_{s-1}   + 3K_0\tconst_3 \right]\\
		&\qquad +   \sum_{s=1}^t \bG_{s+1:t} (\eta_s{\rm U}_{s} -\eta_{s-1}{\rm U}_{s-1} ). 
	\end{aligned}
        \end{equation*}
	Notice that 
        \begin{equation*}
	\begin{aligned}
		& \sum_{s=2}^t \bG_{s+1:t} \big( \eta_s {\rm U}_{s} - \eta_{s-1} {\rm U}_{s-1} \big) \\
		& = \sum_{s=2}^t \bG_{s+1:t} \big( \eta_s {\rm U}_{s} + (1 - \eta_s K_0) ( \eta_{s-1} {\rm U}_{s-1} - \eta_{s-1} {\rm U}_{s-1} ) - \eta_{s-1} {\rm U}_{s-1} \big) \\
		& = \sum_{s=2}^t \Big\{ \Big( \bG_{s+1:t} \eta_s {\rm U}_{s} - \bG_{s:k} \eta_{s-1} {\rm U}_{s-2} \Big) - \eta_s \eta_{s-1} K_0{\rm U}_{s-1}  \Big\} \leq \eta_t{\rm U}_{t} \leq \eta_t\overline{\Delta}.
	\end{aligned}
        \end{equation*}
	On the other hand, Lemma~\ref{lem:o_gamma_k} implies that $\sum_{s=1}^t \bG_{s+1:t} \eta_{s-1}\eta_s \le \kappa\sum_{s=1}^t \bG_{s+1:t}  \eta_s^2  \leq \frac{2\kappa}{K_0} \eta_t$.
	Hence, the following is obtained
        \begin{equation*}
	\begin{aligned}
		{\rm U}_{t+1} 
		& \le  \bG_{1:t} \Delta_1
		+ \frac{2\kappa}{K_0} \left[ (\tconst_2  + 2 K_0   )\overline{\Delta}  + 3 K_0\tconst_3 \right] \eta_t
		+ \eta_t    \overline{\Delta}  \\
		&\le \prod_{s=2}^t \left( 1 -K_0 \eta_s \right) \Delta_1 
		+ \eta_t \left[ \left(   \frac{2\kappa\tconst_2}{K_0} + (4\kappa + 1)    \right) \overline{\Delta} + 6\kappa\tconst_3  \right]\\
  &	\le \prod_{s=2}^t \left( 1 -K_0 \eta_s \right) \Delta_1 
		+ \eta_t \left[  \left(1 + 8 \kappa \right) \overline{\Delta} + 6\kappa\tconst_3  \right],
	\end{aligned}
        \end{equation*}
 where the last inequality uses
 $\kappa \tconst_2/K_0 = \kappa (\varkappa + K_0)/K_0 \le 2 \kappa$.
\end{proof}

%% file: Journal_version/proof_of_regret_lemma.tex
\subsection{Proofs for Regret Bound}
\begin{proof}{Proof of Lemma~\ref{lem:bd_trans_err_of_regret}}
    For any state \( x \) in the state space $\mathcal{X}$ and any parameter \( \vartheta \), we denote by \( \{X_n(x, \vartheta)\}_{n=0}^\infty \) the Markov chain generated with \( x \) as the initial state and \( \gP_{\vartheta} \) as the transition kernel. Similarly, we use \( \{X_n(x)\}_{n=0}^\infty \) to represent the state sequence generated during the invocation of the stream SGD algorithm. By Assumptions~\ref{ass:W-contraction} and~\ref{ass:cont_cost}, it follows that for all $j \ge 1$,
    \begin{equation*}
    \begin{aligned}
        \EB~ c(\vartheta, X_j(x,\vartheta))& - \EB~  c(\vartheta, X_\infty(\vartheta))\\
        &= \EB~ c(\vartheta, X_j(x,\vartheta)) - \EB_{x^\prime \sim \pi_\vartheta}c(\vartheta, X_j(x^\prime, \vartheta))\\
        &\le \int L_c W_{\norm{\cdot},1}(X_j(x,\vartheta), X_j(x^\prime, \vartheta)) \rd \pi_{\vartheta}(x^\prime)\\
        &\le 2CL_c \int  e^{-cj} \rd \pi_{\vartheta}(x^\prime) \le C^\prime e^{-cj},
    \end{aligned}
    \end{equation*}
    where the first inequality holds because the cost function $c(\vartheta, x) \in \gF_{\norm{\cdot}, 1}$, where $L_c$ denotes the Lipschitz continuity constant of $c(\vartheta, x)$, and $C^\prime \leq 2CL_c$.  

Now, assume that \( x_{t-j} \) is the state at the \( (t-j) \)-th iteration, obtained through the transition dynamics \( \prod_{r = 1}^{t-j} P_{\vartheta_r} \), where \( \vartheta_r \) is generated by the stream SGD algorithm. Substituting \( x_{t-j} \) as the initial state into the inequality derived above, we have

    \begin{equation*}
    \begin{aligned}
        \EB_{x_{t-j}}\left[\gP_{\vartheta}^j c(\vartheta, x_{t-j})\right] &- \EB~[c(\vartheta, X_\infty(\vartheta))]\\ &= \EB_{x_{t-j}}\left[
        \EB\left[c\big(\vartheta, X_j(x_{t-j},\vartheta)\big)\mid \gF_{t-j}\right] - \EB [c\big(\vartheta, X_\infty(x,\vartheta)\big)]
        \right]\\
        &\le C^\prime  e^{-cj}.
    \end{aligned}
    \end{equation*}
Thus, for any given iteration index $t$, let $j = \left\lceil \frac{2}{c}\log t + \frac{\log C'}{c}\right\rceil$. Then, we have the following result
    \[
    \EB_{x_{t-j}}\left[\gP_{\vartheta_{t}}^j c(\vartheta_t, x_{t-j})\right] - \EB~ c(\vartheta_t, X_\infty(\vartheta_t)) \le t^{-2}.
    \]
    By applying the telescoping technique and the mixing bound above, we obtain
    \begin{equation}\label{eq:lem_trans_err_eq1}
    \begin{aligned}
        \EB~ c(\vartheta_t, X_t) &- \EB~ c(\vartheta_t, X_\infty(\vartheta_t))\\
        &= \EB\left[\EB[c(\vartheta_t, X_t)\mid \gF_{t-j}]\right] - 
        \EB_{x_{t-j}}\gP_{\vartheta_t}^j c(\vartheta_t, x_{t-j})\\
        &\qquad + \EB_{x_{t-j}}\gP_{\vartheta_t}^j c(\vartheta_t, x_{t-j}) - 
        \EB~ c(\vartheta_t, X_\infty(\vartheta_t))\\
        &\le \left\{\EB\left[\EB[c(\vartheta_t, X_t)\mid \gF_{t-j}]\right] - 
        \EB_{x_{t-j}}\gP_{\vartheta_t}^j c(\vartheta_t, x_{t-j})\right\} + \gO(t^{-2}).
    \end{aligned}
    \end{equation}
    For the first term, by making use of the telescoping technique again and noting that $X_t \overset{d}{=} X_{j}(x_{t-j})$, we can derive that,
    \begin{equation}\label{eq:lem_trans_err_eq2}
    \begin{aligned}
        \EB\left[\EB[c(\vartheta_t, X_t)\mid \gF_{t-j}]\right] &- 
        \EB_{x_{t-j}}\gP_{\vartheta_t}^j c(\vartheta_t, x_{t-j})\\
        &= \EB_{x_{t-j}}\left[\prod\limits_{i=t-j+1}^{t} \gP_{\vartheta_{i}} c(\vartheta_t, x_{t-j})\right] - \EB_{x_{t-j}}\left[\gP_{\vartheta_t}^j c(\vartheta_t, x_{t-j})\right]\\
        &= \ssum{k}{1}{j} \EB_{x_{t-j}}\left[\left(\prod\limits_{i=t-j+1}^{t-k} \gP_{\vartheta_{i}}\right) (\gP_{\vartheta_{t-k+1}} - \gP_{\vartheta_{t}})\gP_{\vartheta_t}^{k-1}
        c(\vartheta_t, x_{t-j})
        \right].
    \end{aligned}
    \end{equation}
    According to Lemma~\ref{lem:prop_L_mf}, since $c(\vartheta, x)$ is $L_c$-Lipschitz, we can find a universal constant $C$ and a function $V(x) \in \gV$ such that for any $1 \le k \le j$,
    \[
    (\gP_{\vartheta_{t-k}} - \gP_{\vartheta_t})\gP_{\vartheta_t}^{k-1} c(\vartheta_t, x_{t-j}) \le C\norm{\vartheta_{t-k} - \vartheta_t}e^{-c(k-1)}(V(x_{t-j}) + 1).
    \]
    Plugging this inequality into \eqref{eq:lem_trans_err_eq2} yields
    \begin{equation}\label{eq:lem_trans_err_eq3}
    \begin{aligned}
        \EB\left[\EB[c(\vartheta_t, X_t)\mid \gF_{t-j}]\right] &- 
        \EB_{x_{t-j}}\gP_{\vartheta_t}^j c(\vartheta_t, x_{t-j})\\
        &= \ssum{k}{1}{j} \EB_{x_{t-j}}\left[\prod\limits_{i=t-j}^{t-k-1} \gP_{\vartheta_{i}} (\gP_{\vartheta_{t-k}} - \gP_{\vartheta_{t}})\gP_{\vartheta_t}^{k-1}
        c(\vartheta_t, x_{t-j})
        \right]\\
        &\le \ssum{k}{1}{j} Ce^{-c(k-1)} \EB_{x_{t-j}}\left[\norm{\vartheta_{t-k} - \vartheta_t}\prod\limits_{i=t-j}^{t-k-1} \gP_{\vartheta_{i}}(V(x_{t-j}) + 1)
        \right]\\
        &\overset{(a)}{\le}
        \ssum{k}{1}{j} C^\prime \frac{k}{t}e^{-c(k-1)} \EB\left[ V(X_{t-k}) + 1
        \right]\overset{(b)}{\le}
        \frac{C^{''}}{t}.
    \end{aligned}
    \end{equation}
    Here $(a)$ holds by using the following inequality that
    \[
    \EB\left[\norm{\vartheta_{t-k} - \vartheta_t}\mid \gF_{t-j}\right] = \EB\left[
    \norm{\ssum{i}{t-k}{t-1}\eta_i H(\vartheta_i, x_i)}\bigg|~ \gF_{t-j}
    \right] \le \frac{2L_c k}{t}
    \]
    and $(b)$ follows from the fact that $\EB[V(x_{t-j})]$ are uniformly bounded for any $V\in \gV$. Ultimately, the lemma is proved by combining \eqref{eq:lem_trans_err_eq1} and \eqref{eq:lem_trans_err_eq3}.
\end{proof}

%% file: Journal_version/inference.tex
\subsection{Proofs for Functional Central Limit Theorem}
\label{proof:fclt}
In the following, we present the proof idea of Theorem~\ref{thm:fclt}.
Recall that the update rule is $\vartheta_{t+1} = \Pi_{\Theta} \left(\vartheta_t - \eta_t  H(\vartheta_t, x_t) \right)$.
We first provide an almost sure convergence for $\{\vartheta_t\}_{t \ge 0}$ which would be useful later on.
\begin{lemma}
	\label{lem:a.s.}
	Under the conditions of Theorem~\ref{thm:fclt}, we have $\vartheta_{t} \overset{a.s.}{\to} \vtheta$ and $\eta_t H(\vartheta_t, x_t) \overset{a.s.}{\to} 0$.
\end{lemma}
\begin{proof}{Proof of Lemma~\ref{lem:a.s.}}
	The proof can be found in~\ref{proof:lem:a.s.}
\end{proof}
We then introduce
\begin{equation}
	\label{eq:w}
	\sw_t = \frac{\vartheta_{t+1} -\vartheta_{t} }{\eta_t} -  H(\vartheta_t, x_t).
\end{equation}
to incorporate the truncation effect of projections on the iterate updates.
As we will show in Lemma~\ref{lem:error-analysis}, Lemma~\ref{lem:a.s.}, together with the condition that $\vtheta$ is an inner point of the domain $\Theta$, implies $\sw_t$ is non zero only for a finite number of iterations.
Here, we still maintain the dependence on $\sw_t$ and obtain
\[
\vartheta_{t+1} = \vartheta_t - \eta_t  (H(\vartheta_t, x_t) + \sw_t).
\]
By utilizing the Possion equation, we decompose $ H(\vartheta_t, x_t)$ into two terms:
\begin{equation*}
\begin{aligned}
	H(\vartheta_t, x_t) 
	&=  \nabla f (\vartheta_t)
	+ \left[ H(\vartheta_t, x_t)  - \nabla f (\vartheta_t)\right]\\
	&=\nabla f (\vartheta_t)
	+ \left[  U_H(\vartheta_t, x_t)  - P_{\vartheta_t}U_H(\vartheta_t, x_t)  \right].
\end{aligned}
\end{equation*}
Inspired by the martingale-residual-coboundary decomposition in~\citep{liang2010trajectory,li2023online}, we then further decompose $U_H(\vartheta_t, x_t)  - P_{\vartheta_t}U_H(\vartheta_t, x_t)$ into three terms:
\begin{equation}
\begin{aligned}
	\label{eq:decompose0}
		U_H(\vartheta_t, x_t)  - P_{\vartheta_t}U_H(\vartheta_t, x_t)
		&= \underbrace{\vphantom{ \left(\frac{a^{0.3}}{b}\right)}   \left[  U_H(\vartheta_t, x_t) - P_{\vartheta_{t}} U_H(\vartheta_t, x_{t-1}) \right]}_{martingale} \\
		& \qquad +  \underbrace{	\vphantom{ \left(\frac{a^{0.3}}{b}\right)}  \left[ \frac{\eta_{t+1}}{\eta_t} P_{\vartheta_{t+1}} U_H(\vartheta_{t+1}, x_t) - P_{\vartheta_{t}}  U_H(\vartheta_t, x_{t}) \right]  }_{residual}
		\\
		& \qquad +  \underbrace{ \vphantom{ \left(\frac{a^{0.3}}{b}\right)}   \left[  P_{\vartheta_{t}}  U_H(\vartheta_t, x_{t-1})-  \frac{\eta_{t+1}}{\eta_t} P_{\vartheta_{t+1}}  U_H(\vartheta_{t+1}, x_t) \right]}_{coboundary}.
\end{aligned}
\end{equation}
The telescoping structure in the coboundary term motivates us to introduce an auxiliary process $\{\ttheta_t\}_{t\ge0}$ to remove its effect where 
\[
\ttheta_t = \vartheta_t -\eta_tP_{\vartheta_{t}}U_H(\vartheta_{t}, x_{t-1}).
\]
As a result, we have
\begin{equation*}
\begin{aligned}
	\ttheta_{t+1} = \ttheta_t - &\eta_t \left( \nabla f (\vartheta_t)  + \sw_t +  U_H(\vartheta_t, x_t) - P_{\vartheta_t} U_H(\vartheta_t, x_{t-1}) \right) \\
	- &\eta_t \left( \frac{\eta_{t+1}}{\eta_t} P_{\vartheta_{t+1}} U_H(\vartheta_{t+1}, x_t) - P_{\vartheta_t} U_H(\vartheta_t, x_{t})  \right).  
\end{aligned}
\end{equation*}

We then focus on the convergence of $\{\ttheta_t\}_{t\ge0}$ and simplify the last equation by introducing the following shortcuts: $\De_t = \ttheta_t - x^{\star}$ and
\begin{subequations}
\begin{align}
	\sr_t  &=  \nabla f (\vartheta_t)- G \De_t, \label{eq:r} \\
	\su_t  &= U_H(\vartheta_t, x_t) - P_{\vartheta_t} U_H(\vartheta_t, x_{t-1}), \label{eq:u}  \\
	\snu_t &=\frac{\eta_{t+1}}{\eta_t} P_{\vartheta_{t+1}} U_H(\vartheta_{t+1}, x_t) - P_{\vartheta_t} U_H(\vartheta_t, x_{t}). \label{eq:nu}
\end{align}
\end{subequations}
With the notation, the update rule becomes
\begin{equation}\label{eq:help0}
	\De_{t+1} = \De_t - \eta_t  \left[
	G \De_t + \sr_t + \su_t +  \snu_t  + \sw_t
	\right]  \notag 
	= (I - \eta_t G) \De_t + \eta_t \left[ \sr_t + \su_t + \snu_t + \sw_t\right].
\end{equation}	
The following lemma explains the reason why we perform the decomposition~\eqref{eq:decompose0}.
It shows, while $\{H(\vartheta_t, x_t)  - g(\vartheta_t)\}_{t\ge 0}$ is not a martingale difference sequence, the decomposed $\{\su_t\}_{t \ge 0}$ is.
Furthermore, $\{\su_t\}_{t \ge 0}$ admits an FCLT via a standard argument of multidimensional martingale FCLT (e.g., Theorem 2.1 in~\citet{whitt2007proofs}).
The remaining terms namely $\{\sr_t\}_{t\ge 0}, \{\snu_t\}_{t\ge 0}$ and $\{\sw_t\}_{t\ge 0}$ have negligible effects because they vanish asymptotically.

\begin{lemma}[Properties of decomposed terms]
	\label{lem:error-analysis}
	Under the same conditions of Theorem~\ref{thm:fclt}, 
	\begin{enumerate}
		\item \label{lem:error-r}  As $T \to \infty$, $\frac{1}{\sqrt{T}} \sum_{t=1}^T \EB \| \sr_t\| \to 0$.
		\item  \label{lem:error-u} $\{ \su_t \}_{t \ge 0}$ is a martingale difference sequence satisfying $\sup_{t \ge 0}\EB\|\su_t\|^p < \infty$ with $p \in (2, 4]$
		Furthermore, the following FCLT holds $	\frac{1}{\sqrt{T}} \sum_{t=1}^{\lfloor{Tr}\rfloor} \su_t \overset{w}{\to} S^{1/2} W(r)$.
		\item \label{lem:error-nu}  As $T \to \infty$, $\frac{1}{\sqrt{T}} \sum_{t=1}^T \EB \| \snu_t\| \to 0.$
		\item \label{lem:error-w}  As $T \to \infty$, $ \frac{1}{\sqrt{T}} \sum_{t=1}^T \|\sw_t\| \overset{a.s.}{\to} 0$.
	\end{enumerate}
\end{lemma}
\begin{proof}{Proof of Lemma~\ref{lem:error-analysis}}
	The proof can be found in~\ref{proof:lem:error-analysis}.
\end{proof}

Setting $B_t = I - \eta_t G$  and recurring~\eqref{eq:help0} give
\[
\De_{t+1} = \left(\prod_{j =1}^tB_j\right) \De_0 + \sum_{j=1}^t \left(\prod_{i=j+1 }^tB_i\right) \eta_j \left[ \sr_j + \su_j + \snu_j + \sw_j \right].
\]
Here we use the convention that $\prod_{j = t+1}^tB_j =I$ for any $t \ge 0$.
As a result, for any $r \in [0, 1]$,
\begin{align*}
	\widetilde{\ph}_T(r) 
	&:= \frac{1}{\sqrt{T}} \sum_{t=1}^{\floor{Tr}} \widetilde{\Delta}_t  
	=\frac{1}{\sqrt{T}} \sum_{t=1}^{\floor{Tr}} \left\{   \left(\prod_{j = 1}^tB_j\right) \De_0 + \sum_{j=1}^t \left(\prod_{i=j+1 }^tB_i\right) \eta_j \left[ \sr_j + \su_j + \snu_j + \sw_j\right]  \right\}\\
	&=\frac{1}{\sqrt{T}} \sum_{t=1}^{\floor{Tr}}  \left(\prod_{j = 1}^tB_j\right) \De_0 
	+ \frac{1}{\sqrt{T}} \sum_{j=1}^{\floor{Tr}} \sum_{t=j}^{\floor{Tr}} \left(\prod_{i=j+1 }^tB_i\right) \eta_j \left[ \sr_j + \su_j  + \snu_j + \sw_j\right]. 
\end{align*}
In the following, for simplicity we define
\begin{equation}
	\label{eq:A}
	A_{j}^n : =\sum\limits_{t=j}^{n}\left(\prod\limits_{i=j+1}^{t}B_i\right) \eta_j.
\end{equation}
Using the notation in~\eqref{eq:A}, we further simplify the last equation as
\begin{equation*}
	\label{eq:phi1}
	\widetilde{\ph}_T(r) 
	= \frac{1}{\sqrt{T}\eta_1} A_1^{\floor{Tr}} B_1\De_1
	+ \frac{1}{\sqrt{T}} \sum_{j=1}^{\floor{Tr}} A_j^{\floor{Tr}} \left[ \sr_j  + \su_j  + \snu_j +\sw_j\right].
\end{equation*}
Arrangement yields
\begin{align}\label{eq:bpsi}
	\widetilde{\ph}_T(r)  - \frac{1}{\sqrt{T}} \sum_{j=0}^{\floor{Tr}} G^{-1}  \su_t
	&= \frac{1}{\sqrt{T}\eta_1} A_1^{\floor{Tr}} B_1\De_1
	+ \frac{1}{\sqrt{T}} \sum_{t=1}^{\floor{Tr}} A_t^{\floor{Tr}}(\sr_t + \snu_t + \sw_t ) \notag
	\\
	&\qquad + \frac{1}{\sqrt{T}} \sum_{t=1}^{\floor{Tr}} \left(A_t^T -G^{-1}\right)\su_t
	+ \frac{1}{\sqrt{T}} \sum_{t=0}^{\floor{Tr}} \left(A_t^{\floor{Tr}} -A_t^T\right)\su_t \nonumber \\
	&:= \Bpsi_0(r) + \Bpsi_1(r) + \Bpsi_2(r) +  \Bpsi_3(r).
\end{align}

By~\eqref{eq:bpsi}, we are ready to prove Theorem~\ref{thm:fclt}.
We introduce the uniform norm (or maximum norm) $\vertiii{\cdot}$ for a random process $\phi$ defined in $[0, 1]$:
\[
\vertiii{\phi} = \sup_{r \in [0, 1]} \|\phi(r)\|.
\]
\begin{itemize}
	\item First, from Lemma~\ref{lem:error-analysis}, the functional weak convergence follows that $\frac{1}{\sqrt{T}} \sum_{t=1}^{\floor{Tr}} G^{-1}\su_t \overset{w}{\to} \Bpsi(r) = S^{1/2} G^{-1} W(r)$ uniformly over $r \in [0, 1]$.
	\item Second, by a similar argument in Lemma~\ref{lem:bounded-growth}, we can prove that
 $\EB \left\|P_{\vartheta_{t}}U_H(\vartheta_{t}, x_{t-1})\right\|^2 \le 2(\sigma_U^2 +4C_{\Theta}^2 C_U^2)$ is uniformly bounded and thus
	\begin{align*}
		\EB\vertiii{ \widetilde{\ph}_T-\ph_T} 
		&\le \frac{1}{\sqrt{T}} \sum_{t=1}^T \eta_t \EB\| P_{\vartheta_t} U_H(\vartheta_t, x_{t-1}) \|  \to 0.
	\end{align*}
	It implies the random function $\ph_T$ has the same asymptotic behavior as $\widetilde{\ph}_T$, i.e., ${\ph}_T =\widetilde{\ph}_T + o_{P}(1)$.
	\item To complete the proof, it suffices to show that 
	\begin{equation}
		\label{eq:fclt-final}
		\vertiii{\widetilde{\ph}_T  -  \Bpsi} =
		\sup_{r \in [0 ,1]}\left\|
		\widetilde{\ph}_T(r)  -  \Bpsi(r)
		\right\| = o_{P}(1).
	\end{equation}
	In this way, one has $\widetilde{\ph}_T = \Bpsi + o_{P}(1)$ and thus $\ph_T = \Bpsi  + o_{P}(1)$ due to Slutsky's theorem.
	Lemma~\ref{lem:residual-terms} validates~\eqref{eq:fclt-final} by showing the four separate terms $\sup_{r \in [0, 1]} \|\Bpsi_k(r)\|(0 \le k \le 3)$ respectively converge to zero in probability.
	We then complete the proof.
	\begin{lemma}
		\label{lem:residual-terms}
		Under the same conditions of Lemma~\ref{lem:error-analysis}, for all $0 \le k \le 3$, when $T \to \infty$,
		\[
		\vertiii{\Bpsi_k} =\sup_{r \in [0, 1]} \|\Bpsi_k(r)\| = o_{P}(1).
		\]
	\end{lemma}
	\begin{proof}{Proof of Lemma~\ref{lem:residual-terms}}
		The proof can be found in~\ref{proof:lem:residual-terms}.
	\end{proof}
\end{itemize}

\subsection{Proof of Lemma~\ref{lem:a.s.}}
\label{proof:lem:a.s.}
\begin{proof}{Proof of Lemma~\ref{lem:a.s.}}
	Recall that $\FM_t = \sigma(\{ x_{\tau}\}_{\tau \le t})$ is the $\sigma$-field generated by all randomness before iteration $t$ and $\vartheta_{t+1} \in \FM_{t}$.
	We denote $\EB_t[ \cdot] = \EB[\cdot|\FM_{t-1}]$.
	By a similar argument in Lemma~\ref{lem:itera}, it follows that
        \begin{equation*}
	\begin{aligned}
		\EB_t \left\|\vartheta_{t+1}-\vtheta\right\|^{2} 
		&\le \EB_t\left\|\vartheta_{t}-\vtheta-\eta_t H( \vartheta_{t}, x_{t} )\right\|^{2} \\
		& = \left\|\vartheta_{t}-\vtheta\right\|^{2}- \, 2 \eta_{t} \langle\vartheta_t -\vtheta, P_{\vartheta_t}H( \vartheta_t, x_{t-1} )  \rangle + \eta_{t}^2  P_{\vartheta_t}\left\|  H(\vartheta_t, x_{t-1})  \right\|^{2} \\
		&\le \left( 1- 2\eta_t K_0 \right) \left\|\vartheta_{t}-\vtheta\right\|^{2} + 2 \eta_t^2 P_{\vartheta_t}\left\|  H(\vartheta_t, x_{t-1})  \right\|^{2} \\
		&\qquad + 2\eta_t \|\vartheta_t -\vtheta\| \|P_{\vartheta_t}H( \vartheta_t, x_{t-1}) - \nabla f (\vartheta_t)\|
	\end{aligned}
        \end{equation*}
	where the last inequality uses the strong convexity of $g$ and the smoothness of $H$.
	
	Note that we have the following intermediate results.
	\begin{itemize}
		\item $\sum_{t=1}^{\infty} \eta_t^2 P_{\vartheta_t}\left\|  H(\vartheta_t, x_{t-1}) \right\|^{2} < \infty$ almost surely. This is because Lemma~\ref{lem:bounded-growth} implies that
		\[
		\EB \sum_{t=1}^{\infty} \eta_t^2 P_{\vartheta_t}\left\|  H(\vartheta_t, x_{t-1}) \right\|^{2}  =  \sum_{t=1}^{\infty} \eta_t^2 \EB \left\|  H(\vartheta_t, x_{t}) \right\|^{2} \le \gO(1) \cdot \sum_{t=1}^{\infty} \eta_t^2 < \infty.
		\]
  As a byproduct, $\sum_{t=1}^{\infty} \eta_t^2 \left\|  H(\vartheta_t, x_{t}) \right\|^{2} < \infty$ almost surely, which implies that $\eta_t H(\vartheta_t, x_{t}) \overset{a.s.}{\to} 0$.
		\item By Lemma~\ref{assump:U},
  \begin{equation*}
  \begin{aligned}
            \|P_{\vartheta_t}H( \vartheta_t, x_{t-1}) - \nabla f (\vartheta_t)\| &\le \|P_{\vartheta_t}H( \vartheta_t, x_{t-1})\| + \|\nabla f (\vartheta_t)\| \\
            &\le\|P_{\vartheta_t}H( \vartheta_t, x_{t-1})\| + K_1 \|\vartheta_t-\vtheta\|.
  \end{aligned}
  \end{equation*}
		\item $\sum_{t=1}^{\infty} \eta_t \|\vartheta_t-\vtheta\| \|P_{\vartheta_t}H( \vartheta_t, x_{t-1}) \| < \infty$ almost surely.
  Again this is because it has a finite expectation:
            \begin{equation*}
		\begin{aligned}
			\EB \sum_{t=1}^{\infty} &\eta_t \|\vartheta_t-\vtheta\| \|P_{\vartheta_t}H( \vartheta_t, x_{t-1}) \|\\
			&\le  \sum_{t=1}^{\infty} \frac{\eta_t}{2}  \left(
			\EB \frac{\EB\|\vartheta_t-\vtheta\|^2 }{\sqrt{\eta_t}} + \sqrt{\eta_t}  \EB 
			\|P_{\vartheta_t}H( \vartheta_t, x_{t-1}) \|^2
			\right)
			\le \gO(1) \cdot  \sum_{t=1}^{\infty} \eta_t^{1.5} < \infty,
		\end{aligned}
            \end{equation*}
		where the last inequality uses $\EB\|\vartheta_t-\vtheta\|^2 = \gO(\eta_t)$ that is inferred from Theorem~\ref{thm:L2-convergence} and $\EB \|P_{\vartheta_t}H( \vtheta, x_{t-1})\|^2 = \gO(1)$ from Lemma~\ref{lem:bounded-growth}.
		\item By a similar argument, we can also show $\sum_{t=1}^{\infty} \eta_t\|\vartheta_t-\vtheta\|^2 < \infty$ almost surely.
	\end{itemize}
	
	Hence, it follows that
	\[
	\EB_t \left\|\vartheta_{t+1}-\vtheta\right\|^{2} \le (1-2\eta_tK_0) \left\|\vartheta_{t}-\vtheta\right\|^{2} + R_t,
	\]
	where $R_t$ is a summable non-negative random sequence defined by  
    \[
		R_t := 2K_1\eta_t\left\|\vartheta_{t}-\vtheta\right\|^2 + 2 \eta_t^2 P_{\vartheta_t}\left\|  H(\vartheta_t, x_{t-1}) \right\|^{2} +
		2\eta_t \|\vartheta_t-\vtheta\| \|P_{\vartheta_t}H( \vartheta_t, x_{t-1})\|.
    \]
	As discussed, $\sum_{t=1}^\infty R_t$ is finite almost surely.
	Then by the theorem of Robbins and Siegmund~\citep{robbins1971convergence}, the almost sure convergence of $\vartheta_t -\vtheta$ follows.
\end{proof}

\subsection{Proof of Lemma~\ref{lem:error-analysis}}
\label{proof:lem:error-analysis}
\begin{proof}{Proof of Lemma~\ref{lem:error-analysis}}
In the following, we use $a \precsim b$ or $a = \gO(b)$ to denote $a \le C b$ for an unimportant positive constant $C > 0$ (which doesn't depend on $t$) for simplicity.
\begin{enumerate}
    \item By the local linearity in Lemma~\ref{assump:U}, it follows that
        \begin{equation*}
	\begin{aligned}
		\| \sr_t\| &= \|\nabla f (\vartheta_t) - G\De_t\|\\
		&\le \|\nabla f (\vartheta_t) - G (\vartheta_t -\vtheta)\| + \eta_t \|G\| \cdot \| P_{\vartheta_t} U_H(\vartheta_t, x_{t-1})\|\\
		&\le 
		\left\{ \begin{array}{ll}
			L_G \cdot\|\vartheta_t -\vtheta\|^2 + \eta_t \|G\|\cdot  \| P_{\vartheta_t} U_H(\vartheta_t, x_{t-1})\|&  \text{if} \ \|\vartheta_t -\vtheta\| \le \delta_G\\
			(K_1 +  \| G\|  )\cdot \| \vartheta_t -\vtheta\| + \eta_t \|G\| \cdot\| P_{\vartheta_t} U_H(\vartheta_t, x_{t-1})\|& \text{if} \ \|\vartheta_t -\vtheta\| \ge \delta_G
		\end{array}
		\right.\\
		&\le \max\left\{L_G , \frac{K_1 +  \| G\|}{\delta_G}\right\}\|  \vartheta_t -\vtheta  \|^2 +\eta_t \|G\|  \cdot \| P_{\vartheta_t} U_H(\vartheta_t, x_{t-1})\|
	\end{aligned}
        \end{equation*}
By Theorem~\ref{thm:L2-convergence}, it follows that $\EB \|  \vartheta_t -\vtheta  \|^2 = \gO(\eta_t)$.
By Lemma~\ref{lem:bounded-growth}, we have $\EB \left\|P_{\vartheta_{s-1}}U_H(\vartheta_{s-1}, x_{s-1})
\right\| = \gO(1)$.
	As a result, when $T \to \infty$,
	\[
	\frac{1}{\sqrt{T}} \sum_{t=1}^T \EB \| \sr_t\|
	\precsim \frac{1}{\sqrt{T}}  \sum_{t=1}^T\EB\|\vartheta_t -\vtheta\|^2 + \frac{1}{\sqrt{T}}  \sum_{t=1}^T \eta_t
	\precsim \frac{1}{\sqrt{T}}  \sum_{t=1}^T  \eta_t \to 0.
	\]
 \item By the data generation mechanism, $x_t \sim P_{\vartheta_t}(x_{t-1}, \cdot)$, we have that $\su_t  = U_H(\vartheta_t, x_t) - P_{\vartheta_t}U_H(\vartheta_t, x_{t-1})$ is a martingale difference sequence.
	By using the Possion equation and Lemma~\ref{lem:bounded-growth}, we know that $\sum_{t=1}^{\floor{Tr}} \su_t $ is square integrable for all $r \in [0, 1]$.
	
	In order to figure out the asymptotic behavior of $\sum_{t=1}^{\floor{Tr}} \su_t$, we decompose $\su_t$ into two parts  
	$\su_t = \su_{t,1} + \su_{t,2}$ where
        \begin{equation}
	\begin{aligned}
		\label{eq:U-12}
		\su_{t, 1} &=  \left[  U_H(\vartheta_t, x_t) - P_{\vartheta_t} U_H(\vartheta_t, x_{t-1}) \right] - \left[  U_H(\vtheta, x_t) -  P_{\vartheta_t} U_H(\vtheta, x_{t-1}) \right],\\
		\su_{t, 2}&=  \left[  U_H(\vtheta, x_t) -  P_{\vartheta_t} U_H(\vtheta, x_{t-1}) \right].
	\end{aligned}
        \end{equation}
	It's clear that both $\{ \su_{t,1} \}_{t \ge 0}$ and $\{ \su_{t,2} \}_{t \ge 0}$ are also martingale difference sequences.
	
	We assert that $\frac{1}{\sqrt{T}} \sum_{t=1}^{\floor{Tr}} \su_t$ has the same asymptotic behavior as $\frac{1}{\sqrt{T}} \sum_{t=1}^{\floor{Tr}} \su_{t,2}$ due to
	\[
	\EB \sup_{r \in [0, 1]} \left\|\frac{1}{\sqrt{T}} \sum_{t=1}^{\floor{Tr}} \su_t- \frac{1}{\sqrt{T}} \sum_{t=1}^{\floor{Tr}} \su_{t,2} \right\|^2=
	\EB \sup_{r \in [0, 1]} \left\|\frac{1}{\sqrt{T}} \sum_{t=1}^{\floor{Tr}} \su_{t,1} \right\|^2= o(1).
	\]
	This is because Holder's inequality implies that
        \begin{equation*}
	\begin{aligned}
		\EB\|\su_{t,1}\|^2
		&\le \EB\|U_H(\vartheta_t, x_t)-U_H(\vtheta, x_t)\|^2
		= \EB P_{\vartheta_t} \|U_H(\vartheta_t, x_{t-1})-U_H(\vtheta, x_{t-1})\|^2\\
		&\le \EB \left[ L(x_{t-1})^2 \cdot \|\vartheta_t-\vtheta\|^2  \right] 
		\le (\EB L(x_{t-1})^p)^{\frac{2}{p}} \cdot (\EB  \|\vartheta_t-\vtheta\|^{\frac{2p}{p-2}})^{1-\frac{2}{p}}\\
		&\le \gO(1) \cdot (\EB  \|\vartheta_t-\vtheta\|^2)^{1-\frac{2}{p}} 
		\le \gO(1) \cdot \eta_t^{1-\frac{2}{p}},
	\end{aligned}
        \end{equation*}
and Doob's martingale inequality implies that
	\[
	\EB \sup_{r \in [0, 1]} \left\|\frac{1}{\sqrt{T}} \sum_{t=1}^{\floor{Tr}} \su_{t,1} \right\|^2
	\le \frac{1}{T} \sum_{t=1}^T\EB\|\su_{t,1}\|^2
	\precsim \frac{1}{T} \sum_{t=1}^T \eta_t^{1-\frac{2}{p}} \to 0.
	\]

In the following, we then focus on the partial-sum process $\frac{1}{\sqrt{T}} \sum_{t=1}^{\floor{Tr}} \su_{t,2}$.
We will verify the sufficient conditions for the martingale central limit theorem.

	For one thing, by Lemma~\ref{assump:U}, $\{\su_{t,2}\}_{t\ge 0}$ has uniformly bounded $p$-th order moments. 
	This is because of Jensen's inequality:
        \begin{equation*}
	\begin{aligned}
			\sup_{t \ge 1}\EB\|\su_{t,2}\|^p
			&\le 2^{p-1}\sup_{t \ge 1} \left[ \EB\| U_H(\vtheta, x_t)\|^p  + \EB\|P_{\vartheta_t}U_H(\vtheta, x_{t-1})\|^p \right] \\
			&\le 2^p\sup_{t \ge 1}  \EB\| U_H(\vtheta, x_t)\|^p  < \infty.
	\end{aligned}
        \end{equation*}
	As a result, for any $\eps > 0$, as $T$ goes to infinity,
        \begin{equation*}
	\begin{aligned}
		\EB \left\{
		\sum_{t=1}^{T} \EB\left[  \left\| \frac{\su_{t,2}}{\sqrt{T}} \right\|^2 {1}_{\{\|\su_{t,2}\| \ge \sqrt{T}\eps \} }\bigg| \FM_{t-1}  \right]\right\}
		&\le \frac{1}{\eps^{\frac{p}{2}-1} T^{\frac{p}{2}}}	\EB \left\{\sum_{t=0}^{T} \EB\left[  \left\| \su_{t,2}\right\|^p \bigg| \FM_{t-1}  \right] \right\} \\
		&\le \frac{\sup_{t \ge 0}\EB\|\su_{t,2}\|^p}{\eps^{\frac{p}{2}-1} T^{\frac{p}{2}-1}}   \to 0,
	\end{aligned}
        \end{equation*}
	which implies
	\[
	\sum_{t=1}^{T} \EB\left[  \left\| \frac{\su_{t,2}}{\sqrt{T}} \right\|^2 {1}_{\{\|\su_{t,2}\| \ge \sqrt{T}\eps \} }\bigg| \FM_{t-1}  \right] \overset{p}{\to} 0.
	\]
	
	For another thing, we notice that
        \begin{equation*}
	\begin{aligned}
		&\EB \left[  \left[  U_H(\vtheta, x_t) - P_{\vartheta_t} U_H(\vtheta, x_{t-1}) \right]  \left[  U_H(\vtheta, x_t) - P_{\vartheta_t} U_H(\vtheta, x_{t-1}) \right]^\top \big|\FM_{t-1}   \right]\\
		=& \EB \left[U_H(\vtheta, x_t)U_H(\vtheta, x_t)^\top \big|\FM_{t-1}   \right] -  P_{\vartheta_t} U_H(\vtheta, x_{t-1}) P_{\vartheta_t} U_H(\vtheta, x_{t-1})^\top \\
		=& P_{\vartheta_t} \left[ U_H(\vtheta, x_{t-1})U_H(\vtheta, x_{t-1})^\top \right]-  P_{\vartheta_t} U_H(\vtheta, x_{t-1}) P_{\vartheta_t} U_H(\vtheta, x_{t-1})^\top\\
  =& \Var_{\vartheta_t}(U_H(\vtheta, x_{t-1})).
	\end{aligned}
        \end{equation*}
	Hence, we have
	\begin{equation}\label{eq:conv_asy_var}
	\frac{1}{T} \sum_{t=1}^T \EB[\su_{t, 2}\su_{t, 2}^\top|\FM_{t-1}] =
 \frac{1}{T} \sum_{t=1}^T \Var_{\vartheta_t}(U_H(\vtheta, x_{t-1}))
 \overset{p}{\to} S.
	\end{equation}
	Hereto, we have shown $\{\su_{t,2}\}_{t \ge 0}$ satisfies the Lindeberg-Feller conditions for martingale central limit theorem.
	Then the martingale FCLT follows from Theorem 4.2 in~\citet{hall2014martingale} (or Theorem 8.8.8 in~\citet{durrett2013probability}, or Theorem 2.1 in~\citet{whitt2007proofs}).
	Therefore, we have
	\[
	\frac{1}{\sqrt{T}} \sum_{t=1}^{\floor{Tr}} \su_{t,2} \overset{w}{\to} S^{1/2} W(r)
	\ \text{and} \ 		\frac{1}{\sqrt{T}} \sum_{t=1}^{\floor{Tr}} \su_{t} \overset{w}{\to} S^{1/2} W(r).
	\]
	Finally, it is easy to show $\su_{t, 1}$ also has uniformly bounded $p$-th order moment and thus $\sup_{t \ge 1} \EB\|\su_{t}\|^p < \infty$.

 \item  By the definition in~\eqref{eq:nu} and Lemma~\ref{assump:U} ,we have 
        \begin{equation}
	\begin{aligned}
		\label{eq:help-nu-0}
		\|\snu_t\| &= \left\| \frac{\eta_{t+1}}{\eta_t} P_{\vartheta_{t+1} }U_H(\vartheta_{t+1}, x_t) - P_{\vartheta_{t}} U_H(\vartheta_t, x_{t})\right\| \nonumber \\
		&\le  \left\|  P_{\vartheta_{t+1} } U_H(\vartheta_{t+1}, x_t) - P_{\vartheta_{t} } U_H(\vartheta_t, x_{t})\right\|
		+ \left\| \frac{\eta_{t+1}-\eta_t}{\eta_t} P_{\vartheta_{t+1} } U_H(\vartheta_{t+1}, x_t) \right\| \nonumber \\
		&\le L_U(x_t) \cdot  \|\vartheta_{t+1} - \vartheta_t\| + \left| \frac{\eta_{t+1}-\eta_t}{\eta_t} \right| \cdot \left\|P_{\vartheta_{t+1} } U_H(\vartheta_{t+1}, x_t) \right\|  \notag \\
		& \precsim L_U(x_t) \cdot \|\vartheta_{t+1} - \vartheta_t\| + o(\eta_t) \left\|P_{\vartheta_{t+1} } U_H(\vartheta_{t+1}, x_t) \right\|  \notag
	\end{aligned} 
        \end{equation}
	From another hand, it follows from Lemma~\ref{lem:bounded-growth} that 
        \begin{equation*}
	\begin{aligned}
				\EB \left[L_U(x_t)\|\vartheta_{t+1} - \vartheta_t\|  \right]
		&\le \eta_t \cdot  \EB \left[L_U(x_t) \| H(\vartheta_t, x_t) \|\right]\\
		&\le  \eta_t \cdot  \sqrt{\EB L_U(x_t)^2 \cdot \EB \| H(\vartheta_t, x_t) \|^2}\\
		&\le  \eta_t \cdot \gO(1),
	\end{aligned}
        \end{equation*}
	and $\EB\left\|P_{\vartheta_{t+1} } U_H(\vartheta_{t+1}, x_t) \right\| = \gO(1)$.
	Therefore, it follows that 
	\[
	\EB \|\snu_t\| = \gO(\eta_t),
	\]
	as a result of which, as $T \to \infty$,
	\[
	\frac{1}{\sqrt{T}}\sum_{t=1}^T \EB\|\snu_t\| 
	\precsim 	\frac{1}{\sqrt{T}}\sum_{t=1}^T \eta_t \to 0.
	\]
 \item By Lemma~\ref{lem:a.s.}, we know that both $\vartheta_t -\vtheta$ and $\eta_t H(\vartheta_{t}, x_t)$ converge to zero almost surely.
Because $\vtheta$ is an inner point of $\Theta$, there must exist a sufficiently small number $r > 0$ such that the ball $\mathcal{B}(\vtheta, r)$ that centers at $\vtheta$ and has a radius $r$ is still contained in $\Theta$.
Then, as long as $t$ is sufficiently large, we have $\|\vartheta_t -\vtheta\| \le r/2$ and $\|\eta_t H(\vartheta_{t}, x_t)\| \le r/2$, which implies that 
$\vartheta_t -\eta_t H(\vartheta_{t}, x_t) \in  \mathcal{B}(\vtheta, r) \subset \Theta$ by the triangle inequality.
As a result, $\vartheta_t -\eta_t H(\vartheta_{t}, x_t)$ will not go out of the domain $\Theta$ and the projection $\Pi_{\Theta}(\cdot)$ seems to disappear, i.e.,
\[
\vartheta_{t+1} = 
\vartheta_t -\eta_t H(\vartheta_{t}, x_t).
\]
The last equality is equivalent to $\sw_t = 0$.
Hence, it holds almost surely that $\sw_t$ is non-zero for only a finite number of iterations and of course $\frac{1}{\sqrt{T}}\sum_{t=1}^T \|\sw_t\|$ converges to zero almost surely.
\end{enumerate}
\end{proof}

\subsection{Proof of Lemma~\ref{lem:residual-terms}}
\label{proof:lem:residual-terms}
\begin{proof}{Proof of Lemma~\ref{lem:residual-terms}}
	We will analyze the four separate terms $\sup_{r \in [0, 1]} \|\Bpsi_k(r)\|(0 \le k \le 3)$ respectively.
	
	\paragraph{For $\Bpsi_0$}
	Lemma~\ref{lem:A} shows $A_{j}^n$ is uniformly bounded.
	As $T \to \infty$,
	\[
	\sup_{r \in [0, 1]} \|\Bpsi_0(r)\| = 
	\frac{1}{\sqrt{T} \eta_1} \sup_{r \in [0, 1]} \|A_1^{\floor{Tr}} B_1\De_1\| \le  \frac{C_A}{\sqrt{T} \eta_1}\|B_1\De_1\| \to 0.
	\]
	
	\paragraph{For $\Bpsi_1$}
	Recall that 
	$\Bpsi_1(r) = \frac{1}{\sqrt{T}} \sum_{t=0}^{\floor{Tr}} A_t^{\floor{Tr}} (\sr_t + \snu_t + \sw_t)$.
	Since $\|A_j^n\| \le C_A$ uniformly for any $n \ge j \ge 0$, it follows that as $T \to \infty$, 
	\[
	\sup_{r \in [0, 1]} \|\Bpsi_1(r)\|  \le \frac{C_A}{\sqrt{T}}  \sum_{t=0}^T \left( \|\sr_t\| + \|\snu_t\| + \|\sw_t\|\right) \overset{p}{\to} 0
	\]
	where the last inequality uses Lemma~\ref{lem:error-analysis}.
	
	\paragraph{For $\Bpsi_2$}
	Recall that $\Bpsi_2(r) = \frac{1}{\sqrt{T}} \sum_{t=0}^{\floor{Tr}} \left(A_t^T -G^{-1}\right)\su_t$ with $\su_t$ a martingale difference.	
	In the following, we set $z_t = \Bpsi_2(t/T)$ (indexed by $t \in [T]$) for simplicity.
	It is clear that $\{z_t, \FM_t \}_{t \in [T]}$ forms a square integrable martingale difference sequence.
	As a result $\{ \|z_t\|, \FM_t \}_{t \in [T]}$ is a submartingale due to $\EB[\|\su_t\||\FM_{t-1}] \ge \|\EB[\su_t|\FM_{t-1}] \| = \|\su_{t-1}\|$ from conditional Jensen's inequality.
	By Doob’s maximum inequality for submartingales (which we use to derive the following $(*)$ inequality), 
        \begin{equation*}
	\begin{aligned}
		\EB\sup_{r \in [0, 1]} \|\Bpsi_{2}(r)\|^2
		&=\EB \sup_{t \in [T]} \|z_t\|^2 
		\overset{(*)}{\le} 4 \EB\|z_T\|^2\\
		&= \frac{4}{T} \sum_{t=0}^T \EB\|\left(A_t^T -G^{-1}\right)\su_t\|^2 \\
		&\le  4\sup_{t \ge 0} \EB\|\su_t\|^2 \cdot  \frac{1}{T} \sum_{t=0}^T\|A_t^T -G^{-1}\|^2 \to 0
	\end{aligned}
        \end{equation*}
	where we use Lemma~\ref{lem:A} and the fact that $\|A_t^T -G^{-1}\|$ is uniformly bounded by $C_0 + \|G^{-1}\|$ to prove $ \frac{1}{T} \sum_{t=0}^T\|A_t^T -G^{-1}\|^2 \to 0$.

	\paragraph{For $\Bpsi_3$}
	Recall that $\Bpsi_3(r) = \frac{1}{\sqrt{T}} \sum_{t=0}^{\floor{Tr}} \left(A_t^{\floor{Tr}} -A_t^{T}\right)\su_t$ with $\su_t$ a martingale difference.	
	Notice that for any $n \in [T]$
	\begin{equation*}
		\begin{aligned}
			\sum_{t=0}^{n}(A_t^T- A_t^{n} )\su_t
			&= \sum_{t=0}^{n} \sum_{j=n+1}^T \left(\prod_{i=t+1}^j B_i\right)\eta_t \su_t
			= \sum_{j=n+1}^{T} \sum_{t=0}^n \left(\prod_{i=t+1}^j B_i\right)\eta_t \su_t\\
			&=\sum_{j=n+1}^{T}  \left(\prod_{i=n+1}^j B_i\right)\sum_{t=1}^n \left(\prod_{i=t+1}^n B_i\right)
			\eta_t \su_t\\
			&=\frac{1}{\eta_{n+1}}A_{n+1}^TB_{n+1}\sum_{t=0}^n \left(\prod_{i=t+1}^n B_i\right)
			\eta_t \su_t.
		\end{aligned}
	\end{equation*}
	By Lemma~\ref{lem:A}, $\|A_{n+1}^TB_{n+1}\| \le C_A(1+
	\eta_1\| G\|)$ for any  $T \ge n \ge 0$.
	Hence,
        \[
		\sup_{r \in [0, 1]}\| \Bpsi_3(r)\|
		= \sup_{n \in [T]} \left\|  \frac{1}{\sqrt{T}} \sum_{t=0}^{n} \left(A_t^{n} -A_t^{T}\right)\su_t \right\|
        \precsim \sup_{n \in [T]} \left\|  \frac{1}{\sqrt{T}}
		\frac{1}{\eta_{n+1}}\sum_{t=0}^n \left(\prod_{i=t+1}^n B_i\right)
		\eta_t \su_t \right\|.
        \]
	Lemma~\ref{lem:error} implies that the r.h.s. of the last inequality converges to zero in probability.
\end{proof}

\subsection{Other Technical Lemmas}

\begin{lemma}
	\label{lem:A}
	Let $B_i := I - \eta_iG$ and $-G$ is Hurwitz (i.e., $\mathrm{Re} \lambda_i(G) > 0$ for all $i \in [d]$).
	For any $n \ge j$, define $A_{j}^{n}$ as
	\begin{gather}
		\tag{\ref{eq:A}}
		A_{j}^n =\sum\limits_{t=j}^{n}\left(\prod\limits_{i=j+1}^{t}B_i\right) \eta_j.	
	\end{gather}
	When $\{\eta_t\}_{t \ge 0}$ satisfies $\eta_t \downarrow 0, t \eta_t \uparrow \infty, \frac{\eta_{t-1}-\eta_t}{\eta_{t-1}} = o(\eta_{t-1})$ and $\sum_{t=1}^{\infty} \eta_t t^{-1/2} < \infty$, then there exist $C_A > 0$ such that $A_j^n$ is uniformly bounded with respect to both $j$ and $n$ for $0 \le j \le k$ (i.e., $\|A_j^n\| \le C_A$ for any $n \ge j \ge 0$), and
	\[
	\frac{1}{n} \sum_{j=0}^n \|A_{j}^{n} - G^{-1}\| \to 0 
	\ \text{as} \ n \to \infty.
	\]
\end{lemma}
\begin{proof}{Proof of Lemma~\ref{lem:A}}
	See Lemma B.7 in~\citep{li2023online} or Lemma 1 in~\citep{polyak1992acceleration}.
\end{proof}

\begin{lemma}
	\label{lem:error}
	Let $\{\varepsilon_t\}_{t \ge 0}$ be a  martingale difference sequence adapting to the filtration $\FM_t$.
	Define an auxiliary sequence $\{ y_t \}_{t\ge0}$ as follows: $y_0 = 0$ and for $t\ge 0$,
	\begin{equation*}
		y_{t+1} =(I - \eta_tG) y_t + \eta_t\varepsilon_t.
	\end{equation*}
	It is easily verified that
	\begin{equation*}
		y_{t+1} = \sum_{j=0}^t \left(\prod\limits_{i=j+1}^{t}\left( I - \eta_iG \right)\right) \eta_j  \varepsilon_j.
	\end{equation*}
	Let $\{\eta_t\}_{t \ge 0}$ satisfy the conditions in Assumption~\ref{assmpt: more condtion FCLT}.
	If $\mathrm{Re} \lambda_i(G) > 0$ for all $i \in [d]$ and $\sup_{t \ge 0}\EB \|\varepsilon_t\|^{p} < \infty$ for $p>2$,
	then we have that when $T \to \infty$, 
	\begin{equation*}
		\frac{1}{\sqrt{T}}	\sup_{r \in [0, 1]} \frac{y_{\floor{(T+1)r}}}{\eta_{\floor{(T+1)r}}} \overset{p}{\to} 0.
	\end{equation*}
\end{lemma}
\begin{proof}{Proof of Lemma~\ref{lem:error}}
	See Lemma 3.3 in~\citep{li2023online}.
\end{proof}

%% file: Journal_version/Convergence_of_the_Asymptotic_Covariance.tex
In this section, we will prove Proposition~\ref{prop: covariance convergence}. 
In fact, if all $x_t$ are generated from the same transition kernel, the convergence of the average of the conditional variances immediately follows by using the mixing property and Birkhoff's ergodic theorem. However, in our case, the transition kernel that $x_t$ generated from changes with iteration, making $\{x_t\}$ no longer a time-homogeneous Markov process. 
However, by making use of Assumption~\ref{ass:continue_P} and~\ref{ass:W-contraction}, we can derive the following two properties.
    \begin{itemize}
        \item For any two series of parameters $\{\vartheta_{t}^{(i)}\}_{t = 0}^\infty,~ i=1,2$ with corresponding stochastic process generated by $X_{t+1}^{(i)}\sim \gP_{\vartheta_{t}^{(i)}}(\cdot | X_{t}^{(i)})$ and $X_0^{(i)} = x \in \Omega$ (it is possible that $\vartheta_{t+1}^{(i)} \in \sigma\left(\{X_l^{(i)}\}_{l=0}^{t}\right)$). Then there exists a coupling between these two stochastic processes such that
        \[
        \EB\norm{X_t^{(1)} - X_t^{(2)}} \precsim (x^{\alpha_1} + 1)t^{\alpha_2}\ssum{s}{0}{t-1}\left\{\EB_x\norm{\vartheta_{s}^{(1)} - \vartheta_{s}^{(2)}}^2\right\}^{\frac{1}{2}},
        \]
        where $\alpha_1$, $\alpha_2$ and $\delta$ are three positive definite constants.
        \item For any two stochastic processes $X_t^{(i)},~ i = 3,4$ who are both generated as $X_{t+1}^{(i)} \sim \gP_{\vartheta_t}(\cdot | X_{t}^{(i)})$ but with different initialization $X_0^{(3)} = x_3,~ X_0^{(4)} = x_4$. Then there exists a coupling such that
        \[
        \EB \norm{X_t^{(3)} - X_t^{(4)}} \precsim (x_3^\beta + x_4^\beta + 1)e^{-\iota t},
        \]
        where $\beta,~\iota$ are positive definite constants.
    \end{itemize}

Then by the telescoping technique, we can construct a series of auxiliary processes $\{X_t^{(l)}\}_{t=0}^T$ with $l = 1,\cdots, T-1$ as follows,
\begin{itemize}
    \item when $t = 0, \cdots, l$, we set $X_t^{(l)} = X_t^*$;
    \item when $t = l+1, \cdots, T$, we set $X_{t+1}^{(l)} \sim \gP_{\vartheta_t}(\cdot| X_t^{(l)})$ and share the same randomness with the generated mechanism of $X_{t+1}^{(0)} \sim \gP_{\vartheta_t}(\cdot | X_t^{(0)})$.
\end{itemize}
Fix an integer $m$ whose value will be determined later on and define $k = [T/m]$. Then
\begin{equation}\label{eq:tele_from_tt_i_to_tt_star}
\begin{aligned}
    \frac{1}{T}\ssum{t}{1}{T}\{f(X_t^{(0)}) - f(X_t^{*})\} 
    &= \frac{1}{T}\ssum{l}{0}{k-1}\ssum{t}{lm+1}{(l+1)m}\{(f(X_t^{(0)}) - f(X_t^{(lm)})) + (f(X_t^{(lm)}) - f(X_t^{*}))\} \nonumber \\
    &= \frac{1}{T}\ssum{l}{0}{k-1}\ssum{t}{lm+1}{(l+1)m}\{(f(X_t^{(0)}) - f(X_t^{(lm)})) + (f(X_t^{(lm)}) - f(X_t^{((l+1)m)}))\}.
\end{aligned}
\end{equation}
In the above equation, we separate the summation into $k$ parts. For every part, we use different auxiliary processes to perform the telescoping. Note the last equation holds because $lm+1 \le t \le (l+1)m \Rightarrow X_t^{((l+1)m)} = X_t^*$.
To prove this lemma, we only need to show the formula of \eqref{eq:tele_from_tt_i_to_tt_star} converges to zero in absolute value.

For any given $l$ and $t \in [lm+1, (l+1)m]$, utilizing Assumption~\ref{ass:W-contraction} yields
\begin{equation*}
\begin{aligned}
    \EB \left|f(X_t^{(0)}) - f(X_t^{(lm)})\right| 
    &\le L_f\EB \left( \norm{X_{lm}^{(0)}}^\beta + \norm{X_{lm}^*}^\beta + 1
    \right)e^{-\iota (t - lm)}\\
    &\precsim e^{-\iota (t - lm)}.
\end{aligned}
\end{equation*}
On the other hand, as for the second part in every term of the double summation \eqref{eq:tele_from_tt_i_to_tt_star}, by noting $X_{lm}^{(lm)} = X_{lm}^{((l+1)m)} = X_{lm}^*$, we can use the first property of $\gP_{\vartheta_1}$ and get
\begin{equation*}
\begin{aligned}
    \EB \left|f(X_t^{(lm)}) - f(X_t^{((l+1)m)})\right| 
    &\precsim
    \EB \left(\norm{X_{lm}^*}^{\alpha_1} + 1\right)(t - lm)^{\alpha_2}\ssum{s}{0}{t - lm}\left\{\EB[\norm{\vartheta_{s + lm} - \vtheta}^2|\gF_{lm}]\right\}^{1/2}\\
    &\precsim m^{\alpha_2}\left(\EB \norm{X_{lm}^*}^{2\alpha_1} + 1\right)^{1/2}\ssum{s}{0}{t - lm}\left(\EB \norm{\vartheta_{s+lm} - \vtheta}^2\right)^{1/2}\\
    &\precsim m^{\alpha_2 + 1/2}\sup\limits_{lm < t \le (l+1)m}\left(\EB \norm{\vartheta_t - \vtheta}^2\right)^{1/2}\\
    &\le m^{\alpha_2 + 1/2}(lm)^{-\frac{\gamma}{2}} = m^{\alpha_2 + 1/2 - \gamma/2} l^{-\gamma/2} =: m^{\alpha_3}l^{-\gamma/2},
\end{aligned}
\end{equation*}
where the last inequality holds by using the convergence rate of $\EB \norm{\vartheta_t - \vtheta}$.

Now we can analyze the right hand of \eqref{eq:tele_from_tt_i_to_tt_star}. It follows that
\begin{equation*}
\begin{aligned}
    &\EB \left|
    \frac{1}{T}\ssum{l}{0}{k-1}\ssum{t}{lm+1}{(l+1)m}\left\{
    (f(X_t^0) - f(X_t^{(lm)})) + (f(X_t^{(lm)}) - f(X_t^{((l+1)m)}))
    \right\}
    \right|\\
    &\precsim \frac{1}{T}\ssum{l}{0}{k-1}\ssum{t}{lm+1}{(l+1)m}[e^{-\iota (t-lm)} + m^{\alpha_3}l^{-\gamma/2}]\le \frac{1}{T}\ssum{l}{0}{k-1}[C + m^{\alpha_3 + 1}l^{-\gamma/2}]\\
    &\precsim \frac{Ck}{T} + \frac{m^{\alpha_3 + 1}k^{1 - \gamma/2}}{T} = \frac{C}{m} + m^{\alpha_3}k^{-\gamma/2}.
\end{aligned}
\end{equation*}
Without loss of generality, we suppose $\alpha_3 + 1 \ge 0$ and let $m = T^{\frac{\gamma}{2\alpha_3 + 2 + \gamma}}$. Then $m^{\alpha_3}k^{-\gamma/2} = \gO\left(T^{-\frac{\gamma}{2\alpha_3 + 2 + \gamma}}\right) = o(1)$.
    Ultimately we have
    \[
        \frac{1}{T}\ssum{t}{1}{T}f(X_t^{(0)}) \overset{\gL_1}{\longrightarrow} \frac{1}{T}\ssum{t}{1}{T}f(X_t^*).
    \]

%% file: Journal_version/proof_of_queue.tex
\subsection{Proofs for the Queueing Example}\label{sec:proof_of_queue}

\subsubsection{Proof of Proposition~\ref{prop:verify_queue}}
In the following, we verify that the gradient estimator \( H(\vartheta, x) \) and the queue dynamics of the GI/GI/1 system satisfy the assumptions outlined in Section~\ref{sec: main results} for the pricing and capacity sizing problem.

We begin by confirming Assumption~\ref{ass:continue_H}. Specifically, this verification is based on the exact form of \( H(\vartheta, x) \) given by equation~\eqref{eq: Q gradient}. The details are provided in the following lemma.
\begin{lemma}
    For the GI/GI/1 system, Assumption~\ref{ass:continue_H} holds by choosing $V(x) = \norm{x}$ with $x = (w,y)$.
\end{lemma}
This part of the proof can be readily obtained by incorporating the specific form of \( H(\vartheta, x) \) in the queueing system, and therefore, we omit the details.

Next, we verify the Lipschitz continuity of the transition kernel with respect to the parameter \( \vartheta \), as required by Assumption~\ref{ass:continue_P}. This property is established in the following lemma.
\begin{lemma}[Parameter continuity in GI/GI/1 queues]\label{lmm:para_cont_queue}
    Consider two GI/GI/1 queues initialized at the same state \((w, y)\) with parameter settings \(\vartheta_i = (\mu_i, p_i)\) for \(i = 1, 2\), where \(\mu_i\) and \(p_i\) represent the service rate and arrival probability, respectively. Let \(P_{\vartheta_i}\) denote the transition kernel corresponding to parameter \(\vartheta_i\). 
    
    For any \(\kappa < \frac{\iota_w}{4} \wedge \frac{\mu\theta}{4}\) and \(\alpha > 0\), there exists a universal constant \(L > 0\) such that:
    \[
    \gW_{\rho, V}(\delta_{w,y}P_{\vartheta_1}, \delta_{w,y}P_{\vartheta_2}) \leq L \norm{\vartheta_1 - \vartheta_2}\big(1 + V(w, y)\big),
    \]
    where:
    \begin{itemize}
        \item \(\gW_{\rho, V}\) is the Wasserstein-type divergence for the cost function \(\rho(x, y)(1 + V(x) + V(y))\),
        \item \(\rho(x, y)\) denotes the Euclidean norm, and
        \item \(V(w, y) = \norm{y}^{\alpha} + e^{\kappa w}\) is a potential function reflecting the system's state.
    \end{itemize}
\end{lemma}

Before establishing the Wasserstein contraction property for the GI/GI/1 queue system, we present a key lemma that highlights the uniformly light-tailed properties of the waiting time sequence \( \{W_n\}_{n=1}^\infty \) and the busy period sequence \( \{Y_n\}_{n=1}^\infty \). These properties are critical for selecting a potential function \( V \) that satisfies Assumptions~\ref{ass:continue_P} and \ref{ass:W-contraction}.

\begin{lemma}[Uniformly light-tailed properties]\label{lem:subexp_of_w&y}
    Under Assumptions~\ref{assmpt: uniform} and \ref{assmpt: light tail}, with fixed parameters \((\lambda, \mu)\), the non-negative random processes \(\{W_n\}_{n=1}^\infty\) and \(\{Y_n\}_{n=1}^\infty\), initialized at \((w, y)\) and following the update rule \eqref{eq:gnr rule of w&y}, are uniformly sub-exponential. Specifically, there exist positive constants \( C_w, \iota_w, C_y, \iota_y \) such that, for any \( n \geq 0 \),
    \[
    \PB(W_n \geq a) \leq C_w e^{-\iota_w (a - w)_+}, \quad 
    \PB(Y_n \geq a) \leq C_y e^{-\iota_y (a - y - \iota^\prime_y w)_+}, \quad \forall a > 0.
    \]
\end{lemma}

Building on these results, we establish the Wasserstein contraction property through the following two lemmas. These results leverage the definition of Wasserstein-type divergence~\eqref{eq:def_w_dist}, combined with the synchronous coupling technique and the update rule~\eqref{eq:gnr rule of w&y}.

\begin{lemma}[Contraction property for waiting times]\label{lmm:init_lip_of_W_n}
Consider two GI/GI/1 queues with parameters \((\mu, p) \in \mathcal{B}\), synchronously coupled with initial waiting times \( W_0^1 \) and \( W_0^2 \). Suppose Assumption~\ref{assmpt: light tail} holds. Then, for any \( \kappa \leq \iota_w/4 \wedge \mu\theta/4 \) and \( \alpha > 0 \), there exist constants \( \iota > 0 \) and \( C > 0 \) such that, conditional on \( W_0^1 \) and \( W_0^2 \),
\begin{equation*}
\begin{aligned}
\EB&\left[\norm{W_n^1-W_n^2}\left(\ssum{i}{1}{2}\norm{Y_n^i}^\alpha + e^{\kappa W_n^i} \right) \bigg|~ W_0^1, W_0^2\right]\\ 
&\quad\quad \leq C e^{-\iota n}\norm{W_0^1 - W_0^2}\left(\ssum{i}{1}{2} \norm{Y_0^i}^\alpha +  e^{2\kappa W_0^i}\right).
\end{aligned}
\end{equation*}
\end{lemma}

\begin{lemma}[Contraction property for busy periods]\label{lmm:init_lip_of_Y_n}
Consider two GI/GI/1 queues with parameters \((\mu, p) \in \mathcal{B}\), synchronously coupled with initial states \((W_0^1, Y_0^1)\) and \((W_0^2, Y_0^2)\). Suppose Assumption~\ref{assmpt: light tail} holds. Then, for any \( \kappa \leq \frac{\iota_w}{8} \wedge \frac{\mu\theta}{8} \wedge \frac{\gamma (\lambda^{-1} - \mu^{-1})}{8} \) and \( \alpha > 0 \), there exist constants \( \iota > 0 \) and \( C > 0 \) such that:
\begin{equation*}
\begin{aligned}
    &\EB \left[\norm{Y_n^1 - Y_n^2}\left(\ssum{i}{1}{2} \norm{Y_n^i}^\alpha + e^{\kappa W_0^i}\right) \bigg|~ \{(W_0^i, Y_0^i\}_{i=1}^2\right]\\
    &\le C\left\{\norm{Y_0^1 - Y_0^2} + \norm{W_0^1 - W_0^2}\right\}\left(\ssum{i}{1}{2}\norm{Y_0^i}^{2\alpha} + e^{2\kappa W_0^i}\right) e^{-\iota n}.
\end{aligned}
\end{equation*}
\end{lemma}

By combining Lemmas~\ref{lmm:init_lip_of_W_n} and \ref{lmm:init_lip_of_Y_n}, we verify that Assumption~\ref{ass:W-contraction} holds for the metric \( \gW_{\norm{\cdot},V} \). Similarly, for the \( \gW_{\mathbbm{1},V} \)-type Wasserstein-type divergence, its contractive property can also be established using synchronous coupling and analogous arguments, as detailed in \citet{chen2023online}.
Building on these results, we construct the required Lyapunov function. Specifically, the lemmas suggest two potential forms for the function \( V \): \( V(w, y) = y + w \) and \( V(w, y) = y + e^{\kappa w} \). Since \( w \precsim e^{\kappa w} \), we adopt \( V(w, y) = y + e^{\kappa w} \) as the potential function, where \( \kappa = \frac{\iota_w}{8} \wedge \frac{\mu\theta}{8} \wedge \frac{\gamma (\lambda^{-1} - \mu^{-1})}{8} \). With this construction, we establish Proposition~\ref{prop:verify_queue}.


\subsubsection{Proof of Lemma~\ref{lmm:para_cont_queue}}
\begin{proof}{Proof of Lemma~\ref{lmm:para_cont_queue}}
Using the triangle inequality, we have 
\[
\norm{(W_1, Y_1) - (W_2, Y_2)} \leq \abs{W_1 - W_2} + \abs{Y_1 - Y_2}.
\]
This allows us to verify separately that both the \( W \) and \( Y \) components satisfy the bound specified in Assumption~\ref{ass:continue_P}. Under synchronous coupling, the dynamics for \( W \) and \( Y \) can be analyzed independently to ensure the required Lipschitz continuity properties.

We focus on the analysis for $W$ as an example.
It follows that
    \begin{equation*}
    \begin{aligned}
    \gW_{\rho,V}(\delta_w P_{\vartheta_1}, \delta_w P_{\vartheta_2}) &= \EB \left[
    \abs{\left(w - \frac{U}{\lambda_1} + \frac{V}{\mu_1}\right)_+ - \left(w - \frac{U}{\lambda_2} + \frac{V}{\mu_2}\right)_+} \big(1 + \ssum{i}{1}{2}V(W_i,Y_i)\big)
    \right]\\
    &\overset{(a)}{\le} \EB \left[\left(\abs{\frac{1}{\lambda_1} - \frac{1}{\lambda_2}}U + \abs{\frac{1}{\mu_1} - \frac{1}{\mu_2}}V\right)\big(1 + \ssum{i}{1}{2}V(W_1,Y_1)\big)\right]\\
    &\precsim \norm{\vartheta_1 - \vartheta_2}\EB \left[  
    (U+V)\left(1 + \ssum{i}{1}{2}\left\{\abs{y + \frac{U}{\lambda_i}}^\alpha + e^{{\kappa}\left(w - \frac{U}{\lambda_i} + \frac{V}{\mu_i}\right)_+}
    \right\}
    \right)
    \right].
    \end{aligned}
    \end{equation*}
    Using Jensen's inequality in conjunction with Assumption~\ref{assmpt: light tail}, we have that
    \begin{equation*}
    \begin{aligned}
        \EB &\left[  
    (U+V)\left(1 + \ssum{i}{1}{2}\left\{\abs{y + \frac{U}{\lambda_i}}^\alpha + e^{{\kappa}\left(w - \frac{U}{\lambda_i} + \frac{V}{\mu_i}\right)_+}
    \right\}
    \right)
    \right]\\
    &\precsim \EB\left[
    (U+V)\left(1 + \abs{y}^\alpha + \frac{\abs{U}^\alpha}{\underline{\lambda}} + e^{\frac{{\kappa}}{\underline{\mu}}V}e^{{\kappa}w}\right)
    \right]
    \precsim \left(1 + \abs{y}^\alpha + e^{{\kappa}w}\right).
    \end{aligned}
    \end{equation*}
      The last inequality follows from the fact that both \( \EB\left[(U+V)\left(1 + \frac{\abs{U}^\alpha}{\underline{\lambda}}\right)\right] \) and \( \EB\left[(U+V)e^{\frac{{\kappa}V}{\underline{\mu}}}\right] \) are uniformly bounded. Returning to the initial bound, we conclude that
    \[
    \gW_{\rho,V}\left(\delta_w P_{\vartheta_1}, \delta_w P_{\vartheta_2}\right) \precsim \norm{\vartheta_1 - \vartheta_2}\left(1 + V(w,y)\right).
    \]
    Similar routines can be used on the analysis of $\gW_{\rho, V}(\delta_y P_{\vartheta_1}, \delta_y P_{\vartheta_2})$ to get the same bound. Consequently, the lemma has been proved.
\end{proof}

\subsubsection{Proof of Lemma~\ref{lem:subexp_of_w&y}} \label{sec:prf_of_subexp_wy}

\begin{proof}{Proof of Lemma~\ref{lem:subexp_of_w&y}}
We first introduce several notations which are also used in the proof of Lemma~\ref{lmm:init_lip_of_Y_n}.
We define \( \tau_n \) as the largest index \( \tau \) in \([0, n]\) such that \( W_\tau = 0 \), that is,
\begin{equation*}
\tau_n = \sup\{-1 < \tau \le n:~ W_\tau = 0\}.
\end{equation*}
Here we use $\tau_n = -1$ to represent the event that $\inf_{t\in [n]}W_t > 0$, then it is obvious that $-1\le \tau_n \le n$ from this definition.
We define
\begin{subequations}
\label{eq:R}
\begin{align}
      R_n &:=  \frac{V_n}{\mu} - \frac{U_n}{\lambda}\\
    r &:= -\EB R_n =  \frac{1}{\lambda} - \frac{1}{\mu} > 0\\
    \Tilde{R}_n&:= R_n + r.  
\end{align}
\end{subequations}
Clearly, $\Tilde{R}_n$ is a zero mean i.i.d random sequence. 
 An important observation is that 
 \begin{equation*}
\ssum{i}{\tau_n}{\tau} R_i \ge 0~\text{for any}~n \ge \tau \ge \tau_n.
 \end{equation*}

    \paragraph{Sub-exponential of W.} We first handle the bound of the waiting time $W_n(w,y) =: W_n$.
    Owing to the universal bound, we have
    \begin{equation*}
    \begin{aligned}
        \PB(W_n > a) &= \ssum{\tau}{-1}{n}\PB(W_n > a, \tau_n = \tau)
        = \PB(W_n > a,\tau_n = -1) + \ssum{\tau}{0}{n-1}\PB(W_n > a,\tau_n = \tau),
    \end{aligned}
    \end{equation*}
    where the last equation holds because $W_n > a \ge 0$ implies $\tau_n \neq n$. We consider the two parts of the above equation separately. Note that $\tau_n = -1$ implies $W_n = w + \ssum{i}{1}{n} R_n$. This means if $a > w$, then there exists a constant $\iota_1$ such that
    \begin{equation*}
    \begin{aligned}
        \PB(W_n > a;&\tau_n = -1) \le \PB\left(w + \ssum{i}{1}{n} R_n > a\right) = \PB\left(\ssum{i}{1}{n}R_n > a - w\right)\\
        &= \PB\left( \ssum{i}{1}{n}\tilde{R}_n > a - w + rn \right)
        \overset{(a)}{\le} \exp\left\{- \iota_1\left(\frac{(a-w+rn)^2}{n}\right)\wedge (a-w+rn)\right\}\\
        &\le \exp\left\{ -\iota_1 (r\wedge 1)(a-w+rn) \right\} = e^{-\iota_1(r\wedge 1)(a-w)}e^{-\iota_1(r\wedge 1)rn}
        \le e^{-\iota_1(r\wedge 1)(a-w)_+}.
    \end{aligned}
    \end{equation*}
    At $(a)$ we use Bernstein's inequality with respect to the i.i.d sub-exponential sequence $\{\tilde{R}_i\}_1^n$ with the mean of zero. While if $a \le w$, $e^{-\iota_1(r\wedge 1)(a-w)_+} = 1$, which means the inequality $\PB(W_n > a;\tau_n=-1) \le e^{-\iota_1(r\wedge 1)(a-w)_+}$ is still correct.
    On the other hand, for $\ssum{\tau}{0}{n-1}\PB(W_n>a;\tau_n = \tau)$, we have
    \begin{equation}
        \begin{aligned}
            & \ssum{\tau}{0}{n-1}\PB(W_n > a, ~ \tau_n = \tau)
            \le \ssum{\tau}{0}{n-1}\PB\left(\ssum{i}{\tau}{n-1} R_i > a\right)
            = \ssum{\tau}{0}{n-1}\PB\left(\ssum{i}{\tau}{n-1} \tilde{R}_i > a + r (n - \tau)\right) \nonumber \\
            &\overset{(a)}{\le} \ssum{\tau}{0}{n-1} e^{- \iota_1 \left(\frac{(a + r (n-\tau))^2}{n - \tau}\right)\wedge (a + r (n-\tau))}
            \overset{(b)}{\le}  e^{-\iota_1(r \wedge 1) a}\ssum{\tau}{0}{n-1} e^{-\iota_1 (r\wedge 1)r (n-\tau)}
            \le \frac{1}{\iota_1 (r\wedge 1)r}e^{-\iota (r \wedge 1) a}. 
        \end{aligned}
    \end{equation}
    Here $(a)$ uses Bernstein's inequality and $(b)$ holds for $\frac{(a + r(n-\tau))}{n-\tau} \ge r$.

    By taking $\iota_w$ as $\iota_1(r\wedge 1)$, we can get
    \begin{equation*}
    \begin{aligned}
        \PB(W_n > a) &= \PB(W_n > a;\tau_n = -1) + \ssum{\tau}{0}{n-1}\PB(W_n > a;\tau_n = \tau)\\
        &\le e^{-\iota_w(a-w)_+} + \frac{1}{\iota_w r}e^{-\iota_w a}\\
        &\le \left(1 + (\iota_w r)^{-1}\right)e^{-\iota_w (a-w)_+},
    \end{aligned}
    \end{equation*}
    and the uniformly sub-exponential property of $W_n$ is verified.

    \paragraph{Sub-exponential of Y.} 
    At a high level, the derivation for the bound of $Y_n$ is similar, i.e., to consider the latest hitting time $\tau_n$ where $W_{\tau_n}$ reaches zero before $n$. 
    Our goal is to prove the following bound,
    \[
    \PB(Y_n > a) \le C e^{-\iota_y(a-y-\iota_y^\prime w)_+}
    \]
    with some positive universal constants $C,\iota_y$ and $\iota_y^\prime$. In fact, when $a \le y + \frac{4}{\lambda r}w$, it holds by taking $C=1$ and $\iota_y^\prime \ge \frac{4}{\lambda r}$. So in the rest of this proof, we assume that $a > y + \frac{4}{\lambda r}w$ always holds.
    
    First, we still need to analyze the condition that $\tau_n = -1$ exclusively. It is obvious in this scenario it's true that $Y_n = y + \ssum{i}{0}{n-1} \frac{U_i}{\lambda}$. At this time we should adopt two different analyzing schemes according to the relative magnitude between $n$ and $\frac{\lambda (a-y)}{4} + \frac{2w}{r}$.

    If $n < \frac{\lambda(a-y)}{4} + \frac{2w}{r}$, we can find a positive constant $\iota_2$ such that
    \begin{equation}\label{eq:bound of P(Y_n > a; tau_n=-1)}
    \begin{aligned}
        \PB\left(Y_n > a; \tau_n = -1\right) &\le \PB\left(y + \ssum{i}{0}{n-1}\frac{U_i}{\lambda} > a\right)
        = \PB\left( \ssum{i}{0}{n-1}\frac{U_i}{\lambda} > a-y \right) = \PB\left( \ssum{i}{0}{n-1}\frac{U_i - 1}{\lambda} > a-y - \frac{n}{\lambda} \right)\\
        &\overset{(b)}{\le} \exp\left\{-\iota_2\frac{\left(a-y - \frac{n}{\lambda}\right)^2}{n}\wedge \left(a-y - \frac{n}{\lambda}\right)\right\}.
    \end{aligned}
    \end{equation}
     $(b)$ uses Bernstein's inequality with respect to i.i.d sub-exponential sequence $\left\{ (U_i - 1)/\lambda \right\}_{i=0}^{n-1}$ with the mean of zero.
    Because $n < \frac{\lambda(a-y)}{4} + \frac{2w}{r}$, we have
    \begin{equation*}
    \begin{aligned}
        a - y -\frac{n}{\lambda} &\ge \frac{3}{4}(a - y) - \frac{2}{\lambda r}w > \frac{3}{4}\left( a - y -\frac{4}{\lambda r}w \right)\\
        \frac{a-y-\frac{n}{\lambda}}{n} &\ge \frac{\frac{3}{4}(a-y) - \frac{2}{\lambda r}w}{\frac{1}{4}(a-y) + \frac{2w}{\lambda r}}
        \overset{(a)}{\ge} \frac{\frac{a-y}{4}}{\frac{(a-y)}{4}+ \frac{2w}{\lambda}} \overset{(b)}{\ge} \frac{1}{3},
    \end{aligned}
    \end{equation*}
    where both $(a)$ and $(b)$ use the fact $a > y + \frac{4}{\lambda r}w$.
    Back to \eqref{eq:bound of P(Y_n > a; tau_n=-1)}, now we can derive
    \[
    \PB(Y_n > a; \tau_n = -1) \le e^{-\frac{\iota_2}{3}\left(a-y-\frac{4}{\lambda r}w\right)}.
    \]
    If $n \ge \frac{\lambda(a-y)}{4} + \frac{2}{r}w$, we instead consider what information we can get from $W_n$. At this time, not only $W_n = w + \ssum{i}{0}{n-1}R_i$ holds, but also $W_n > 0$ is true. This yields the following bound with some positive constant $\iota_1$ so that
    \begin{equation*}
    \begin{aligned}
        \PB(Y_n > a;\tau_n = -1) 
        &\le \PB(\tau_n = -1)\\
        &\le \PB\left( w + \ssum{i}{0}{n-1}R_i > 0\right) = \PB\left(\ssum{i}{0}{n-1}\tilde{R}_n > rn - w\right)\\
        &\overset{(a)}{\le} e^{-\iota_1\frac{(rn-w)^2}{n}\wedge (rn-w)}\overset{(b)}{\le} e^{-\iota_1\left(\frac{r}{2}(rn-w)\right)\wedge (rn-w)}\\
        &= e^{-\iota_1(\frac{r}{2}\wedge 1)(rn-w)} \le e^{- \iota_3 \left(a-y-\frac{4}{\lambda r}w\right)}.
    \end{aligned}
    \end{equation*}
    Here $(a)$ holds by using Bernstein's inequality, $(b)$ holds because $\frac{rn-w}{n} = r - \frac{w}{n} \ge \frac{r}{2}$. $\iota_3$ is defined as $\frac{\iota_1\lambda r(r\wedge 2)}{8}$.
    Consequently, by letting $C=1,~\iota_y \le \iota_3\wedge \frac{\iota_2}{3}$ and $\iota_y^\prime \ge \frac{4}{\lambda r}$, we have 
    \[
    \PB(Y_n > a;\tau_n = -1) \le Ce^{-\iota_y(a - y - \iota_y^\prime w)_+}.
    \]

    Then we put our focus on the condition that $\tau_n \ge 0$. Note at this time $Y_n = \ssum{i}{\tau_n}{n-1} \frac{U_i}{\lambda}$ and $W_n = \ssum{i}{\tau_n}{n-1}R_i$ (since $W_{\tau_n} = Y_{\tau_n} = 0$). Define $\tilde{\tau} = \left(n - \frac{\lambda a}{2}\right)_+$, we then have
    \begin{equation}\label{eq:subexp_y_1}
    \begin{aligned}
        \PB(Y_n > a) & = \ssum{\tau}{0}{n-1}\PB(Y_n > a, \tau_n = \tau)
        = \sum\limits_{\tau \le \tilde{\tau}} \PB(Y_n > a, \tau_n = \tau) + \sum\limits_{\tau > \tilde{\tau}} \PB(Y_n > a, \tau_n = \tau).
    \end{aligned}
    \end{equation}
    When $\tau > \tilde{\tau}$, we have a constant $\iota_2 > 0$ satisfying
    \begin{equation*}
    \begin{aligned}
    \PB(Y_n > a, \tau_n = \tau) &\le \PB\left(\ssum{i}{\tau}{n-1} U_i/\lambda > a\right)
    = \PB\left( \ssum{i}{\tau}{n-1}\frac{U_i - 1}{\lambda} > a - \frac{n - \tau}{\lambda}\right)\\
    &\overset{(a)}{\le} \exp\left\{
    -\iota_2 \frac{(a - (n-\tau)/\lambda)^2}{n-\tau}\wedge \left(a - \frac{n-\tau}{\lambda}\right)
    \right\} \\
    &\overset{(b)}{\le} \exp\left\{ -\frac{\iota_2}{\lambda \vee 1} \left( a - \frac{n - \tau}{\lambda}\right) \right\}
    \overset{(c)}{\le} \exp \left\{ - \frac{\iota_2 a}{2(\lambda \vee 1)}\right\},
    \end{aligned}
    \end{equation*}
    where $(a)$ holds due to Bernstein's inequality and $(b)-(c)$ follow from the condition $\tau > \tilde{\tau}$ which implies $a > \frac{2(n -\tau)}{\lambda}$.

    When $\tau \le \tilde{\tau}$, by Bernstein's inequality, the following can be obtained with some constant $\iota_1$
    \begin{equation*}
    \begin{aligned}
        \PB(Y_n > a, \tau_n = \tau) &\le \PB\left(\ssum{i}{\tau}{n-1}R_i > 0\right)
        = \PB\left(\ssum{i}{\tau}{n-1}\tilde{R}_i > r(n -\tau)\right)\\
        &\le \exp\left\{ -\iota_1 (r\wedge 1)r(n-\tau) \right\}.
    \end{aligned}
    \end{equation*}
    
    We define $\iota_4$ as $\min \left\{\frac{\iota_2}{2(\lambda \vee 1)}, \iota_1 r(r\wedge 1)\right\}$ so that we have
    \[
    \PB(Y_n > a, \tau_n = \tau) \le   \begin{cases}
      e^{-\iota_4 (n -\tau)}  & \text{if}~\tau \le \tilde{\tau}; \\
      e^{-\iota_4 a}  & \text{if}~\tau > \tilde{\tau}.\\
    \end{cases}  
    \]
    Plugging the above results into \eqref{eq:subexp_y_1} yields
    \begin{equation*}
    \begin{aligned}
        \ssum{\tau}{0}{n-1}\PB(Y_n &> a;\tau_n = \tau) = \sum\limits_{\tau \le \tilde{\tau}} \PB(Y_n > a, \tau_n = \tau) + \sum\limits_{\tau > \tilde{\tau}} \PB(Y_n > a, \tau_n = \tau)\\
        &\le \ssum{\tau}{0}{[\tilde{\tau}]} e^{-\iota_4 (n -\tau)} + \ssum{\tau}{[\tilde{\tau}] + 1}{n-1} e^{-\iota_4 a} = \ssum{i}{\lambda a/2}{n} e^{-\iota_4 i} + (n -\tilde{\tau})e^{-\iota_4 a}\\
        &\le e^{-\frac{\iota_4 \lambda a}{2}} \ssum{i}{0}{n - \lambda a/2} e^{-\iota_4 i} + \frac{\lambda a}{2}e^{-\iota_4 a} \le \left(\frac{\lambda a}{2} + 1\right) e^{- \left(\frac{\lambda}{2}\wedge 1\right)\iota_4 a} \le C e^{-\iota_y a}.
    \end{aligned}
    \end{equation*}
    The last inequality is derived by letting $C$ satisfying $\frac{\lambda a}{2}+ 1 \le Ce^{\frac{\lambda \wedge 2}{4}\iota_4 a},~\forall a > 0$ and $\iota_y := \frac{\lambda \wedge 2}{4}\iota_4$.
    Here, we assume that $\lambda a/2$ is an integer. In cases where it is not, we can simply replace it with its floor or ceiling value.
    Finally, by combining the results of the different cases $\tau_n = -1$ and $\tau_n \ge 0$, we have the following bound with $C,~\iota_y$ and $\iota_y^\prime$,
    \[
    \PB(Y_n > a) \le Ce^{-\iota_y(a-y-\iota_y^\prime w)_+}.
    \]
\end{proof}

\subsubsection{Proof of Lemma~\ref{lmm:init_lip_of_W_n}}
\begin{proof}{Proof of Lemma~\ref{lmm:init_lip_of_W_n}}
    We define the hitting times by $\tau_i := \min\{t: W_t^i = 0\}, i\in \{1,2\}$. Without loss of generality, we assume $W_0^1 < W_0^2$, which implies $\tau_1 \le \tau_2$.
    It is not difficult to find that $n \ge \tau_2$ implies $W_n^1 = W_n^2$. At this time, we have $\norm{W_n^1 - W_n^2}(e^{\kappa W_n^1} + e^{\kappa W_n^2}) = 0$. Besides this, another finding is that when $n < \tau_2$, $\norm{W_n^1 - W_n^2} \le \norm{W_0^1 - W_0^2}$ under the synchronous coupling. As a result, we have
   \begin{equation*}
   \begin{aligned}
       \EB\left[\norm{W_n^1-W_n^2}(e^{\kappa W_n^1} + e^{\kappa W_n^2}) \mid W_0^1, W_0^2\right] 
       &= \EB\left[\norm{W_n^1-W_n^2}(e^{\kappa W_n^1} + e^{\kappa W_n^2})\mathbbm{1}_{\{n \le \tau_2\}} \mid W_0^1, W_0^2\right]\\
       &\leq
       \EB\left[\norm{W_0^1-W_0^2}(e^{\kappa W_n^1} + e^{\kappa W_n^2})\mathbbm{1}_{\{n \le \tau_2\}} \mid W_0^1, W_0^2\right]\\
       &= \norm{W_0^1-W_0^2} \EB\left[(e^{\kappa W_n^1} + e^{\kappa W_n^2})\mathbbm{1}_{\{n \le \tau_2\}} \mid W_0^1, W_0^2\right].
   \end{aligned}
   \end{equation*}
   Define several notations
   \[
   R_n:= \frac{V_n}{\mu} - \frac{U_n}{\lambda},~ r:= -\EB R_n = \frac{1}{\lambda} - \frac{1}{\mu} >0,~ \text{and } \tilde{R}_n := R_n + r.
   \]
   Given $n \le \tau_2$, we have $ W_n^i = W_0^i + \ssum{i}{1}{n}R_i$ for $i = 1,2$. Hence,
   \begin{equation*}
   \begin{aligned}
   \EB[(e^{\kappa W_n^1} + e^{\kappa W_n^2})\mathbbm{1}_{\{n \le \tau_2\}}\mid W_0^1, W_0^2] 
   &\le 2e^{\kappa W_0^2} \EB\left[ \exp\left\{\kappa \ssum{i}{1}{n} R_i\right\}\mathbbm{1}\left\{W_0^2 + \ssum{i}{1}{n}R_n \ge 0\right\}\Big| W_0^1, W_0^2\right]\\
   &\overset{(a)}{=} 2 \int_0^\infty \kappa e^{\kappa s}\PB\left(\ssum{i}{1}{n}R_i \ge s - W_0^2\right)d s \\
   &= 2 \int \kappa e^{\kappa s}\PB\left(\ssum{i}{1}{n}\tilde{R}_i \ge rn + s - W_0^2\right) ds \\
   &\overset{(b)}{\le} 2\int_0^\infty \kappa e^{\kappa s}e^{-(\gamma r n + \mu\theta s - \mu\theta W_0^2)_+}ds.
   \end{aligned}
   \end{equation*}
   Here $(a)$ holds for Fubini's lemma and $(b)$ holds by combining Assumption~\ref{assmpt: light tail} with Bernstein's inequality.
   As for the exponential term $e^{\kappa s -(\gamma r n + \mu\theta s - \mu\theta W_0^2)_+}$, because $\kappa < (\iota_w \wedge \mu\theta)/4$, we have $2\kappa < (\iota_w \wedge \mu\theta)/2$. This implies $\frac{2\kappa}{\mu\theta} < 1$. Then the following holds
   \[
   e^{\kappa s -(\gamma r n + \mu\theta (s-W_0^2))_+} \le e^{\kappa s - \frac{2\kappa}{\mu\theta}(\gamma r n + \mu\theta (s-W_0^2))_+} \le e^{-\iota n}e^{2\kappa W_0^2}e^{-\kappa s}
   \]
   with $\iota := \frac{2\kappa \gamma r}{\mu\theta} > 0$.
   This yields
   \[
   \int_0^\infty \kappa e^{\kappa s - (\gamma r n  + \mu\theta(s - W_0^2))_+} \le \kappa e^{2\kappa W_0^2} e^{-\iota n}\int_0^\infty e^{- \kappa s} = e^{2\kappa W_0^2} e^{-\iota n}.
   \]

   Combining the above results, we finally get
   \begin{equation*}
   \begin{aligned}
       &\EB\left[\norm{W_n^1-W_n^2}(e^{\kappa W_n^1} + e^{\kappa W_n^2}) \mid W_0^1, W_0^2\right] \le \norm{W_0^1 - W_0^2}\EB\left[(e^{\kappa W_n^1} + e^{\kappa W_n^2})\mathbbm{1}\{n \le \tau_2\}\mid W_0^1, W_0^2\right]\\
       &\le 4\norm{W_0^1 - W_0^2}e^{2\kappa W_0^2}e^{-\iota n} \le 4\norm{W_0^1 - W_0^2}(e^{2\kappa W_0^2} + e^{2\kappa W_0^1})e^{-\iota n}.
   \end{aligned}
   \end{equation*}
   The lemma is concluded by letting $C = 4$.
\end{proof}

\subsubsection{Proof of Lemma~\ref{lmm:init_lip_of_Y_n}}
\begin{proof}{Proof of Lemma~\ref{lmm:init_lip_of_Y_n}}
    We first define the hitting times by $\tau_i := \min\{t : W_t^i = 0\},~ i\in \{1,2\}$.
    The following are several useful notations
    \[
    R_n :=  \frac{V_n}{\mu} - \frac{U_n}{\lambda},~r := -\EB R_n =  \frac{1}{\lambda} - \frac{1}{\mu} > 0,~\text{and}~\Tilde{R}_n:= R_n + r.
    \]
    Clearly, $\Tilde{R}_n$ is a zero mean i.i.d random sequence. These quantities are crucial for our following analysis.
    We start the proof by considering three cases based on whether $W_0^1, W_0^2$ and $Y_0^1, Y_0^2$ are equal. 
    
    \paragraph{Case 1: $W_0^1 = W_0^2 =: W_0$ and $Y_0^1 \neq Y_0^2$}

    In this scenario, it is straightforward to observe that $W_n^1 = W_n^2$ by utilizing the recursion~\eqref{eq:gnr rule of w&y} and the nature of synchronous coupling. This implies $\tau_1 = \tau_2 =: \tau$ and we denote it by $\tau$ for simplicity. Moreover, from the recursion of $Y_n$, we can deduce $Y_\tau^1 = Y_\tau^2 = 0$. Consequently, for any $n \geq \tau$, it is always true that $\norm{Y_n^1 - Y_n^2} = 0$. On the other hand, if $n < \tau$, a crucial observation is that $\norm{Y_n^1 - Y_n^2}$ does not increase as $n$ varies from $0$ to $\tau$ due to the synchronous coupling.
    Hence,
    \[
    \EB \norm{Y_n^1 - Y_n^2}(e^{\kappa W_n^1} + e^{\kappa W_n^2}) = \EB \norm{Y_n^1 - Y_n^2} (e^{\kappa W_n^1} + e^{\kappa W_n^2})\mathbbm{1}_{\{n < \tau\}} \le \norm{Y_0^1 - Y_0^2}\EB e^{\kappa W_n}\mathbbm{1}\{n <\tau \}.
    \]
    Given $n < \tau$, we have $W_n = W_0 + \ssum{i}{1}{n}R_i$. Hence, by making use of Fubini's formula,
    \begin{equation*}
    \begin{aligned}
        \EB e^{\kappa W_n}\mathbbm{1}\{n < \tau\} &= e^{\kappa W_0}\EB \exp\left\{\kappa \ssum{i}{1}{n} R_i\right\}\mathbbm{1}\left\{W_0 + \ssum{i}{1}{n}R_n \ge 0\right\}\\
        &=  \int_0^\infty \kappa e^{\kappa s}\PB\left(\ssum{i}{1}{n}R_i \ge s - W_0\right)\rd s = \int \kappa e^{\kappa s}\PB\left(\ssum{i}{1}{n}\tilde{R}_i \ge rn + s - W_0\right) \rd s \\
   &\le \int_0^\infty \kappa e^{\kappa s}e^{-(\gamma r n + \mu\theta s - \mu\theta W_0)_+}\rd s.
    \end{aligned}
    \end{equation*}
     As for the exponential term $e^{\kappa s -(\gamma r n + \mu\theta s - \mu\theta W_0)_+}$, because $\kappa < (\iota_w \wedge \mu\theta)/4$, we have that $2\kappa < (\iota_w \wedge \mu\theta)/2$. This implies $\frac{2\kappa}{\mu\theta} < 1$. Then the following holds
   \[
   e^{\kappa s -(\gamma r n + \mu\theta (s-W_0))_+} \le e^{\kappa s - \frac{2\kappa}{\mu\theta}(\gamma r n + \mu\theta (s-W_0))_+} \le e^{-\iota n}e^{2\kappa W_0}e^{-\kappa s}
   \]
   with $\iota := \frac{2\kappa \gamma r}{\mu\theta} > 0$.
   This yields
   \[
   \int_0^\infty \kappa e^{\kappa s - (\gamma r n  + \mu\theta(s - W_0))_+} \le \kappa e^{2\kappa W_0} e^{-\iota n}\int_0^\infty e^{- \kappa s} = e^{2\kappa W_0} e^{-\iota n}.
   \]
   Consequently, in this case, we have
   \[
   \EB \norm{Y_n^1 - Y_n^2}(e^{\kappa W_n^1} + e^{\kappa W_n^2}) \le \norm{Y_0^1 - Y_0^2}e^{2\kappa W_0 - \iota n} \le \norm{Y_0^1 - Y_0^2}(e^{2\kappa W_0^1} + e^{2\kappa W_0^2}) e^{-\iota n}.
   \]
    
    \paragraph{Case 2: $W_0^1 \neq W_0^2$ and $Y_0^1 = Y_0^2 =: Y_0$}

    Without loss of generality, we assume $W_0^1 < W_0^2$, which implies $\tau_1 \le \tau_2$. When $n < \tau_1$ or $ n > \tau_2$, by using the synchronous coupling, we have $Y_n^1 = Y_n^2$. When $n\in [\tau_1 , \tau_2]$, simple computation yields 
    \[
    \norm{Y_n^1 - Y_n^2} \le \norm{Y_n^2} = \norm{Y_0 + \ssum{t}{0}{n-1}\frac{U_t}{\lambda}}.
    \]
    Therefore,
    \begin{equation}\label{eq:dvd_by_yn}
    \begin{aligned}
    &\EB \norm{Y_n^1 - Y_n^2}(e^{\kappa W_n^1} + e^{\kappa W_n^2}) \le \EB \norm{Y_n^2}(e^{\kappa W_n^1} + e^{\kappa W_n^2})\mathbbm{1}_{\{\tau_1 \le n \le \tau_2\}}\\
    &\le 2 \EB \left( \norm{Y_0} + \ssum{t}{0}{n-1}\frac{U_t}{\lambda} \right)e^{\kappa W_n^2}\mathbbm{1}_{\{\tau_1 \le n \le \tau_2\}}.
    \end{aligned}
    \end{equation}
    We then deal with $\EB U_t e^{\kappa W_n^2}\mathbbm{1}_{\{\tau_1 \le n \le \tau_2\}},~\forall t \in [n]$ respectively. In fact, 
    \[\EB U_te^{\kappa W_n^2}\mathbbm{1}_{\{\tau_1 \le n \le \tau_2\}}= \underbrace{\EB U_t\mathbbm{1}_{\{U_t \le n\}}e^{\kappa W_n^2}\mathbbm{1}_{\{\tau_1 \le n \le \tau_2\}}}_{=: \gU_1(t)}
    + \underbrace{\EB U_t\mathbbm{1}_{\{U_t > n\}}e^{\kappa W_n^2}\mathbbm{1}_{\{\tau_1 \le n \le \tau_2\}}}_{=: \gU_2(t)}.
    \]
    For $\gU_1(t)$, it can be bounded as $\gU_1(t) \le n \EB e^{\kappa W_n^2}\mathbbm{1}\{\tau_1 \le n \le \tau_2\}$. For $\gU_2(t)$, because $\kappa < \iota_w/4$, we have that $2\kappa < \iota_w/2$. By using Young's inequality, the following derivation holds,
    \begin{equation*}
    \begin{aligned}
        &\gU_2(t) = \EB U_t\mathbbm{1}_{\{U_t > n\}}e^{\kappa W_n^2}\mathbbm{1}{\{-W_0^2 \le \inf\limits_{k \le n}\ssum{i}{0}{k-1}R_i \le -W_0^1\}}\\
        &\le \EB \left(2 U_t^{2}\mathbbm{1}\{U_t > n\} + e^{2\kappa W_n^2}\right)\mathbbm{1}\{-W_0^2 \le \inf\limits_{k\le n }\ssum{i}{0}{k-1}R_i \le - W_0^1\}\\
        &= 2\EB \left\{
        U_t^{2} \mathbbm{1}_{\{U_t > n\}} \PB \left(
        -W_0^2 \le \inf\limits_{k \le n}\ssum{i}{0}{k-1}R_i \le -W_0^1 \bigg| \sigma\left(\left\{U_i\right\}_{i=1}^t\right)
        \right)
        \right\}\\
        &\quad + \EB e^{2\kappa W_n^2}\mathbbm{1}\{-W_0^2 \le \inf\limits_{k\le n }\ssum{i}{0}{k-1}R_i \le - W_0^1\}.
        \end{aligned}
        \end{equation*}
        For the first term of the above inequality, we have
        \begin{equation*}
        \begin{aligned}
        &2\EB \left\{
        U_t^{2} \mathbbm{1}_{\{U_t > n\}} \PB \left(
        -W_0^2 \le \inf\limits_{k \le n}\ssum{i}{0}{k-1}R_i \le -W_0^1 \bigg| \sigma\left(\left\{U_i\right\}_{i=1}^t\right)
        \right)
        \right\}\\
        &\overset{(a)}{\le} 2\EB U_t^{2} \mathbbm{1}_{\{U_t > n\}}\left\{
        \ssum{k}{0}{n-1} \PB \left(-W_0^2 \le \ssum{i}{0}{k-1}R_i \le -W_0^1 \bigg| \sigma(\{U_i\}_{i=0}^t)\right)
        \right\}\\
        &= 2\EB U_t^{2} \mathbbm{1}_{\{U_t > n\}}\left\{
        \ssum{k}{1}{n}\PB\left(
        -W_0^2 {-} \ssum{i}{0}{k-2}R_i {+} \frac{U_{k-1}}{\lambda} \le \frac{V_k}{\mu} \le -W_0^1 {-} \ssum{i}{0}{k-2}R_i {+} \frac{U_{k-1}}{\lambda}
        \bigg| \sigma(\{U_i\}_{i=0}^t)\right) \right\}\\
        &\overset{(b)}{\le} {\mu np_V}\abs{W_0^1 - W_0^2}\cdot\EB U_t^{2} \mathbbm{1}_{\{U_t > n\}} = {\mu np_V}\abs{W_0^1 - W_0^2}\int_{n}^{\infty}x\PB(U_t > x) \rd x\\
        &\overset{(c)}{\le} \frac{\mu np_V}{\eta}\abs{W_0^1 - W_0^2}e^{-\iota n}.
    \end{aligned}
    \end{equation*}
    Here $(a)$ holds for the union bound, $(b)$ is supported by combining the assumption that $p_{V}(x) \le p_V$ and the independence between $V_k$ and the other random variables in $\{V_i, i \le k-1\}\cup \{U_i, i \le t\}$, and $(c)$ is true owing to Assumption~\ref{assmpt: light tail} for some small enough $\iota$.
    
    Then we focus on controlling the second term, i.e., $\EB e^{\kappa W_n^2}\mathbbm{1}\{\tau_1 \le n \le \tau_2\}$ (Here we use the fact that $\{\tau_1\le n\le \tau_2\} = \left\{-W_0^2 \le \inf\limits_{k\le n}\ssum{i}{0}{k-1}R_i \le - W_0^1\right\}$). The routine on bounding $\EB e^{2\kappa W_n^2}\mathbbm{1}\{\tau_1 \le n \le \tau_2\}$ is similar.
    We express the event $\left\{-W_0^2 \le \inf\limits_{k\le n-1} \ssum{i}{0}{k} R_i \le -W_0^1 \right\}$ as a union set of some tractable disjoint events by discussing the index where the smallest value of the partial sum sequence $\ssum{i}{1}{k}R_i$ locates. That is,
    \begin{equation*}
        \mathbbm{1}\left\{ -W_0^2 \le \inf\limits_{k\le n-1}\ssum{i}{1}{k} R_i \le - W_0^1 \right\}
        = \ssum{k}{0}{n-1} \mathbbm{1}\left\{ -W_0^2 \le \ssum{i}{0}{k} R_i \le W_0^1; ~ k = \argmin\limits_{r \le n-1} \ssum{i}{0}{r}R_i\right\}.
    \end{equation*}
    For the sake of brevity in the following proof,
    we denote the event $\left\{-W_0^2 \le \ssum{i}{0}{k} R_i \le -W_0^1\right\}$ as $\gA_k$. And denote the event $\{k = \argmin\limits_{r\le n-1}\ssum{i}{0}{r}R_i\}$ as $\gB_k$.
    For $k \in [1, n/2]$, i.e., we know $\argmin\limits_{r\le n} \ssum{i}{1}{r} R_i \le \frac{n}{2}$, then it holds that $\ssum{i}{k}{r} R_i \ge 0$ for all $r \in [k, n-1]$. 
    Additionally, given $\gA_k \cap \gB_k$,
    \begin{equation*}
    \begin{aligned}
        e^{\kappa W_n^2} = \exp\left\{\kappa W_0^2 + \kappa\ssum{i}{0}{n-1}R_n\right\} &=\exp\left\{\kappa W_0^2 + \kappa \ssum{i}{0}{k}R_i\right\}\exp\left\{\kappa\ssum{i}{k+1}{n-1}R_i\right\}\\
        &\le e^{\kappa(W_0^2 - W_0^1)}\exp\left\{\kappa\ssum{i}{k+1}{n-1}R_i\right\}.
    \end{aligned}
    \end{equation*}
    Consequently,
    \begin{equation*}
    \begin{aligned}
        \EB e^{\kappa W_n^2}\mathbbm{1}\left\{ 
        \gA_k; ~ \gB_k
        \right\}&\le 
        e^{\kappa W_0^2} \EB \exp\left\{\kappa \ssum{i}{k+1}{n-1} R_i \right\}
        \mathbbm{1} \left\{ \gA_k; \ssum{i}{k+1}{n-1} R_i \ge 0 \right\} \\
        &\overset{(a)}{=} e^{\kappa W_0^2}\PB\left(\gA_k\right) \EB \exp\left\{\kappa \ssum{i}{k+1}{n-1}R_i\right\}\mathbbm{1}\left\{\ssum{i}{k}{n-1} R_i \ge 0\right\} \\
        &\le \mu p_V \norm{W_0^1 - W_0^2}e^{\kappa W_0^2} \EB \exp\left\{\kappa \ssum{i}{k+1}{n-1}R_i\right\}\mathbbm{1}\left\{\ssum{i}{k+1}{n-1}R_i \ge 0\right\}
        \\
        &\overset{(b)}{\le} \mu p_V \norm{W_0^1 - W_0^2} e^{\kappa W_0^2} \int_0^\infty \kappa e^{\kappa s}\PB\left(\ssum{i}{k+1}{n-1}\tilde{R}_i \ge r(n-k) + s\right)\rd s
        \\
        &\le \mu p_V \norm{W_0^1 - W_0^2}e^{\kappa W_0^2} \int_0^\infty \kappa\exp \left\{ - \gamma r(n - k) - (\gamma - \kappa)s \right\}\rd s\\
        &\le \mu p_V \norm{W_0^1 - W_0^2}e^{\kappa W_0^2} \exp \left\{ - \frac{\gamma r}{2}n \right\},
    \end{aligned}
    \end{equation*}
    where $(a)$ holds because $R_i = \frac{V_i}{\mu} - \frac{U_i}{\lambda},~ i \in [n]$ are generated identically and independently, $(b)$ holds by using Fubini's theorem, and the last inequality holds as long as $\kappa < \frac{\gamma}{2}$.

    For $k\in (n/2, n]$, using the definition $k = \argmin\limits_{r\le n-1} \ssum{i}{0}{r} R_i$, we have $\ssum{i}{1}{n/2} R_i \ge -W_0^2$. This implies
    \begin{align*}
        &\EB e^{\kappa W_n^2}\mathbbm{1}\left\{ \gA_k , \gB_k \right\} 
        \le
        \EB e^{\kappa (W_0^2 - W_0^1)} \exp \left\{\ssum{i}{k+1}{n-1}R_i\right\}\mathbbm{1}\left\{
        \gA_k; \ssum{i}{1}{n/2}R_i \ge - W_0^2\right\}\\
        & \le e^{\kappa W_0^2}\PB\left(\ssum{i}{1}{n/2}R_i \ge - W_0^2\right) \PB\left( 
        \gA_k \bigg| \sigma \left(\left\{R_i\right\}_{i=1}^{n/2}\right)
        \right)\EB \exp\left\{\kappa\ssum{i}{k+1}{n-1}R_i\right\}\\
        & \overset{(a)}{\le} e^{\kappa W_0^2}\exp \left\{ {-} \left(\frac{\gamma rn}{2} {-} \mu \theta W_0^2\right)_+\right\}\\
        &\quad\quad \times\PB \left( -W_0^2 - \ssum{i}{1}{k-1} R_i + \frac{U_k}{\lambda} \le \frac{V_k}{\mu} \le -W_0^1 - \ssum{i}{1}{k-1}R_i + \frac{U_k}{\lambda} \bigg| \sigma \left(\left\{R_i\right\}_{i=1}^{n/2}\right) \right)\\
        &\le \mu p_V \norm{W_0^1 - W_0^2} e^{\kappa W_0^2} \exp \left\{- \left(\frac{\gamma r}{2}n - \mu\theta W_0^2\right)_+\right\}\\
        &\overset{(b)}{\le} \mu p_V \norm{W_0^1 - W_0^2}e^{2\kappa W_0^2} e^{-\iota n}.
    \end{align*}
    Here $(a)$ holds according to Assumption~\ref{assmpt: light tail}, and $(b)$ holds as long as $\iota \le \frac{ \gamma r \kappa}{2\mu\theta}$.
    
    So far, we have bounded every case of $\EB U_t e^{\kappa W_n^2}\mathbbm{1}\{\tau_1 \le n \le \tau_2\}$ and get
    \begin{equation*}
    \begin{aligned}
        \EB U_t e^{\kappa W_n^2 \mathbbm{1}\{\tau_1 \le n \le \tau_2\}} &\le \frac{\mu n p_V}{\eta \epsilon}\norm{W_0^1 - W_0^2} e^{-\iota n}\\
        &\quad + \mu p_V \norm{W_0^1 - W_0^2} e^{-\frac{\gamma r}{2}n} + \mu p_V \norm{W_0^1 - W_0^2} e^{2\kappa W_0^2}e^{-\iota n}\\
        &{\le} \frac{3\mu n p_V}{\eta}\norm{W_0^1 - W_0^2} e^{2\kappa W_0^2} e^{-\iota n}.
    \end{aligned}
    \end{equation*}

    Plugging this inequality into \eqref{eq:dvd_by_yn} implies that at the Case 2,
    \begin{equation*}
    \begin{aligned}
        \EB \norm{Y_n^1 - Y_n^2}(e^{\kappa W_n^1} + e^{\kappa W_n^2}) &\le \left( \norm{Y_0} + \frac{3n}{\eta} \right)\mu p_V \norm{W_0^1 - W_0^2} n e^{2\kappa W_0^2} e^{-\iota n}\\
        &\le \mu p_V\left(\norm{Y_0}^{2} + \frac{4}{\eta }\right)\norm{W_0^1 - W_0^2}e^{4 \kappa W_0^2}e^{-\iota n}.
    \end{aligned}
    \end{equation*}
    Up to now, we have completed the analysis in Case 2.

    \paragraph{Case 3: $W_0^1 \neq W_0^2$ and $Y_0^1 \neq Y_0^2$}

    At this case, we note that $Y_n^1 = Y_n^1(W_0^1, Y_0^1)$ and $Y_n^2 = Y_n^2(W_0^2, Y_0^2)$.
    To reuse the existing results in Case 1 and Case 2, we introduce a surrogate sequence $Y_n^{3/2}$ whose initial state is $(W_0^1, Y_0^2)$ and let $Y_n^{3/2}$ evolve under the same synchronous coupling with $Y_n^i, ~ i \in \{1,2\}$. Then by leveraging Jensen's inequality,
    \[
    \EB \norm{Y_n^1 - Y_n^2}(e^{\kappa W_0^1} + e^{\kappa W_0^2}) \le \left\{ \EB \norm{Y_n^1 - Y_n^{3/2}}(e^{\kappa W_0^1} + e^{\kappa W_0^2}) + \EB \norm{Y_n^{3/2} - Y_n^2}(e^{\kappa W_0^1} + e^{\kappa W_0^2}) \right\}.
    \]
    Note that the first term has been analyzed in Case 1 and the second term has been analyzed in Case 2.
    We complete the proof by applying the existing bounds.
\end{proof}

\subsubsection{A sufficient condition of Assumption~\ref{assmpt: convexity} for the queueing example
}

In this section, we provide sufficient conditions under the queueing example to ensure Assumption 1 holds. These conditions are simpler and more general than those given in EC. 5 of \citep{chen2023online}.

%% file: Journal_version/proof_of_inventory.tex
\subsection{Proofs for the Inventory Example}

In the following, we verify that the gradient estimator \( H(\vartheta, x) \) and the dynamics of the inventory management satisfy the assumptions outlined in Section~\ref{sec: main results} for the inventory control problem.

\subsubsection{Proof of Proposition~\ref{prop:verify_inventory}}
In this setting, the control parameter is \( \vartheta = S \), and the Markov chain state is \( x_t = (X_t, Y_t) \), as we will take the expectation over the i.i.d. demand variable \( D_t \). 
To verify Assumption~\ref{ass:continue_H}, note that \( \EB_{D_t} \big[H(S, (X_t, Y_t, D_t))\big] \) represents the expected stochastic gradient estimator \( H(S, (X_t, Y_t, D_t)) \) with respect to \( D_t \). For verification, it suffices to consider:
\[
\gP_{S} H(S, (X_t, Y_t)) 
:=\gP_{S} \EB_{D_t}[H(S, (X_t, Y_t, D_t))]
= hF_D\big(I_t(S)\big)I_t^\prime(S) - b\big(1-F_D(I_t(S))\big)I_t^\prime(S),
\]
where \( F_D \) is the cumulative distribution function of \( D_t \). Given this form, Assumption~\ref{ass:continue_H} is naturally satisfied, provided \( D_t \) has a density. Specifically:

\begin{lemma}\label{lem:cont_H_inv}
    Under Assumption~\ref{assmpt:demmand}, in the inventory management problem, for a fixed base-stock level \( S \), Assumption~\ref{ass:continue_H} holds for the gradient estimator \( H(S, (X_t, Y_t, D_t)) \), with the potential function chosen as \( V(x, y) \equiv 1 \).
\end{lemma}

As for Assumption~\ref{ass:continue_P}, we have the following lemma,
\begin{lemma}\label{lem:cont_P_inv}
    Consider two inventory dynamics with the same initialization state $(x,y)$ and two base-stock levels $S_i,~ i = 1,2$. Denote their respective transition kernel as $P_{\vartheta_i}$. Then there exists some universal constant $L$ such that the following holds:
    \[
    \gW_{\norm{\cdot},1}(\delta_{x,y}P_{S_1}, \delta_{x,y}P_{S_2}) \le L \abs{S_1 - S_2}.
    \]
    Here we take the potential function $V(x,y) \equiv 1$.
\end{lemma}

In the final step, we analyze how to verify the Wasserstein contractility property of the inventory dynamics.
For any fixed base stock level $S$, if we regard $I_t(i)$ as a function of the initialization $i$, then what we are curious about now is the Lipschitz constant of the function $I_t(i)$. Thus, let us begin to analyze the path-wise derivative of $I_t(i)$ with respect to $i$. Then, for two inventory dynamics initialized with different states $i_1,~ i_2$ but updated using the same demand sequence $\{D_t\}$,  by the mean value theorem for integrals,
\begin{equation*}
\begin{aligned}
\EB \abs{I_t(i_1) - I_t(i_2)} &= 
\EB \abs{(i_2 - i_1)\int_0^1 dI_t\big((1-\lambda)i_1 + \lambda i_2\big) \rd \lambda}\\
&\le
\abs{i_1 - i_2}\int_0^1 \EB\abs{dI_t\big(
(i_2 - i_1)\lambda + i_1
\big)}\rd \lambda,
\end{aligned}
\end{equation*}
which means we only need to analyze the $l_1$-norm of the derivative, i.e., $\EB\abs{dI_t(i)}$. And this is exactly what the following lemma addresses.
\begin{lemma}\label{lem:I_t_Wcontraction_gen}
    Under Assumption~\ref{assmpt:demmand}, we can find two universal constants $\iota > 0$ and $C > 0$, such that
    \[
    \EB\abs{dI_t(i)} \le C e^{-\iota t}.
    \]
\end{lemma}
Equipped with Lemma~\ref{lem:I_t_Wcontraction_gen}, we derive the following relationship
\begin{equation}\label{eq:I_t_Wcontraction_gen}
    \EB\abs{I_t(i_1) - I_t(i_2)} \le C\abs{i_1 - i_2}e^{-\iota t},
\end{equation}
where \( C > 0 \) and \( \iota > 0 \) are universal constants. Eqn.~\eqref{eq:I_t_Wcontraction_gen} demonstrates that the sequence \( I_t \) satisfies the Wasserstein contraction property. 

Building on this, we further analyze the derivative sequence \( I_t' \), which also satisfies the Wasserstein contraction property. 
The following lemma formalizes this result.

\begin{lemma}\label{lem:I_t^prime_Wcontraction_gen}
    Consider two inventory systems with the same base-stock level \( S \), synchronously coupled with initial states \( (x_1, x_1^{\prime}) \) and \( (x_2, x_2^{\prime}) \), respectively. Under Assumption~\ref{assmpt:demmand}, there exist universal constants \( \iota > 0 \) and \( C > 0 \) such that:
    \[
    \EB \left[\abs{I_t^{1,\prime} - I_t^{2,\prime}} \mid (x_i, x_i^{\prime}),~ i = 1,2\right] \le C\left\{
    \abs{i_1 - i_2} + \abs{i_1^\prime - i_2^\prime}
    \right\}e^{-\iota t}.
    \]
\end{lemma}


The proofs of these two lemmas are deferred to~\ref{sec:prf_genI_Wcon} and~\ref{sec:prf_genIpr_Wcon}, respectively. Based on the update rule of the base-stock policy~\eqref{eq:base_stock_update}, we know that the new order quantity $Q_t$ is uniquely determined by $I_t$ and the previous order quantities. Therefore, from the above lemma, we can naturally deduce the Wasserstein contractility property of the $Q_t$ sequence with respect to the $Q_t^\prime$ sequence. Combining the points above, we can verify that Assumption~\ref{ass:W-contraction} holds for inventory dynamics that satisfy Assumption~\ref{assmpt:demmand}.

\subsubsection{Proof of Lemma~\ref{lem:cont_H_inv} and Lemma~\ref{lem:cont_P_inv}}
First, we verify Lemma~\ref{lem:cont_H_inv}. Denote the initial states of the two sample paths as $(x_j,y_j),~ j=1,2$ with $x_j = (q_{-1}^{(j)},\cdots, q_{-\tau+1}^{(j)}, i^{(j)})$ and $y_j = \left({q_{-1}^{(j)\prime}},\cdots, {q_{-\tau + 1}^{(j)\prime}}, {i^{(j)\prime}}\right)$. Then,
\[
P_{S_j}H\big(S_j, (x_j,y_j)\big) = h F_D\big( i^{(j)}\big) i^{(j)\prime} - b\left(1 - F_D\big(i^{(j)}\big)\right)i^{(j)\prime},\quad j = 1,2.
\]
Hence we can get
\begin{equation*}
\begin{aligned}
    &P_{S_1} H\big(S_1, (x_1,y_1)\big) - P_{S_2} H\big(S_2, (x_2,y_2)\big)\\
    &\quad = 
    h F_D\big(i^{(1)}\big) i^{(1)\prime} - b\big(1 - F_D(i^{(1)})\big) i^{(1)\prime} - \left\{
    h F_D\big(i^{(2)}\big) i^{(2)\prime} - b\big(1 - F_D(i^{(2)})\big) i^{(2)\prime}
    \right\}\\
    &\quad =
    h\left\{F_D(i^{(1)}) - F_D(i^{(2)})\right\} i^{(1)\prime} + h F_D(i^{(2)})\big(i^{(1)\prime} - i^{(2)\prime}\big)\\
    &\quad\quad\quad\quad\quad\quad\quad+ b\left\{F_D(i^{(1)}) - F_D(i^{(2)})\right\}i^{(1)\prime}
     + b\big(1 - F_D(i^{(2)})\big)\big(i^{(1)\prime} - i^{(2)\prime}\big)\\
     &\quad \le 
     L\left\{\abs{i^{(1)} - i^{(2)}} + \abs{i^{(1)\prime} - i^{(2)\prime}}\right\},
\end{aligned}
\end{equation*}
where the last inequality holds by utilizing $i^{(j)\prime} \in \{0,1\}$,~ $F_D(\cdot)$ is Lipschitz continuous and $F_D(\cdot) \in [0,1]$ these three facts.

As for Lemma~\ref{lem:cont_P_inv}, denote the common initial state of the two trajectories as 
\[
x := (q_{-1},\cdots, q_{-\tau+1},i),~~ y = (q_{-1}^\prime, \cdots, q_{-\tau+1}^\prime, i^\prime).
\]
We assume that both trajectories receive the same demand variable $D$ in the first iteration. Then by using the update rule of the inventory dynamics under base-stock policy,
\begin{equation*}
\begin{aligned}
    \EB \left[\abs{I_0^{(1)} - I_0^{(2)}}\right] &= \EB \left[\abs{(i-D)_+ + q_{-\tau+1} - \big((i-D)_+ + q_{-\tau+1}\big)}\right] = 0\\
    \EB \left[\abs{q_0^{(1)} - q_0^{(2)}}\right] &= \EB\left[\abs{(S_1 - I_0 - x_{1:\tau-1}\cdot \mathbf{1}^{\tau-1})_+ - (S_2 - I_0 - x_{1:\tau-1}\cdot \mathbf{1}^{\tau-1})_+}\right] \le \abs{S_1 - S_2}.
\end{aligned}
\end{equation*}
Here $x_{1:\tau-1}$ represents the sub-vector consisting of the first $\tau - 1$ components of $x$. Similarly, we can derive the following continuity result for the derivative sequence.
\begin{equation*}
\begin{aligned}
    \EB \left[\abs{I_0^{(1)\prime} - I_0^{(2)\prime}}\right] &= \EB\left[
    \abs{i^\prime\mathbbm{1}\{D < i\} + q_{-\tau+1}^\prime - i^\prime\mathbbm{1}\{D < i\} + q_{-\tau+1}^\prime}\right] = 0\\
    \EB \left[
    \abs{q_0^{(1)\prime} - q_0^{(2)\prime}}
    \right] &= \EB\left[\left|\left(1 - i^\prime\mathbbm{1}\{D < i\} + y_{1:\tau-1}\mathbf{1}^{\tau-1}\right)\mathbbm{1}
    \{S_1 \ge (i-D)_+ + x_{1:\tau-1}\cdot \mathbf{1}^{\tau-1}\}
    \right.\right.\\
     &~-
     \left.\left.
     \left(1 - i^\prime\mathbbm{1}\{D < i\} + y_{1:\tau-1}\mathbf{1}^{\tau-1}\right)\mathbbm{1}
    \{S_2 \ge (i-D)_+ + x_{1:\tau-1}\cdot \mathbf{1}^{\tau-1}\}\right|\right]\\
    &\lesssim \abs{\PB\left(S_1 \ge (i-D)_+ +x_{1:\tau-1}\cdot \mathbf{1}^{\tau-1}\right) - 
    \PB\left(S_2 \ge (i-D)_+ +x_{1:\tau-1}\cdot \mathbf{1}^{\tau-1}\right)}\\
    &\lesssim \abs{S_1 - S_2}.
\end{aligned}
\end{equation*}
Here the last inequality holds by using the Lipschitz continuous of the cumulative distribution function $F_D$.
\subsubsection{Preliminaries before proving Lemma~\ref{lem:I_t_Wcontraction_gen}}
Before verifying the Wasserstein contractive property of the inventory dynamics, we first introduce an auxiliary sequence \( \{dI_t\}_{t=1}^\infty \).

Define
\(
dI_t(i) := \frac{\rd}{\rd i}I_t(i),~ dQ_t(i) := \frac{\rd}{\rd i}Q_t(i).
\)
According to the update rule of $(I_{t+1}, Q_t)$ under the base-stock policy, we can write down the update rule of $(dI_{t+1}(i), dQ_{t}(i))$,
\begin{equation}
\label{eq:inv_drvt_update_init}
\begin{aligned}
    dI_{t+1} &= dI_t\cdot \mathbbm{1}\{I_t \ge D_t\} + dQ_{t-\tau + 1};\\
    dQ_{t} &= - \left(dX_t\cdot \mathbf{1}^\tau\right)\cdot \mathbbm{1}\{X_t\cdot \mathbf{1}^\tau \le S\}.
\end{aligned}
\end{equation}
Here $dX_t = (dQ_{t-1}, dQ_{t-2}, \cdots, dQ_{t-\tau+1}, dI_t)$. Furthermore, we define two stopping times $\varsigma$ and $\varrho$ as follows, both of which play key roles in the incoming analysis.
\begin{equation*}
\begin{aligned}
    \varsigma &:= \inf\left\{t \ge 0: X_t\cdot \mathbf{1}^\tau \le S\right\};\\
    \varrho &:= \inf\left\{t \ge \varsigma: dI_{t+l} = 0,~\forall l = 0,1,\cdots, \tau \right\}
\end{aligned}
\end{equation*}

The following lemma establishes key properties of two stopping times in the inventory system.

\begin{lemma}\label{lem:gen_lead_time}
    Consider an inventory system \( I_t(i) \) with a replenishment lead time \( \tau \), base-stock level \( S \), and initialization \( i \). The following properties hold:
    \begin{itemize}
        \item For any \( t > \varsigma \):
        \[
        I_t + \sum_{j=t - \tau + 1}^{t-1} Q_j \leq S, \quad \text{and} \quad I_t + \sum_{j=t - \tau + 1}^{t} Q_j = S.
        \]
        Additionally, the incremental changes satisfy that $dI_t + \sum_{j=t-\tau + 1}^{t} dQ_j = 0.$
        \item For any \( t \geq \varrho \), the inventory change is zero: $dI_t = 0.$
    \end{itemize}
\end{lemma}
This part of the proof is similar to \citet[Lemma~1]{Huh2009}, and thus, we omit the proof here.

Now, using Lemma~\ref{lem:gen_lead_time}, we decompose $\EB \abs{dI_t}$ into two parts:
\begin{align*}
\EB\abs{dI_t} &\le \EB\abs{dI_t}\mathbbm{1}\{t/2 \le \varsigma\} + \EB\abs{dI_t}\mathbbm{1}\{\varsigma< t/2 < t \le \varrho\} + \EB\abs{dI_t}\mathbbm{1}\{t> \varrho\}\\
&= \underbrace{\EB\abs{dI_t}\mathbbm{1}\{t/2 \le \varsigma\}}_{\gT_1} + \underbrace{\EB\abs{dI_t}\mathbbm{1}\{\varsigma< t/2 < t \le \varrho\}}_{\gT_2}.
\end{align*}

Then we start to analyze how to control $\gT_1$ and $\gT_2$ separately.

\paragraph{Analysis for $\gT_1$.}
\input{Analysis_for_gT_1}

\paragraph{Analysis for $\gT_2$.}
\input{Analysis_for_gT_2}

\subsubsection{Proof of Lemma~\ref{lem:I_t_Wcontraction_gen}}\label{sec:prf_genI_Wcon}
After the analysis of $\gT_1$ and $\gT_2$, we can easily prove the lemma by combining \eqref{eq:E_It1-It2_3} and \eqref{eq:E_It1-It2_4}. The only remaining task is to verify the claims we used in the analysis.

\begin{proof}{Proof of Claim~\ref{clm:linearity}}
    For this claim, we prove it by induction on \( t \).
    
    Firstly, we have $dX_0^1 + dX_0^2 = dx_1 + dx_2 = dX_0^3$. Suppose the claim holds for $0, 1, \cdots , t$. For the iteration $t+1$, since we couple the two sequences \( dI_t^1 \) and \( dI_t^2 \) using the same randomness, we have
    \begin{equation*}
    \begin{aligned}
    dI_{t+1}^1 +& dI_{t+1}^2 \overset{(a)}{=} dI_{t}^1 \cdot \mathbbm{1}\{I_t \ge D_t\} + dQ_{t-\tau+1}^1 + dI_{t}^2 \cdot \mathbbm{1}\{I_t \ge D_t\} + dQ_{t-\tau+1}^2\\
    &= (dI_t^1 + dI_t^2)\mathbbm{1}\{I_t \ge D_t\} + (dQ_{t-\tau+1}^1 + dQ_{t-\tau+1}^2)
    \overset{(b)}{=} dI_t^3\mathbbm{1}\{I_t \ge D_t\} + dQ_{t-\tau+1}^3 = dI_{t+1}^3.
    \end{aligned}
    \end{equation*}

   Here, $(a)$ holds because, apart from the initialization of the pathwise derivative (i.e., $dX_0^i,~ i = 1,2$), the other components (including $X_0^i,~i=1,2$ and the demand variable $D_t$) are exactly the same for both inventory dynamics. Meanwhile, $(b)$ follows the inductive hypothesis. Similarly, we can calculate
   \begin{equation*}
   \begin{aligned}
       dQ_{t}^1 + dQ_{t}^2 &= - (dX_t^1 \cdot \mathbf{1}^\tau)\mathbbm{1}\{X_t\cdot \mathbf{1}^\tau \le S\} - (dX_t^2 \cdot \mathbf{1}^\tau)\mathbbm{1}\{X_t\cdot \mathbf{1}^\tau \le S\}\\
       &= -(dX_t^1 + dX_t^2)\cdot\mathbf{1}^\tau \mathbbm{1}\{X_t\cdot \mathbf{1}^\tau \le S\} = - (dX_t^3 \cdot \mathbf{1}^\tau)\mathbbm{1}\{X_t\cdot \mathbf{1}^\tau \le S\} = dQ_{t}^3.
   \end{aligned}
   \end{equation*}
   Therefore, we have proved that $dX_{t+1}^1 + dX_{t+1}^2 = dX_{t+1}^3$, which concludes the proof.
\end{proof}

\begin{proof}{Proof of Claim~\ref{clm:alnt_dq}}
    We proceed by induction. First, the result holds naturally for the initialization state $(X_0, dX_0) = (x,dx)$. Assume that the conclusion holds for all time points prior to time \( t-1 \).
    Then, consider the time point $t$. Define $t_1 = \max\left\{s\in [0,t-1], dQ_{s}\neq 0\right\}$. From the content of Claim~\ref{clm:alnt_dq}, we know that \(dQ_{t_1} = \pm dq_0\). Whithout loss of generality, we suppose that \(dQ_{t_1} = dq_0\). Recall the update rule of $dI_t,~ dQ_t$,
    \begin{equation*}
    \begin{aligned}
    dI_{t} &= dI_{t-1}\cdot \mathbbm{1}\{I_{t-1} \ge D_{t-1}\} + dQ_{t-\tau}\\
    dQ_t &= -(dI_t + dQ_{t-1} + \cdots + dQ_{t - \tau + 1})
    \end{aligned}
    \end{equation*}
    Plugging the second equation into the first yields
    \begin{equation*}
    \begin{aligned}
        -\ssum{j}{t-\tau+1}{t}Q_{j} = dI_{t} = dI_{t-1}\cdot\mathbbm{1}\{I_{t-1} \ge D_{t-1}\} + dQ_{t-\tau} = -\left(\ssum{j}{t-\tau}{t-1}Q_j\right)\mathbbm{1}\{I_t > D_t\} + dQ_{t-\tau}.
    \end{aligned}
    \end{equation*}
    Rearranging leads to
    \[
    dQ_{t} = -(dQ_{t-1} + \cdots + dQ_{t-\tau})\mathbbm{1}\{I_{t-1} \le D_{t-1}\}.
    \]
    If $I_{t-1} > D_{t-1}$, then $dQ_t = 0$, which concludes the induction process. Otherwise, $I_t \le D_{t-1}$, then
    \[
    dQ_{t} = -(dQ_{t-1} + \cdots + dQ_{t-\tau}).
    \]
    If $t_1 < t - \tau$, we have $dQ_{t-1} = \cdots = dQ_{t-\tau}$, which implies $dQ_t = 0$. Otherwise, $t_1 \ge t - \tau$, then we can decompose the summation $\ssum{j}{t-\tau}{t -1}dQ_j$ as follows,
    \[
    \ssum{j}{t-\tau}{t - 1}dQ_j = \ssum{r}{0}{R}(dQ_{j_{2r}} + dQ_{j_{2r+1}}).
    \]
    Here, \(dQ_{j_{r}}\) represents all the non-zero elements in the sum, with \(j_{1} < j_{2} < \cdots < j_{2R+1} = t_1\). If there is an even number of non-zero elements between \(t - \tau\) and \(t - 1\), then from the statement of Claim~\ref{clm:alnt_dq}, we know that \(dQ_{j_{2r}} + dQ_{j_{2r+1}} = -dq_0 + dq_0 = 0\), which gives 
    \[dQ_{t_2} = -\ssum{r}{0}{R}(dQ_{j_{2r}}+ dQ_{j_{2r+1}}) = 0 \neq dq_0 ,\] 
    verifying our claim. If there is an odd number of non-zero elements, we set \(dQ_{j_0} = 0\). From the sign pattern of each adjacent pair of non-zero elements in this decomposition, we know that \(dQ_{j_1} = dQ_{2R+1} = dQ_{t_1} = dq_0\). In this case, we have 
    \[dQ_{t_2} = -dQ_{j_1} + \ssum{r}{1}{R}(dQ_{j_{2r}} + dQ_{j_{2r+1}}) = - dq_0,\]
    which also completes the induction process. Therefore, by completing the induction, Claim~\ref{clm:alnt_dq} holds.
\end{proof}

\begin{proof}{Proof of Claim~\ref{clm:I_s>D_s}}
    From the update rule \eqref{eq:inv_drvt_update_init} Under the condition that was given in Claim~\ref{clm:I_s>D_s}, we have that for all $l\in \{0,1,\cdots, \tau\}$,
    \begin{equation*}
    \begin{aligned}
    dI_{t+l} &= dI_{t+l - 1}\mathbbm{1}\{I_{t+l-1} \ge D_{t+l-1}\} + dQ_{t -\tau+l} = dI_{t+l-1} + dQ_{t+l -\tau}\\
    &= \cdots = dI_{t+l-\tau}+ \ssum{j}{t-2\tau +l+1}{t+l-\tau} dQ_{j} = 0.
    \end{aligned}
    \end{equation*}
    Here the last equality holds by using Lemma~\ref{lem:gen_lead_time}. And the claim is concluded by recalling the definition of the stopping time $\varrho$.
\end{proof}

\subsubsection{Proof of Lemma~\ref{lem:I_t^prime_Wcontraction_gen}}\label{sec:prf_genIpr_Wcon}

We prove the lemma by considering the following two cases based on the difference in initialization: \(i_1^\prime \neq i_2^\prime\), and \(i_1^\prime = i_2^\prime\) with \(i_1 \neq i_2\).

\paragraph{Case 1: $i_1^\prime \neq i_2^\prime$:} In this setting, $\abs{i_1^\prime - i_2^\prime} = 1$. Thus, we only need to prove that
\[
\EB\left[\abs{I_t^{1,\prime} - I_t^{2,\prime}}\Big| (x_i, x_i^\prime),~i=1,2\right] \le Ce^{-\iota t}
\]
holds for some $C,~ \iota$. This has been proven in \citep[Theorem 3]{Huh2009}.

\paragraph{Case 2: $i_1^\prime = i_2^\prime$ and $i_1 \neq i_2$:}
First, we consider the random indicator $\mathbbm{1}\{X_{t+1}^{1,\prime} \neq X_{t+1}^{2,\prime}\}$ with $X_{t+1}^{\prime} = (Q^\prime_{t},\cdots, Q^\prime_{t-\tau + 1}, I^\prime_{t+1})$. Note that \( X_{t+1}^{1, \prime} \neq X_{t+1}^{2, \prime} \) implies that these two vectors differ in at least one component. Using basic set operations, we obtain,
\begin{equation}\label{eq:x1pr_neq_x2pr}
\begin{aligned}
    \mathbbm{1}\{X_{t+1}^{1,\prime} \neq X_{t+1}^{2,\prime}\} 
    &= \mathbbm{1}\left(\left(\bigcup\limits_{l = 1}^{\tau - 1}\left\{Q_{t-l}^{1,\prime} \neq Q_{t-l}^{2,\prime}\right\}\cup \left\{Q_t^{1,\prime} \neq Q_t^{2,\prime}\right\}
    \cup
    \left\{
    I_{t+1}^{1,\prime} \neq I_{t+1}^{2,\prime}
    \right\}\right)\right)\\
    &=\mathbbm{1}\left(\bigcup\limits_{l=1}^{\tau-2}\{Q_{t-l}^{1,\prime}\neq Q_{t-l}^{2,\prime}\}\right) + \mathbbm{1}\left(
    \bigcap\limits_{l=1}^{\tau - 2}\{Q_{t-l}^{1,\prime} = Q_{t-l}^{2,\prime}\}; Q_t^{1,\prime} \neq Q_t^{2,\prime}
    \right)\\
    & \quad + \mathbbm{1}\left(
    \bigcap\limits_{l=1}^{\tau - 1} \{Q_{t-l}^{1,\prime} = Q_{t-l}^{2,\prime}\}; Q_{t}^{1,\prime} = Q_{t}^{1,\prime}; I_{t+1}^{1,\prime} \neq I_{t+1}^{2,\prime}
    \right).
\end{aligned}
\end{equation}
For the events in the first two terms in \eqref{eq:x1pr_neq_x2pr}, namely \( \bigcup\limits_{l=1}^{\tau - 2}\{Q_{t-l}^{1, \prime} \neq Q_{t-l}^{2, \prime}\} \) and \( \bigcap\limits_{l=1}^{\tau - 2}\{Q_{t-l}^{1, \prime} = Q_{t-l}^{2, \prime}\} \cap \{Q_t^{1, \prime} \neq Q_t^{2, \prime}\} \), both lead to \( \{X_t^{1,\prime} \neq X_t^{2,\prime}\} \). This is because the first event implies that the first \( \tau - 2 \) components of \( X_t^{1, \prime} \), i.e., \( (Q_{t-l}^{1, \prime}, \cdots, Q_{t-l}^{1, \prime}) \), differ from the corresponding components of \( X_t^{2, \prime} \), i.e., \( (Q_{t-l}^{2, \prime}, \cdots, Q_{t-l}^{2, \prime}) \). For the second event, combining it with the sequence update relationship that \( I_t^{i, \prime} + Q_{t}^{i, \prime} + \cdots + Q_{t - \tau + 1}^{i, \prime} = 1\) for \(\forall i = 1, 2 \), we have that:
\begin{align*}
I_{t}^{1,\prime} = 1 - \ssum{l}{1}{\tau - 2}Q_{t-l}^{1,\prime} - Q_t^{1,\prime} = 1 - \ssum{l}{1}{\tau - 2}Q_{t-l}^{2,\prime} - Q_t^{1,\prime}
\neq
1 - \ssum{l}{1}{\tau - 2}Q_{t-l}^{1,\prime} - Q_t^{2,\prime}
= I_t^{2,\prime}.
\end{align*}
To analyze the event in the last term in \eqref{eq:x1pr_neq_x2pr}, we will show that
\begin{equation}
\label{eq:last-term-event}
\bigcap\limits_{l=0}^{\tau - 1} \{Q_{t-l}^{1,\prime} = Q_{t-l}^{2,\prime}\}\bigcap \{I_{t+1}^{1,\prime} \neq I_{t+1}^{2,\prime} \} \subseteq
\{X_t^{1, \prime} \neq X_t^{2, \prime}\} \cup \left\{\mathbbm{1}\{I_t^1 \ge D_t\} \neq \mathbbm{1}\{I_t^2 \ge D_t\}\right\}.
\end{equation}
To show this, we apply the update rule for \( I_{t+1}^{\prime} \) and obtain
\[
I_{t+1}^{i, \prime} = I_t^{i, \prime} \mathbbm{1}\{I_t^i > D_t\} + Q_{t - \tau + 1}^{i, \prime}, \quad \forall i = 1,2.
\]
We assert under the last event in \eqref{eq:x1pr_neq_x2pr}, we must have either $X_t^{1, \prime} \neq X_t^{2, \prime}$ or $\mathbbm{1}\{I_t^1 \ge D_t\} \neq \mathbbm{1}\{I_t^2 \ge D_t\}$.
If $X_t^{1, \prime} = X_t^{2, \prime}$ while the second is not true, there are two cases.
\begin{itemize}
    \item \textbf{Case 1}: \( \mathbbm{1}\{I_t^1 > D_t\} = \mathbbm{1}\{I_t^2 > D_t\} = 1 \), then $I_t^{1, \prime} = I_{t+1}^{1, \prime} - Q_{t - \tau + 1}^{1, \prime} = I_{t+1}^{1, \prime} - Q_{t - \tau + 1}^{2, \prime} \neq I_{t+1}^{2, \prime} - Q_{t - \tau + 1}^{2, \prime} = I_t^{2, \prime}$, which 
    contradicts with the fact that $I_t^{1, \prime}=I_t^{2, \prime}$  from the condition $X_t^{1, \prime} = X_t^{2, \prime}$.
    \item \textbf{Case 2}: \( \mathbbm{1}\{I_t^1 > D_t\} = \mathbbm{1}\{I_t^2 > D_t\} = 0 \), then  \(Q_{t - \tau + 1}^{1, \prime} = I_{t+1}^{1, \prime} \neq I_{t+1}^{2, \prime} = Q_{t - \tau + 1}^{2, \prime} \), leading to a contradiction with $X_t^{1, \prime} = X_t^{2, \prime}$.
\end{itemize}
 Thus, we prove \eqref{eq:last-term-event}. Furthermore, noting that the three terms in~\eqref{eq:x1pr_neq_x2pr} are pairwise disjoint, we have that
\begin{equation}\label{eq:X_t^prime_diff}
\begin{aligned}
     \mathbbm{1}\{X_{t+1}^{1,\prime} \neq X_{t+1}^{2,\prime}\} 
     &\le \mathbbm{1}\{X_t^{1,\prime} \neq X_t^{2,\prime}\} +
     \mathbbm{1}\left\{\mathbbm{1}\{I_t^1 \ge D_t\} \neq \mathbbm{1}\{I_t^2 \ge D_t\}\right\}\\
     &=
     \mathbbm{1}\{X_t^{1,\prime} \neq X_t^{2,\prime}\}
     + \mathbbm{1}\left(\{I_t^1 < D_t\le I_t^2\}\cup \{I_t^2 < D_t\le I_t^1\}\right).
\end{aligned}
\end{equation}
Taking expectation to the Eqn.~\eqref{eq:X_t^prime_diff} yields
\begin{equation}\label{eq:X_t^prime_diff_1}
\begin{aligned}
    \PB(X_{t+1}^{1,\prime} &\neq X_{t+1}^{2,\prime}) \le 
    \PB(X_{t}^{1,\prime} \neq X_{t}^{2,\prime}) + \EB\abs{F_D(I_t^1) - F_D(I_t^2)}\\
    &\le
    \PB(X_{t}^{1,\prime} \neq X_{t}^{2,\prime}) +
    L\EB\abs{I_t^1 - I_t^2} \le \PB(X_{t}^{1,\prime} \neq X_{t}^{2,\prime}) + LC\abs{i_1 - i_2}e^{-\iota t},
\end{aligned}
\end{equation}
where the last inequality holds due to \eqref{eq:I_t_Wcontraction_gen}.

To analyze the quantity $\abs{I_{t}^{1,\prime} - I_{t}^{2,\prime}} = \mathbbm{1}\{I_{t}^{1,\prime}\neq I_{t}^{2,\prime}\}$, we consider two surrogate sequences $I_l^{3,\prime}$ and $Q_l^{3,\prime},~ l = 0,\cdots, t$, which are recursively defined as follows,
\begin{equation*}
\begin{aligned}
    &I_{l+1}^{3,\prime} = I_l^{3,\prime}\mathbbm{1}\{I_l^1 \ge D_l\} + Q_{l-\tau+1}^{3,\prime},~
    Q_{l}^{3,\prime} = I_{l-1}^{3,\prime}\mathbbm{1}\{I_{l-1}^1< D_{l-1}\},\quad l\in [0,t/2];\\
    &I_{l+1}^{3,\prime} = I_l^{3,\prime}\mathbbm{1}\{I_l^2 \ge D_l\} + Q_{l-\tau+1}^{3,\prime},~
    Q_{l}^{3,\prime} = I_{l-1}^{3,\prime}\mathbbm{1}\{I_{l-1}^2< D_{l-1}\},\quad l\in [t/2+1,t].
\end{aligned}
\end{equation*}
Then we have the following decomposition,
\begin{equation*}
\begin{aligned}
    \EB\abs{I_{t}^{1,\prime} - I_t^{2,\prime}} \le \EB \abs{I_{t}^{1,\prime} - I_t^{3,\prime}} + \EB \abs{I_{t}^{2,\prime} - I_t^{3,\prime}}.
\end{aligned}
\end{equation*}
For \( \EB \abs{I_{t}^{1,\prime} - I_t^{3,\prime}} \), note that the sequences \( I_l^{j,\prime} \) for \( j = 1, 3 \) share the same indicator variable \( \mathbbm{1}\{I_l^1 \ge D_l\} \) during the first \( t/2 \) update steps. Consequently, this ensures that \( X_{t/2}^{1,\prime} = X_{t/2}^{3,\prime} \). Furthermore, using the result from~\eqref{eq:X_t^prime_diff_1}, we obtain
\begin{equation*}
\begin{aligned}
    \EB\abs{I_t^{1,\prime}-I_t^{3,\prime}} &= \PB(I_t^{1,\prime}\neq I_t^{3,\prime}) \le \PB(X_t^{1,\prime} \neq X_t^{3,\prime})\\
    &\le
    \PB(X_{t-1}^{1,\prime}\neq X_{t-1}^{2,\prime}) + C\abs{i_1 - i_2}e^{-\iota t/2} \\
    &\le \cdots \le \PB(X_{t/2}^{1,\prime}\neq X_{t/2}^{3,\prime}) + Ct\abs{i_1 - i_2}e^{-\iota t/2}\\
    &= Ct\abs{i_1 - i_2}e^{-\iota t/2}.
\end{aligned}
\end{equation*}

For \( \EB \abs{I_t^{2,\prime} - I_t^{3,\prime}} \), note that the sequences \( I_{l}^{j,\prime} \) for \( j = 2, 3 \) utilize the same indicator variable \( \mathbbm{1}\{I_l^2 \ge D_l\} \) during the last \( t/2 \) update steps. As a result, we can derive that 
\[
\PB(X_{t}^{2,\prime} \neq X_{t}^{3,\prime} \mid \gF_{t/2}) \le C \mathbbm{1}\{X_{t/2}^{2,\prime} \neq X_{t/2}^{3,\prime}\}e^{-\iota t/2},
\]
by applying Theorem 3 from \citet{Huh2009}. Combining this with~\eqref{eq:X_t^prime_diff_1} yields
\begin{equation*}
\begin{aligned}
    \EB\abs{I_t^{2,\prime} -I_t^{3,\prime}}&= \PB(I_t^{2,\prime}\neq I_t^{3,\prime})\le \PB(X_t^{2,\prime}\neq X_t^{3,\prime})\\
    &\le Ce^{-\iota t/2}\PB\left(
    X_{t/2}^{2,\prime} \neq X_{t/2}^{3,\prime}
    \right)\\
    &\le
    Ce^{-\iota t/2}\left\{\PB\left(
    X_{t/2-1}^{2,\prime} \neq X_{t/2-1}^{3,\prime}\right) + C\abs{i_1 - i_2}\right\}\\
    &\le \cdots \le Ce^{-\iota t/2}\left\{
    \PB(x_1^\prime \neq x_2^\prime) + Ct\abs{i_1 - i_2}
    \right\} \\
    &= C^2\abs{i_1 - i_2}e^{-\iota t/4}.
\end{aligned}
\end{equation*}
Here the last equality holds because the analysis is conducted under the case \( i_1^\prime = i_2^\prime \) and \( i_1 \neq i_2 \).
The lemma is concluded by choosing $C,~ \iota$ as $C^2,~ \iota/4$, respectively.

    

%% file: Analysis_for_gT_1.tex
Note that when $t \le \varsigma$, $X_t\cdot \mathbf{1}^\tau > S$. Then the replenishment quantity for this period $Q_t$ satisfies
\[
Q_t = (S - X_t\cdot\mathbf{1}^\top)_+ = 0,~ \forall t \le \varsigma.
\]
Combining this with $X_t \cdot \mathbf{1}^\tau > S$ yields $I_t > 0$.
Furthermore, we can see that $I_t = (I_{t-1} - D_{t-1})_+ + Q_{t-\tau} = I_{t-1} - D_{t-1} < I_{t-1}$, where the last inequality holds because $D_{t-1}$ is a non-negative random variable. This implies that we can recursively obtain $I_t = I_0 - \ssum{j}{0}{t-1}D_{j} = i - \ssum{j}{0}{t-1}D_j$. Since the sequence $\{D_t\}_{t=1}^\infty$ is independent of the initialization $i$, we have $dI_t = dI_0 = \frac{\rd}{\rd i}i=1$. Hence,
\[
\gT_1 = \EB \abs{dI_t}\mathbbm{1}\{t/2 \le \varsigma\}
= \PB(t/2 \le \varsigma) = \PB(I_{t/2} \ge S) = \PB\left(i - \ssum{j}{0}{t/2-1}D_{j}>S\right).
\]

To bound the probability $\PB\left(i - \ssum{l}{0}{t/2-1}D_l > S\right)$, we use the Assumption~\ref{assmpt:demmand}. Because the demand variable $D$ is a positive random variable, it's obviously that $\EB \exp \left\{-D\right\} \le e^{-2\iota}$ for some positive number $\iota$. Hence,
\begin{equation*}
\begin{aligned}
\PB\left(i - \ssum{l}{0}{t/2-1}D_l > S\right) &= \PB\left(\ssum{l}{0}{t/2-1}-D_l > S-i\right)
= \PB\left(\exp\left\{\ssum{l}{0}{t/2-1}-D_l\right\} > \exp\{S-i\}\right)\\
&\le e^{i - S}\EB\exp\left\{\ssum{l}{0}{t/2-1}{-D_l}\right\} \overset{(a)}{=}e^{i -S}\prod\limits_{l=0}^{t/2-1}\EB e^{-D_l}\overset{(b)}{\le} e^{\bar{M} - S -\iota t}.
\end{aligned}
\end{equation*}
Here, equality \((a)\) holds because \( D_l, l \ge 0 \) are mutually independent, and inequality \((b)\) holds due to the assumption \( I_0^2 \le \bar{M} \).
Therefore, now we have proved that
\begin{equation}\label{eq:E_It1-It2_3}
\gT_1 = \EB \abs{dI_t}\mathbbm{1}\{t/2 \le \varsigma\} \le C e^{-\iota t}
\end{equation}
by taking $C$ as $e^{\bar{M} - S}$.

%% file: Analysis_for_gT_2.tex
To control the magnitude of $\gT_2$, we should first give $\abs{dI_t}$ a uniform bound. For the sake of this, we need the following claims to help us achieve our goal.

\begin{claim}\label{clm:linearity}
    Consider two inventory dynamics $(X_t^1, dX_t^1)$ and $(X_t^2, dX_t^2)$, whose initializations are $(x, dx_1)$ and $(x, dx_2)$ respectively. Furthermore, both of them share the same demand random sequence $\{D_t\}_{t=0}^\infty$. Then we have $dX_t^1 + dX_t^2 = dX_t^3$ with $dX_t^3$ being the corresponding derivative sequence of an inventory dynamics with initialization $(x, dx_1 + dx_2)$ and randomness $\{D_t\}_{t=0}^\infty$.
\end{claim}

\begin{claim}\label{clm:alnt_dq}
Consider an inventory dynamics system \((X_t, dX_t)\) initialized at \((x, dx)\), where the \(\tau\)-dimensional vector \(dx = (dq_{-1}, \dots, dq_{-\tau + 1}, di_0)\) has only one non-zero component. Specifically, let \(dx = dq \cdot \mathbf{e}_1\), where \(\mathbf{e}_k = (0, \dots, 0, 1, 0, \dots, 0)\) is a unit vector with the 1 in the \(k\)-th position. For the sequence \(\{dQ_t\}\), we assert the following: 
\begin{itemize}
\item Each \(dQ_t\) can only take one of the values \(-dq\), \(0\), or \(dq\).  
\item The values \(-dq\) and \(dq\) alternate in the sequence. That is, between any two adjacent \(dq\), there must be a \(-dq\).
\end{itemize}
\end{claim}

We decompose the inventory dynamics $(X_t, dX_t)$ into $\tau$ dynamics $(X_t, dX_t^k),~ k = 1,2,\cdots, \tau$, whose initializations are $dx_k = dq_k \cdot \mathbf{e}_k,~ k = 1,2,\cdots, \tau-1$ and $dx_{\tau} = di\cdot \mathbf{e}_{\tau}$. Then according to the Claim~\ref{clm:linearity}, we know that $\ssum{k}{1}{\tau}dX_t^k = dX_t$. While for every $dX_t^k$, by Claim~\ref{clm:alnt_dq}, we can get $\abs{dQ_t^k} \le \abs{dq_k},~ \forall k = 1,\cdots, \tau-1$ and $\abs{dQ_\tau} \le \abs{di}$. Combining these two results leads to
\begin{equation*}
\abs{dQ_t} = \abs{\ssum{k}{1}{\tau}dQ_t^k} \le \ssum{k}{1}{\tau}\abs{dQ_t^k} \le \ssum{k}{1}{\tau-1}\abs{dq_k} + \abs{di} = \norm{dx}_1.
\end{equation*}

Then we combine Lemma~\ref{lem:gen_lead_time}, Claim~\ref{clm:linearity} and Claim~\ref{clm:alnt_dq} and get
\[
\abs{dI_t} = \abs{\ssum{j}{t-\tau}{t}dQ_j} = \abs{\ssum{j}{t-\tau}{t}\ssum{k}{1}{\tau}dQ_j^k}\le
\ssum{k}{1}{\tau}\abs{\ssum{j}{t-\tau}{t}dQ_j^k} \le \ssum{k}{1}{\tau-1}\abs{dq_k} + \abs{di} = \norm{dx}_1.
\]
Here the last inequality follows from the fact that $dq_k$ and $-dq_k$ present alternatively in the non-zero position of the sequence $\{dQ_t^k\}$, which means that, except for the first or last elements, the non-zero elements in the middle of the summation $\ssum{j}{t-\tau}{t}dQ_t^k$ can cancel out in pairs according to their signs. While it is not difficult to derive that $dX_0 = (0,0,\cdots, 0, 1)$ by leveraging the update rule of $dI_t$ and $dQ_t$. So we ultimately get
\[
\abs{dI_t} \le \norm{dx}_1 = 1,~ \abs{dQ_t} \le \norm{dx}_1 = 1.
\]

Back to the analysis of $\gT_2$, now we can write down that
\[
\gT_2 = \EB \abs{dI_t}\mathbbm{1}\{\varsigma < t/2 < t < \varrho\} \le \PB\left(\varsigma < t/2 < t < \varrho\right)
\]

The following claim plays a key role for the purpose of controlling the probability $\PB(\varsigma < t/2 < t < \varrho)$.

\begin{claim}\label{clm:I_s>D_s}
    If there exists a time point $t \ge \varsigma + \tau$ such that $I_s \ge D_s,~\forall s\in\{t-\tau, t -\tau +1, \cdots, t + \tau - 1\}$, then we have $t \ge \varrho$.
\end{claim}

For any time point $t$, define the \textit{absorbing set} as
\[
\gA_t := \left\{\omega: I_s \ge D_s,~\forall s\in \{t-\tau, \cdots, t + \tau - 1\}\right\}.
\]
Then Claim~\ref{clm:I_s>D_s} tells us that 
\[
\bigcup_{\varsigma+ \tau \le s \le t}\gA_s \subseteq \{t \ge \varrho\}.
\]
However, recall the update rule of the on-hand inventory $I_t$, we have
\begin{equation*}
\begin{aligned}
I_{t} &= (I_{t-1} - D_{t-1})_+ + Q_{t-\tau}
\ge I_{t-1} + Q_{t - \tau} - D_{t-1}\\
&\ge \cdots
\ge I_{t-\tau} + \ssum{j}{t - 2\tau + 1}{t- \tau}Q_{j} - \ssum{j}{t -\tau}{t-1}D_j
\ge S - \ssum{j}{t - \tau}{t-1} D_j.
\end{aligned}
\end{equation*}
Here the last inequality holds by using the first point of Lemma~\ref{lem:gen_lead_time}. Therefore, it's obvious that
\[
\gA_t = \bigcap\limits_{s= t -\tau}^{t+\tau-1} \{I_s \ge D_s\} \supseteq \bigcap\limits_{s=t-\tau}^{t+\tau -1} \left\{S - \ssum{j}{s-\tau}{s-1}D_j \ge D_s\right\} \supseteq \left\{\ssum{j}{t-2\tau}{t+\tau -1}D_j \le S\right\} =: \gB_t.
\]

Consequently, through some basic set operations, we are able to obtain
\begin{equation*}
\begin{aligned}
    \{\varsigma < t/2 < t < \varrho\} &= \{\varsigma < t/2\}\backslash \{t \ge \varrho\} \subseteq 
    \{\varsigma < t/2\}\bigg\backslash \left(
    \bigcup\limits_{s = \varsigma+\tau}^t \gA_s
    \right)\\
    &= \{\varsigma \le t/2\}\cap \left(\bigcap\limits_{s = \varsigma + \tau}^t \gA_s^c\right) \subseteq \{\varsigma \le t/2\}\cap \left(\bigcap\limits_{s = \varsigma + \tau}^t \gB_s^c\right) \subseteq \bigcap\limits_{k = 0}^{\left[\frac{t - 2\tau}{6\tau}\right]} \gB_{\frac{t}{2}+ (3k+1)\tau}.
\end{aligned}
\end{equation*}
Since \(\{D_t\}_{t\ge 0}\) is an i.i.d. sequence, it is not difficult to verify that \(\gB^c_{\frac{t}{2}+ (3k+1)\tau},~ k \ge 0\) are mutually independent random events. From this, we derive that
\begin{equation*}
\begin{aligned}
    \gT_2 &\le \PB(\varsigma < t/2 < t < \varrho) \le
    \PB\left(
    \bigcap\limits_{k=0}^{\left[\frac{t-2\tau}{6\tau}\right]} \gB^c_{\frac{t}{2}+(3k+1)\tau}
    \right) = \prod\limits_{k=0}^{\left[\frac{t-2\tau}{6\tau}\right]}\PB\left(\gB^c_{\frac{t}{2}+(3k+1)\tau}\right)\\
    &= \prod\limits_{k=0}^{\left[\frac{t-2\tau}{6\tau}\right]} \PB\left(
    \ssum{j}{\frac{t}{2}+3k\tau}{\frac{t}{2}+(3k+2)\tau -1}D_j > S 
    \right) \overset{(a)}{\le} \prod\limits_{k=0}^{\left[\frac{t-2\tau}{6\tau}\right]} \left(1 - \prod\limits_{j = \frac{t}{2}+3k\tau}^{\frac{t}{2}+(3k+2)\tau - 1}\PB\left(D_j \le \frac{S}{3\tau}\right)\right)\\
    &\overset{(b)}{\le}
    \prod\limits_{k=0}^{\left[\frac{t-2\tau}{6\tau}\right]}
    \left(
    1 - \PB\left(D \le \frac{\underline{M}}{3\tau}
    \right)^{3\tau}
    \right) {\le} \exp\left\{
    - \PB\left(D\le \underline{M}/3\tau\right)^{3\tau}\cdot \frac{t - 2\tau}{6\tau} 
    \right\} = \exp\{-\iota t\}.
\end{aligned}
\end{equation*}
Here $(a)$ follows from the fact
\begin{equation*}
\begin{aligned}
\PB\left(\ssum{j}{a}{b}D_j > S\right) = 1 - \PB\left(\ssum{j}{a}{b}D_j \le S\right) &\le 1 - \PB\left(\bigcap_{a\le j \le b}\{D_j \le S/(b-a)\}\right)\\
&= 1 - \prod\limits_{a\le j\le b}\PB\left(D_j \le \frac{S}{b-a}\right),
\end{aligned}
\end{equation*}
and $(b)$ holds by substituting $S$ with $\underline{M}$. $\iota$ is defined by $\iota := \frac{\PB\left(D \le \underline{M}/3\tau\right)^{3\tau}}{6\tau}$. Thus, we establish that
\begin{equation}\label{eq:E_It1-It2_4}
    \gT_2 = \EB \abs{dI_t}\mathbbm{1}\{\varsigma < t/2 < t < \varrho\} \le \exp\{-\iota t\}.
\end{equation}

%% file: Journal_version/proofs_of_auxiliary_lemmas.tex
\section{Detailed Regret Analysis for Queueing System}\label{sec:queue_regret}
In Section~\ref{sec:regret_analysis}, we derive the cumulative regret bound for a range of online learning problems, including \textbf{pricing and capacity sizing in queueing systems}, using stream SGD. However, in some literature, the cumulative regret for this problem is analyzed in a continuous-time framework. In this section, we formalize this notion of continuous-time regret and provide the corresponding regret bound for stream SGD.
Following the literature on online learning for queues \citep{chen2023online,jia2022resusable}, we define the cumulated regret $R_T$ of stream SGD in the first $T$ iterations as the expected difference between the actually cost and the optimal long-run cost under the optimal policy $(\mu^\star,p^\star)$. 
Note that the queueing system operates in continuous time, so we need to map the discrete iterations back to the real timeline to evaluate the actual revenue performance. 
Before introducing the formal definition of regret, we first describe the real-time dynamic of the queuing system.

\begin{figure}[t]
    \centering
    \includegraphics[width=1.0\textwidth]{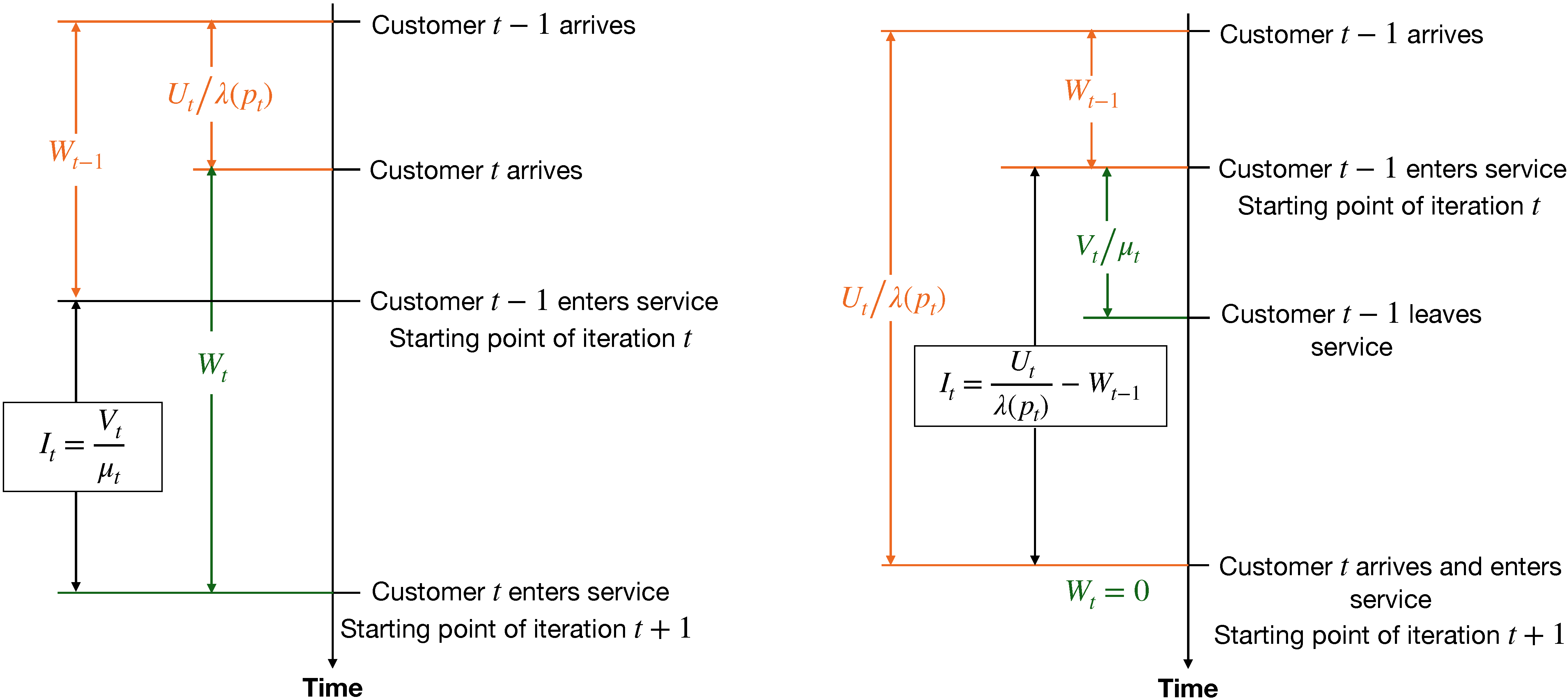}
    \caption{Real-time dynamic of the GI/GI/1 system within a single iteration under stream SGD. 
    The left panel shows the case where the $t$-th customer arrives before the $(t{-}1)$-th customer leaves, while the right panel shows the opposite case.
    }
    \label{fig:service illustration}
\end{figure}

Figure \ref{fig:service illustration} illustrates the dynamic of the queueing system at iteration $t$. 
As the waiting times and observed busy periods are realized when customers enter service, we prescribe that each iteration starts when a customer enters the service and use $I_t$ to denote the time interval between the enter-service times of the consequent customer $t{-}1$ and $t$. 
If the $t$-th customer arrives before the $(t{-}1)$-th customer leaves, then $I_t$ corresponds to the service time of customer $t{-}1$, which equals to $I_t = \frac{V_t}{\mu_t}$.
If the $t$-th customer arrives after the $(t{-}1)$-th customer leaves, we then should take the idle time where no customer receives service into consideration and find that $I_t = \frac{U_t}{\lambda(p_t)} - W_{t-1}$.
In both cases, one can verify that
$$I_t = W_t + \frac{U_t}{\lambda(p_t)} - W_{t-1}.$$
Then, the total time elapsed upon stream SGD finishes $T$ iterations is $\sum_{t=1}^T I_t$ and customer $T+1$ enters service then. As a result, the cumulated regret in $T$ iterations is
\begin{equation} \label{def:regret}
    \widetilde{R}(T) := \EB\left[\ssum{t}{2}{T+1} (h_0(w_t + S_t) - p_t + \zeta(\mu_t)I_t) - f(\mu^\star,p^\star) \ssum{t}{2}{T+1}I_t\right],
\end{equation}
where $S_t:=V_t/\mu_t$ is the service time of customer $t$. For each $t$, the term $(h_0(w_t + S_t) - p_t + \zeta(\mu_t)I_t)$ can be understood as the cost brought by the customer $t$ to the system. The following provides the bound for this new regret, \( \widetilde{R}(T) \), in the continuous-time framework.
\begin{thm}\label{thm:cont_time_regret_queue}
    Consider the pricing and capacity sizing problem. Let $\{(\mu_t,p_t)\}_{t=1}^\infty$ be the parameter sequence outputed by the stream SGD algorithm. For the cumulative regret given in~\eqref{def:regret}, under the step size choice $\eta_t = \frac{\gamma}{t}$,we have
    \[
    \widetilde{R}(T) = \gO(\log T).
    \]
\end{thm}
\begin{proof}{Proof of Theorem~\ref{thm:cont_time_regret_queue}}
Recalling the definition of the cumulated regret $\widetilde{R}_T$ and the GI/GI/1 system generating mechanism, we can derive the following decomposition,
\begin{equation}
\label{eq:regret_decompose}
\begin{aligned}
    \widetilde{R}(T) &= \underbrace{\EB \ssum{t}{{2}}{{T+1}}\left\{ [h_0(w_t {+} S_t) {-} p_t {+} \zeta(\mu_t)I_t] {-} f(\mu_t, p_t)I_t \right\}}_{\widetilde{R}_1(T)} {+} \underbrace{\EB \ssum{t}{2}{T+1}(f(\mu_t,p_t) {-} f(\mu^\star, p^\star))I_t }_{\widetilde{R}_2(T)} \\
    &= \EB \ssum{t}{2}{T+1}\left\{ [h_0(w_t + S_t) - p_t] - \left[ 
    h_0\left( \EB W_\infty(\mu_t, p_t) + \frac{1}{\mu_t} \right) - p_t
    \right] \right\}\nonumber\\
    &\quad + \EB \ssum{t}{2}{T+1}(1 {-} \lambda(p_t)I_t)\left[ 
    h_0\left( \EB W_\infty(\mu_t, p_t) {+} \frac{1}{\mu_t} \right) {-} p_t
    \right] {+} \EB \ssum{t}{2}{T+1}(f(\mu_t,p_t) {-} f(\mu^\star, p^\star))I_t\nonumber\\
    &=
    \underbrace{h_0\ssum{t}{2}{T+1}\left( \EB w_t - \EB W_{\infty}(\mu_t,p_t) \right)}_{\gT_1} + \underbrace{\ssum{t}{2}{T+1}\EB (w_{t+1} - w_t)\lambda(p_t)\Gamma(\mu_t,p_t)}_{\gT_2}\\ 
    &\quad + \underbrace{\EB \ssum{t}{2}{T+1}(f(\mu_t,p_t) {-} f(\mu^\star, p^\star))I_t}_{\gT_3}.  
\end{aligned}
\end{equation}
Here we define $\Gamma(\mu_t,p_t) := h_0\left( \EB W_\infty(\mu_t, p_t) {+} \frac{1}{\mu_t} \right) {-} p_t$ for simplicity.

Here we can see three terms within the final formula, each of which must be managed separately. From an intuitive standpoint, $\gT_1$ arises due to the non-stationary nature of the Markov chain induced by parameter updates in each iteration. $\gT_2$ stems from the randomness of the time intervals between adjacent iterations, denoted as $I_t$. Lastly, $\gT_3$ is rooted in the estimation errors associated with the parameters.

In the following, we analyze the three terms on the r.h.s. of Eqn.~\eqref{eq:regret_decompose} respectively.
For the first term $\gT_1$, only an upper bound for $\EB w_t - \EB W_\infty (\mu_t,p_t)$ is needed, and this is immediately concluded by Lemma~\ref{lem:bd_trans_err_of_regret}. Consequently, we have
\[
\gT_1 = h_0\ssum{t}{2}{T+1}\gO\left(t^{-1}\right) = \gO(\log T).
\]


To analyze $\gT_2$, we recall a previous theoretical result regarding the Lipschitz continuity of asymptotic waiting times and present it in Lemma~\ref{lmm:lip of w_inf}.
Actually, it is a simplified version of Lemma 4 in \citep{chen2023online}.
\begin{lemma}\label{lmm:lip of w_inf}
    For any two pairs of parameters $(\mu_i,p_i)\in [\underline{\mu},\bar{\mu}]\times [\underline{p}, \bar{p}],~i=1,2$, we have
    \[
    \EB W_\infty(\mu_1,p_1) - \EB W_\infty(\mu_2,p_2) = \gO(\abs{\mu_1 - \mu_2} + \abs{p_1 - p_2}).
    \]
\end{lemma}
When Lemma~\ref{lmm:lip of w_inf} is satisfied, the function $\lambda(p)\Gamma(\mu,p)$ exhibits Lipschitz continuity w.r.t. the parameters $(\mu,p)$. Consequently, $\gT_2$ can be reformulated as follows:
\begin{equation*}
\begin{aligned}
    \gT_2 &= \EB w_{T+2}\lambda(p_T)\Gamma(\mu_T,p_T) - \EB w_1\lambda(p_1)\Gamma(\mu_1,p_1)\\
    &+ \ssum{t}{2}{T+1}\left(\lambda(p_{t-1})\Gamma(\mu_{t-1},p_{t-1}) - \lambda(p_t)\Gamma(\mu_t,p_t)\right)w_t\\
    &\precsim 2\left({\lambda(p^\star)}\Gamma(\mu^\star,p^\star) + T^{-\alpha/2}\right)\sup\limits_{k\ge 1}\EB w_k + \ssum{t}{2}{T+1}\EB\norm{\vartheta_{t-1} - \vartheta_t}w_t\\
    &\le 3\lambda(p^\star)\Gamma(\mu^\star,p^\star)\sup\limits_{k\ge 1}\EB w_k + \ssum{t}{2}{T+1}\sqrt{\EB\norm{\vartheta_{t-1} - \vartheta_t}^2}\sqrt{\EB w_t^2}\\
    &\le 3\lambda(p^\star)\Gamma(\mu^\star,p^\star)\sup\limits_{k\ge 1}\EB w_k + \left(\ssum{t}{2}{T+1}\sqrt{\EB\norm{\vartheta_{t-1} - \vartheta_t}^2}\right)\sqrt{\sup\limits_{k \ge 1}\EB w_k^2}\\
    &\precsim \sup\limits_{k\ge 1}\EB w_k + \sqrt{\sup\limits_{k \ge 1}\EB w_k^2} \ssum{t}{2}{T+1}t^{-1} = \gO(\log T).
\end{aligned}
\end{equation*}
The last inequality uses the fact $\sup_{k\ge 1}\EB w_k^j < \infty, (j = 1,2)$, which are direct consequences of Lemma~\ref{lem:subexp_of_w&y}.

The final component $\gT_3$ essentially concerns the convergence rate of $(\mu_t, p_t)$ in terms of the function value $f(\mu,p)$.

Hence, we will draw upon the convergence analysis in Theorem~\ref{thm:L2-convergence}. Additionally, we will make use of the function $f(\mu,p)$'s smooth properties and consider the optimality of $(\mu^\star, p^\star)$.
In fact, we have
\begin{align*}
    \gT_3 &= \ssum{t}{2}{T+1}\EB (f(\mu_t,p_t) - f(\mu^\star,p^\star))I_t \precsim \ssum{t}{2}{T+1}\EB \norm{\vartheta_t - \vtheta}^2 I_t\\
    &= \ssum{t}{2}{T+1}\EB \norm{\vartheta_t - \vtheta}^2\left(w_{t+1} + \frac{U_t}{\lambda(p_t) }- w_t\right) \le \ssum{t}{2}{T+1}\EB \norm{\vartheta_t - \vtheta}^2 \left(\frac{U_t}{\lambda(p_t)} + \frac{V_t}{\mu_t}\right)\\
    &\overset{(a)}{\le} \left(\frac{1}{\underline{\lambda}} + \frac{1}{\underline{\mu}}\right)\ssum{t}{2}{T+1}\EB\norm{\vartheta_t - \vtheta}^2 = \gO(\log T).
\end{align*}
At $(a)$ we make use of $\EB U_t = \EB V_t = 1$ and the fact that $(U_t,V_t)$ are independent with $\vartheta_t$.
Combining the above bounds for $\gT_i~(i = 1,2,3)$ yields the Theorem~\ref{thm:cont_time_regret_queue}.
\end{proof}

%% file: Journal_version/numerical_more.tex
\label{append:more-queue}
In the main paper, we presented results for the M/M/1 queueing system. Here, we provide the corresponding results for the other three settings, which are introduced below. The basic setup remains the same as described in Section \ref{sec:MM1}, with the differences highlighted here.

\paragraph{M/M/1 Queue} 
In this case, the objective function \eqref{eqObjmin} has a closed-form expression as
\[
\min_{(\mu, p) \in \mathcal{B}} \left\{h_0 \frac{\lambda(p)}{\mu - \lambda(p)} + \frac{1}{10} \mu^2 - p \lambda(p) \right\},
\]
from which the exact values of the optimal solutions $(\mu^\star, p^\star)$ and the corresponding objective value $f(\mu^\star, p^\star)$ can be obtained via numerical methods.
We set $\mathcal{B} = [3.5, 10] \times [6.56, 15]$, with which  the optimal solution is $(\mu^\star, p^\star) = (4.0234, 7.103)$.  

\paragraph{M/GI/1 Queue} 
The famous Pollaczek-Khinchine (PK) formula provides a closed-form expression for the steady-state average queue length the for M/GI/1  model:
\begin{equation*}
\EE[Q_{\infty}(\mu,p)] = \rho + \frac{\rho^2}{1-\rho} \frac{1+c_s^2}{2},
\end{equation*}
where $\rho = \frac{\lambda(p)}{\mu}$ is the traffic intensity under the parameter $(p, \mu)$ and $c_s^2 = \Var(V)/[\EE V]^2$ is the squared coefficient of variation (SCV) for the service time random variable $V$.

To introduce an additional degree of flexibility for $c_s$, we adopt the Erlang-$k$ distribution for the service-time distribution whose SCV equals $1/k$.
In the special case when $k=1$, the Erlang-1 distribution is reduced to the Poisson distribution, and the corresponding M/GI/1 queue is reduced to an M/M/1 queue.
When $c_s \ge 1$, we then employ the hyperexponential distribution which is the mixture of two i.i.d. exponential distributions.
We would select the mixture probability and the parameters for the two exponential distributions carefully such that it has zero mean and $c_s$ SCV.
For the M/GI/1 model, we consider a linear staffing cost function: $c(\mu) = c_0 \cdot \mu$.
In our experiments, we fix $c_0=1$ and $\mathcal{B} = [3.5, 7] \times [6.5, 10]$ but vary $k$ to make $c_s$ range from $0.1$ to $3$ equidistantly, with a total of 200 different $c_s$'s.
The optimal solutions $(\mu^\star, p^\star)$ and traffic intensity $\rho^*\equiv \lambda(p^*)/\mu^*$ corresponding to different values of $c_s$, together with the last-iterate performance of the stream SGD algorithm, are reported in 
Noting that M/M/1 corresponds to the case where $c_s=1$, we also examine two specific scenarios where $c_s \in \{0.5, 1.5\}$ for further case study.


\paragraph{GI/M/1 Queue} 
 For the queueing model with general inter-arrival time distributions, we consider an $E_2$/M/1 queue – a system characterized by Erlang-2 interarrival times and exponential service times. The main purpose for us to use this setting is that efficient numerical methods are available for computing the true value of the optimal $(\mu^*,p^*)$ and objective function, employing the matrix geometric approach. In experiments for this GI/M/1 queue, we use linear staffing cost with $c_0=1$ and $\mathcal{B} = [3.5, 7] \times [6.5, 10]$, which yields the optimal decision $(\mu^\star, p^\star) = (3.762, 7.935)$.

\paragraph{Convergence and regret results}

Figures \ref{fig:conv_MG1_0.5}, \ref{fig:conv_MG1_1.5}, and \ref{fig:conv_GM1} present the convergence and regret results for the queueing systems introduced above. The results are consistent with the observations in Section \ref{sec:convergence-MM1}. Using a single sample to compute gradients, our method outperforms other batched SGD methods.

\begin{figure}[t!]
    \centering
    \includegraphics[width=\textwidth]{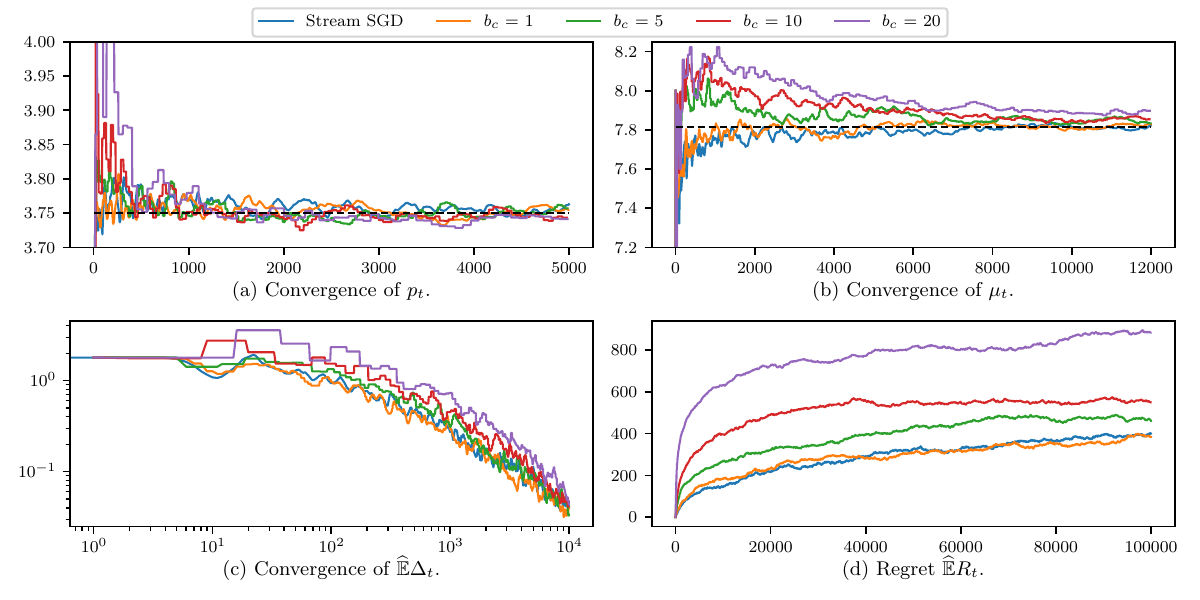}
    \caption{Convergence and regret in the M/GI/1 setting with $c_s=0.5$ over 200 repetitions.}  
    \label{fig:conv_MG1_0.5}
\end{figure}

\begin{figure}[t!]
    \centering
    \includegraphics[width=\textwidth]{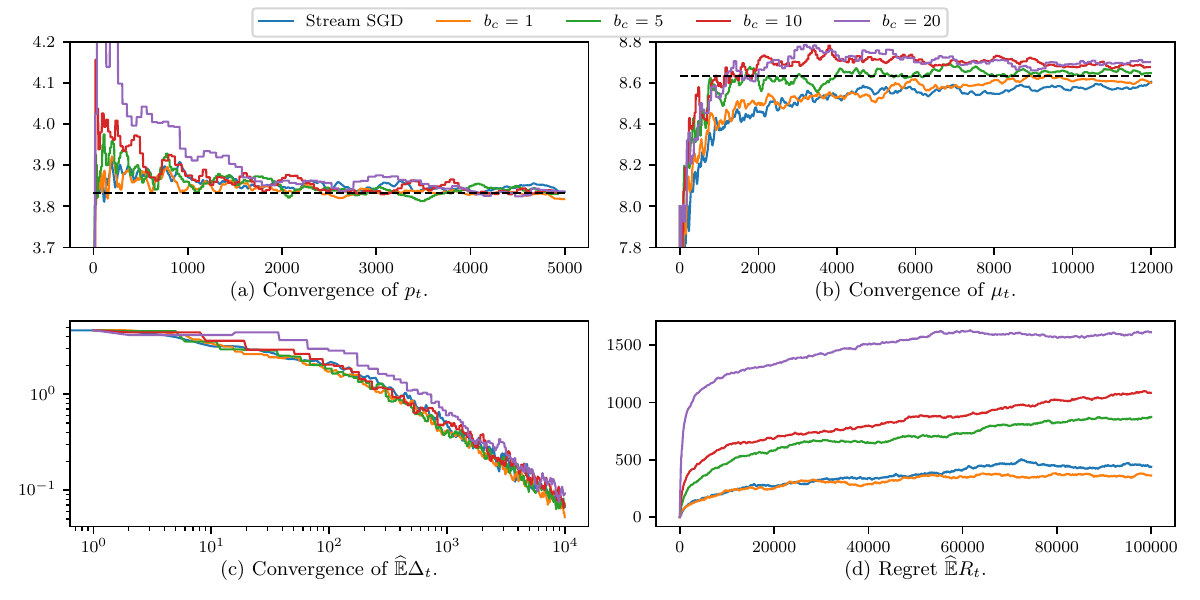}
    \caption{Convergence and regret in the M/GI/1 setting with $c_s=1.5$ over 200 repetitions.}  
    \label{fig:conv_MG1_1.5}
\end{figure}

\begin{figure}[t!]
    \centering
    \includegraphics[width=\textwidth]{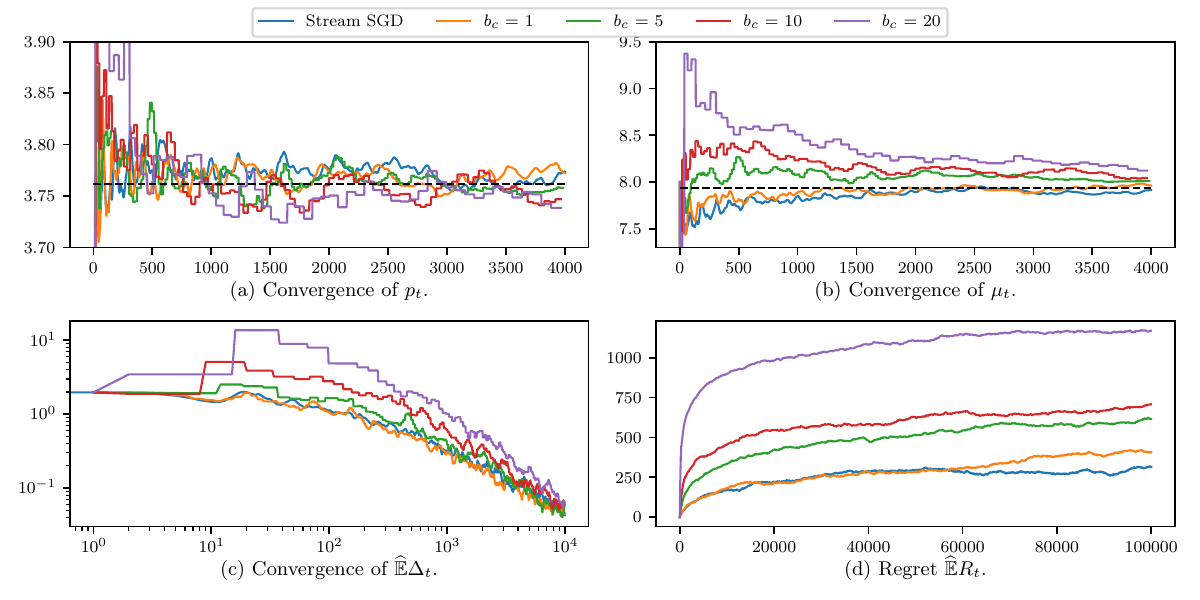}
    \caption{Convergence and regret in the GI/M/1 setting over 200 repetitions.}  
    \label{fig:conv_GM1}
\end{figure}

\paragraph{Robustness check on step sizes}



Although our theoretical framework assumes the same step size $\eta_t$ for all coordinates, we experiment with using different step sizes for each coordinate. This discrepancy does not create significant theoretical challenges. A practical amendment is to replace the standard Euclidean norm with a weighted norm.\footnote{Specifically, for any set of positive weights $w_i > 0$ for $i \in [d]$, if we wish to use separate step sizes $\eta_{t, i} = w_i \eta_{t, 0}$ for the $i$-th coordinate, we can redefine the inner product as $\langle x, y \rangle_w = \sum_{i=1}^d x_i y_i w_i$. Under these assumptions, the theoretical guarantees remain valid as long as the step sizes $\{\eta_{t, i}\}_{i \in [d]}$ converge to zero uniformly at the same rate. This adaptation preserves the key results of our theoretical framework with minimal modifications to the proofs.} 
To test the robustness of step size choices numerically, we allow different yet proportionate step sizes for each coordinate. Specifically, we implement stream SGD using step sizes of the form $\eta_t^p = \frac{c_p}{1+t}$ for updating $p$ and $\eta_t^\mu = \frac{c_\mu}{1+t}$ for $\mu$, with $(c_p, c_\mu) \in \{(1,1), (1,5), (5,1), (5,5), (1,10)\}$. 

Full results of the robustness check on step size are provided in Figure \ref{fig:step_size}.
In the first row, we plot the logarithm of the $L_2$ difference, i.e., $\log \widehat{\mathbb{E}} \Delta_t$, against the logarithm of the number of data samples. The results show that all  curves align well with a 45-degree line, confirming that the $\gO(t^{-1})$ convergence rate holds across all queueing settings and step size combinations. In the second row, we plot the cumulative regret, i.e., ${\widehat{\mathbb{E}}R_t}$, against the logarithm of the number of data samples. Consistent with the theoretical $\gO(\log(t))$ regret bounds, the regret curves become linear in the later iterations, albeit with varying slopes. These results demonstrate the robustness of stream SGD to step size variations across dimensions.

 An interesting observation is that the step size combinations $(1,5)$ and $(1,10)$ slightly outperform the equal step size combination $(1,1)$. A possible explanation is that the objective function is flatter along the staffing level dimension $\mu$ than along the price dimension $p$ (as illustrated in Figure 6 of \cite{chen2023online}). Using a larger step size for $\mu$ may help accelerate convergence by better aligning with the shape of the objective function.

\begin{figure}[t!]
   \centering
   \includegraphics[width=\textwidth]{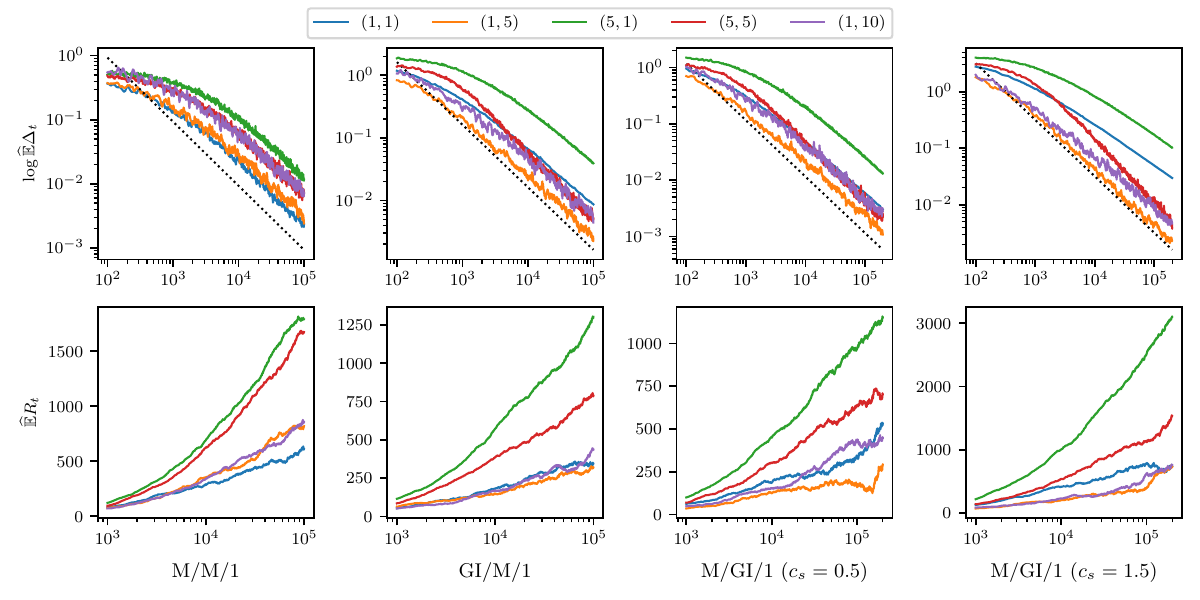}
   \caption{Convergence and regrets with different step size parameters $(c_p, c_{\mu})$ in the four settings over 200 repetitions.}  
  \label{fig:step_size}
\end{figure}

\paragraph{Online inference performance}

To ensure robust inference performance, the step sizes are customized for different models. Recall that the parameter $p$ is updated using $\eta_{t,p} = \frac{c_p}{(1+t)^{\alpha}}$, and $\mu$ is updated using $\eta_{t,\mu} = \frac{c_{\mu}}{(1+t)^{\alpha}}$. We set $c_p = 1$, $\alpha = 0.99$, and $c_{\mu} = 10$ for the M/M/1 model, while $c_{\mu} = 20$ is used for the other three settings. For the M/GI/1 case, the number of iterations is set to $T = 2 \times 10^6$, while for the other three cases, $T$ is set to $5 \times 10^4$.
Figure \ref{fig:inf_more} presents the online inference results for the additional models. Across all settings, our inference method consistently demonstrates strong and stable performance, aligning with the observations in Section \ref{sec:exp-online-inference}.
Finally, Table \ref{table:empirical_coverage_full} provides the full results corresponding to Table \ref{table:empirical_coverage}. Once again, our inference method outperforms the considered batch-mean methods.

\begin{figure}[t!]
    \centering
\includegraphics[width=\textwidth]{figs/inf_result_MM1_0.99_h.pdf}
\includegraphics[width=\textwidth]{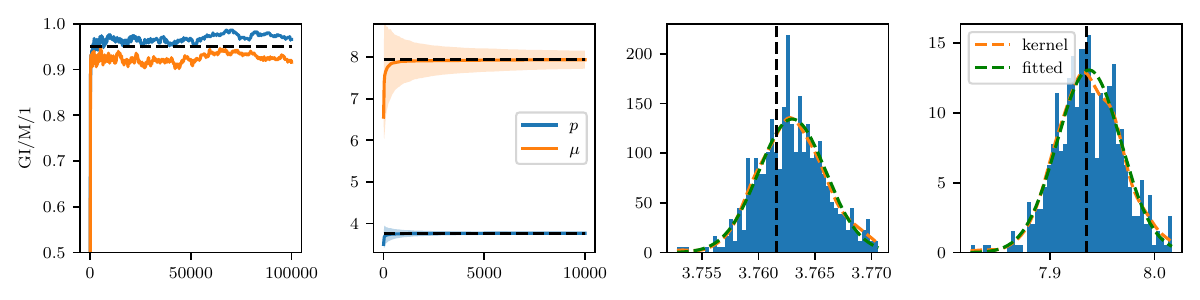}
\includegraphics[width=\textwidth]{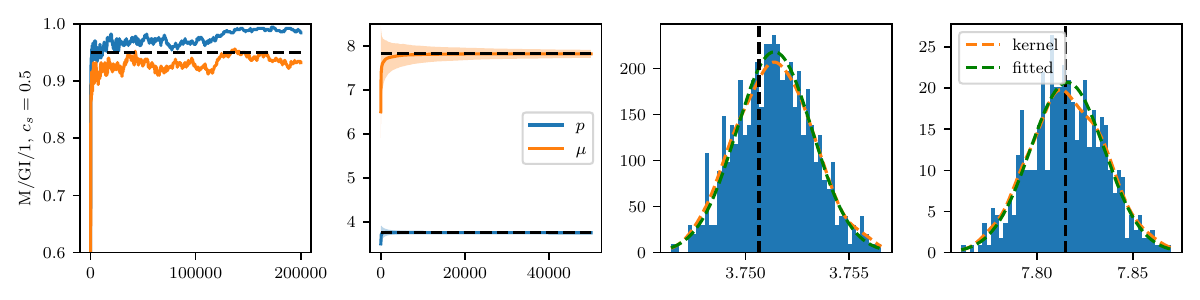}
\includegraphics[width=\textwidth]{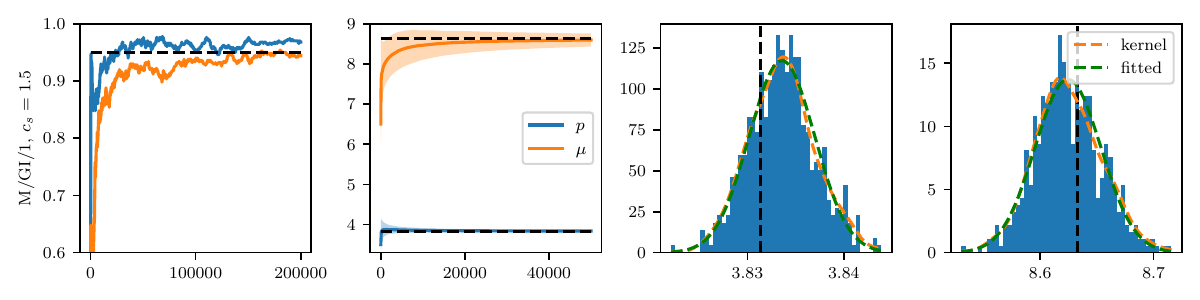}
\caption{Online inference for the full four settings over 500 repetitions.}  
   \label{fig:inf_more}
\end{figure}

\begin{table}[t!]
\caption{Comparison of empirical coverage rates between our method and the batch mean method on the four queueing settings.
The standard errors of coverage rates $\widehat{p}$ are reported in the bracket (computed via $\sqrt{\widehat{p}(1-\widehat{p})/500} \times 100 \%$). 
} 
\centering 
\small
\begin{tabular}{c|c|c|ccccc} 
\toprule
Para.&
Settings&\multicolumn{1}{c|}{Methods} &  $n=T/2^4$ & $n=T/2^3$ & $n=T/2^2$ & $n=T/2$ & $n=T$\\
\midrule \multirow{12}{*}{$p$}&
\multirow{3}{*}{M/M/1}
& Ours & 89.2(1.39)&90.4(1.32)&93.2(1.13)&92.4(1.19)&94.8(0.99)\\ 
&& ES& 80.6(1.77)&86.2(1.54)&90.2(1.33)&91.6(1.24)&92.4(1.19)\\ 
&& IBS& 84.8(1.61)&92.4(1.19)&94.0(1.06)&95.6(0.92)&96.4(0.83)\\ 
\cmidrule{2-8}&
\multirow{3}{*}{\makecell{M/GI/1 \\ $c_s=0.5$}}
& Ours & 97.0(0.76)&96.6(0.81)&95.8(0.9)&97.4(0.71)&96.8(0.79)\\ 
&& ES&96.2(0.86)&97.8(0.66)&96.8(0.79)&97.4(0.71)&97.0(0.76)\\ 
&& IBS& 96.4(0.83)&96.2(0.86)&97.2(0.74)&98.2(0.59)&98.4(0.56)\\ 
\cmidrule{2-8}&
\multirow{3}{*}{\makecell{M/GI/1\\ $c_s=1.5$}}
& Ours & 91.8(1.23)&94.6(1.01)&95.4(0.94)&96.2(0.86)&96.4(0.83)\\ 
&& ES& 88.0(1.45)&92.2(1.2)&94.4(1.03)&96.2(0.86)&97.0(0.76)\\ 
&& IBS& 88.8(1.41)&92.2(1.2)&94.4(1.03)&97.2(0.74)&97.8(0.66)\\ 
\cmidrule{2-8}&
\multirow{3}{*}{\makecell{GI/M/1}}
& Ours & 97.0(0.76)&96.6(0.81)&95.8(0.9)&97.4(0.71)&96.8(0.79)\\ 
&& ES& 96.2(0.86)&97.8(0.66)&96.8(0.79)&97.4(0.71)&97.0(0.76)\\ 
&& IBS&96.4(0.83)&96.2(0.86)&97.2(0.74)&98.2(0.59)&98.4(0.56)\\ 
\midrule
\multirow{12}{*}{$\mu$}&
\multirow{3}{*}{M/M/1}
& Ours & 97.0(0.76)&97.0(0.76)&95.8(0.9)&94.0(1.06)&94.4(1.03)\\ 
&& ES& 97.8(0.66)&97.4(0.71)&96.6(0.81)&93.8(1.08)&93.6(1.09)\\ 
&& IBS& 95.4(0.94)&94.6(1.01)&90.8(1.29)&89.2(1.39)&90.4(1.32)\\ 
\cmidrule{2-8}&
\multirow{3}{*}{\makecell{M/GI/1 \\ $c_s=0.5$}}
& Ours & 91.6(1.24)&91.2(1.27)&92.0(1.21)&93.0(1.14)&93.6(1.09)\\ 
&& ES&90.4(1.32)&89.6(1.37)&90.0(1.34)&91.2(1.27)&90.4(1.32)\\ 
&& IBS& 86.6(1.52)&85.6(1.57)&88.6(1.42)&91.0(1.28)&93.2(1.13)\\ 
\cmidrule{2-8}&
\multirow{3}{*}{\makecell{M/GI/1 \\ $c_s=1.5$}}
& Ours &85.8(1.56)&87.2(1.49)&90.4(1.32)&91.2(1.27)&95.2(0.96)\\ 
&& ES& 76.0(1.91)&83.8(1.65)&85.0(1.6)&86.8(1.51)&89.6(1.37)\\ 
&& IBS& 77.2(1.88)&89.0(1.4)&94.2(1.05)&97.6(0.68)&98.6(0.53)\\ 
\cmidrule{2-8}&
\multirow{3}{*}{\makecell{GI/M/1}}
& Ours & 91.2(1.27)&91.8(1.23)&91.6(1.24)&91.6(1.24)&92.0(1.21)\\ 
&& ES& 91.8(1.23)&91.8(1.23)&90.0(1.34)&89.6(1.37)&90.2(1.33)\\ 
&& IBS& 86.0(1.55)&87.6(1.47)&90.0(1.34)&88.8(1.41)&91.0(1.28)\\ 
\bottomrule
\end{tabular}
\label{table:empirical_coverage_full} 
\end{table}

%% file: Journal_version/additional_results_convergence_rate.tex

{We conducted a series of experiments to further examine the actual convergence rate across a variety of model settings. The basic setup follows the configuration in Section \ref{sec: numermical} and \ref{append:more-queue}, with variations in the initial points $(\mu_0, p_0)$ and in the interarrival and service time distributions. The results are shown in Figure~\ref{fig:new-figure}. The black dotted line represents the fitted line using errors from $10^3$ to $10^5$ steps, with the slope fixed at $-1$. Notably, the fitted line aligns closely with the empirical decay, providing strong evidence that the asymptotic convergence rate is indeed $O(1/t)$.}

\begin{figure}[th]
    \centering
    \includegraphics[width=0.245\textwidth]{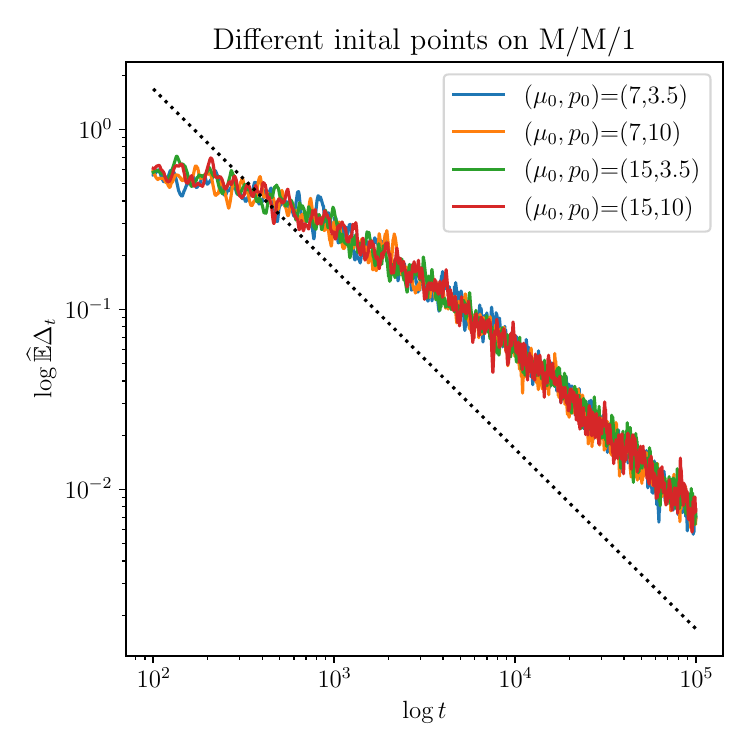}
    \includegraphics[width=0.245\textwidth]{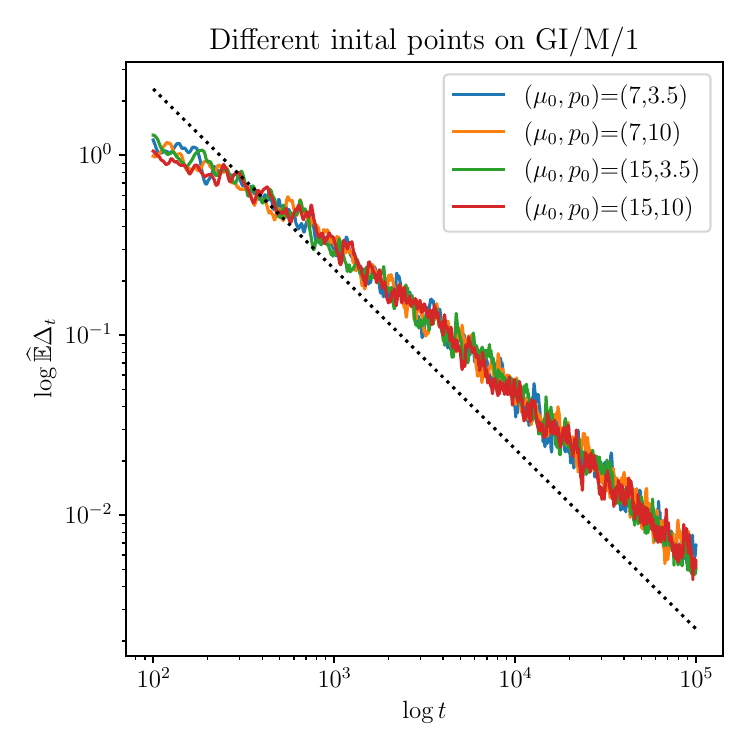} 
    \includegraphics[width=0.245\textwidth]{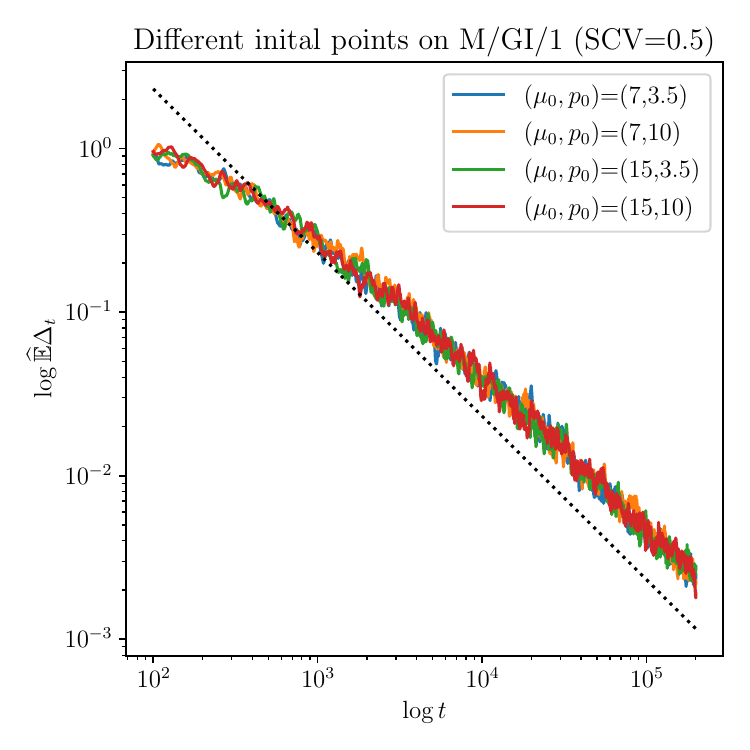}
    \includegraphics[width=0.245\textwidth]{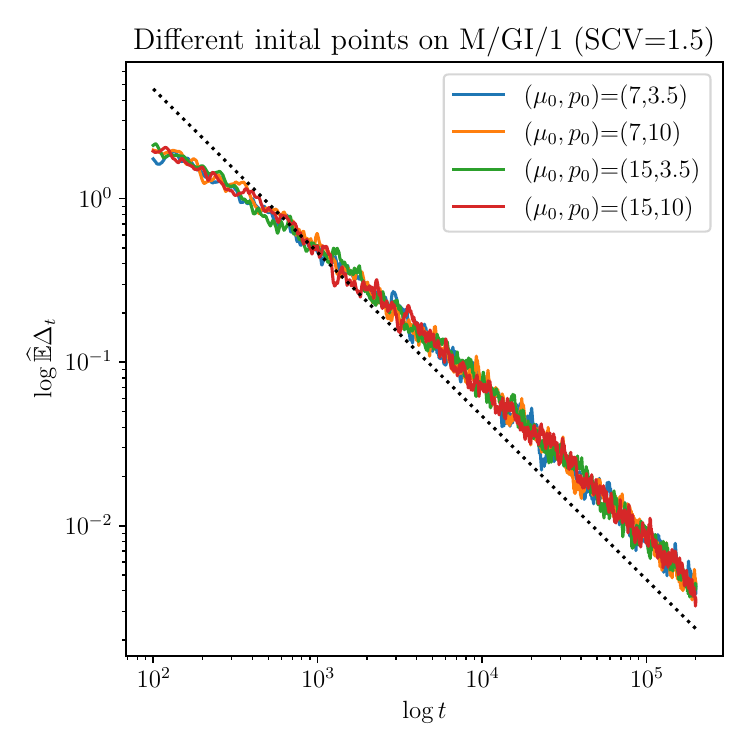} \\
    \caption{{The same asymptotic $O(1/t)$ convergence rate on different initial points and different queue models. The setup is the same as our paper.}}
   \label{fig:new-figure}
\end{figure}

{Next, we focus on M/M/1 queue models but vary the demand function $\lambda(p)$. Recall that the loss function has a closed form on the M/M/1 model:
\begin{equation}
\tag{\ref{eqObjmin}}
\min_{(\mu, p) \in \mathcal{B}} f(\mu, p) \equiv h_0 \frac{\lambda(p)}{\mu - \lambda(p)} + c \cdot \mu - p\lambda(p), \qquad \mathcal{B} \equiv [\underline{\mu}, \bar{\mu}] \times [\underline{p}, \bar{p}],
\end{equation} 
where we set $h_0=c=1$ in this case study.
We test three different demand functions and compute the optimal point $(\mu^\star, p^\star)$ numerically. Note that we slightly enlarge the feasible domain $\mathcal{B}$ such that the initial point could be reasonably far from the optimum while the uniform stability condition holds.
\begin{itemize}
    \item \textbf{Logistic demand function}: $\lambda(p) = M \cdot \frac{\exp(a-p)}{1 + \exp(a-p)}$ with $M=10$ and $a=4.1$.  The optimal solution is $(\mu^\star, p^\star)=(8.184, 3.785)$ and the domain is $[6.56, 15] \times [3.5, 10]$.
    \item \textbf{Linear demand function}: $\lambda(p) = M-b\cdot p$ with $M=10$ and $b=1$.  The optimal solution is $(\mu^\star, p^\star)=(6.321, 5.742)$ and the domain is $[6.1, 15] \times [4, 10]$.
    \item \textbf{Quadratic demand function}: $\lambda(p) = M -p^2/2$ with $M=10$. The optimal solution is $(\mu^\star, p^\star)=(7.449, 3.106)$ and the domain is $[6.871, 15] \times[2.5, 10]$.
\end{itemize}
As shown in Figure \ref{fig:new-demand}, across different initial points and demand functions, the asymptotic convergence rate remains $O(1/t)$, with the black dotted line closely aligning with the tail segments of the empirical error decay curves.}
\begin{figure}[t!]
    \centering
    \includegraphics[width=0.32\textwidth]{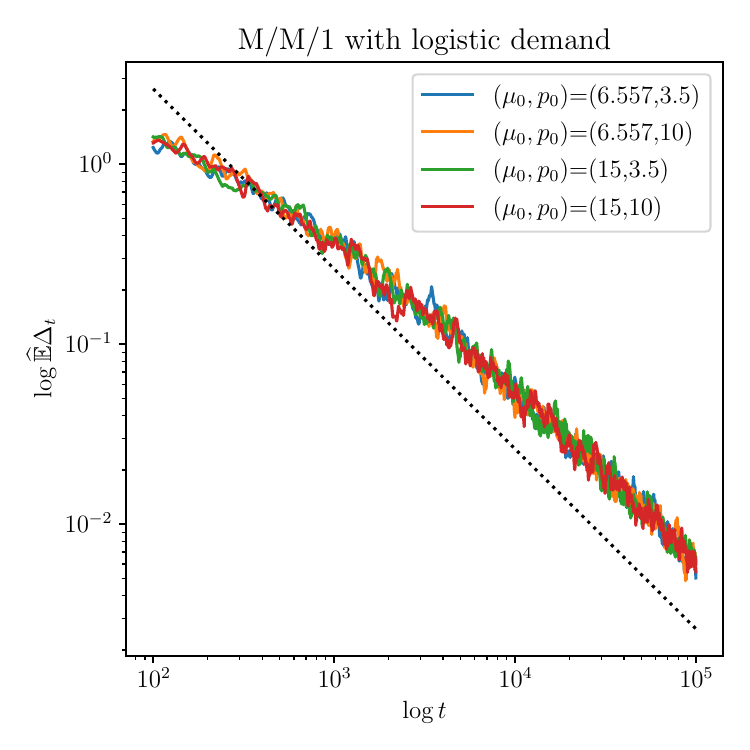}
    \includegraphics[width=0.32\textwidth]{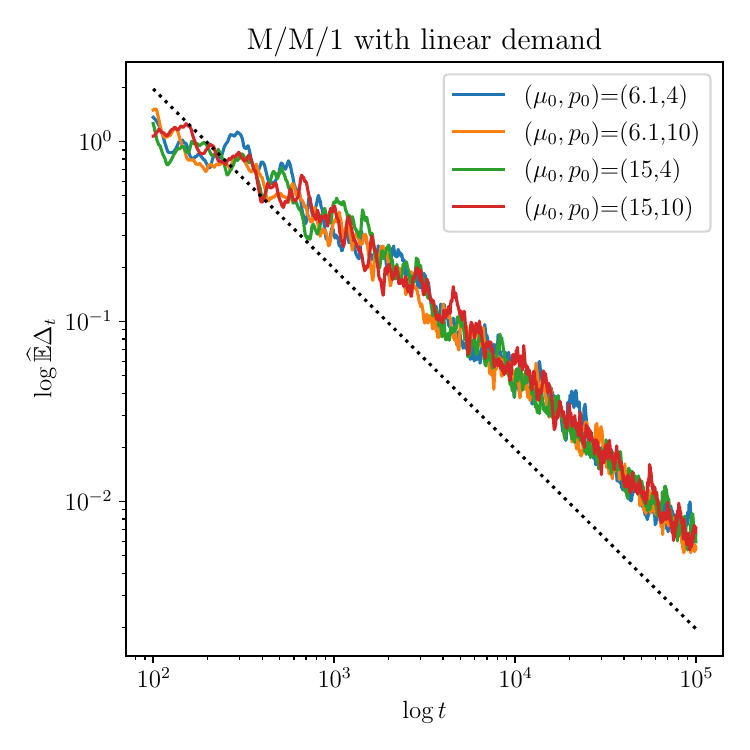} 
    \includegraphics[width=0.32\textwidth]{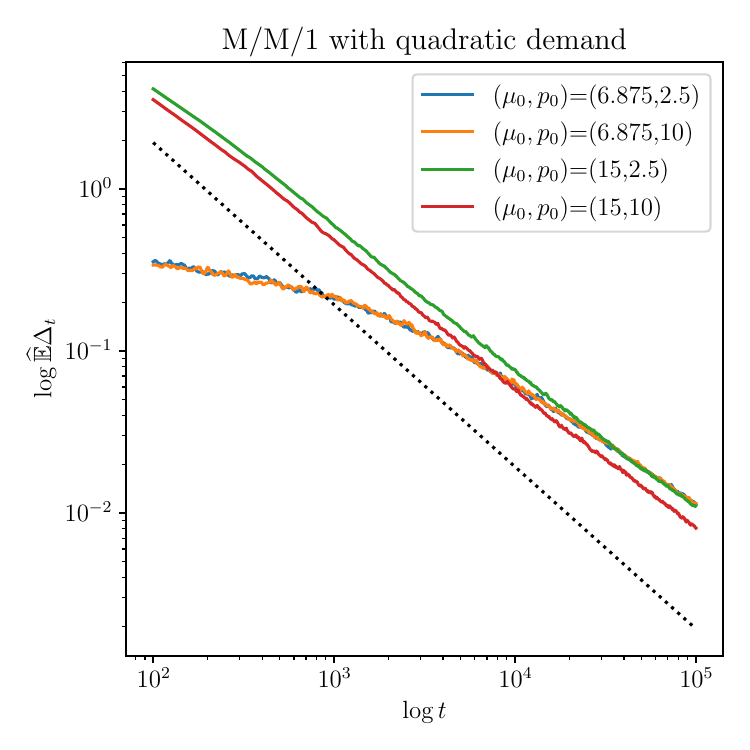}
    \caption{{The same asymptotic $O(1/t)$ convergence rate using different demand function $\lambda$ on the M/M/1 model. }}  
   \label{fig:new-demand}
\end{figure}